\documentclass{amsart}

\usepackage[utf8]{inputenc}
\usepackage{manfnt}
\usepackage{xcolor}
\usepackage[foot]{amsaddr}
\usepackage{bbm}
\usepackage{enumerate}
\usepackage[inline]{enumitem}
\usepackage{amssymb,latexsym,amsmath}
\usepackage[all,v2,cmtip,2cell]{xy}
\UseAllTwocells
\xyoption{v2}

\usepackage{hyperref}
\usepackage{cite}
\usepackage{graphicx}
\newcommand{\sboxtimes}{\, {\mathop{\scriptstyle \boxtimes}}\,}
\newcommand{\wsboxtimes}{\mathop{\widetilde{\sboxtimes}}}
\newcommand{\wtimes}{\mathop{\widetilde{\times}}}
\newcommand{\uwtimes}{\mathbb I_{\wtimes}}
\newcommand{\uAtimes}{\mathbb I_{\times_{_A}}}

\newcommand{\uwsboxtimes}{\mathbb I_{\wsboxtimes}}

\newcommand{\uAotimes}{\mathbb I_{\otimes_A}}

\newcommand\skyk{\ensuremath{\operatorname{skysc}_0(\Bbbk)}}

\newcommand{\sO}{{\mathcal O}}

\newcommand{\sXYmod}{(X \times_S Y)\operatorname{-mod}}

\newcommand{\sAAkmod}{(A \times_\Bbbk  A)\operatorname{-mod}}
\newcommand{\sAalg}{A\operatorname{-alg}}
\newcommand{\hsAalg}{A_{\mathrm {h}}\operatorname{-alg}}
\newcommand{\hsAalggr}{A_{\mathrm {h,gr}}\operatorname{-alg}}
\newcommand{\sSalg}{S\operatorname{-alg}}

\newcommand{\sXmod}{X\operatorname{-mod}}
\newcommand{\sTmod}{T\operatorname{-mod}}
\newcommand{\sYmod}{Y\operatorname{-mod}}
\newcommand{\sATmod}{{A_T\operatorname{-mod}}}
\newcommand{\sAmod}{{A\operatorname{-mod}}}
\newcommand{\sAmodgr}{{A_{\operatorname{gr}}\operatorname{-mod}}}
\newcommand{\sAmodz}{{A_{0}\operatorname{-mod}}}
\newcommand{\sAalggr}{{A_{\operatorname{gr}}\operatorname{-alg}}}
\newcommand{\sAalgz}{{A_{0}\operatorname{-alg}}}
\newcommand{\hsAmod}{{A_{\mathrm h}\operatorname{-mod}}}
\newcommand{\hsAmodgr}{{A_{\mathrm{h,gr}}\operatorname{-mod}}}

\newcommand{\fsAalg}{{A_{\mathrm{f}}\operatorname{-alg}}}

\newcommand{\AQcoalg}[1]{\Comod{#1}_{QA\operatorname{-alg}}}

\newcommand{\AQcoalgsf}[1]{\Comodsf{#1}_{QA\operatorname{-alg}}}

\newcommand{\RQAcoalgsf}[1]{\Rcomod{#1}_{QA\operatorname{-alg}}}

\newcommand{\RatAcomodO}[1]{\Comod{#1}_{rA\operatorname{-mod}}}
\newcommand{\RatAcomodgr}[1]{\Comod{#1}_{rA_{\operatorname{gr}}\operatorname{-mod}}}

\newcommand{\Acomodgr}[1]{\Comod{#1}_{A_{\operatorname{gr}}\operatorname{-mod}}}
\newcommand{\AQcomod}[1]{\Comod{#1}_{QA\operatorname{-mod}}}
\newcommand{\AQpcomod}[1]{\Comod{#1}_{QA_p\operatorname{-mod}}}

\newcommand{\Acomodsf}[1]{\Comodsf{#1}_{A\operatorname{-mod}}}

\newcommand{\sSmod}{S\operatorname{-mod}}

\newcommand\st{\operatorname{st}}

\newcommand\End{\operatorname{End}}
\newcommand\EndO{\operatorname{End_{0}}}
\newcommand\Endgr{\operatorname{End}_{\operatorname{gr}}}

\newcommand\Hom{\operatorname{Hom}}
\newcommand\HomO{\operatorname{Hom_0}}
\newcommand\Homgr{\operatorname{Hom_{gr}}}

\newcommand\aff{\operatorname{aff}}

\newcommand\red{\operatorname{red}}
\newcommand\ant{\operatorname{ant}}

\newcommand\Ker{\operatorname{Ker}}

\newcommand{\Rep}{\operatorname{Rep}}
\newcommand{\RepO}{\operatorname{Rep}_0}
\newcommand{\Sets}{\operatorname{Sets}}

\newcommand\Spec{\operatorname{Spec}}
\newcommand\Speck{\operatorname{Spec}(\Bbbk)}

\newcommand\Sch{\operatorname{Sch}}
\newcommand\Schk{\operatorname{Sch}|\Bbbk}

\newcommand\Schs{\operatorname{Sch}|S}

\newcommand\GextaffA{\ensuremath{\operatorname{GE}|_{_{\operatorname {aff}}}A}}
\newcommand\GextqcA{\ensuremath{\operatorname{GE}|_{_{\operatorname {qc}}}A}}

\newcommand{\MmoraffA}{\operatorname{MM}|_{_{\aff}}A }
\newcommand{\GmoraffA}{\operatorname{GM}|_{_{\aff}}A }
\newcommand{\GmorqcA}{\operatorname{GM}|_{_{\operatorname{qc}}}A }

\newcommand\Schaqc{\ensuremath{\operatorname{Sch}|_{_{\operatorname{qc}}}A}}
\newcommand\Schasqc{\ensuremath{\operatorname{Sch}|_{_{\operatorname{sqc}}}A}}
\newcommand\Schsqc{\ensuremath{\operatorname{Sch}|_{_{\operatorname{sqc}}}S}}

\newcommand\Schpqc{\ensuremath{\operatorname{Sch}|_{_{\operatorname{pqc}}}S}}
\newcommand\Schapsqc{\ensuremath{\operatorname{Sch}|_{_{\operatorname{psqc}}}A}}
\newcommand\Schpsqc{\ensuremath{\operatorname{Sch}|_{_{\operatorname{psqc}}}S}}

\newcommand\Schfpqc{\ensuremath{\operatorname{Sch}|_{_{\operatorname{fpqc}}}S}}

\newcommand\Schfpsqc{\ensuremath{\operatorname{Sch}|_{_{\operatorname{fpsqc}}}S}}

\newcommand\Schkaff{\operatorname{Sch}|_{_{\operatorname{aff}}}\Bbbk}

\newcommand\Schaaff{\operatorname{Sch}|_{_{\operatorname{aff}}}A}
\newcommand\Schaff{\operatorname{Sch}|_{_{\operatorname{aff}}}S}

\newcommand\Schkqc{\operatorname{Sch}|_{_{\operatorname{qc}}}\Bbbk}

\newcommand\Schqc{\operatorname{Sch}|_{_{\operatorname{qc}}}S}

\newcommand{\Lmodmon}[1]{{}_{#1}{\mathsf M}}
\newcommand{\Rmodmon}[1]{\mathsf M_{#1}}

\newcommand{\Rcomod}[1]{\mathsf M^{#1}}
\newcommand{\Lcomod}[1]{{}^{#1}\mathsf M}

\newcommand{\Comodsf}[1]{{}^{#1}\mathsf M}

\newcommand{\Comod}[1]{{}^{#1}\mathsf M}

\newcommand{\Comodfin}[1]{{}^{#1}\mathsf M_{\operatorname{fin}}}

\newcommand{\ShRepO}[1]{\operatorname{{#1}-}QA_{\operatorname{p}}\operatorname{-mod}} 

\newcommand{\ShRepgr}[1]{\operatorname{{#1}-}QA_{\operatorname{p,gr}}\operatorname{-mod}} 
\newcommand{\ShRepgrfin}[1]{\operatorname{{#1}-}C_{\operatorname{lf}}A_{\operatorname{gr}}\operatorname{-mod}} 
\newcommand{\ShRepOfin}[1]{\operatorname{{#1}-}C_{\operatorname{lf}}A\operatorname{-mod}} 
\newcommand{\HBA}{\operatorname{HVB_{gr}}(A)}

\newcommand{\HVB}{\operatorname{HVB}}
\newcommand{\VB}{\operatorname{VB}}
\newcommand{\VBG}{\operatorname{VB}_{\operatorname{gr}}}
\newcommand{\HVBG}{\operatorname{HVB}_{\operatorname{gr}}}

\newcommand{\SimA}{\operatorname{SA-alg}}

\newcommand{\omegagr}{{\omega_{{\operatorname{gr}}}}}
\newcommand{\omegaO}{{\omega_{0}}}
\newcommand{\Autw}{\operatorname{Aut}^\otimes(\omega)}
\newcommand{\AutwO}{\operatorname{Aut}^\otimes(\omega_0)}
\newcommand{\AutwgrO}{\operatorname{Aut}^\otimes_0(\omegagr)}
\newcommand{\Autwgr}{\operatorname{Aut}^\otimes(\omegagr)}

\newcommand{\Endwgr}{\operatorname{End}^{\otimes}(\omegagr)}
\newcommand{\EndwO}{\operatorname{End}^{\otimes}(\omega_0)}
\newcommand{\EndwgrO}{\operatorname{End}^{\otimes}_0(\omegagr)}
\newcommand{\Autwe}{\operatorname{Aut}^\otimes\bigl(\omegagr|_{_{\Rep(\mathcal
      S)_E}}\bigr)}
 \newcommand{\AutwgrOe}{\operatorname{Aut}^\otimes_0\bigl(\omegagr|_{_{\Rep(\mathcal S)_E}}\bigr)}
\newcommand{\AutweO}{\operatorname{Aut}^\otimes\bigl(\omega_0|_{_{\Rep(\mathcal S)_E}}\bigr)}
\newcommand\Aut{\operatorname{Aut}}
\newcommand\Autgr{\operatorname{Aut}_{\operatorname{gr}}}
\newcommand\AutO{\operatorname{Aut}_{\operatorname{0}}}
\newcommand\Authbext{\operatorname{{\mathit{Aut}}}_{\operatorname{gr}}}

\newcommand{\ate}{{\xymatrixcolsep{1.5pc} \xymatrix{1\ar[r]&H\ar[r]&G\ar[r]^(.45){q}&\ar[r] A&0}}}

\newcommand\PP{\operatorname{\mathcal P}}

\newcommand\VBSpec{\operatorname{\mathbb V\mathbb B}}
\newcommand\VBSpecgr{\operatorname{\mathbb V\mathbb B}_{\operatorname{gr}}}

 \newcommand\id{\operatorname{id}}
\newcommand\smallcirc{{\scriptstyle{\circ}\,}}

 \newcommand\op{\operatorname{op}}

 \newcommand{\Alb}{\operatorname{\mathcal A{\it lb}}}

\newcommand\im{\operatorname{Im}}

\theoremstyle{plain}
\newtheorem{thm}{Theorem}
\newtheorem*{thmsans}{Theorem}
\newtheorem{lem}[thm]{Lemma}
\newtheorem{cor}[thm]{Corollary}
\newtheorem{prop}[thm]{Proposition}

\numberwithin{thm}{section}
\numberwithin{equation}{section}
\def\pf{\noindent{\sc Proof. }}

\theoremstyle{definition}

\newtheorem{defn}[thm]{Definition}

\newtheorem{ej}[thm]{Example}
\newtheorem{ejs}[thm]{Examples}
\newtheorem{rem}[thm]{Remark}
\newtheorem{nota}[thm]{Notation}
\newtheorem{summary}[thm]{Summary}

\begin{document}

\title[A representation theory for quasi-compact group schemes]{Quasi-compact group schemes, Hopf sheaves, and their representations}

\author{Alvaro
  Rittatore}
\address{Alvaro
  Rittatore:
   \scriptsize {\rm Centro de Matemática, Facultad de Ciencias,  Universidad de la Rep\'ublica,
   Igu\'a 4225,  11400 Montevideo,  Uruguay.} \emph{e-mail:} {\tt alvaro@cmat.edu.uy}}

\author{Pedro Luis del \'Angel}
\address{Pedro Luis del \'Angel: 
\scriptsize  {\rm  CIMAT,  Jalisco S/N
    Mineral de Valenciana,
      3624 Guanajuato, Guanajuato, M\'exico.}, 
\emph{e-mail:} {\tt luis@cimat.mx}}
 
\author{Walter Ferrer Santos}
\address{Walter Ferrer Santos: 
  \scriptsize  {\rm Departamento de
matemática y aplicaciones,  Centro Universitario Regional Este, Tacuaremb\'o s/n, 2000 Maldonado, Uruguay},  \emph{e-mail:} {\tt wrferrer@cure.edu.uy}}



\begin{abstract}
  We explore the notion of  \emph{representation of an affine
  extension of an abelian variety} --- such an extension is a faithfully flat affine morphism of
  $\Bbbk$--group schemes $q:G\to A$, where  $A$ is an abelian variety.  We characterize
  the categories that arise as the category of representations of  an affine extension
  $q:G\to A$, generalizing the classical results of Tannaka Duality
  established for affine $\Bbbk$--group schemes (that is, when $A=\Speck$). We
  also prove the existence of a contravariant equivalence between
  the category of affine extensions of a given $A$ and the category of
  \emph{faithful commutative Hopf sheaves} on $A$, generalizing in this
  manner the well-known $\op$--equivalence between affine group
  schemes and commutative Hopf algebras.  If  $\mathcal H_q$ is the
  Hopf sheaf on $A$ associated to $ q$,
  the category of representations of $q$ is equivalent to the category
  of \emph{$\mathcal H_q$}--comodules.
\end{abstract}

\maketitle

\begin{small}
\end{small}

\tableofcontents

\section{Introduction}

Roughly speaking, given a certain family of objects $\mathcal R$ (the
``representable objects'') and a fixed basic monoidal category
$\mathcal C$, a ``representation theory'' consists in the association
to an element $r \in \mathcal R$, of a pair
$\bigl(\operatorname{Rep}(r),\, U:\operatorname{Rep}(r) \to \mathcal
C\bigr)$, where $\operatorname{Rep}(r)$ is a monoidal category --- the
\emph{category of representations} of $r$ ---, and $U$ a monoidal functor
(the \emph{forgetful functor}) --- eventually with certain additional
properties depending on the situation under consideration. One aspires
to ``reconstruct'' each $r \in \mathcal R$ in terms of the
corresponding pair $\bigl(\operatorname{Rep}(r), U\bigr)$, and also to
describe intrinsically all the pairs $(\mathcal D, U:\mathcal D \to
\mathcal C)$ that are equivalent to pairs of the above form for some
$r$.  For example, such a platform has been developed within the
following frameworks: categories of groups (abstract, topological,
Lie, affine algebraic, differential); of general algebras; of Lie
algebras, Hopf algebras. In this context the notion of \emph{Tannaka Duality}
is generally presented as an answer to the following questions (see
\cite{kn:joyalstreet}, \cite{kn:formaltan} or \cite{kn:streetbook} for
a precise formulation):

\emph{The Reconstruction Problem:}  can a representable object be
described in terms of its category of representations?

\emph{The Recognition Problem:} can  a category of
representations be described intrinsically?

It is worth mentioning that the theory
     of Tannaka Duality was generalized
     to a categorical context: the relevant concept of ``tannakian
     adjunction'' was developed and some of the classical results were
     generalized and clarified (see \cite{kn:streetbook} and
     \cite{kn:joyalstreet}).

In the case of affine algebraic group schemes over a field (even over
a commutative ring) the theory of \emph{rational representations} has
achieved a considerable degree or maturity and many of its main problems have
been solved and important advances have been done in the theory of its
actions on general schemes.  Examples of significant accomplishments
in the area are: the completion of the structure theory of reductive
affine group schemes (see \cite{kn:bible}), the development of the
geometric methods in invariant theory (see \cite{kn:GIT}), or more
generally the theory of transformation groups (see
\cite{kn:transgrp}).  In particular both the reconstruction and
recognition questions were answered positively: Saavedra first
presented a proof in \cite{kn:saav} which was later observed to have
some mistakes.  A correct proof of the result was produced afterwards
by Deligne and Milne in \cite{kn:delmil}; see also
\cite{kn:delignetannaka90}.

Temporarily, call $\Rep(G)$, the category of finite dimensional
rational representations of an affine group scheme; it is well known that $\Rep(G)$  is monoidal, 
rigid, abelian and $\Bbbk$-linear. Denote as $\omega:\Rep(G)\to
\operatorname{Vect}_\Bbbk$ the corresponding monoidal forgetful
  functor. In this context Saavedra-Deligne-Milne's result can be
stated as follows (see also \cite{kn:joyalstreet}, \cite{kn:formaltan}
or \cite{kn:streetbook} for other perspectives of the same problem):
 
\begin{thmsans}[Tannaka Duality for affine group schemes,
  {\cite[Prop.~2.8,  Thm.~2.11]{kn:delmil}}]
  
\noindent\\ \emph{(1) Reconstruction Theorem.} Let $G$ be an affine
  group scheme. Consider the pair $(\Rep(G),\omega)$ and the group
  scheme of tensor automorphisms of $\omega$, denoted as $\Autw$ --- see
  \cite[page 20]{kn:delmil} for a definition of this group and  compare with Definition
  \ref{defn:endotimes} below. Then  $\Autw$ is an affine  group scheme isomorphic to
  $G$.

  In particular, if $G$
  and $G'$ are affine group schemes 
  such that $(\Rep(G),\omega)$ and $(\Rep(G'),\omega')$ are equivalent
  as monoidal $\Bbbk$--linear categories with a forgetful
  functor, then $G$ and $G'$ are isomorphic group schemes.

\noindent \emph{(2) Recognition Theorem.} Let $\mathcal C$ be a
monoidal, abelian, rigid $\Bbbk$--linear category such that
$\Bbbk=\End(\mathbb I)$, together with an exact, faithful,
$\Bbbk$--linear monoidal functor $\omega: \mathcal C\to
\operatorname{Vect}_{f,\Bbbk}$. Then $(\mathcal C, \omega)$ is
equivalent (as a monoidal category with forgetful functor) to the
category of rational representations of the affine group scheme
$\Autw$.
\end{thmsans}

Let $G$ be a group scheme of finite type (see Definition
\ref{defn:groupscheme}); it is easy to see that the naive attempt to
define the category of representations of $G$ as a direct
generalization of the affine situation, yields a category which does
not fulfill our needs, as it is too small to determine $G$ --- for
example, if $G$ is an anti-affine group (see 
Definition \ref{defi:anti-aff}), then the only morphism of group
schemes  $G\to  \operatorname{GL}(V)$ is the trivial morphism.
 
Motivated by previous work of Brion, Rittatore and others on the
structure of group and monoid schemes (see for example
\cite{kn:ritbr}, \cite{kn:briantaff}, \cite{kn:brisamum},
\cite{kn:brionchev}) and on their actions (\cite{kn:bprit},
\cite{kn:brionSumihiro}), and taking into account the mentioned
obstruction to the naive approach, we propose a representation theory
\emph{not} for isolated group schemes, but for what we call
\emph{affine extensions} of   abelian varieties.   Roughly speaking, an affine extension
 is a generalization of the so-called Chevalley
decomposition for algebraic groups (see Theorem
\ref{thm:chevftgs}):  an affine extension $\mathcal S$ is an exact
sequence of group 
schemes $\ate$, where $A$ is an abelian variety and $H$ is an affine
group scheme --- equivalently, $q$ is an affine, faithfully flat
morphism of quasi-compact group schemes and $H=\Ker(q)$ (see
definitions \ref{defi:exactseq} and \ref{defi:affqcext}). It follows
that $q:G\to A$ is an $H$--torsor (see definitions
\ref{defi:torsor}--\ref{defi:affqcext}). A \emph{representation} for
$\mathcal S$ is built on an homogeneous vector bundle over $A$ ---
by homogeneous we mean that vector bundle $E\to A$ is such that $E_{\overline{\Bbbk}}\cong
t^*_aE_{\overline{\Bbbk}}$ for any translation $t_a:A\to A$ by a
geometric point $a\in A(\overline{\Bbbk})$ (see \cite{kn:brionrepbund}
and Definition \ref{defn:hmogvecbunbrion} for a more conceptual and
precise definition).

The basic nomenclature of the paper as well as the necessary
properties of the category of affine extensions of a fixed abelian
variety $A$ are presented in Chapter \ref{sect:actgrsch}. Therein, we
also present minor indispensable complements to some of the classical
results on the theory of group schemes.

In order to capture (i.e.~\emph{to reconstruct}) the complete structure of the affine extension
$\mathcal S$ with a representation theory supported on a category with
\emph{homogeneous vector bundles as objects}, we need to consider ``more
morphisms'' than the usual ones between vector bundles. This new
category $\HVBG(A)$ is an \emph{enriched category} over the monoidal
category $\bigl(\Schk, \times,
\{*\}=\Spec(\Bbbk)\bigr)$. Moreover, the (scheme of) morphisms
between two objects, denoted as  $\Homgr(E,E')$, is also a homogeneous vector 
bundle over $A$;  we call its structure morphism $d: \Homgr(E,E') \to
A$ the \emph{degree map}.  

The \emph{automorphism group} $\Autgr(E)$ of a general vector bundle
(see Definition \ref{defn:vbgraded}) is a smooth group scheme of
finite type, and the \emph{degree map} $d:\Autgr(E)\to A$ is an affine
morphism of group schemes, with kernel $\AutO(E)$ --- the group of
``classical'' automorphisms of $E$. This result is well known for
algebraically closed fields; more recently it was generalized for
arbitrary fields (see \cite[Lemma 2.8]{kn:brionrepbund}).  We say that
the bundle $E$ is homogeneous when the sequence
\[ 
  \xymatrix{\Authbext(E): & 1\ar[r] &
    \AutO(E)\ar[r]&\Autgr(E)\ar[r]^-d& A \ar[r]& 0,}\]
is exact  --- and hence $\Authbext(E)$ is an affine extension, see \ref{defn:hmogvecbunbrion} for a precise definition.

A \emph{representation} of an affine extension $\mathcal S$: $\ate$
or an \emph{$\mathcal S$--module}, is a morphism of group schemes
$\rho: G\to \Autgr(E)$, where $E$ is a homogeneous vector bundle over
$A$, such that $\rho$ induces the identity on
$A$ (see Definition \ref{defn:repaffext}). 
The \emph{category $\Rep(\mathcal S)$ of $\mathcal
  S$--modules} is the category enriched over $\Schk$  that has as objects the
representation of $\mathcal S$ and  as hom-objects the scheme (in fact
a homogeneous vector bundle) of
$G$--equivariant graded morphisms of homogeneous
vector bundles (see Definition \ref{defn:catrepaffext} and Lemma \ref{lem:Smorfarebundle}). 

If $G$ is an affine group scheme $G$ (i.e.~when $A=\Speck$) then the
above definition corresponds to the category of finite dimensional
rational $G$--modules and the $G$--equivariant morphisms.

The category of representations of an affine extension is monoidal
``in degree zero'', see Remark \ref{rem:hvbnotmonoidal} and Lemma
\ref{lem:tensordualhvb}; and this \emph{weaker} monoidality condition
becomes a decisive ingredient in the proof of the Tannaka Duality
Theorem in our context. We establish a) a version of Tannaka's
reconstruction Theorem \ref{thm:reconstruction}) proving that from the
category of representations of an extension $\mathcal S$ we can define
$\mathcal S$ itself; b) we prove that the mentioned reconstruction is
accomplished in the best possible way when using the category of
representations (Recognition Theorem \ref{thm:recognition}).

Moreover, as expected in view of the results and methods of the affine
case, in order to establish a version of the Tannaka Duality in our
context one should deal with affine extensions as  limits of
affine extensions of finite type (that is, such that the corresponding
groups $H$ and $G$ are group schemes of finite type); in Theorem
\ref{thm:affextisproalg} we prove that \emph{any} affine extension is
such a limit.  On the other hand, by a result of D.~Perrin, if $G$ is
a connected group scheme then $G$ is quasi-compact and it is a limit
of a family $\{ G_\alpha\}_{\alpha\in I}$ of group schemes of finite
type (see \cite[th\'eor\`emes II.2.4 and IV.3.2]{kn:dperrin1}).
Moreover, $G$ fits into an affine extension of an abelian
variety --- this result, stated without a complete proof in
\cite[Corollary V.4.3.1]{kn:dperrin1}, is proved in Corollary
\ref{cor:qcisate} below.

The op-equivalence between the category of affine group schemes and
commutative Hopf algebras --- that to a given  affine group
scheme $G$ associates the algebra of global section $\mathcal O_G(G)$ with
a structure of Hopf algebra induced by the multiplication and inversion morphisms in $G$ ---  is not only an important viewpoint but also a powerful tool in the study of the
representation theory of affine group schemes. Thus, once we have
constructed an adequate representation theory for affine extensions
--- in the sense that the representation theory satisfies a full
Tannaka Duality Theorem ---, we undertake the generalization of this
well-known equivalence to the context of affine extensions, by
developing the notion of \emph{Hopf sheaf over an abelian variety}.

As we are working with quasi-compact morphisms of group schemes $q: G \to A$, we consider
the categories $\Schaqc$ (of quasi-compact schemes over $A$), and
$Q\sAalg$ (of quasi-coherent sheaves of $\mathcal O_A$--algebras) and
the well known covariant equivalence associated to the functors $\PP :
\Schaqc\to Q\sAalg^{\op}$, $\PP(x:X\to A)=x_*(\mathcal O_X)$ and
$\Spec: Q\sAalg^{\op}\to \Schaqc$, where $\Spec(\mathcal F)$ is the
  affine scheme over $A$ associated to the sheaf of $\mathcal
  O_A$--algebras $\mathcal F$. In general these functors are adjoint
  (see Remark \ref{rem:affingen} and \cite[\S 1.2,\S 1.3]{kn:EGAII}),
  but they establish a covariant equivalence when restricted to the
  situation that the objects $x:X \to A$ are affine
  morphisms. Moreover, if $A=\Spec(\Bbbk)$, then we obtain the usual
  equivalence mentioned above.

In order to obtain the necessary generalization of the notion of Hopf
algebra to this context, one needs to observe that $\Schaqc$ and
$Q\sAalg$ are in fact duoidal categories (see Definition
\ref{defi:duoidalcat}) and that the adjoint functors $\PP $ and
$\Spec$ are (strong, colax) monoidal for the given monoidal structures (see
Theorem \ref{thm:mainspec} for a precise statement). The
construction of such duoidal structures is known in the setting of
slice categories, but we take an explicit approach in order to
identify the affine extensions as a certain type  of bimonoids in the
category (see  Lemma \ref{lem:duoschemes}, Proposition
\ref{prop:catextcar} and Theorem \ref{thm:antipoduoidal}). It its
worth noticing that the construction of the duoidal structure on
$\Schaqc$ relies heavily in the fact that $A$ supports a commutative
sum (see Definition \ref{defn:othermonoidal1}).

In the general setting of arbitrary duoidal categories there is no
canonical way to define an antipode, or more generally the notion of
group (Hopf) object.  However, in our particular category $\Schaqc$ we
have an obvious candidate for a group  type object, namely the
quasi-compact morphisms of group schemes $q:G\to A$. Thus, the affine
morphisms of group schemes $q:G\to A$ are the affine group type
objects for the duoidal structure on $\Schaqc$; under the functor
$\PP$, these group objects are in bijection with the (faithful,
commutative) group type objects for the duoidal structure on
$Q\sAalg$, that we call \emph{faithful commutative Hopf sheaves} (see
Definition \ref{defn:hsh} and Theorem \ref{thm:hopssheaf=affext}).

A drawback of the proposed definition of $\mathcal S$--module is that
 it  only  contemplates the finite dimensional objects
  ---   for  affine group schemes,  the
notion of \emph{rational $G$--module} allows to take into account the
infinite dimensional case  (see for example
\cite[Definition 5.3.7]{fer-ritt}).  Indeed, whereas an infinite
dimensional $\Bbbk$--space is a colimit of finite dimensional
sub-spaces --- a directed union of finite dimensional sub-spaces ---,
we need an adequate notion of ``rational infinite dimensional vector
bundle'', convenient for our purposes. The op-equivalence of the
category of affine extensions and the category of flat commutative
Hopf sheaves allows the lifting of this obstruction, by considering
the $\mathcal S$--modules as sheaves.

 More precisely, we proceed as follows:  given an affine scheme  $X=\Spec(B)$, J.-P.~Serre proposed in \cite{kn:serrefvid} the
category of projective $B$-modules as a generalization of the notion
of vector bundle. In
\cite{kn:drinfeldfvid}, V.~Drinfeld generalizes in turn Serre's
proposal, by considering  quasi-coherent, flat sheaves on an scheme
$X$ --- recall that if $X$ is a n\oe therian scheme, then  the category of coherent flat
$\mathcal O_A$--modules, being  the category  locally free of finite rank
$\mathcal O_A$--modules, is equivalent to the category of vector
bundles (see 
\cite[Proposition 2]{kn:serrefvid},  Remark \ref{rem:comodhts1} and
Proposition\ref{prop:equicomodrep0}).  Hence, we  establish the notion of
comodule of a Hopf sheaf taking as support the quasi-coherent, flat
sheaves on $A$. However, in order to exploit the well known
equivalence between the category of vector
bundles and the category of coherent flat
$\mathcal O_A$--modules and to  establish a notion of
\emph{$\mathcal S$--linearized sheaves} for an affine extension
$\mathcal S$, we need to develop the notions of \emph{graded morphisms of
  sheaves} and \emph{homogeneous sheaf} (see definitions
\ref{defn:hsgraded1} and \ref{defn:homogsheaf}); we define in this way
the  \emph{category of
  homogeneous sheaves on $A$ with graded morphisms} as a $\Schk$--category ---
in Lemma \ref{lem:vbspecgr}, we prove that the 
category of homogeneous vector bundles with graded morphism is
equivalent to the category of homogeneous quasi-coherent, flat sheaves
with graded morphisms.

Once the categorical framework above is established, we can consider
  the categories of \emph{$\mathcal S$--linearized sheaves with
    graded morphisms} and of \emph{$\mathcal H_{\mathcal
      S}$--comodules with graded morphisms} (here $\mathcal H_{\mathcal
      S}$ denotes the Hopf sheaf associated to $\mathcal S$ --- this is
    done in sections \ref{subsect:linerarishaeves} and
    \ref{subsection:Hcomod} respectively ---, and
    prove the
equivalence between the category $\Rep(\mathcal S)$ and  the categories of
coherent flat $\mathcal S$--linearized sheaves and $\mathcal H_{\mathcal
      S}$--comodules with graded morphisms (see Theorem
    \ref{thm:ratalg1}  and Proposition \ref{prop:ratshafpobre}). 
Finally, we also
  propose a notion of \emph{rational sheaf} that could be useful in 
  the study of these categories (see definitions
  \ref{defi:rationalsheafrep} and \ref{defi:rationalcomod}).

\bigskip

\noindent {\sc Acknowledgments:} The authors would like to thank ANII (Uruguay), CIMAT
(Mexico), CONACyT (Mexico)  and CSIC (Udelar, Uruguay) 
for partial financial support. We  also thank Michel Brion for several
useful discussions, in particular for pointing us to previous work on
quasi-compact group schemes by D.~Perrin (\cite{kn:dperrin1},
\cite{kn:dperrin2}), and Ignacio L\'opez for his helpful insight  in subjects in
category theory specially in relation to the duoidal situation.

\section{Extensions of abelian varieties by affine group schemes}
\label{sect:actgrsch}

\subsection{Group schemes and their actions}\ %
 \label{subsec:gsact}

 In this section we present some basic definitions and known results
 on quasi-compact group schemes.

 \begin{defn}
\label{defn:groupscheme}
\noindent (1) A \emph{$\Bbbk$--monoid scheme} $M$ --- or \emph{monoid scheme over}
$\Bbbk$ --- is a $\Bbbk$--scheme together with two
$\Bbbk$--morphisms 
 $m_M:M\times M\to M$ and  $e_M:\Spec(\Bbbk)=\{*\} \to M$ (called  the
\emph{multiplication},  and the \emph{unit}
respectively), satisfying
the usual commutative diagrams (of associativity of $m_M$ and unitality
of $e_M$).

\noindent (2) A \emph{$\Bbbk$--group scheme} $G$ --- or \emph{group scheme over}
$\Bbbk$ --- is a $\Bbbk$--monoid scheme together with an
\emph{inversion morphism} $\iota_G:G\to G$ (defined over $\Bbbk$ and satisfying
the corresponding commutative diagrams).

\noindent (3) 
A \emph{morphism of monoid schemes} between $M$ and $M'$
is a morphism of $\Bbbk$--schemes $f: M\to M' $ satisfying the
usual commutative  diagrams:
\[
\xymatrix{
M\times M\ar[r]^-{m_M}\ar[d]_{f\times f}& M\ar[d]^f & 
\Spec(\Bbbk)\ar[r]^(0.7){e_M}\ar[rd]_{e_{M'}} & M\ar[d]^f \\
M'\times M'\ar[r]_-{m_{M'}}& M' & &  M' 
}
\]

If both $M$ and $M'$ are group schemes, we say that $f$ is a
\emph{morphism of group schemes} (in this case, $f\smallcirc
i_G=i_{G'}\smallcirc f$).

\noindent (4) A monoid scheme $M$ is affine (resp.~of finite type,
resp.~smooth) if $M$ is so as scheme.

\noindent (5) An \emph{abelian variety} is a smooth, connected, proper
$\Bbbk$--group scheme of finite type --- an abelian variety is
necessarily a commutative group. An \emph{isogeny} of abelian
varieties is a group homomorphism that is surjective and has finite
kernel.

\end{defn}

Most of the time we abbreviate and omit the mention to the base field
e.g. a $\Bbbk$--scheme is called simply a scheme, and a $\Bbbk$--group
scheme of finite type is referred as a \emph{group scheme of finite
  type}.

 Whenever it is convenient or necessary, we will interpret a group
 scheme $G$ as a representable functor $G:(\Schk)^{\op}\to
 \mathrm{Groups}$ --- where $\Schk$ is the category of $\Bbbk$--schemes and
 $\mathrm{Groups}$ the category of abstract groups.  If $T$ is a
 $\Bbbk$--scheme, then $G(T)$ together with $m(T),i_G(T),e_G(T)$ is
 called the group of the $T$--points of the scheme $G$.

\begin{rem}
  \label{rem:alggrps}
 Traditionally, group schemes of finite type were called ``algebraic
 groups" (cf.~\cite{kn:demgab, kn:waterhouse}), but currently this
 nomenclature does not seem to have a unique connotation (e.g.~in
 \cite{kn:GIT} an algebraic group is a \emph{smooth} group scheme of
 finite type). In order to avoid confusion we prefer to use a more
 explicit, unambiguous, name.
 \end{rem}

\begin{defn} \noindent (1) An \emph{action of a $\Bbbk$--group scheme $G$ on a
    $\Bbbk$-scheme $X$} is a morphism of schemes $a :G\times
  X\to X$, satisfying the usual commutative diagrams. In this
  situation the scheme $X$ is said to be a \emph{$G$--scheme}.

  \noindent (2)   Given two $G$--schemes $X$ and $Y$, a morphism $f:X\to Y$ is \emph{$G$--equivariant} (or a \emph{morphism of $G$--schemes}) if the
following diagram is commutative, where the horizontal arrows are the
corresponding $G$--actions:
\[
\xymatrix{ G\times X \ar[r]\ar[d]_{\operatorname{id}\times
    f}&X\ar[d]^f\\ G\times Y \ar[r]&Y }\]
  \end{defn}

  \begin{rem}\label{rem:morfmonsch}
(1)  It is well known (see for example \cite{kn:matsumuraoort}) that to
  give an action of $G$ on $X$  is equivalent to give a morphism of functors
  (that is, a natural transformation) $\phi: G\to
  \Aut_X$, where
  $\Aut_X:(\Schk)^{\operatorname{op}} \to \mathrm{Groups}$ is the so
  called \emph{automorphism group functor}. Recall 
  that given
  a scheme $T$,  the group $\Aut_X(T) \subseteq
  \text{Aut}_{\Sch|T}(X \times T)$, where  $\Sch|T$ denotes the
  category of schemes over $T$ and we consider $X\times T$ as an
  $T$--scheme given by the projection on the second coordinate,  is
  defined as follows: 
\[
\Aut_X(T)=\bigl\{f: X \times T \to X \times T \text{ isomorphism }\mathrel{:}
  f(x,t)=\bigl(\widetilde{f}(x,t),t\bigr)\,,\ \widetilde{f} : X \times T \to X
  \bigr\}.
\] 

Equivalently, $f \in \Aut_X(T)$ if $f:X\times T\to X\times T $ is an isomorphism and the following diagram
commutes
\[
{\xymatrixcolsep{3pc}
 \xymatrix{
   X\times T \ar[r]^-{f}_{\cong}\ar[d]_{p_2} & X\times T \ar[d]^{p_2} \cr
   T \ar[r]^-{\id_T} & T}}\] 
in which case $\widetilde{f} = p_1 \smallcirc f$ (see for example \cite{kn:matsumuraoort} or \cite{kn:brionchev}).

\noindent  (2) In particular, when $\Aut_X$ is a group scheme, we have
a canonical action $a:\Aut_X\times X\to X$, induced by
$\id_{\Aut_X}:\Aut_X\to \Aut_X$: if $f=(\widetilde{f},p_2) \in \Aut_X(T)$ and $x\in
X(T)$, then $a(f, x)=\widetilde{f}\smallcirc (x,\id_T)$.
\end{rem}

\begin{rem}
\label{rem:scheima}
 (1) If $G$ is a connected group scheme, then $G$ is
  quasi-compact, as follows  from  \cite[Th\'eor\`eme IV.3.2]{kn:dperrin1}).

\noindent (2) Let $f:G\to G'$ be a morphism of group schemes. Then the
\emph{scheme theoretic image} of $f$, denoted as $f(G)$, is the smallest closed subgroup scheme of
$G'$ containing the image of $f$.

\noindent (3) Any group scheme is  separated (because
$e_G:\Spec(\Bbbk)\to G$ is a  closed immersion, see for example
\cite[Tag 045W]{kn:stackproj}).

\noindent (4) Since any  morphism of separated schemes is separated,
any  morphism of group schemes  $f:G\to G'$ is
separated.
\end{rem}

We finish this section by recalling some fundamental results on the structure
of quasi-compact group schemes due to D.~Perrin (see
\cite{kn:dperrin1,kn:dperrin2}).

\begin{thm}[Perrin, {\cite[Th\'eor\`eme II.2.4 and
Th\'eor\`eme V.1.1]{kn:dperrin1}}]
  \label{thm:perrinG0}
Let  $G$ be  a quasi-compact group scheme. Then 

\noindent (1) There  is a unique  irreducible component of $G$ passing trough
$e_G$ --- this component is called the \emph{neutral component of $G$} and denoted as
$G^0$ ---; moreover $G^0$ is geometrically irreducible;

\noindent (2) the inclusion $i:G^0\to G$ is a flat closed immersion;

\noindent (3) $G^0$ is a normal (quasi-compact) subgroup of $G$;

\noindent (4) the quotient $G/G^0$ exists and is an affine group
scheme, with fields as local rings. Moreover, $G/G^0$ is compact,
totally discontinuous, and limit of \'etale finite groups (see Section \ref{section:invlim}).\qed
\end{thm}

\begin{thm}[Perrin, {\cite[Corollaire V.3.2]{kn:dperrin1}}]
Let $G$ be a quasi-compact $\Bbbk$--group scheme, and $K\subset G$ a
closed subgroup scheme. Then, in the following two situations the
quotient $G/K$ exists in the category of $\Bbbk$--schemes:

\noindent (1) $K$ is defined by a sheaf of finitely generated ideals,
in which case $G/K$ is of finite type;

\noindent (2) $K$ is a normal subgroup of $G$, in which case $G/K$ is
a group scheme.\qed
\end{thm}

\begin{prop}
  \label{prop:fmonisflat}
  Let $M$ be a  monoid scheme, $G$ a reduced group scheme
and $f:M\to G$ a quasi-compact dominant morphism of monoid schemes. Then $f$
is flat.
  \end{prop}
  \pf If $M$ is a group scheme, this result is proved in
  \cite[Proposition II.1.3]{kn:dperrin1}. An inspection of the
  proof presented therein, shows that it is still valid for $M$ a
  monoid scheme.\qed

  \begin{thm}[Perrin, {\cite[Proposition II.1.5,  Lemme V.3.3.1 and Corollaire V.3.3]{kn:dperrin1}}]
  \label{thm:perrinigame}\ \\
  Let $f:G\to K$ be a quasi-compact morphism of group schemes; let
  $f(G)$ be the scheme theoretic image of $f$. Then $f(G)\cong
  G/\operatorname{Ker}f$ and the induced morphism 
  $\widetilde{f}:G\to f(G)$ is faithfully flat. In particular, 
    the
  induced morphism $G/\operatorname{Ker}f\to K$ is a closed immersion.\qed
\end{thm}

\begin{cor}
  \label{cor:surimpliesflat}
Let $f:G\to K$ be a quasi-compact morphism of group schemes, with $K$
reduced. Then the following three assertions are equivalent: (a) $f$ is faithfully flat; (b) the map associated to $f$ at the level of sets is surjective; (c) the map $f(\overline{\Bbbk}): G(\overline{\Bbbk})\to
K(\overline{\Bbbk})$ is surjective.

Moreover if $f$ is as above and $K$ is connected, then the restriction
$f|_{_{G^0}}:G^0\to K$ is faithfully flat.
\end{cor}
\pf Indeed, under the hypothesis of this corollary, $K(\overline{\Bbbk})$
is dense in the base space of  $K$.

If $K$ is connected, since $f$ is faithfully flat, it follows that
$G^0$ dominates $K$ (see for example  \cite[Proposition
IV.2.3.4]{kn:EGAIV2}). Since $f|_{_{G^0}}$ factors through a closed
immersion, the result follows. 
\qed

\subsection{Extensions of abelian varieties by affine group
  schemes}\ %
\label{subsect:affext}

\begin{defn}
\label{defi:torsor}
Let $H$ be a $\Bbbk$--group scheme, $X$ an $H$--scheme with action
$a$ and $f:X\to Y$ an $H$--invariant morphism of schemes, $f$ is
an \emph{$(H,a)$--torsor} if:

\noindent (1)  $f$ is quasi-compact and faithfully flat;

\noindent (2) The morphism $H\times X\to X\times_YX$ induced by $a$ and
  the projection over the second coordinate, is an isomorphism; in
  other words, the commutative diagram below is cartesian:
\[
\xymatrix{
H\times X\ar[r]^-a\ar[d]_{p_2}& X\ar[d]^f\\
X\ar[r]^f& {Y}
}
\]

When no confusion arises, we will say that $f$ is an \emph{$H$--torsor} of a \emph{torsor under $H$}.
\end{defn}

\begin{defn}
\label{defi:exactseq}
Let $j:N\to G$ and $q:G\to Q$ be two morphisms of group schemes. The
sequence
 \[
\xymatrix{
\mathcal S: & 1\ar[r] & N\ar[r]^j & G\ar[r]^q & Q\ar[r] & 1
}
\]
is a \emph{short exact sequence of group schemes} if and only if the
following two conditions are satisfied:

\noindent (1) The sequence $\mathcal S$ is left exact; that is, the
sequence $1\to N(T)\to G(T)\to Q(T)$ is exact for every $\Bbbk$-scheme
$T$ --- equivalently, $\operatorname{Ker} j$ is trivial and $j$
induces an isomorphism $\operatorname{Ker} q \cong N$.

  \noindent (2) If $T$ is a scheme and $y\in Q(T)$,
  then there exists a faithfully flat, quasi-compact morphism $f:T'\to
  T $ and $x\in G(T')$ such that $q_{T'}(x)=Q(f)(y)\in Q(T')$.
\end{defn}

\begin{rem}
\label{rem:torsoraffine2}
(1) Notice that   condition (2) of Definition \ref{defi:exactseq} holds
whenever $q:G\to Q$ is an fpqc (i.e.~faithfully flat 
quasi-compact) morphism.

\noindent (2) Moreover, if $q:G\to Q$ is an fpqc morphism, then
clearly $q$ is
an $N$--torsor --- the second condition of Definition
\ref{defi:torsor} is easily proved due to the fact that
all the schemes involved are group schemes.
In particular,   since $q$ is an $N$--torsor, $q$ is a categorical
quotient (see for example \cite[\S 2.6]{kn:brionchev}). 
\end{rem}

    \begin{ej}
\label{ej:exactandtorsor}
  Let $G$ be a connected group scheme and $H\subset G$ a normal closed
  subgroup scheme. Then it follows  from \cite[Corollaire
  IV.3.3]{kn:dperrin1}  that  $G/H$ is a group scheme and the quotient
  map $q: G\to G/H$ is a faithfully flat quasi-compact morphism.  In
  particular, the sequence ${\xymatrixcolsep{1.5pc}
    \xymatrix{1\ar[r]&H\ar[r]&G\ar[r]^-{q}&\ar[r] G/H&0}}$ is
  exact. 
  \end{ej}

\begin{defn}
\label{defi:affqcext}
Let $A$ be an abelian variety. A \emph{group extension of $A$} is a
short exact sequence $
\mathcal S$: $\ate$. If moreover $q:G\to A$ is a faithfully flat
quasi-compact morphism we say that $\mathcal S$ is a
\emph{quasi-compact group extension of $A$}; if $q$ is a faithfully
flat affine morphism, we say that $\mathcal S$ is an 
 \emph{an affine
   group extension of $A$}.
\end{defn}

\begin{rem}
  \label{rem:forequiv}
  Let  $A$ be an abelian variety. If 
   $q:G\to A$ is a 
  surjective quasi-compact  morphism of group schemes, then $G$ is a
  quasi-compact group scheme and $q$ is a faithfully flat morphism by
  \cite[Proposition II.1.3]{kn:dperrin1}, since $A$ is a reduced group scheme 
  (see Corollary \ref{cor:surimpliesflat}). It follows that
  \[
    \mathcal S_q :
   \xymatrix{1\ar[r]&\operatorname{Ker}(q)\ar[r]&G\ar[r]^(.45){q}&\ar[r]
     A&0}
   \]  
is a quasi-compact extension of $A$.

On the other hand,  if $G$ is a quasi-compact group scheme and
$H\subset G$ is a 
  normal subgroup scheme such that $A=G/H$ 
  is an abelian variety, then the canonical 
projection $q:G\to A$ is an $H$--torsor, and the corresponding exact
sequence is  a quasi-compact  extension (see \cite[Corollaire IV.3.3
and Proposition II.1.3]{kn:dperrin1}).
\end{rem}

\begin{rem}\label{rem:torsoraffine1}
  \noindent (1) Let 
  $\mathcal S$: $\ate$ be a quasi-compact extension. Then $\mathcal S$
  is an  
  affine extension  if and only if $H$ is an affine group scheme. See
  \cite[III, \S\ 3,2.5/6]{kn:demgab}, or  \cite[\S\ I.5.7]{kn:jantzen2} for a similar result for $H$--torsors.

  \noindent (2) By definition,  if a group scheme $G$ fits into an 
  affine extension then $G$ is quasi-compact; see Corollary
  \ref{cor:qcisate} below 
  for a partial converse due to D.~Perrin
  (\cite{kn:dperrin1}).

\noindent (3) It is well known that if $\ate$ is an affine
extension then $G$ is of finite type  if and only if $H$
is of finite type,  see for example \cite[Proposition 2.6.5]{kn:brionchev}; if this is the case, we say that the extension
is of \emph{finite type}.

\noindent (4) Let $A$ be an abelian variety and  $\mathcal S$: $\ate$
a short exact sequence. It follows from Remark \ref{rem:forequiv} that  a
$\mathcal S$ is  a  quasi-compact 
(resp.~affine)
extension if and only if $\mathcal S$ is left exact and  $q$ a
surjective quasi-compact 
(resp.~affine) morphism.
\end{rem}

We complete the definition of the category of quasi-compact
 (resp.~affine) extensions of an affine variety $A$ by 
 defining its morphisms (see also Section \ref{subsect:affextschoverA}).

\begin{defn}\label{defn:catextensions}
  Let $A$ be an abelian variety.

  \noindent (1) The \emph{category $\GextqcA$ of quasi-compact group extensions of $A$}   has as objects the
  quasi-compact extensions of $A$ and as \emph{morphisms}
  $\phi:\mathcal S \to \mathcal S'$ between two quasi-compact
  extensions of $A$, the commutative diagrams of the form:
\begin{equation}
\label{eq:morfext}
\raisebox{5ex}{\xymatrix{
\mathcal S:\ar[d]^-\phi &
1\ar[r] & N\ar[r]^j\ar[d]_-{f_N} & G \ar[r]^q\ar[d]_-f & A\ar[r]\ar@{=}[d] & 0\\
\mathcal S': &
1\ar[r] & N'\ar[r]^{j'} & G'\ar[r]^{q'} & A\ar[r] & 0
} }
\end{equation}
where 
$f_N$ and $f$ are morphisms of group schemes.

\noindent (2) The \emph{category $\GextaffA$ of affine extensions of $A$} is
defined as the full subcategory of $\GextqcA$ with objects the affine extensions.

\noindent (3) If ${\mathbf P}$ is a class of morphisms of schemes
(affine, quasi-compact, finitely presented, etc.) we say that the
morphism $\phi: \mathcal S \to \mathcal S'$ is of class ${\mathbf
  P}$ if and only if $f$ is of class ${\mathbf P}$.
\end{defn}

\begin{nota}\label{nota:central} In the situation of a diagram such as \eqref{eq:morfext}, the morphism  $f: G \to G'$ will be called, the \emph{mid morphism} of $\mathrm{\phi}$. 
  \end{nota}

\begin{rem}\label{rem:fisaffine}
   It is evident that it is equivalent to give a diagram as
  \eqref{eq:morfext} or a commutative triangle of  morphisms of group
  schemes as below, with $q$ and 
  $q'$ affine morphisms. 
\[
\xymatrix@=10pt{
G \ar[rr]^{f}\ar[dr]_(.4){q} & & G'\ar[ld]^(.4){q'} \cr
 & A &}
\]

Indeed, if $f:G\to G'$ is as above, then $\Ker(f) \subset \Ker(q)$ and the
restriction $f|_{\Ker(q)}:\Ker(q)\to \Ker(q')$ makes sense as
$f\bigl(\Ker(q)\bigr) \subset \Ker(q')$. 

In particular, notice that if $\phi: \mathcal S\to \mathcal S'$ is a
morphism of affine extensions then the mid morphism $f$ is an
affine morphism, since it is a morphism in the category $\Schaaff$.
\end{rem}

  \begin{rem}\label{rem:isomorphism}
 \noindent (1) The composition of morphisms in $\GextqcA$ and the
 identity morphism are 
  defined in the obvious manner.

  \noindent (2) Clearly $\mathcal S$ and $\mathcal S'$ are
  \emph{isomorphic} if and only if the maps $f_N$ and $f$ are
  isomorphisms --- this last assertion is equivalent the assertion
  that $f$ is an isomorphism (compare with Theorem
  \ref{thm:perrinigame}, Remark \ref{rem:forequiv} and \S~\ref{subsect:affextassch}). 
\end{rem}

\begin{defn}
  \label{defn:smoothext}
 If in Definition \ref{defi:affqcext} the  group scheme $H$
  is smooth, then the canonical projection $q:G\to A$ is a smooth
  morphism; in this situation we say that the  extension is
  \emph{smooth}.

\end{defn}

\begin{ejs}
\label{exam:trivials}

(1) If $G$ is an affine group scheme, then $G$ can be viewed in a
canonical way as an affine extension of the trivial abelian variety
$\Spec{\Bbbk}=\{*\}$ by means of  the   short exact sequence
$\xymatrixcolsep{1.5pc} \xymatrix{1\ar[r] & G\ar[r]^(.43){\id}
  &G\ar[r] &\Spec(\Bbbk)\ar[r] &0}$.

\noindent (2) An abelian variety $A$ can be thought as an affine
extension in a natural way as: $\xymatrixcolsep{1.5pc} \xymatrix{
  0\ar[r] & 0\ar[r]& A\ar[r]^(.45){\id} &A\ar[r] &0}$.

\noindent (3) If $\mathcal S$: $\ate$ is an affine extension and
$f:A\to A$ is an isogeny (i.e.~a surjective morphism of abelian
varieties with finite kernel), then ${\xymatrixcolsep{1.5pc}
  \xymatrix{1\ar[r]&\operatorname{Ker}(f\smallcirc
    q)=q^{-1}\bigl(\operatorname{Ker}(f)\bigr)\ar[r]&G\ar[r]^(.45){f\smallcirc
      q}&\ar[r] A&0}}$ is an affine extension of $A$ (see Remark
\ref{rem:torsoraffine1}).

In particular if $f:A\to A$ is an isogeny, 
then $\xymatrixcolsep{1.5pc} \xymatrix{ 0\ar[r] &
  \operatorname{Ker}(f)\ar[r]& A\ar[r]^(.45){f} &A\ar[r] &0}$ is an
affine extension.
\end{ejs}

\begin{rem}
\label{rem:pushes}
Let $\mathcal S$: $\ate$ be an affine extension and $\ell:H\to H'$ be
a morphism of affine group schemes. Assume moreover that $H\subset G$
is central in $G$ and that $\ell(H)\subset H'$ is central in
$H'$. Then $\Gamma_{\ell} (H)$, the scheme theoretic image of the ``graph''
morphism $\Gamma_\ell=(\operatorname{inc},
\ell\smallcirc \iota_H)=(\operatorname{inc}\times \ell)\smallcirc \Delta_H : H\to G\times H'$, is
a central subgroup scheme of $G \times H'$ --- here $\Delta_H: H\to
H\times H$ denotes the diagonal embedding, $\Delta_H(h)=(h,h^{-1})$. Therefore, the quotient
$G\times^HH'=(G\times H')/\Gamma_\ell(H)$ is a quasi-compact group
scheme and fits into an affine extension, that we denote
$\ell_*\mathcal S$ --- it is also possible to deduce the existence of
$\ell_* \mathcal S$ from the properties of the \emph{induced space}
(see Theorem \ref{thm:indesp1}). Moreover, $\ell$ yields a morphism
$\lambda: \mathcal S\to \ell_* \mathcal S$ of affine extensions:
\[
\xymatrix{
\mathcal S:\ar[d]^-\lambda & 1\ar[r]& H\ar[r]\ar[d]_-\ell & G\ar[r]^q \ar[d]^-j
& A\ar[r]\ar@ {=}[d]& 0 \\
\ell_*\mathcal S: & 1\ar[r]& H'\ar[r]& G\times^H H'\ar[r]_-{\pi_{H'}} & A\ar[r]& 0 
}
\]
where $j:G\to G\times^HH'$ is given by $j(g)=[g,1]:=\pi_{G\times H'}(g,1)$, with
$\pi_{G\times H'}:G\times H'\to G\times^H H'$  the canonical
projection, and $\pi_{H'}: G\times ^HH'\to A$ is given by
$\pi_{H'}\bigl([g,h']\bigr)= q(g)$. Indeed, note that $\pi_{H'}$ is
well defined and that $\pi_{H'}\bigl([g,h']\bigr)=0$ if and only if
$g\in H$, therefore $\Ker(\pi_{H'})= \bigl\{ [1,h']\mathrel{:} h'\in
H'\bigr\}=H'$. 

Note that if $H'$ is smooth, then  $\ell_*\mathcal S$ is a smooth
extension.
\end{rem}

\begin{defn}
A \emph{closed immersion} of the affine extension
  $\mathcal T$ into the affine extension $\mathcal S$ (both
  extensions of $A$) is a morphism $\phi: \mathcal T\to \mathcal S$ of
  affine extensions
\[
\xymatrix{\mathcal T:\ar[d]_-{\phi}&1\ar[r]&
  H'\ar[r]\ar@{^{(}->}[d]_-{f|_{_{H'}}} & G'\ar[r]^{q'}
  \ar@{^{(}->}[d]^-{f} & A\ar[r]\ar@{=}[d]& 0 \\ \mathcal S :&
  1\ar[r]& H\ar[r]& G\ar[r]_{q} & A\ar[r]& 0}
\]
such that the vertical morphism $f:G'\to G$ (and therefore
$f|_{_{H'}}:H'\to H$) is a closed
immersion.

In particular, if   $G'\subset G$ is a closed subgroup scheme such
that $q(G')=A$ and $H'=\operatorname{Ker}(q|_{_{G'}})$, then $\mathcal
T$: ${\xymatrixcolsep{1.5pc}
  \xymatrix{1\ar[r]&H'\ar[r]&G'\ar[r]^(.45){q|_{_{G'}}}&\ar[r] A&0}}$ is an
affine extension and  the inclusion $\mathcal T\hookrightarrow \mathcal S$ is a
closed immersion, we say that $\mathcal T$ is a \emph{(closed, affine)
  sub-extension} of $\mathcal S$.
\end{defn}

  The following theorem was first announced by C.~Chevalley in the 1950s
and published in 1960 in \cite{kn:chetheo}. We present here a slightly
more general version, due to M.~Brion (see
\cite[Theorem 2]{kn:brionchev} and Corollary
\ref{cor:qcisate} below).

\begin{thm}[Chevalley, Raynaud, Brion]
\label{thm:chevftgs}
  Every $\Bbbk$--group scheme of finite type $G$ has a smallest normal
  subgroup scheme $G_{\aff}$ such that the quotient $G/G_{\aff}$ is
  proper. Moreover, $G_{\aff}$ is affine and connected, and the associated short sequence of
  group schemes over $\Bbbk$ is exact (see Definition
  \ref{defi:exactseq})       
\begin{equation}
\label{eqn:chev}
\xymatrix{
1\ar[r]& G_{\aff}\ar[r]& G\ar[r]^(0.36){q}&G/G_{\aff}\ar[r]&0.
}
\end{equation}

If 
$\Bbbk$ is perfect and $G$ is smooth, then $G_{\aff}$ is smooth as
well, and its formation commutes with field extensions --- that is,
if $\Bbbk\subseteq \mathbb K$, then $G(\mathbb K)_{\aff}=G_{\aff}(\mathbb K)$.
In particular, if $G$ is a  connected   group scheme of finite type
over a perfect field, then  $G$ fits in  a (smooth) affine  extension
of the abelian variety $A=G/G_{\aff}$.\qed
\end{thm}

Since every  group scheme of finite type is an extension of a smooth
group scheme by an
infinitesimal group scheme (see \cite[Proposition 2.9.2]{kn:brionchev}),
Theorem \ref{thm:chevftgs}  implies that any connected  group scheme
of finite type 
fits in an affine extension.

\begin{cor}
  \label{cor:qcisate}
Let $G$ be a  connected  group scheme of finite type. Then there exists an
affine extension of an abelian variety  $\mathcal S$: $\ate$, with $H$ a connected affine
group scheme. 
\end{cor}

\proof
See \cite[Corollary 4.3.7]{kn:brionchev}.\qed

\begin{rem}\label{rem:perrinimprov}
  It follows from Perrin's Approximation
  Theorem (\cite[Th\'eor\`eme V.3.1]{kn:dperrin1}, see Theorem \ref{thm:perrinquasicomp}
  below) that Theorem \ref{thm:chevftgs} and its Corollary
  \ref{cor:qcisate} imply that   \emph{any connected} 
  group scheme fits into an affine extension (see  \cite[Corollary
  V.4.3.1]{kn:dperrin1}).
 \end{rem}
 
\begin{nota}\label{nota:chevdecomp} If $G$ is a smooth group scheme of finite type, then the
  sequence \eqref{eqn:chev} is known as the \emph{Chevalley
    decomposition} of $G$.
\end{nota}

\begin{rem}
\label{rem:extnounica}
  If  $G$ is a smooth group scheme of finite type over a perfect field
  $\Bbbk$,  then $G_{\aff}$ is the largest normal, affine, connected,  
  smooth, subgroup scheme of $G$ (see for example
  \cite{kn:brionchev}).

The following uniqueness result (for
  $\Bbbk$ a perfect field) follows easily.  Assume that a given smooth
  group scheme $G$ fits in an exact sequence
\[
\xymatrix{
1\ar[r]& H\ar[r]& G\ar[r]&G/H\ar[r]&0
,}
\]
with $H$ affine connected and $G/H$ an abelian variety. Then there are isomorphisms $f_1:H \cong G_{\aff}$
and $f_2: G/H \cong A$, such that the diagram of short exact sequences is commutative:
\[
\xymatrix{ 1\ar[r]& H\ar[r]\ar[d]_{f_1}&
  G\ar[r]\ar[d]_{\id}&G/H\ar[r]\ar[d]_{f_2}&0\\ 1\ar[r]& G_{\aff}\ar[r]&
  G\ar[r]&A\ar[r]&0. }
\]

It follows that $G_{\aff}$ is the \emph{unique} normal,  affine,
connected, smooth, subgroup scheme $H$ such that $G/H$ is proper.
Indeed, if $H$ is such a group, then $G_{\aff}\subset H$ by the
Chevalley decomposition theorem \ref{thm:chevftgs}, and $H\subset
G_{\aff}$ by the preceding remark.
\end{rem}

\begin{lem}
  \label{lem:descchevred}
Let $G$ be a group scheme of finite type over a perfect field
  $\Bbbk$ and assume that $G$ fits
in an exact sequence of group schemes
\[
\xymatrix{
1\ar[r]& H\ar[r]& G\ar[r]^-{q_G}&G/H\ar[r]&0
}
\]
with $H$ an affine connected normal subgroup scheme and $G/H$ proper.
Then 
\[
\xymatrix{
1\ar[r]& H_{\red}\ar[r]& G_{\red}\ar[r]&(G/H)_{\red}\ar[r]&0
}
\]
is the Chevalley decomposition of $G_{\red}$.
\end{lem}

\pf
By construction $G_{\aff}\subset H$ and it follows that
$(G_{\aff})_{\red}\subset H_{\red}$. Now, since $H_{\red}$ is affine
and connected, then its Albanese variety is $\Alb(H_{red})=\Spec(\Bbbk)$
and so $q_G(H_{red})=\{0\}\subset A$; therefore, $H_{\red}\subset G_{\aff}$. 
\qed

  \begin{rem}
    In Lemma \ref{lem:descchevred} the condition that $\Bbbk$ is
    a perfect field cannot be omitted. Indeed, if $\Bbbk$ is not perfect, then
    $G_{\red}$ is not necessarily a group scheme, as it is shown in \cite[Example
    2.5.5]{kn:brionchev}.
  \end{rem}

The Chevalley decomposition of smooth group schemes
has the 
  following functorial property.
  
\begin{lem}
  \label{lem:chevcase}
  If $f:G\to G'$ is a morphism of smooth group schemes of finite
  type, then
  their Chevalley decompositions fit in 
  the following commutative diagram:
  \begin{equation}
    \label{eqn:chevcase}
\raisebox{5ex}{\xymatrix{
\mathcal G:\ar[d]^\phi &1\ar[r] & G_{\aff}\ar[r]\ar@{.>}[d]_{f|_{_{G_{\aff}}}}&
G\ar[r]^-{q}\ar[d]^{f}&
Q\ar[r]\ar@{.>}[d]^{\widetilde{f}}&
0\\
\mathcal G': &1\ar[r] & G'_{\aff}\ar[r]&
G'\ar_-{q'}[r]& Q' \ar[r]& 0
}}
\end{equation}

If $f$ is a faithfully flat morphism, then the vertical arrows of the
diagram above are faithfully flat morphisms. Moreover, if $f$ is
affine and faithfully flat, then $\widetilde{f}$ is an isogeny.

\end{lem}

\pf
 Since $G$ is a  group scheme  of finite type the image of
  $q'\smallcirc f$ is closed in $Q'$; thus $q'\smallcirc f(G) \subset Q'$ is a proper group and
  it follows that $G_{\aff}\subset \Ker(q'\smallcirc f)$. By the
  universal property of the quotient, it follows  that $q'\smallcirc f$
  induces a morphism $\widetilde{f}:Q\to Q'$ that fits in
  Diagram \eqref{eqn:chevcase}.

Assume now that $f$ is faithfully flat. Then
  $\widetilde{f}\smallcirc q= q'\smallcirc f$ is faithfully flat, $q
  $ being a faithfully flat morphism it follows that
  $\widetilde{f}$ is faithfully flat. Since 
  $f(G_{\aff})\subset G'_{\aff}$ is a closed (therefore affine)
  subscheme and  $f$ is  faithfully flat, it follows that 
$f(G_{\aff})$ is an affine  normal 
subgroup scheme of $G'$ --- recall that if  $g'\in G'(T)$ then there exists a
faithfully flat quasi-compact morphism $f:T'\to T$ and a point $g\in
G(T')$ such that $f(T')(g)=g'$. The  faithfully flat  morphism 
$G\to G'/f(G_{\aff})$ factors through $Q$ and so
$G'/f(G_{\aff})$ is a proper group scheme. The minimality of $G'_{\aff}$ then
implies that $f(G_{\aff})=G'_{\aff}$; that is,
$f|_{_{G_{\aff}}}$ is faithfully flat.

If $f$ is an affine morphism, then  $\Ker (q'\smallcirc f)$
is an affine closed subgroup scheme of $G$. It follows that
$\Ker(\widetilde{f})=q\bigl(\Ker (q'\smallcirc
f)\bigr)\cong \Ker(q'\smallcirc f) /G_{\aff}$ is a closed affine
subgroup scheme of an abelian variety, and therefore is an affine  (and
hence a finite) subgroup scheme of $Q$.
 \qed

 \subsection{Quasi-compact extensions as schemes over an abelian variety}\ %
 \label{subsect:affextassch}
 
 As follows from Remark \ref{rem:forequiv},  to give a quasi-compact
 (resp.~affine) 
 extension over an abelian variety $A$  is equivalent to give a
 surjective, quasi-compact (resp.~affine), 
 morphism of group schemes $q:G\to A$.

 On the other hand and concerning the arrows, given two  surjective
 quasi-compact (resp.~affine) morphisms of
 group schemes $q:G\to A$ and $q':G'\to A$, and a 
 morphism of group schemes $f:G\to G'$ such that $q'\smallcirc f=q$.  It
 easily follows that $f\bigl(\operatorname{Ker}(q)\bigr)\subset 
\operatorname{Ker}(q')$. Hence, $f$ induces a morphism of extensions
\[
  \xymatrix{
    \mathcal S:\ar[d]^{\widetilde{f}} &
1\ar[r] & \operatorname{Ker}(q)\ar[r]\ar[d]_{f|_{_{\operatorname{Ker}(q)}}} & G \ar[r]^q\ar[d]_f & A\ar[r]\ar@{=}[d] & 0\\
\mathcal S': &
1\ar[r] & \operatorname{Ker}(q')\ar[r] & G'\ar[r]_{q'} & A\ar[r] & 0
} 
\]

Therefore, the category $\GextqcA$ of quasi-compact (resp.~$\GextaffA$
of affine) extensions is
equivalent to a 
subcategory of $\Schaqc$ (resp.~$\Schaaff$), that has as objects the
separated (see Remark \ref{rem:scheima}) quasi-compact
(resp.~affine) surjective 
(and therefore faithfully flat) morphisms of group schemes $q:G\to A$
and as morphisms $f: (q:G\to A)\to (q':G\to A)$, the morphisms in $\Schaqc$ (resp.
~$\Schaaff$) that are also morphisms of group schemes $f:G\to
G'$. 

\begin{rem}
The reader should be aware that under the above equivalence, affine
extensions \emph{do not} correspond to affine \emph{group schemes over
  the 
  scheme $A$} --- recall that the product of a group scheme over
$A$ is a morphism $m:G\times_AG\to G$.

In Section \ref{subsect:affextschoverA} we will introduce a 
structure of duoidal category  on $\Schaqc$ (see Definition
\ref{defi:duoidalcat} and 
Lemma \ref{lem:duoschemes}), such that  the  quasi-compact (resp.~affine)
extensions correspond to the group type objects for this category
(resp.~the group type objects that are affine over $A$).  See Proposition
\ref{prop:catextcar} and Theorem \ref{thm:antipoduoidal}. 
\end{rem}

\begin{nota}
  \label{nota:affA}
In view of the above equivalence, we will abuse of notation and say that
a surjective, quasi-compact (resp.~affine) morphism of group schemes $q:G\to A$ is a
\emph{quasi-compact} (resp.~\emph{affine}) \emph{extension} of $A$.

In what follows, we will freely use both points of view (affine
extensions as short exact sequences or as surjective affine morphism
of group schemes) depending on  which one is better adapted to the
particular result or definition.

\end{nota}

\subsection{Rosenlicht decomposition for affine extensions}\ %
\label{subsect:antiaff}

\begin{defn}
\label{defi:anti-aff}
\noindent (1) A group scheme $G$ defined over a field $\Bbbk$ is called
\emph{anti-affine} if $\mathcal O_G (G)=\Bbbk$.

\noindent (2) An affine extension $\mathcal S$: $\ate$, with $G$ anti-affine, is
said to be of \emph{anti-affine type}. Equivalently, an affine extension of
anti-affine type is a surjective affine morphism $q:G\to A$, with $G$
an anti-affine group scheme.
\end{defn}

Whereas the notion of anti-affine group scheme already appeared
(implicitly) in the work of Rosenlicht (\cite{kn:rosdesc}) and Serre
(\cite{kn:serrealb}) in the late 50s, it was not regularly studied
until about 10 years ago.  In \cite{kn:briantaff}, Brion began a
thorough study of anti-affine group schemes, generalizing earlier
results by Rosenlicht on the decomposition of a group scheme of finite
type (see \cite{kn:rosdesc} and \cite{kn:rosextab} and Theorem
\ref{thm:rosenftgs2} below) --- the classification of anti-affine
groups was obtained simultaneously  by Brion (\emph{op.~cit.}) and C.~Sancho de Salas and F.~Sancho de Salas
(\cite{kn:lossancho}, see also \cite{kn:sanchoesp}).

In this section (see Theorem \ref{thm:rosenftgs2}), we present a
generalization of the Rosenlicht decomposition to the setting of
affine extensions of an abelian variety as well as some related
properties. We begin by recalling the results on anti-affine group
schemes that will be used in what follows; for other properties, in
particular for a complete classification theorem, see
\cite{kn:briantaff}, \cite{kn:brisamum} and \cite{kn:brionchev}.

\begin{rem} 
(1) It is well known (see for example \cite[Chapters 2 and 5]{kn:brisamum}), that an anti-affine
group scheme is connected and commutative.

\noindent (2) If $G$ is an anti-affine  group scheme  of finite type, then $G$
is smooth (see for example \cite[Lemma 3.3.2]{kn:brionchev}).

\noindent (3) In particular, if $G$ is an anti-affine  smooth group scheme of finite type, using the Chevalley decomposition (Theorem \ref{thm:chevftgs})
we deduce that $G$ is the 
extension of   a proper group scheme $A$ by a commutative affine
group scheme  of finite type. This result was much improved by Brion in \cite[Section
5.5]{kn:brionchev} (see also \cite{kn:briantaff}): the affine subgroup $G_{\aff}$  and the group
scheme  $A$ appearing therein are  smooth (i.e.~$A$ is an
abelian variety).
\end{rem}

\begin{defn}
  \label{def:affinization}
The \emph{affinization functor} $\operatorname{Aff}: \Schkqc \to
\Schkaff$ (the codomain is the category of affine $\Bbbk$--schemes)
defined at the level of objects as
$\operatorname{Aff}(X)=\Spec\bigl(\mathcal O_X(X)\bigr) $, and
$\operatorname{Aff}(f:X\to Y): \operatorname{Aff}(X)\to
\operatorname{Aff}(Y)$ is defined as the morphism
$\Spec(f^\#_Y):\Spec\bigl(f_*(\mathcal O_X)(Y)\bigr)=\Spec\bigl(\mathcal O_X(X)\bigr) \to
\Spec\bigl(\mathcal O_Y(Y)\bigr)$. In this situation $\operatorname{Aff}(X)$ is
called the affinization of $X$.
\end{defn}

\begin{rem} \label{rem:gantqc} We list some of the properties of this functor for immediate use, see
\cite[III.3.8]{kn:demgab}, \cite[\S\ V.4.2]{kn:dperrin1}, and
\cite[\S\ 3.2]{kn:brionchev}) for general references.
  Later in \ref{prop:affschqcs} we deal with the properties of this
  functor in the more general context of schemes over a fixed scheme
  $S$.

\noindent (1)  There is an adjunction $\operatorname{Aff} \dashv
  \operatorname{inc}$ as:
  ${\xymatrixcolsep{3em}\xymatrix{ \Schkqc \ar@<5pt>[r]^-{\operatorname{Aff}}
      \ar@<-5pt>@{<-}[r]_-{\operatorname{inc}}\ar@{}[r]|-\bot& \Schkaff}}$.

\noindent  (2) The counit of this adjunction is an isomorphism and the
      unit is given by a family of morphisms of schemes $\eta_X:X \to
      \operatorname{Aff}(X): \Schk \to \Schkaff$ that satisfies the
      following universal property.

      For any morphism $f:X \to Y$ with $X \in \Sch, Y \in \Schkaff$ there
      is a unique morphism $\widehat{f}$ that makes commutative the
      diagram below: \[\xymatrix@=15pt{X \ar[rr]^-{\eta_X}\ar[rd]_{f}&&
        \operatorname{Aff}(X)\ar@{.>}[ld]^{\widehat{f}}\\&Y&}\]

The morphism $\eta_X:X\to \operatorname{Aff}(X)$ is called the
\emph{affinization morphism} of $X$; if $U=\Spec\bigl(\mathcal O_X(U)\bigr)\subset X$ is an affine open
subset, then $\eta_X|_{_U}$ is the morphism induced by the
restriction $\mathcal O_X(X)\to \mathcal O_X(U)$.

  \noindent (3) In the case that $G$ is a quasi-compact group scheme
  so is $\operatorname{Aff}(G)$, and the adjunction restricts to the
  category of group schemes. In particular the unit $\eta_G$ is a
  morphism of group schemes and its kernel is a closed subgroup scheme
  of $G$, that we  call $G_{\ant}$.
\end{rem}

In \cite{kn:dperrin1}, Perrin considered the affinization morphism in
the situation of quasi-compact group schemes:

\begin{prop}
\label{prop:keretag}
  Let $G$ be a quasi-compact group scheme. Then

  \noindent (1) The affinization morphism is a faithfully flat
  morphism (of group schemes).

  \noindent (2)  $\Ker(\eta_G)$ is
a geometrically reduced, connected anti-affine 
group scheme, contained in the center of $G^0$.

\noindent (3) $\Ker(\eta_G)$ is the smallest normal
subgroup scheme $K$ such that the quotient $G/K$ is
affine.

\noindent (4)  $\Ker(\eta_G)$ is the
largest anti-affine subgroup scheme of $G$.

\end{prop}
\pf
Assertions (1) and (2) are the content of \cite[Th\'eor\`eme
V.4.2]{kn:dperrin1}. The remaining assertions  can be proved easily
using the universal 
property of the affinization morphism $\eta$.
Indeed,   let $K\subset G$ be a
normal subgroup scheme such that $G/K$ is an affine group; let
$q_K:G\to G/K$ be the canonical projection. Then there
exists  a morphism of group 
schemes $f:\operatorname{Aff}(G)\to G/K$ such that $f\smallcirc \eta_G=
q_K$. It follows that $G_{\ant}=\operatorname{Ker}(\eta_G)\subset
\operatorname{Ker}(q_K)=K $.

On the other hand, it  is clear that if we consider  a subgroup $L
\subseteq G$ that is anti-affine, then the inclusion $L\subset G$
induces a closed immersion $\operatorname{inc}: \operatorname{Aff}(L)\hookrightarrow
\operatorname{Aff}(G)$, such that  $\eta_G|_L=
\operatorname{inc}\smallcirc \eta_L$. Since $L$ is anti-affine, it
follows that  $\eta_L$ is the  trivial morphism; thus, $L
\subseteq \Ker(\eta_G)$.
\qed

As noted in the introduction of this section, Rosenlicht
decomposition was generalized by Brion (for smooth group schemes of
finite  type, see   \cite[theorems 1 and
5.1.1, Proposition 3.3.5]{kn:brionchev}).
In Theorem  \ref{thm:rosenftgs2}  we produce a  Rosenlicht decomposition for
 affine extensions --- that will be improved afterwards in Theorem
 \ref{thm:rossmoorhnew}. Before proving the existence of such
 decomposition, we mention 
  an easy technical result on faithfully flat morphisms for which we write
  the proof for the lack of an adequate reference.

\begin{lem}
\label{lem:ffandpoints}
Let $i:X\to Y$ be a closed immersion of $\Bbbk$--schemes and $f:T'\to T$ a faithfully flat morphism. If
$y:T\to Y$ is a $T$--point of $Y$, and $x:T'\to X$ a $T'$ point of $X$ such that $i \smallcirc x= y\smallcirc f$, then
there exists  a $T$--point $ \widetilde{x}:T\to X$  such that the
following diagram is commutative
\[
\xymatrix{
T'\ar[d]_x \ar[r]^f& T\ar[d]^y\\
X\ar@{<-}[ur]^{\widetilde{x}}\ar@{^(->}[r]_i& Y
}
\]
\end{lem}
\pf 
By the universal property of the fibred product we obtain a
commutative diagram:
\[
\xymatrix@C=12pt{T'\ar@/_1pc/[rdd]_{x}\ar@/^1pc/[rrrd]^f\ar[rd]|-\ell&&&\\
&X \times_Y T \ar[d]_{p_1}\ar@{^(->}[rr]^(.60){p_2} && T\ar[d]^y\\&
X\ar@{^(->}[rr]_i&& Y.
}
\]
where $p_2$ is again a closed immersion. Since $f$ is faithfully flat, then
$p_2$ is an isomorphism (since $p_2$ is also surjective with
$p_2^{\#}$  injective) and the result follows. \qed

\begin{thm}[Rosenlicht decomposition of  affine 
  extensions]\ 
\label{thm:rosenftgs2}\\
Let $\mathcal S$: $\ate$ be an affine extension  and
$G_{\ant}=\operatorname{Ker}(\eta_G)$. Then: 

\noindent (1) 
 The restriction $m|_{_{G_{\ant}\times H}}:
G_{\ant}\times H  \to G$ is a faithfully flat morphism of group schemes,
with  kernel
$\Delta(G_{\ant}\cap H)$, the image of the diagonal embedding 
$\Delta: G_{\ant}\cap H\to G_{\ant}\times H$, $\Delta(z)=(z,z^{-1})$:
\[
  \xymatrix{
    1\ar[r]& G_{\ant}\cap H\ar[r] & G_{\ant}\times H\ar[rr]^-{
      m|_{_{G_{\ant}\times H}}} &&  G\ar[r]& 1.
  }
\]

In particular, $G=G_{\ant}H=HG_{\ant}\cong 
(G_{\ant}\times H)/(G_{\ant}\cap H)$.

\noindent (2) The restriction  $q|_{_{G_{\ant}}}:G_{\ant}\to A$ is
faithfully flat and  induces a closed sub-extension  of anti-affine
type of $\mathcal S$:  
\[\xymatrix{
 \mathcal S_{\ant}:\ar[d]&  1 \ar[r]& G_{\ant} \cap H \ar[r]\ar@{^(->}[d] &
  G_{\ant} \ar@{^(->}[d]\ar[r]^-{q|_{_{G_{\ant}}}}& A
  \ar@ {=}[d]\ar[r]&0\\
\mathcal S:&1 \ar[r]& H \ar[r]& G \ar[r]^{q}& A \ar[r]&0}
\]

\noindent (3) If moreover $G$ is connected of finite type, then:

\noindent (a) $(G_{\ant})_{\aff}\subset G_{\ant}\cap H$, and
  $(G_{\ant})_{\aff}$ is a  normal subgroup of finite index.

\noindent (b) If $\Delta=(\id_{(G_{\ant})_{\aff}},i|_{(G_{\ant})_{\aff}}): (G_{\ant})_{\aff}\to
G_{\ant}\times H $, $\Delta(z)=(z,z^ {-1})$,  is the diagonal embedding of
$(G_{\ant})_{\aff}$ in $G_{\ant}\times H$, then the quotient 
$G'=(G_{\ant}\times H)/\Delta\bigl((G_{\ant})_{\aff}\bigr)$  exists and it is a
group scheme  of finite type. Moreover, the 
canonical morphism  $f:G'\to G$ is an isogeny --- in the
terminology of Section \ref{section:indesp},
$G'=G_{\ant}\times^{(G_{\ant})_{\aff}}H$, the 
\emph{induced space} for the canonical action of $(G_{\ant})_{\aff}$ on $H$. 
\end{thm}
\pf
We first prove Assertion (1) in the case that  $G$ is connected, situation in
which one can follow closely the proof of \cite[Theorem
  5.1.1]{kn:brionchev}, where it is proved for the case that $G$ is a
smooth scheme of finite type --- the fact that the quotient $G/H=A$ is
an abelian variety allows this transcription.

Since $G_{\ant}$ is central in $G$, it is clear that  $m|_{_{G_{\ant}\times H}}: G_{\ant}\times H\to
  G$ is a quasi-compact morphism of group schemes, with
  $\Ker(m|_{_{G_{\ant}\times H}})=\Delta_{G_{\ant}\cap H}(G_{\ant}\cap H)=
  \bigl\{
  (h,h^{-1})\mathrel{:} h\in G_{\ant}\cap H\bigr\}$; therefore its
  schematic image  is the image of the closed immersion
  $(G_{\ant}\times H)/\Delta_{G_{\ant}\cap H}(G_{\ant}\cap H)$   by
  \cite[Corollaire V.3.3]{kn:dperrin1}.    
  Hence,  $G_{\ant}H$ is a
normal closed subgroup of $G$ and the quotient $G\to G/(G_{\ant}H)$
factors through  morphisms $G/H\to G/(G_{\ant}H)$ and $G/G_{\ant}\to
G/(G_{\ant}H)$
\[
  \xymatrix {G\ar[r]\ar[d]& A=G/H\ar[d]\\
  \operatorname{Aff}(G)=G/G_{\ant}\ar[r]& G/(G_{\ant}H)
  }
  \]

It follows that $G/(G_{\ant}H)$ is on the one hand a quotient of the
(connected) abelian variety $A$ --- and hence it is a connected
abelian variety --- and on  
the other hand  $G/(G_{\ant}H)$ is a quotient of an affine group scheme by a normal
closed subgroup --- and hence affine by \cite[VIb
11.17]{kn:SGA3-1}. We deduce that  $G/(G_{\ant}H)=\Spec(\Bbbk)$  and therefore
$G=G_{\ant}H$.

In order to prove Assertion (2) we can assume that $G$ is connected,
since $G_{\ant}\subset G^0$, and ${\xymatrixcolsep{1.5pc}\xymatrix{
    1\ar[r]& H\cap G^0\ar[r]& G^0\ar[r]^{q|_{_{G^0}}}&
    A\ar[r]& 0}}$ is an affine extension  (see Corollary \ref{cor:surimpliesflat}). 
Since $A=q(G)=q( G_{\ant}\cdot H)=q(G_{\ant})$, it
    follows that $A\cong G_{\ant}/(G_{\ant}\cap H)$, and hence Assertion (2)
    is proved.

In order to prove (1) in the non--connected situation, assume that
$G$ is not necessarily connected and  let $g:T\to G\in G(T)$ be a
$T$--point. Since $\mathcal S_{\ant}$ is an affine extension, there
exists a faithfully flat morphism $h:T'\to T$ and $b:T'\to
G_{\ant} \in G_{\ant}(T')$
such that $q\smallcirc b= q\smallcirc g\smallcirc h : T'\to A\in A(T')$. It
follows that $ (g\smallcirc h)b^{-1}\in H(T')$; therefore, $g\smallcirc
h \in G_{\ant}(T')H(T')\subset G_{\ant}H(T')$. But the morphism 
$m|_{_{G_{\ant}\times H}}:G_{\ant}\times H\to G$ is a morphism of
quasi-compact group
schemes and therefore has closed image $G_{\ant}H$. Applying 
 Lemma
\ref{lem:ffandpoints}, we deduce that $g\in G_{\ant}H(T)$. Hence
$G_{\ant}H=G$.\qed

The proof of part (a) of assertion (3) appears in \cite[Theorem
  5.1.1]{kn:brionchev}.
In order to prove   part  (b)  consider   $\mathcal G_{\ant}$, the
  Chevalley decomposition of $G_{\ant}$:
  \[
     \xymatrix{
   \mathcal G_{\ant}:&  0\ar[r]&     (G_{\ant})_{\aff}\ar[r]& G_{\ant}\ar[r]^{\widetilde{q}}& A\ar[r]& 0
 }
\]

Notice first that in general $\widetilde{q}\not= q$, since $(G_{\ant})_{\aff}\not=
G_{\ant}\cap H$. As observed in Remark \ref{rem:pushes}, since $(G_{\ant})_{\aff}$ is central in $G$ as well as in $H$, $\widetilde{q}$
induces an affine extension of finite type $\mathcal S'$:
${\xymatrixcolsep{1.5pc} \xymatrix{ 1\ar[r]& H \ar[r]& G'\ar[r]^{q'}&
    A\ar[r] & 0}}$. It is clear that the multiplication morphism $G_{\ant}\times H\to
G$ induces a morphism of group schemes $f: G'=(G_{\ant}\times H)/\Delta\bigl((G_{\ant})_{\aff}\bigr)\to G$,  with finite
Kernel $  \Ker (f)\cong
(G_{\ant}\cap H)/(G_{\ant})_{\aff}$.
\qed

\begin{nota}
   In view of Proposition \ref{prop:keretag}, Remark \ref{rem:gantqc} and Theorem \ref{thm:rosenftgs2},
   from now on if $ G$ is a 
quasi-compact group scheme, we denote $G_{\ant}=\Ker(\eta_G)$. 
\end{nota}

\begin{rem}
\label{rem:G=HG0}
  Let $\mathcal S$: $\ate$ be an affine extension. Then
  $G^0= (G^0\cap H)G_{\ant}$ and $G=HG^0=HG_{\ant}$. 

  Indeed, 
${\xymatrixcolsep{2pc} \xymatrix{1\ar[r]&H\cap
    G^0\ar[r]&G^0\ar[r]^(0.47){q|_{_{G^0}}}&\ar[r] A&0}}$ is a closed
sub-extension of $\mathcal S$ and the result follows from Theorem
\ref{thm:rosenftgs2}.
\end{rem}

\begin{rem}
  \label{rem:rosengen}
Let $G$ be a connected group scheme of finite type. Consider the
Chevalley decomposition of $G$: ${\xymatrixcolsep{1.5pc}
  \xymatrix{1\ar[r]&G_{\aff}\ar[r]&G\ar[r]^(.45){q}&\ar[r] A&0}}$.

In this situation we have that: the Rosenlicht decomposition of $G$ is
$G=G_{\aff}G_{\ant}$; $(G_{\ant})_{\aff} \subseteq G_{\aff}\cap
G_{\ant}$ is of finite index; and the product morphism
$m:G_{\ant}\times G_{\aff} \to G$ induces an isogeny $(G_{\ant}\times
G_{\aff})/\Delta\bigl((G_{\ant})_{\aff}\bigr)\to G$.

Hence, Theorem \ref{thm:rosenftgs2} can be viewed as a generalization
to affine extensions of the above well known decomposition (see
\cite{kn:brionchev}[Thm. 5.1.1]).
  
\end{rem}

 The following lemma is an easy consequence of Lemma \ref{lem:chevcase}.

\begin{lem}
\label{lem:antiafffitycase}
If $f:G\to G'$ is a morphism of quasi-compact group schemes, then
$f(G_{\ant})\subset G'_{\ant}$. If $G,G'$ are of finite type, then  the
morphism $f$ induces the following commutative diagram of Chevalley
decompositions:
\[
  \xymatrixcolsep{3pc}
\xymatrix{
\mathcal G_{\ant}:\ar[d]^\phi &0\ar[r] & (G_{\ant})_{\aff}\ar[r]\ar[d]_{f|_{_{(G_{\ant})_{\aff}}}}&
G_{\ant}\ar[r]^-{\widetilde{q}}\ar[d]^{f|_{_{G_{\ant}}}}&
A\ar[r]\ar[d]^{\widetilde{f}}&
0\\
\mathcal G'_{\ant}: &0\ar[r] & (G'_{\ant})_{\aff}\ar[r]&
G'_{\ant}\ar_-{\widetilde{q}'}[r]& A' \ar[r]& 0
}
\]

 Moreover, if $f$ is 
faithfully flat,  then  the vertical arrows of the
diagram above are also faithfully flat. In particular, if $f$
 is an affine faithfully flat morphism,  then $\widetilde{f}$ is
 an isogeny. 
\end{lem}

\pf
 The morphism  $\eta_{G'}\smallcirc f: G\to \operatorname{Aff}(G')$
factors through $\eta_G$. It follows that $f(G_{\ant})\subset
G'_{\ant}$.
  Since in the finite type case 
  $G_{\ant}$ is a smooth group scheme, all assertions follow from Lemma
  \ref{lem:chevcase}. 
\qed

\subsection{Filtered systems of affine  extensions}\ %
\label{section:invlim}
  
 We begin with some considerations about limits of filtered systems of
 schemes over $\Bbbk$. Recall that a \emph{filtered system of schemes} consists of
 a family $\bigl\{(X_\alpha,f_{\alpha,\beta})\mathrel{:} \alpha,\beta\in I\,,\ \alpha \geq
 \beta\,,\,f_{\alpha,\beta}:X_\alpha \to X_\beta\bigr\}$, where  $I$ is a finitely upper
 bounded --- or filtered --- poset, $\{X_\alpha\mathrel{:}\alpha \in I\}$
 is a family of schemes and $f_{\alpha,\beta}$ are morphisms of schemes
 --- called the \emph{connecting (or transition) morphisms} --- such
 that if $\alpha \geq \beta \geq \gamma$, then
 $f_{\alpha,\gamma}=f_{\beta,\gamma}\smallcirc f_{\alpha,\beta}$ and
 $f_{\alpha,\alpha}=\id_{X_\alpha}$.  If the
 limit of such a system of schemes
 exists, we use the following notation: $(X, f_\alpha: \alpha \in I)
 =\lim_{\alpha \in I} \{X_\alpha\,,\, f_{\alpha,\beta}: \alpha ,\beta
 \in I\}$ and when the rest of the ingredients are clear we write
 $X=\lim X_\alpha$. A family of morphisms $g_\alpha:Z \to
 X_\alpha\,,\,\alpha \in I$ is said to be \emph{compatible} with the
 filtered system $\bigl\{(X_\alpha,f_{\alpha,\beta})\bigr\}$ if for all $\alpha
 \geq \beta$, $f_{\alpha,\beta} g_\alpha=g_\beta$.

    Next, we recall some known properties of the limit in the above
    situation; the basic references are
    \cite[\S 8]{kn:EGAIV3} and \cite[Tag 01YT]{kn:stackproj}.

    \begin{rem}
\label{rem:surjectivity}
Let $\bigl\{(X_\alpha,f_{\alpha,\beta})\bigr\}_{\alpha,\beta\in I}$ be a filtered
system of schemes, and assume
that the transition morphisms $f_{\alpha,\beta}:X_\alpha\to X_\beta$ are
affine.  Then:

\noindent (1) $X=\lim X_\alpha$ exists in the category of schemes and the
   induced morphisms $f_\alpha$ are affine.  Moreover, if all the
   $X_{\alpha}$ are affine, so is the limit $X$ (see
   \cite[Tag 01YX]{kn:stackproj}).

   \noindent (2) If the transition morphisms  $f_{\alpha,\beta}$ are
   surjective for all $\alpha,\beta\in I$ then the
   morphisms $f_\alpha:\lim X_{\alpha} \to X_\alpha$ are surjective
   for all $\alpha\in I$
   --- this is an easy exercise on limits on the category of
   topological spaces.

   Moreover, if $Z$ is a quasi-compact $\Bbbk$--scheme and $g_\alpha:Z\to
X_\alpha$, $\alpha \in I$, is a family of compatible morphisms that
are also surjective, then the induced morphism $g:Z\to X$ is
surjective.

\noindent (3)   If  the transition morphisms $f_{\alpha,\beta}$ are
faithfully flat for all $\alpha,\beta\in I$, then the morphisms
$f_\alpha$ are faithfully flat 
for all $\alpha$.

Moreover, if $Z$ is  a quasi-compact scheme and $g_\alpha:Z\to X_\alpha$,
$\alpha \in I$, is a family of faithfully flat compatible
morphisms, then the induced morphism $g:Z\to X$ is faithfully flat.

\noindent (4) For any $\alpha \in I$ and
$U_\alpha \subseteq X_\alpha$ open subscheme, we have that
$f_\alpha^{-1}(U_\alpha)=\lim_{\beta \geq
  \alpha}f_{\beta,\alpha}^{-1}(U_\alpha)$
as schemes.
\end{rem}

\begin{nota}
  \label{nota:filteredsystem}
  In what follows we work with filtered systems in the  categories of
  group schemes over $\Bbbk$ and 
  $\GextaffA$,   with the additional condition that the transition
  morphisms  are affine faithfully flat. We introduce for clarity
    the following notations.

  \noindent (1) A filtered system of group schemes  $\{G_\alpha,
  f_{\alpha,\beta}\}$ is \emph{affine} (resp.~\emph{faithfully flat}) if the transition
  morphisms are affine (resp.~faithfully flat) morphisms of group schemes.

  \noindent (2) A \emph{filtered system of affine extensions} is a
  family     $\bigl\{(\mathcal S_\alpha, \phi_{\alpha,\beta})\mathrel{:} \alpha, \beta \in
      I\bigr\}$, where $I$ is a filtered poset, such that for all $\alpha \leq
      \beta$    
        the transition morphisms  $\phi_{\alpha,\beta}$
 \[
     \xymatrix{ \mathcal S_\alpha:\ar[d]_{\phi_{\alpha,\beta}} &
        1\ar[r] & H_{\alpha}\ar[d] \ar[r] & G_{\alpha}
        \ar[r]^{q_{\alpha}}\ar[d]^{f_{\alpha,\beta}} & A\ar@{=}[d]
        \ar[r] & 0 \cr \mathcal S_\beta: & 1\ar[r] & H_{\beta} \ar[r]
        & G_{\beta} \ar[r]_{q_{\beta}} & A \ar[r] & 0 }
      \]
are morphisms of affine extensions, and  the family $\{G_\alpha,
f_{\alpha,\beta}\}$ is a filtered system of group schemes --- notice
that in particular this implies that
$\phi_{\alpha,\gamma}=\phi_{\beta,\gamma}\smallcirc \phi_{\alpha,\beta}$ and
 $\phi_{\alpha,\alpha}=\id_{\mathcal S_\alpha}$ for all $\alpha \geq
 \beta \geq \gamma$.

 A filtered system of affine extensions is \emph{affine}
 (resp.~\emph{faithfully flat})  if the filtered system of group schemes $\{G_\alpha,
f_{\alpha,\beta}\}$ is so --- that is, in accordance with Definition
      \ref{defn:catextensions}, the transition
morphisms $\phi_{\alpha,\beta}$  are affine, faithfully flat. 
\end{nota}

\begin{rem}\label{rem:particularlimit}
 It follows from Remark \ref{rem:fisaffine} that a filtered system of
affine extensions is always affine.
\end{rem}

\begin{prop} \label{prop:flatness} 
  Let $\bigl\{ (\mathcal S_\alpha, \phi_{\alpha,\beta}:\mathcal
  S_\alpha\to \mathcal S_\beta)\mathrel{:}\alpha \geq \beta \in
  I\bigr\}$ be an (affine) faithfully flat 
  filtered system in $\GextaffA$ (see Notation \ref{nota:filteredsystem}). Then:
 
\noindent (1) The limits $G:=\lim G_\alpha$ and $H:=\lim H_\alpha$
exist in the category of group schemes and $H$ is affine. Moreover, the
morphisms $f_\alpha: G \to G_\alpha$ are affine and faithfully flat.

\noindent (2) If we call $q:G \to A$ the morphism induced by the
compatible morphisms  $q_\alpha:G_\alpha \to A$, then $\Ker(q)=H$ and the sequence
$\mathcal S:$ $\ate$ coincides with $\lim \mathcal
S_\alpha$ in $\GextaffA$.
\end{prop}
\pf (1) The existence of the limits in the category of schemes and the
affineness of $H$ as well as the assertions about the maps $f_\alpha$
follow from Remark \ref{rem:surjectivity}, taking also
into account \cite[Proposition II.1.3]{kn:dperrin1} (see Remark \ref{rem:forequiv}).
   
\noindent (2)  First observe  that $q:G \to A$ is affine and faithfully
flat, since $q=q_\alpha f_\alpha$ with $q_\alpha$ and $f_\alpha$
affine and faithfully flat.

Next, we prove that the sequence $\mathcal S$ is left exact, that is
$\Ker(q)=H$.  The commutative diagrams 
\[
{\xymatrixcolsep{1.2pc}
\xymatrixrowsep{1.2pc}
\xymatrix{
H\ar[r]\ar[rd] & H_\alpha\ar[r]\ar[d] & G_\alpha \ar[d]\cr
 & H_\beta\ar[r] & G_\beta
}}
\] 
induce an injective morphism of group schemes $H\to G$. As $H_\alpha =
\Ker(q_\alpha)$ for all $\alpha$, then $H\subset\Ker(q)$. Moreover, if
$h:K\to G$ is a morphism of group schemes such that $q\smallcirc h=0$,
then $q_\alpha\smallcirc f_\alpha\smallcirc h = 0$ for all $\alpha$.
Therefore, 
$\im(f_\alpha\smallcirc h)\subset H_\alpha$ for all $\alpha$ and
$f_{\alpha,\beta}\smallcirc f_\alpha\smallcirc h = f_\beta\smallcirc h$,
so $h :K\to G$ factors through $H$ and the proof of the left exactness is finished.

Thus, $\mathcal S\in \GextaffA$ and for all $\alpha\in I$ we have
morphisms of affine extensions $\phi_\alpha:\mathcal S\to \mathcal
S_\alpha$, compatible with the transition morphisms
$\phi_{\alpha,\beta}$: 
\[
{\xymatrixrowsep{1.5pc} 
\xymatrix{ 
\mathcal S:\ar[d]^(.4){\phi_\alpha}&1\ar[r] & H\ar[d]^(.4){f_\alpha|_{_H}} \ar[r] & G \ar[r]^{q}\ar[d]^(.4){f_\alpha} & A\ar@{=}[d] \ar[r] & 0 \cr
\mathcal S_\alpha: & 1\ar[r] & H_\alpha \ar[r] & G_{\alpha} \ar[r]_{q_{\alpha}} & A \ar[r] & 0
}}
\]
where $f_\alpha$, $\alpha \in I$, are the morphisms  associated to
$G=\lim_\alpha G_\alpha$.

In order to prove that $(\mathcal S,\phi_\alpha)=\lim \mathcal
S_\alpha$, let $\mathcal S'$: ${\xymatrixcolsep{1.5pc}\xymatrix{ 1\ar[r]&
    H'\ar[r]& G'\ar[r]^-{q'}& A\ar[r]& 0}}$ be an affine extension and 
$\phi'_\alpha: \mathcal S' \to \mathcal S_\alpha$ be a compatible
family of morphisms of affine extensions. Then we have  compatible
families  of morphisms $f'_\alpha:G'\to G_\alpha$ and
$f'_\alpha|_{_{H'}}: H'\to H_\alpha$ , for the filtered
systems $\{G_\alpha, f_{\alpha,\beta}\}$ and  $\{H_\alpha,
f_{\alpha,\beta}|_{_{H_\alpha}}\}$ respectively,  that factors through
morphisms $f':G'\to G$ and $f'':H'\to H$. It is easy to see that
$f''=f'_{_{H'}}$, and therefore $f'$ induces a morphism of extensions
$\phi':\mathcal S'\to \mathcal S$, that factors the morphisms
$\phi'_{\alpha}$. \qed

\begin{defn}
\label{defn:proalg}
 An affine  extension $\mathcal S\in \GextaffA$  is called  \emph{pro-algebraic}
if there exists an (affine) faithfully flat  filtered system  of
 affine  extensions of finite type $\{ \mathcal S_\alpha, \phi_{\alpha,\beta}
 \mathrel{:}\alpha,\beta\in I\}$ such that such that $\mathcal S=\lim
 \mathcal S_\alpha$. 
\end{defn}

  \begin{rem}
The term \emph{pro-algebraic} has its roots in  the fact that usually
group schemes of finite type are called algebraic groups, see Remark
\ref{rem:alggrps}. 
  \end{rem}

\begin{ejs}
\label{ej:affiisproalg}
(1) Any affine group scheme $G$ is the limit of an affine faithfully flat filtered system of
 affine   group
schemes of finite type (see  for example
\cite[Page 24]{kn:waterhouse}). In terms of affine
 extensions, this well known result reads as follows:  let
$G=\lim G_\alpha$ and consider
the affine  extensions 
$\mathcal G_{\aff}$: ${\xymatrixcolsep{1.5pc} \xymatrix{ 1\ar[r] &
    G\ar[r]^(.4){\id} &G\ar[r] &0\ar[r] &0}}$, and $\mathcal G_{\alpha,\aff}$: $\xymatrixcolsep{1.5pc} \xymatrix{1\ar[r] &
    G_\alpha\ar[r]^(.4){\id} &G_\alpha\ar[r] &0\ar[r] &0}$. Then $\mathcal G_{\aff}=\operatorname{lim} \mathcal G_{\alpha,\aff}$.

\noindent (2) Let $\mathcal S:$ $\ate$ be an affine 
  extension, with $H$ central. Then $H=\lim H_\alpha$, where
  $\{H_{\alpha}\}_{\alpha\in I}$ is an (affine) faithfully flat
  filtered system of affine 
group schemes of finite type; denote the transition morphisms by
$f_{\alpha,\beta} :H_\alpha \to H_\beta$ and the canonical
projections by $f_\alpha:H\to H_\alpha$. Since $H$ is central and
the canonical maps are faithfully flat, we can apply   Remark
\ref{rem:pushes} and construct the push-forward by $f_\alpha$,
obtaining an affine faithfully flat filtered system of affine extensions as follows:
\[
  \xymatrix{ (f_{\alpha})_*(\mathcal S):\ar[d]^{\phi_{\alpha,\beta}} & 1\ar[r] & H_\alpha \ar[r]\ar[d]_{f_{\alpha,\beta}}
    &G\times^HH_\alpha\ar[r]^-{\pi_{H_\alpha}}\ar[d]^{\overline{f_{\alpha,\beta}}}
    & A\ar[r] \ar@{=}[d]&0\\
  (f_{\beta})_*(\mathcal S): & 1\ar[r] & H_\beta \ar[r]
  &G\times^HH_\alpha\ar[r]^-{\pi_{H_\alpha}} & A\ar[r] &0
}
\]
where $G\times^HH_\alpha=(G\times H_\alpha)/\Delta(H)$  and
$\overline{f_{\alpha,\beta}}: G\times^HH_{\alpha}\to 
G\times^HH_{\beta}$ is the morphism induced by $G\times H_\alpha\to
G\times^HH_{\beta}$, $(g,h_\alpha)\mapsto
\bigl[g,f_{\alpha,\beta}(h_\alpha)\bigr]$ --- notice that if $h\in
H$, then $\bigl[gh^{-1},f_{\alpha,\beta}(f_\alpha(h)h_\alpha)\bigr]= 
\bigl[gh^{-1},f_\beta(h)f_{\alpha,\beta}(h_\alpha)\bigr]=\bigl[g,f_{\alpha,\beta}(h_\alpha)\bigr]\in
G\times^HH_\beta$.  
 
Thus, if $H$ is central, the affine extension $\mathcal S$ is
pro-algebraic. In particular, commutative affine extensions are
pro-algebraic.
\end{ejs}

Let $\mathcal S$:
$\ate$ be an affine extension; then $G$ is a quasi-compact group
scheme (see Remark \ref{rem:torsoraffine1}). Conversely, Perrin proved
in \cite[Corollary V.4.3.1]{kn:dperrin1} that if $G$ is a 
\emph{connected} quasi-compact group scheme, then $G$ fits into an
affine  extension. This result is a
consequence of the Chevalley decomposition of group schemes of finite
type (Theorem \ref{thm:chevftgs}), together with Perrin's
Approximation Theorem below, that in particular shows that any
quasi-compact connected group scheme $G$ is pro-algebraic, that is $G$
is the limit of a directed system of group schemes of finite type.

\begin{thm}[Perrin's Approximation Theorem, {\cite[Th\'eor\`eme
V.3.1]{kn:dperrin1}}]
\label{thm:perrinquasicomp}
Let $G$ be a quasi-compact group scheme. Then there exists an affine faithfully flat filtered system of group schemes of finite type
$\{G_\alpha, f_{\alpha,\beta}\mathrel{:}\ alpha,\beta\in I\}$, such
that  $G \cong \operatorname{lim}_\alpha G_\alpha$. In particular, the
canonical morphisms $f_\alpha:G\to G_\alpha$   are faithfully flat,
with $\Ker(f_\alpha)$ an affine closed subgroup.\qed
\end{thm}

Combining Theorem \ref{thm:rosenftgs2} with the well known
approximation theorem for affine group schemes, we can refine
Perrin's approximation Theorem and show that  any affine extension is pro-algebraic.

\begin{thm}
  \label{thm:affextisproalg}
  Let $\mathcal S$: $\ate$ be an affine extension. Then $\mathcal S$ is pro-algebraic.
  \end{thm}
  \pf
  Let $G=G_{\ant}H=G_{\ant}\times^{G_{\ant}\cap H}H$ be the Rosenlicht
  decomposition of $\mathcal S$ (Theorem \ref{thm:rosenftgs2}). Since
  $H$ is an affine scheme, there exists an affine faithfully flat
  filtered system $\{H_\alpha, p_{\alpha,\beta}\mathrel{:}
  \alpha,\beta\in I\}$,  such that $H=\lim H_\alpha$; let 
  $p_\alpha:H\to H_\alpha$ be the canonical 
  morphisms (of group schemes, faithfully flat). Since $G_{\ant}$ is central, $G_{\ant}\cap H$ is central
 in $H$; it follows that $p_\alpha(G_{\ant} \cap H)$ is central in
 $H_\alpha$. By Remark \ref{rem:pushes}, we have  morphisms of affine extensions
 \[
   \xymatrix{
  \mathcal S_{\ant}\ar[d]^{\phi_\alpha}&   1\ar[r]&  G_{\ant}\cap H\ar[r]\ar[d]_{p_\alpha}&  G_{\ant}\ar[d]_{f_\alpha}\ar[r]&
     A\ar[r]\ar@{=}[d]& 0\\
\mathcal S_\alpha  &   1\ar[r]&  H_\alpha\ar[r]&  G_{\ant}\times^{G_{\ant}\cap H}H_\alpha\ar[r]&
     A\ar[r]& 0
   }
   \]
 where $f_\alpha(z)=[z,1]$ for all $z\in G_{\ant}$ (see Example
 \ref{ej:affiisproalg}). These morphisms 
 clearly extend to morphisms
 \[
   \xymatrix{
\mathcal S \ar[d]^{\lambda_\alpha} &    1\ar[r]&
H\ar[r]\ar[d]_{p_\alpha}&  G=G_{\ant}H\ar[d]_{\ell_\alpha}\ar[r]& 
     A\ar[r]\ar@{=}[d]& 0\\
\mathcal S_\alpha&     1\ar[r]&  H_\alpha\ar[r]&  G_{\ant}\times^{G_{\ant}\cap H}H_\alpha\ar[r]&
     A\ar[r]& 0
   }
   \]
where $\ell_\alpha:G\to  G_{\ant}\times^{G_{\ant}\cap H}H_\alpha$ is the
morphism of group schemes induced by 
$\pi_{G_{\ant}\times H_\alpha}\smallcirc (\id_{G_{\ant}}\times
  p_\alpha):G_{\ant}\times H\to G_{\ant}\times H_\alpha \to
  G_{\ant}\times^{G_{\ant}\cap H}H_\alpha$. 

Moreover, the morphisms $\id_{G_{\ant}} \times p_{\alpha\beta}$ induce
affine, faithfully flat morphisms of group schemes
$\ell_{\alpha,\beta}:G_\alpha= G_{\ant}\times^{G_{\ant}\cap
  H}H_\alpha\to G_\beta = G_{\ant}\times^{G_{\ant}\cap H}H_\beta$.
Therefore, the family $\{ \mathcal S_\alpha, \ell_{\alpha,\beta}\}$
conforms an affine faithfully flat filtered system of affine extensions of finite type;
let $\mathcal G$ be its limit.

We prove now that $\mathcal G\cong \mathcal S$. Indeed, the morphisms
$\lambda_\alpha$ induce a faithfully flat morphism
\[
   \xymatrix{
 \mathcal S\ar[d]^{\lambda} &     1\ar[r]&   H\ar[r]\ar[d]_{\ell|_{_H}}&  G=G_{\ant}H\ar[d]_{\ell}\ar[r]&
     A\ar[r]\ar@{=}[d]& 0\\
 \mathcal G &     1\ar[r]&  \widetilde{H} \ar[r]&  \widetilde{G} \ar[r]&
     A\ar[r]& 0
   }
 \]
 
But by construction $\widetilde{H}=H=\lim H_\alpha$ and
$\ell|_{_H}=\id_H$. Moreover, by the
commutativity of the diagram above, $\operatorname{Ker}(\ell)\subset
H$. Hence, $\ell$ is injective and $\mathcal S\cong \mathcal G$.
\qed

Once we have established that any affine extension is pro-algebraic,
we can estate Rosenlicht decomposition (Theorem \ref{thm:rosenftgs2})
in terms of affine sub-extensions.

\begin{lem}\label{lem:minimal}
 Let $\bigl\{\mathcal S_\alpha, \phi_{\alpha,\beta}: \mathcal
 S_\alpha\to \mathcal S_\beta; \alpha \geq \beta \bigr\}$, be an
 (affine) faithfully flat 
   filtered system of affine extensions of finite type, and assume
   that all the extensions are Chevalley decompositions. Consider the
   limit $\mathcal S:=\lim \mathcal S_\alpha$
  \[
  \xymatrix{\mathcal S: & 1\ar[r]&H\ar[r]&G\ar[r]^-q&\ar[r]
    A&0.}
\]

Then $H$ is minimal among the affine
subgroups schemes $H'\subset G$ such that the quotient scheme $G/H'$ exists and is proper.
\end{lem}

\pf Let $H'\subset H$ be an affine subgroup scheme of
  $H$  such that $G/H'$ exists and is proper. If $f_\alpha:G\to
  G_\alpha$ denotes   as usual the mid morphism of $\phi_\alpha$
  (recall the notations in 
\ref{nota:filteredsystem} and \ref{nota:central}), then   $f_\alpha(H')$,
  the scheme theoretic image of $H'$ by $f_\alpha$, is a
  closed affine subscheme of $G_{\alpha,\aff}$. Since $f_\alpha$ is faithfully
  flat for all $\alpha$,  then $f_{\alpha}|_{_H}: H\to G_{\alpha,
    \aff}$ is faithfully flat (see Lemma \ref{lem:chevcase}), and we can
  factor $f_\alpha$  
  to a faithfully flat  morphism $\overline{f}_\alpha 
  : G/H'\to G_\alpha/f_\alpha(H')$.

Since $G/H'$ is proper, it follows that $G_\alpha/f_\alpha(H')$ is
 also proper, and therefore, by minimality of
$G_{\alpha,\aff}$, $f_\alpha(H')
=G_{\alpha,\aff}$. It follows that  $\{
f_\alpha|_{_{H'}}:H'\to G_{\alpha,\aff}\}$ is a compatible family of
faithfully flat morphism, with $\operatorname{inc}:H'\to H=\lim
G_{\alpha,\aff}$ as induced morphism. We deduce from that
$H'=H$, since $\operatorname{inc}$ is a faithfully flat morphism
(see Remark \ref{rem:surjectivity}).  \qed

\begin{lem}
\label{lem:affiandlim}
The affinization functor $\operatorname{Aff}:\Schkqc\to \Schkaff$
preserves limits of affine  faithfully flat filtered systems of
quasi-compact group schemes. 
 \end{lem}
\pf Let $\{G_\alpha$, $f_{\alpha,\beta}:G_\alpha\to G_\beta, \alpha
\geq \beta \in I\}$ be an affine faithfully flat filtered system of quasi-compact group
schemes, with limit the quasi-compact group scheme $G=\lim
G_\alpha$. Then the transition morphisms induce morphisms
$\widetilde{f_{\alpha,\beta}}=\operatorname{Aff}(f_{\alpha,\beta}):
\operatorname{Aff}(G_\alpha)\to \operatorname{Aff}(G_{\beta})$, such
that the solid part of the following diagram is commutative:
  \begin{equation}\label{eqn:etalim}
    \xymatrix{G \ar@{..>}[r]^{f_\alpha}\ar@{..>}[d]_{\lim \eta_{G_\alpha}}&
      G_\alpha\ar[r]^{f_{\alpha,\beta}}\ar[d]^{\eta_{G_\alpha}}&
        G_\beta\ar[d]^{\eta_{G_\beta}}\\\lim \operatorname{Aff}(G_{\alpha})\ar@{..>}[r]_(.5){\widetilde{f_\alpha}}&
         \operatorname{Aff}(G_\alpha)\ar[r]_{\widetilde{f_{\alpha,\beta}}}
           &
        \operatorname{Aff}(G_\beta)
      }
      \end{equation}

      Since the morphisms $\eta_{G_\alpha}, \eta_{G_\beta}$ and
      $f_{\alpha,\beta}$ are faithfully flat (see
      Prop. \ref{prop:keretag}), it follows that
      $\widetilde{f_{\alpha,\beta}}$ is faithfully flat. Thus, the
      limit $\lim \operatorname{Aff}(G_{\alpha})$ exists and it is an
      affine group scheme fitting into the left part of diagram
      \eqref{eqn:etalim}. Moreover, the (faithfully flat) morphisms
      $\eta_{G_\alpha}\smallcirc f_{\alpha}:G\to
      \operatorname{Aff}(G_{\alpha})$ induce a faithfully flat
      morphism $f: G\to L$ (see Remark \ref{rem:surjectivity}), that
      factorizes through $\eta_G$, by the universal property of the
      affinization morphism as shown in the diagram below.
\begin{equation}\label{eqn:etalim2}
    \xymatrix{& G\ar[ld]_{\eta_G} \ar[r]^{f_\alpha}\ar[d]^{\lim
        \eta_{G_\alpha}}& G_\alpha\ar[d]^{\eta_{G_\alpha}}
      \\ \operatorname{Aff}(G)\ar[r]_(0.38){f}&\lim
      \operatorname{Aff}(G_{\alpha})\ar[r]_(.5){\widetilde{f_\alpha}}&
      \operatorname{Aff}(G_\alpha) }
      \end{equation}

      It is then clear that $f_{\alpha}\bigl(\operatorname{Ker}(f\smallcirc
\eta_G)\bigr)\subset G_{\alpha,\ant}$. Therefore, $\operatorname{Ker}(f\smallcirc
\eta_G) \subseteq G_{\ant}$ and as the other inclusion is evident, it follows that  $\lim \operatorname{Aff}(G_{\alpha}) = G/G_{\ant}=\operatorname{Aff}(G)$.
\qed

\begin{thm}
  \label{thm:limantiaff}
  Let $\mathcal S=\lim \mathcal S_\alpha$: $\ate$ be an affine
  extension, with $\mathcal S_\alpha$: ${\xymatrixcolsep{1.5pc}\xymatrix{ 1\ar[r]&
      H_\alpha \ar[r]& G_\alpha\ar[r]^{q_{\alpha}}&  A\ar[r]&
      0}}$ an (affine) faithfully flat filtered system of finite type in $\GextaffA$, with  transition morphisms   $\phi_{\alpha,\beta}:\mathcal S_\alpha\to \mathcal S_\beta$. Let us call 
$G_{\alpha,\ant}=\Ker(\eta_{G_{\alpha}})\subset
G_\alpha$ (see Remark \ref{rem:gantqc} and Theorem \ref{thm:rosenftgs2}). Then
the morphisms 
  $f_{\alpha,\beta}|_{_{G_{\alpha,{\ant}}}}$ define an affine
  faithfully flat filtered system
  for the family $\{\mathcal S_{\alpha,{\ant}}\}_{\alpha \in I}$
\[
  \xymatrix{
  \mathcal S_{\alpha,\ant}:\ar[d]^{\widetilde{\phi}_{\alpha,\beta}}
  & 1\ar[r]& G_{\alpha,\ant}\cap
  H_\alpha\ar[r]\ar[d]_{f_{\alpha,\beta}|_{_{G_{\alpha,\ant}\cap H}}}&
  G_{\alpha,\ant}\ar[rr]^{q_\alpha|_{_{G_{\alpha,\ant}}}}\ar[d]^{f_{\alpha,\beta}|_{_{G_{\alpha,{\ant}}}}}&
  & A\ar[r]\ar@{=}[d]&  0\\ 
  \mathcal S_{\beta,\ant}: 
  & 1\ar[r]& G_{\beta,\ant}\cap H_\beta\ar[r]&
  G_{\beta,\ant}\ar[rr]_{q_\beta|_{_{G_{\beta,\ant}}}}& &A\ar[r]&  0
  }
\]
with limit $\mathcal S_{\ant}=\lim \mathcal S_{\alpha,\ant}$: ${\xymatrix {
    1\ar[r]& H\cap G_{\ant}\ar[r]&  G_{\ant}\ar[r]^{q|_{_{G_{\ant}}}}& A\ar[r]& 0}}$.
\end{thm}
\pf First we prove that in the above context, $\lim G_{\alpha,\ant}=
G_{\ant}$.   By Lemma
\ref{lem:antiafffitycase}, $f_{\alpha,\beta}(G_{\alpha,\ant})\subset
G_{\beta,\ant}$ for all $\alpha \geq \beta$ and
$f_{\alpha,\beta}|_{_{G_{\alpha,\ant}}}$ is an affine  faithfully flat
morphism; thus  the limit  $\lim
G_{\alpha,\ant}$ exists (see  Remark \ref{rem:surjectivity}); let $L=\lim
G_{\alpha,\ant}$ and 
$\widetilde{f_\alpha}:L\to G_{\alpha,\ant}$ be the canonical morphisms. Then the family
$\widetilde{f_{\alpha}}$ induces a morphism $\widetilde{f}: L\to G$.

On the
other hand, since  $f_\alpha(G_{\ant})\subset G_{\alpha,\ant}$, it
follows that there exists a morphism $h:G_{\ant}\to L$, such
that $\widetilde{f_\alpha}\smallcirc h=f_\alpha|_{_{G_{\ant}}}$
for all $\alpha$. By the universal property of the limit $G=\lim
G_\alpha$, we deduce that $\widetilde{f}
\smallcirc h = \operatorname{inc}: G_{\ant} \subseteq G$. Hence, it
suffices to prove that  $L$ is anti-affine, since if this is  the case
then  $\widetilde{f}(L)$ is anti-affine and therefore
$ \widetilde{f}(L)\subset G_{\ant}$. Then  $G_{\ant}=L=\lim G_{\alpha,\ant}$.
But applying Lemma
\ref{lem:affiandlim} we see that $\operatorname{Aff}(\lim
G_{\alpha,\ant})=\lim \operatorname{Aff}(G_{\alpha,\ant})=\Spec(\Bbbk)$.

Now we prove that $G_{\ant} \cap H=\lim G_{\alpha,\ant} \cap
H_\alpha$. Since $f_{\alpha,\beta}|_{_{H_\alpha}}:H_\alpha\to H_\beta$
is an affine 
morphism (see Remark \ref{rem:particularlimit}), it follows that
the family $\{ G_{\alpha,\ant}\cap
  H_\alpha, f_{\alpha,\beta}|_{_{G_{\alpha,\ant}\cap
  H_\alpha}}:G_{\alpha,\ant}\cap H_\alpha\to G_{\beta,\ant}\cap
H_\beta\mathrel{:} \alpha,\beta\in I\}$ is an affine filtered system
of group schemes.  Then the limit  $N=\lim G_{\alpha,\ant}\cap
H_\alpha$ is a group scheme, and the restriction morphisms
$f|_{_{G_{\ant}\cap H}}:G_{\ant}\cap H\to G_{\alpha, \ant}\cap
H_\alpha$ induce a  morphism $\ell: G_{\ant}\cap H\to
N$.  But it is clear that $N\subset G_{\ant}\cap H$ --- since
$0=q_\alpha\smallcirc f_\alpha : N\to A$ for all $\alpha$ and that $
q_\alpha(N)\subset G_{\alpha,\ant}$ ---; therefore $N=G_{\ant}\cap H$.
\qed

As a consequence of Theorem \ref{thm:limantiaff}, we have the
following result.

\begin{lem}
Let $\mathcal S=\operatorname{lim} \mathcal S_\alpha$, $\mathcal
S'=\operatorname{lim} 
\mathcal S'_\alpha$ be two  affine  extensions, where $\mathcal
S_\alpha, \mathcal S'_\alpha$ are of finite type, 
and $\phi:\mathcal S\to \mathcal S'$ a morphism of affine  extensions: 
\[
{
\xymatrixrowsep{1.5pc}
\xymatrix{
\mathcal S:\ar[d]_(.4){\phi} &0\ar[r] & H\ar[r]\ar[d]& 
G\ar[r]^q\ar[d]^(.4) f& A\ar@{=}[d]\ar[r]& 0\\ 
\mathcal S':& 0\ar[r] & H'\ar[r]& 
G'\ar[r]_{q'}& A\ar[r]& 0
}}
\]

 Then $ f$ induces by restriction a morphism of affine extensions
\[
  \xymatrix{
    \mathcal S_{\ant}:\ar[d]& 0\ar[r]&
    G_{\ant}\cap H\ar[r]\ar[d]_{ f|_{_{G_{\ant}\cap H}}}&
    G_{\ant}\ar[r]^-{q|_{_{G_{\ant}}}}\ar[d]^{ f|_{_{G_{\ant}}}}& 
    A\ar[r]\ar@{=}[d]& 0\\
 \mathcal S'_{\ant}: & 0\ar[r]&  G'_{\ant}\cap H'\ar[r]& G'_{\ant}\ar[r]_-{q'|_{_{G_{\ant}}}}&
    A \ar[r]& 0
  }
\]

 Moreover, if $ f$ is
 faithfully flat  then  $ f|_{_{G_{\ant}}}:
G_{\ant}\to G'_{\ant}$ is faithfully flat.
\end{lem}

\pf
By Lemma \ref{lem:antiafffitycase}, we have that $f(G_{\ant}\subset
G'_{\ant}$ and that $f|_{_{G_{\ant}}};
G_{\ant}\to G_{\ant}$ is a faithfully flat morphism.   
\qed

\begin{thm}[Rosenlicht decomposition of  affine
  extensions, revisited]\label{thm:rossmoorhnew} \ \\
Let $\mathcal S$: $\ate$ be an affine 
extension,  with $G$ connected. Let $\{\mathcal S_\alpha,
\phi_{\alpha,\beta}\mathrel{:}\alpha,\beta\in I\}$ be an (affine)
faithfully flat filtered system of affine extensions of finite type, with  
 $\mathcal S_\alpha$:
${\xymatrixcolsep{1.5pc}\xymatrix { 1\ar[r]&
    H_\alpha\ar[r]&G_\alpha\ar[r]^{q_\alpha}&A\ar[r]&0}}$.
Then:

 \noindent (1) Let $\mathcal G_{\alpha,\ant}$:
  ${\xymatrixcolsep{1.5pc}\xymatrix { 1\ar[r]&
      (G_{\alpha,\ant})_{\aff}\ar[r]&G_{\alpha,\ant}\ar[r]^(.55){\widetilde{q_\alpha}}&A\ar[r]&0}}$
  be the Chevalley decomposition of $G_{\alpha, \ant}$, for $\alpha\in
  I$. Then $G_{\ant}\cap H$ contains $K=\lim (G_{\alpha,\ant})_{\aff}
  $ as a closed subgroup scheme, and $(G_{\ant}\cap H)/K$ is finite. 
  
\noindent (2) The induced space 
$G'=G_{\ant}\times^KH$ is a quasi-compact group scheme, and the
canonical morphism (induced by the multiplication) $f:G'\to G$ is
an isogeny, and  $\Ker(f)=(G_{\ant}\cap H)/K$.
\end{thm}
\pf
First, observe that since $G_{\ant}=\lim G_{\alpha,\ant}$ (by Theorem
\ref{thm:limantiaff}) and $H=\lim 
H_\alpha$ (by Proposition \ref{prop:flatness}), it follows that
$G_{\ant}\cap H=\lim G_{\alpha,\ant}\cap H_\alpha$. 
Let $\nu_\alpha:K\to (G_{\alpha, \ant})_{\aff}$ be the
canonical morphisms. Then, by Theorem \ref{thm:rosenftgs2} the family
$\{\nu_\alpha\}$ induces a compatible family $\{ \widetilde{\nu}_\alpha: K\to
G_{\alpha,\ant}\cap H_\alpha\}$, with in turn induces a  morphism
$\nu:K\to G_{\ant}\cap H$.  In order to prove that $K\subset
G_{\ant}\cap H$ is a closed subgroup of finite type, observe that
$\Ker (\nu)=\nu^{-1}\bigl( \cap_\alpha \Ker (f_\alpha|_{_{H\cap G_{\ant}}})\bigr)$ (again by  Theorem \ref{thm:limantiaff}
and Proposition \ref{prop:flatness}). It  follows from the
compatibility conditions that $\Ker(\nu)= \bigcap_\alpha\bigl( \Ker(
(\nu_\alpha|_{_{H\cap G_{\ant}}})=\Spec(\Bbbk)$.

In order to prove that $K$ is of finite index in $G_{\ant}\cap H$,  we
follow the procedure presented in \cite[Theorem 5.1.1]{kn:brionchev}.

Since the filtered system $\bigl\{(G_{\alpha,\ant})_{\aff},
f_{\alpha,\beta}|_{_{G_{\alpha,\ant})_{\aff}}}\bigr\}$ is affine and
faithfully flat (see Lemma \ref{lem:antiafffitycase}), we have the
following   commutative diagrams of groups schemes,  with exact sequences
of group schemes as rows and affine faithfully flat vertical arrows,
for any 
$\beta\geq \alpha$ --- notice that by Theorem \ref{thm:rosenftgs2}
$\overline{G_{\alpha}}$ is a finite group scheme ---:  
\[
{\xymatrixrowsep{1.5pc}
\xymatrix
{
1\ar[r]&
(G_{\alpha,\ant})_{\aff}\ar[r]\ar[d]_{f_{\alpha,\beta}|_{_{(G_{\alpha,\ant})_{\aff}}}}&
G_{\alpha}\ar[rr]^-{\overline{q_{\alpha}}}\ar[d]^{f_{\alpha,\beta}}&& 
\overline{G_{\alpha}}= G_\alpha/
(G_{\alpha,\ant})_{\aff}\ar[r]\ar[d]^{\overline{f_{\alpha,\beta}}}&
1\\
1\ar[r]& (G_{\beta,\ant})_{\aff}\ar[r]&G_{\beta}\ar[rr]^-{\overline{q_{\beta}}}&&
\overline{G_{\beta}}=G_\beta/ (G_{\beta,\ant})_{\aff}\ar[r]& 1
}}
\]

Taking limits, we deduce that $G/K\cong \lim \overline{G_\alpha}$ ---
since the compatible morphisms $\widetilde{q_\alpha}\smallcirc
f_\alpha: G\to \overline{G_\alpha}$   induce a faithfully flat morphism
$G\to \lim \overline{G_\alpha}$ with Kernel equal to $K$. Thus we have
the following commutative diagram of groups schemes, with exact
sequences as rows and  faithfully flat vertical arrows:
\[
{\xymatrixrowsep{1.5pc}
\xymatrix
{
1\ar[r]& K\ar[r]\ar@{->>}[d]_{f_{\alpha}|_{_K}}&G\ar[r]\ar@{->>}[d]^{f_{\alpha}}& G/K\ar@{->>}[d]^{\overline{f_{\alpha}}}\ar[r]& 1\\
1\ar[r]& (G_{\alpha,\ant})_{\aff}\ar[r]&G_{\alpha}\ar[r]&
\overline{G_{\alpha}}=G_\alpha/ (G_{\alpha,\ant})_{\aff}\ar[r]& 1
}}
\]

Sine $G=G_{\ant}H$ (see Theorem \ref{thm:rosenftgs2}), it follows that
$G/K=(G_{\ant}H)/K= (G_{\ant}/K)(H/K)$.  
Since $H/K$
is the limit of the affine  group schemes of finite type
$H_{\alpha}/(G_{\alpha, \ant})_{\aff}$, it is an affine group
scheme. On the other hand, $G_{\ant}/K\cong
\bigl(G_{\ant}/(G_{\ant}\cap H)\bigr)/ \bigl((G_{\ant}\cap H)/ K\bigr)= A/\bigl((G_{\ant}\cap H)/ H\bigr)$ is
an  abelian variety (see Theorem \ref{thm:limantiaff}). Therefore,
$(H\cap G_{\ant})/K\cong (H/K)\cap
(G_{\ant}/K)$ is finite. 

The remaining assertions follow easily.
\qed

  \subsection{$H$--torsors and induced spaces}\ %
  \label{section:indesp}
  
Let $G$ be a smooth group scheme of finite type over an algebraically
  closed field $\Bbbk$, $H\subset G$ a 
closed subgroup scheme and $X$ a  quasi-projective scheme equipped 
with an $H$--action. Serre proved in 
\cite{kn:serreefa}  that the diagonal action $H\times (G\times X)\to
G\times X$, 
$h\cdot (g,x)=(gh^{-1}, h\cdot x)$ has a geometric quotient, that we
denote as 
$G\times^HX$. If  $(g,x)\in G\times X$, then we denote by  $[g,x] $
the class of 
$(g,x)$ in the quotient. Then $G\times^HX$ is a $G$--scheme (with action
given by $g'\cdot [g,x]=[g'g,x]$), and  the canonical projection
$G\times^HX\to G/H$, induced by 
$[g,x]\mapsto gH$,  is a fiber bundle, with fibers isomorphic to
$X$. We call $G\times^H X$ the  \emph{induced space}.

Later on, Serre's result was generalized in several directions: let
$H$ be a group scheme of finite type and $Y$ an $H$--scheme (for
a right $H$--action), such that the geometric quotient   $Y\to Y/H$
exists (in the sense of GIT, \cite[pages
  3,4]{kn:GIT}). If $X$ is an $H$--scheme, we are concerned 
with the existence of the quotient for the diagonal action, that we
denote as $\pi_{Y\times X}: Y\times X\to Y\times^HX:=(Y\times X)/H$.
In  \cite[Proposition 7.1]{kn:GIT}, Mumford gives sufficient
conditions in terms of the existence of an ample $H$--linearized line bundle on $H\times
X$  (see Definition \ref{def:Gbundle} below)  in order to guarantee the
existence of $Y\times ^HX$  (by means of 
``fppf descent'' techniques) ---  see \cite[\S\ 3.3]{kn:brionlinrev}
for a detailed proof of how to apply Mumford's result in order to
prove the existence of $Y\times^HX$. In  \cite[Chapter
I.5]{kn:jantzen1}, 
Jantzen studies this problem in the context of schemes over a
commutative ring $R$.  It is also worth noting that in  
\cite{kn:bb-induced} Bialynicki-Birula studied the existence of the
induced space $Y\times^H X$ for locally isotrivial (in the finite
\'etale topology) $H$--torsors $Y\to Y/H$, in the context of algebraic
spaces --- of course, some additional hypothesis must be made on $Y$.

Let $\mathcal S$: $\ate$ be an affine extension and $V$ a finite
dimensional $H$--module. Then $q:G\to A$
  is affine and faithfully flat. Thus,  we are in the setting of
  fpqc descent  (see \cite[Expos\'e VIII]{kn:SGA1}), and  we can
  guarantee the existence of the quotient $\pi_{G\times V} : G\times V\to
  G\times^HV$, as in Theorem \ref{thm:indesp1} below --- notice that
  this result  follows from the works cited above, but we couldn't find
  it as a precise statement. 

\begin{thm}
\label{thm:indesp1}
Let $\mathcal S$: $\ate$ be a  affi\-ne 
extension of the abelian 
variety $A$,  $V$ be a finite dimensional $H$--module, and consider  the
diagonal $H$--action $a_{G\times V}:=\bigl(m\smallcirc (p_2, i
\smallcirc p_{1}), a_V\smallcirc p_{13}\bigr):H\times (G\times V)\to
G\times V$, 
  where $p_{13}$ is the projection in the first
  and third coordinates and $a_V:H\times V\to V$ is the action
  associated to the representation.
  Then the scheme $G \times V$ endowed with the
  $H$--action $a_{G\times V}$ admits a geometric quotient
  $(G\times^HV, \pi_{G\times V}: G \times V \to
  G\times^HV)$ in the category of
  schemes over $\Bbbk$, in the sense of GIT, \cite[pages
  3,4]{kn:GIT}.
  Moreover, $E_V:=G\times^HV$ is a \emph{$G$--linearized} vector
  bundle with fibers isomorphic to $V$ --- that is, $E_V$ admits a left $G$--action, linear on the
  fibers, such that the canonical
  projection $\pi_V: E_V\to 
  A$ is a $G$--equivariant morphism.
  \end{thm}

  \pf
The existence of  the
  quotient $E_V$, as well as the fact that the  
fibers of $\pi_V:E_V\to A$ are isomorphic to $V$,  follow directly
from fpqc descent  (see \cite[Expos\'e VIII, Theorem 2.1]{kn:SGA1}).
Moreover, the affine morphism $G\times V \to G$ can be seen as the
bundle associated to the free sheaf $\sO_G^{\oplus\, \dim V}$ and the
local triviality of $E_V\to A$ follows from 
\emph{loc.cit.} Expos\'e VIII, Theorem 1.1 and Corollary 1.2.
Finally, it is clear that $G\times (G\times^HV)\to 
G\times^HV$, (induced by $g'\cdot (g,v)=[g'g,v]$), is an action linear
on the fibers, and that $\pi_V$ 
is a $G$--equivariant fibration. 
\qed

\begin{nota}
  Let $\mathcal S$: $\ate$ be an affine extension and $V$ an
  $H$--module. If $f:G\times V\to Y$ is $H$--invariant, then by the
  universal property of the quotient there exists a unique morphism 
$\widetilde{f}: E_V\to Y$  such that $\widetilde{f}
\smallcirc \pi_{G \times     V}=f$. We will abuse notations  denote 
  $\widetilde{f}\bigl([g,v]\bigr)=f(g,y)$.
  \end{nota}

  \begin{rem}
    \label{rem:fpqcquot}
    (1) Notice that $\pi_{G\times V}:G\times V\to G\times^H V$ is an $H$--torsor.
    
\noindent (2) By definition of geometric quotient, $G\times^H V$
    represents the quotient of the fpqc sheaf $G\times V$ by the
    pre-relation $j=(a_{G\times V},p_{23}):H\times (G\times V)\to (G\times V)\times
    (G\times V)$ (see  \cite[Tags 022O and 02VE]{kn:stackproj}). It follows in particular that
    $\pi_{G\times V}: G\times V\to G\times^HV$ is a categorical quotient in the
    category of fpqc sheaves (see  \cite[pages 3,4]{kn:GIT}).

    \noindent (3) In this context, recall that a morphism (of schemes,
    resp.~fpqc sheaves)  $f:G\times V \to Z$ is
     $H$--invariant 
     if $f\smallcirc a_{G\times V} =f\smallcirc
    p_{23}: H\times (G\times V)\to Z$.
      \end{rem}

\section{A finite dimensional representation theory for affine  extensions}
\label{sect:repaffext}

\subsection{Homogeneous vector bundles over an abelian variety}\ %
\label{subsec:homogvecbun}

In this section we recall some basic facts on the category of
\emph{homogeneous vector bundles} over an abelian variety (see
Definition \ref{defn:hmogvecbunbrion} below). The study of homogeneous vector bundles over an abelian variety was
initiated by Atiyah in 1956 (see
\cite{kn:atiyahsh}, \cite{kn:atiyahconn}, \cite{kn:atiyahvbec}). Later on, Miyanishi,
Mukai and others generalized Atiyah's original results (for
homogeneous vector bundles over elliptic curves  over
  $\mathbb C$) to a more general
setting  --- for homogeneous vector bundles over
  an abelian variety $A$, over an arbitrary algebraically closed field $\Bbbk$ ---, giving
a nice description of the corresponding category and 
its main properties (see for example \cite{kn:miy}, \cite{kn:Mu78} and
\cite{kn:bprit}). Recently, Brion in \cite{kn:brionrepbund} introduced the definition and first properties of the category of
homogeneous vector bundles over an arbitrary field $\Bbbk$. In what
follows we take Brion's definition as departure point in order to
enlarge the category of homogeneous vector bundles by
introducing new morphisms (see Definition \ref{defn:vbgraded}
below). For this, we adapt  the approach taken in \cite{kn:bprit}
(where the authors dealt with homogeneous vector bundles over an
algebraically closed field) to this more general
context.

\begin{defn}
  \label{defn:homogvecbundl}
  If $T$ is an scheme, then  the \emph{category of
    vector bundles with base $T$}, denoted as $\VB_0(T)$, is defined as
    follows:

  \begin{enumerate} \item \emph{Objects}: the family of vector bundles
    with base $T$. Recall that a vector bundle is a pair $(E,\pi)$ with $\pi:E \to
    T$ a morphism that is locally trivial in the
    Zariski topology: there exists an open covering $\{U_i\}_{i\in I}$
    of $T$ and isomorphisms  $\psi_i:\mathbb A_{U_i}^n\to \pi^{-1}(U_i)$ compatible
    with $\pi$ --- called the trivializations of the
    bundle ---  such  that for any affine open subset
    $V=\Spec(R)\subset U_i\cap V_j$, the ``transition morphisms''
    $\psi_j^{-1}\smallcirc \psi_i|_{_{\mathbb A^n_{V}}}:\mathbb
    A^n_{V}\to \mathbb A^n_{V}$ are given by linear automorphisms of
    $R[x_1,\dots , x_n]$. If $t\in T$, the $n$-dimensional $\Bbbk(t)$--vector space
    $\pi^{-1}(t):=E\times_T\Spec\bigl(\Bbbk(t)\bigr)$ is called the
    fiber of $E$ at $t \in T$. Notice that in particular the morphism
    $\pi$ is affine.

    \item  \emph{Arrows}: if  $\pi:E\to  T, \pi':E'\to  T$ are  vector
      bundles  over $T$,  a  \emph{morphism of  vector  bundles} is a
      morphism $f:\pi \to  \pi'$ of schemes over $T$,  i.e.~a morphism
      of  schemes $f:  E  \to E'$  such  that $\pi'\smallcirc  f=\pi$
      i.e. the diagram
\[\xymatrix{ E\ar@{->}[r]^f\ar@{->}_(.4)\pi[d]
    &E'\ar@{->}^(.4){\pi'}[d]\\ T\ar@{->}[r]_{\id_T}& T
  }\]
is commutative, and  for any pair of trivializations $\psi'_j: \mathbb A^m_{V_j}\to
      {\pi'}^{-1}(V_j)$ and $\psi_i:\mathbb A^n_{U_i}\to\pi^{-1}(U_i)$ and any
      open affine subset $V=\Spec(R)\subset V_j\cap U_i$, the morphism
      ${\psi'_j}^{-1}\smallcirc f\smallcirc \psi_i|_{_{\mathbb A^n_{V}}}: \mathbb A^n_V\to \mathbb A^m_V$ is
      given by a linear endomorphism $R[x_1,\dots, x_n]\to R[y_1,\dots,y_m]$. 
    \end{enumerate}
    \end{defn}

    \begin{nota}
      \label{nota:homogvecbundl}
      (1) In the literature, the vector bundles (defined as above) 
are sometimes called \emph{geometric vector bundles}, see for example
\cite[Definition 11.5]{kn:gortz} or
\cite[Exer. 5.18]{kn:hartshorne}. Later in the consideration of
\emph{Hopf sheaves}, it is convenient to view the 
vector bundles from a more general perspective (see
Section \ref{sect:repassheaves}). We regard them in terms of the
relative spectrum of a locally free sheaf (see Section
\ref{sect:quasicompandsheaves} and Remark \ref{rem:basicpropvb}). 

\noindent (2) The pair $(E,\pi)$ is abbreviated as $E$ and if $t\in
T$,  the fiber
$\pi^{-1}(t)$  is denoted as $E_t$.  For further compatibility we
denote the $\Bbbk$-vector space of arrows  between two vector bundles $E,E'$ as
$\HomO(E,E')$ (see Definition 
\ref{defn:vbgraded} and Lemma \ref{lem:homgrhomog}).  A morphism $f \in
\HomO(E,E')$  restricts to the fibers defining a
$\Bbbk(t)$--linear map written as $f_t:E_t \to E'_t$.

\noindent (3) If $(E,\pi)$ is a vector bundle, we denote
$\End_0(E)=\Hom_0(E,E)$, and $\Aut_0(E)\subset \End_0(E)$ the group of
automorphisms  of $E$.
\end{nota}

\begin{defn}\label{defn:zerosectionO}
  Given a vector bundle $\pi: E \to T$ then, using the local
  characterization of a vector bundle, one can define a morphism of
  schemes, called the \emph{zero section}, $\sigma_E : T \to E$
such that: i) $\pi \smallcirc \sigma=\id_T$, ii) if $f: (\pi:E \to T)
\to (\pi': E' \to T)$ is a morphism in $\HomO(E,E')$, then $f
\smallcirc \sigma_E = \sigma_{E'}$. In other words, $\sigma_E(t)=0\in
E_t$ for all $t\in T$.
\end{defn}

\begin{rem}\label{rem:basicpropvb}
  (1) It is standard knowledge that the category $\VB_0(T)$ is
  op-equivalent with the category of locally free
  finitely generated sheaves of $\mathcal O_T$--modules (see
  e.g. \cite[Def. 1.7.8]{kn:EGAII} and the basic general constructions
  of Section \ref{sect:quasicompandsheaves} below).

\noindent (2) It is also well known that $\VB_0(T)$ is a 
monoidal, rigid, $\Bbbk$--linear category, with unit object $p_2:
\mathbb A^1_{\Bbbk}\times T=\Bbbk\times T\to T$ and final object the
trivial bundle $\Speck \times T$.

\noindent (3) In particular, given the vector bundles $E,E'$, $\HomO(E,E')$ is a
$\Bbbk$--vector space, and naturally it supports a $\Bbbk$--scheme.
Thus, $\VB_0(T)$ can be seen as a category enriched over $\Sch|\Bbbk$ in
a canonical way (compare with Definition \ref{defn:gradedhom} and
Remark \ref{rem:compogr}).
\end{rem}

If $\pi:E\to A$ is a vector bundle over an abelian variety,  it is
well known that $\Aut_0(E)$  can be
endowed with a structure of a group scheme of finite type (see 
  for example \cite{kn:matsumuraoort} and \cite[\S
  2.3]{kn:brionrepbund}).  Moreover, it is 
   it is possible to define  a group scheme $\Autgr(E)$ (of
   \emph{graded automorphisms})  as follows. 

   \begin{nota}
(1) If $X, T$ are $\Bbbk$--schemes, then the canonical projection   $p_2:
X\times T\to T$ endows $X\times T$ with a structure of $T$-scheme,
that we denote as $X_T=(X\times T,p_2)$.

\noindent  (2) If  $\pi: E\to S$ is a vector bundle, then $\pi_T=\pi\times \id_T:
E_T\to S_T$ is a vector bundle.

\noindent (3) If $A$ is an abelian variety and $\ell\in A(T)$ is a
$T$--point, then we define the \emph{translation by $\ell$} as the
morphism of $T$--schemes $t_\ell=(s\times \id_T)\smallcirc (\ell\smallcirc p_2,
\id_{A\times T}):A_T\to A_T$  --- notice that if $(b,l)\in A_T$,  then
$t_\ell (b,t)=(\ell+b,t)$. 
\end{nota}

\begin{defn}
  \label{defn:autgr}
 If $A$ is an abelian variety and $\pi:E\to A$ is a vector bundle,
 define a functor $\Autgr(E): 
 (\Schk)^{\op} \to \mathrm{Groups}$ as follows:

 \noindent (1)   {\em Objects.}
If $T \in \operatorname{Obj}(\Schk)$, then
$\Autgr(E)(T)$ is the group of pairs $(f, \ell)$, where $f:E_T\to E_T$
is  a $T$--automorphism  and $\ell:T\to A$ is a $T$--point.
such that:

\begin{enumerate}
  \item[(i)] the diagram below commutes
\[
{\xymatrixcolsep{3.5pc}
\xymatrixrowsep{1.5pc}
 \xymatrix{
E_T=E\times T \ar[r]^{f}\ar[d]_{\pi\times \id_T} & E_T=E\times T \ar[d]^{\pi\times\id_T} \\
A_T= A\times T \ar[r]_{t_\ell} & A_T=A\times T
}}
\]

\item[(ii)] the morphism of $A_T$-schemes $\widehat{f}:E_T\to
  t_\ell^*E_T$, defined by the pullback diagram given by base change, is
  an isomorphism of $A_T$--vector bundles in $\Sch|T$ --- in other
  words $\widehat{f}$ is
  an isomorphism in   $\VB_0(A_T)$. 
\[\xymatrix{
  E_T\ar[rd]|-{\widehat{f}}\ar@/_1pc/[rdd]_{\pi \times \id_T}\ar@/^1pc/[rrd]^{f}
  &&\\&t_\ell^*(E_T)\ar[d]\ar[r]&E_T\ar[d]^{\pi \times \id_T}\\&A_T\ar[r]_{t_\ell}&A_T }\]
\end{enumerate}

\noindent (2) {\em hom-objects.} The functor $\Autgr(E)$ at the level of
hom-objects is defined  as follows: 
   if $g: T'\to T$ is a morphism of $\Bbbk$--schemes then the map
   $\Autgr(E)(g): \Autgr(E)(T) \to \Autgr(E)(T')$ is given by 
\[
\Autgr(E)(g)(f,\ell)=
\bigl(p_1\smallcirc (f\smallcirc(\id_E \times g),p_2),\ell\smallcirc g\bigr) \in
\Autgr(E)(T').
\]
\end{defn}

\begin{rem}
  \label{rem:autgr}
  \noindent (1)  The product  in $\Autgr(E)(T)$ (\emph{the composition of graded
    automorphisms}) is defined by the following
  rule: $(f',\ell')(f,\ell)=\bigl(f'\smallcirc f, s\smallcirc (\ell,\ell')\bigr)$.

  \noindent (2) Notice that by construction, if $(f,\ell)\in
  \Autgr(E)(T)$, then   $f=(\widetilde{f}, p_2)$, where
  $\widetilde{f}:E\times T\to E$ is a 
  morphism of $\Bbbk$--schemes. It is equivalent to give $f$ or $\widetilde{f}$ provided that the new map makes the diagram below commutative,
 \[
{\xymatrixcolsep{3.5pc} \xymatrixrowsep{1.5pc} \xymatrix{ E\times
    T \ar[rr]^{\widetilde{f}}\ar[d]_{\pi\times \id_T} && E
    \ar[d]^{\pi\times\id_T} \\ A\times T \ar[r]_(,4){\id_A \times \ell} & A \times A \ar[r]_s& A.}}
\]
  To simplify notations we write indistinctly $f: E_T \to E_T$ or $f:E
  \times T \to E$ in accordance with the context.

  \noindent (3) Let $(f, \ell)\in \Autgr(E)(T)$ and $g:T'\to T$. If 
  $\Autgr(E)(g)(f,\ell)= (h,\ell\smallcirc g)\in \Autgr(E)(T')$, then
  $h=(\widetilde{h},p_2):E_{T'}\to E_{T'}$, with 
  \[
    \widetilde{h}(e,t')=\bigl(p_1\smallcirc (f\smallcirc(\id_E \times g)\bigr)(e,t')=
    \widetilde{f}\bigl(e,g(t')\bigr) \in E.
    \]
\end{rem}

\begin{rem}
  \label{rem:autgrisrep}
\noindent (1) In \cite[Lemma 2.8]{kn:brionrepbund} it is proved that
the functor $T \to \Autgr(E)(T)$ is representable by a group scheme of
finite type, denoted as $\Autgr(E)$. It follows that the projection
$d:\Autgr(E)\to A$, given by $d(T) (f,\ell)= \ell\in A(T)$ for any
$(f,\ell)\in \Autgr(E)(T)$, is a morphism of group schemes.  It is
clear that $\operatorname{Ker}(d)=\AutO(E)$, i.e.~the smooth, affine
and connected group scheme of finite type consisting of all the
automorphisms of the vector bundle $E$.

\noindent (2) In the particular case where $A=\Spec(\Bbbk)$ and $T=\Spec(R)$
for $R$ a commutative $\Bbbk$--algebra, it is clear that $E$ is a
$\Bbbk$--vector space and $f\in \Autgr(E)\bigl(\Spec(R)\bigr)$ is
determined by a morphism $f_R :E\times_{\Bbbk} \Spec(R)\to
E\times_{\Bbbk} \Spec(R)$, linear on the fibers, which is equivalent
to give an $\Bbbk$--linear automorphism  $\widehat{f}:E\to E$.

\noindent (3) Consider the canonical action $a: \Autgr(E)\times E\to
E$ of the group scheme $\Autgr(E)$ on $E$ as described in Remark
\ref{rem:morfmonsch}. If $T$ is a $\Bbbk$--scheme and
$\bigl(g,d(g)\bigr) \in \Autgr(E)(T)$, then  $g=(\widetilde{g},p_2) :
E_T\to E_T$. If  $e\in E(T)$, then $a(g,e)=\widetilde{g}\,
\smallcirc (e,\id_T): T \to E \times T \to E \in E(T)$. A direct
computation shows that 
the degree $d\bigl(e \to \widetilde{g}\,\smallcirc(e,\id)\bigr)=d(e \to a(g,e))= d(g)$.

Therefore, the diagram in Definition \ref{defn:autgr}, (i) in this context reads as:
\[
   \raisebox{5ex}{\xymatrix{
    \Autgr(E)\times E\ar[r]^(.68)a\ar[d]_{d\times \pi}& E\ar[d]^\pi\\
    A\times A\ar[r]^-s& A.
    }}
  \]
  \end{rem}

\begin{defn}
  \label{defn:hmogvecbunbrion}
  Let $A$  be an abelian variety. A vector bundle $\pi:E\to A$ is
  called \emph{homogeneous} if the induced morphism of group schemes
  $d: \Autgr(E)\to A$ is faithfully flat --- i.e.~if $d$ is surjective,
  in view of Theorem \ref{thm:perrinigame}.

  The \emph{category $\HVB_0(A)$} is defined as the full subcategory of
  $\VB_0(A)$ that has as objects the homogeneous vector bundles.
\end{defn}
    
\begin{rem}
  \label{rem:algclosforhomog}
(1) In view of Corollary \ref{cor:surimpliesflat}, a vector bundle
$\pi:E\to A$ is homogeneous if and only if for any geometric point
$b\in A\bigl(\,\overline{\Bbbk}\,  \bigr)$, there exists an isomorphism of $\overline{\Bbbk}$--vector
bundles   $E_{\overline{\Bbbk}}\to t_b^*E_{\overline{\Bbbk}}$.

\noindent (2) Since $A$ is an abelian variety, it follows that  if
$\pi:E\to A$ is a homogeneous vector bundle, then the short exact
sequence:
\[
  \xymatrix
  {
    \Authbext(E): & 1\ar[r]&  \AutO(E)\ar[r]& \Autgr(E)\ar[r]^-d& A\ar[r]&
  0
  }
\]
is a smooth  affine extension of $A$, of finite type. In particular,
$\Autgr(E)$ is a smooth group scheme of finite type.

\noindent (3) 
It follows from Remark \ref{rem:G=HG0} that
$\Autgr(E)=\AutO(E)\Autgr(E)^0$ and therefore $\Autgr(E)$ is a
connected group scheme.

\end{rem}

\begin{lem}
\label{lem:homfibbun} Let $\mathcal S$: $\ate$ be an  affine extension of
the abelian variety $A$, and let $V$ be a finite dimensional
$H$--module.  Then the vector bundle $\pi_V: E_V=G\times^HV\to A$ is
homogeneous. Conversely, if $\pi_E:E\to A$ is an arbitrary homogeneous vector
bundle, then $E\cong \Autgr(E)\times^{\AutO(E)}E_0$, where
$E_0=\pi^{-1}(0)$ is as usual the fiber over $0\in A$.
\end{lem}
\pf Indeed, it follows from
Theorem \ref{thm:indesp1} that $\pi_V:E_V\to A$ is a vector bundle and
that $G$ acts linearly on $E_V$. Since the $G$--action of $E_V$
induces a morphism 
of affine extensions $\mathcal S\to \Authbext(E)$, it follows (for
example from Remark \ref{rem:algclosforhomog}) that the vector bundle
$E_V$  is homogeneous.

Conversely if $\pi_E:E\to A$ is a homogeneous vector bundle, by
the first part of the lemma $\Autgr(E)\times^{\AutO(E)}E_0$ is a
homogeneous vector bundle, and clearly the restriction of the action
$\Autgr(E)\times E \to E$ to $\Autgr(E)\times E_0\to E$ induces the
required isomorphism (see Remark \ref{rem:algclosforhomog}).  \qed

\begin{defn}
The vector bundle $E_V$ is called the \emph{homogeneous vector bundle
associated to the 
$H$--module $V$}.

\end{defn}

\begin{rem}\label{rem:hbvabelian}
It is clear from the definition that if $E, E'\in \HVB_0(A)$ are
homogeneous vector bundles, then $E\oplus E'$, $E\otimes E'$ and
$E^\vee$ also are objects of $\HVB_0(A)$ and these operations (and
the corresponding morphisms) endow this category with a
$\Bbbk$--linear and monoidal rigid structure --- this was proved in
\cite{kn:miy} and \cite{kn:Mu78}) when $\Bbbk$ is an algebraically
closed field, then by Remark
\ref{rem:algclosforhomog} we deduce the general case. Moreover, in
\cite[Theorem 2.9 and Corollary 
2.10]{kn:brionrepbund}, Brion proved that $\HVB_0(A)$ is also an
abelian category, stable by direct summands.  
  \end{rem}

Next we enlarge the family of arrows in the categories $\VB_0(A)$ and
$\HVB_0(A)$ taking instead of arrows of degree zero, arrows of arbitrary
degree $a \in A$. We use the same notations and abbreviations than in
Definition \ref{defn:autgr}, and remarks \ref{rem:autgr} and \ref{rem:autgrisrep}.

\begin{defn} \label{defn:gradedhom}
Let $\pi: E\to A$ and $\pi':E' \to A$ be two  vector bundles. We define
the \emph{graded homomorphisms functor} $\Homgr(E,E') : \Sch^{\op}\to
\Sets$ as follows.

\noindent (1)  If $T\in \Sch$, then
$\Homgr(E,E')(T)$ is the set of pairs $(f, \ell)$, where
  $f:E_T\to E'_T$ is a $T$--morphism and $\ell\in A(T)$ such that:

\begin{enumerate}
  \item[(i)] the diagram below commutes
\[
{\xymatrixcolsep{3.5pc}
\xymatrixrowsep{1.5pc}
 \xymatrix{
E_T=E\times T \ar[r]^{f}\ar[d]_{\pi \times \id_T} & E'_T=E'\times T \ar[d]^{\pi'\times \id_T} \cr
A_T= A\times T \ar[r]_{t_\ell} & A_T=A \times T
}}
\]

\item[(ii)] The induced morphism of $A_T$-schemes $ E_T\to t_\ell^*E'_T$ (see diagram below) is a
  morphism in $\VB_0(A_T) $.
 \[\xymatrix{
  E_T\ar[rd]|-{\widehat{f}}\ar@/_1pc/[rdd]_{\pi \times \id_T}\ar@/^1pc/[rrd]^{f}
  &&\\&t_\ell^*(E'_T)\ar[d]\ar[r]&E'_T\ar[d]^{\pi' \times \id_T}\\&A_T\ar[r]_{t_\ell}&A_T. }\]
 
 \end{enumerate}

\noindent (2)
If $g:T'\to T$ is a morphism of schemes and $(f, \ell)\in
\Homgr(E,E')(T)$, with $f=(\widetilde{f},p_2):E\times T\to E\times T$, then
\[
  \Homgr(E,E')(g)(f)=
\bigl((\widetilde{f}\smallcirc (\id_E\times g),p_2), \ell\smallcirc g \bigr).
\]

Notice that $\bigl(\widetilde{f}\smallcirc (\id_E\times
g),p_2\bigr):E_{T'}\to E'_{T'}$ is a morphism of $A_{T'}$--schemes.
\end{defn}

\begin{rem}\label{rem:compogr}
  (1) By construction (and descent theory, see
  \cite[I.2.2.7]{kn:demgab} of \cite[Tag 0238]{kn:stackproj}) $\Homgr(E,E')$ is a fpqc sheaf.

\noindent (2) Let $\pi:E\to A$, $\pi': E'\to A$ and  $\pi'':E''\to A$ be three vector
bundles. Then the composition of morphisms induces a natural composition morphism (a natural transformation between the functors)  $\Homgr(E,E')\times \Homgr(E',E'')\rightarrow \Homgr(E, E'')$. 

\noindent (3) Notice that the family $d(T):\Homgr(E,E')(T)\to A(T)$ produces also a morphism
(natural transformation) $d:\Homgr(E, E')\rightarrow A$. 
\end{rem}

\begin{ej}
  \label{ej:pz}
  Given two homogeneous vector bundles $E,E'$, the \emph{zero
    morphism} $0:E\to E'$ is given by $0= \sigma_{E'}\smallcirc \pi_E$
  (see Definition \ref{defn:zerosectionO}). We generalize this
  construction to arbitrary degrees as follows.

   If $\ell\in A(T)$,  it is
clear that the pair $(\zeta_\ell,\ell)$ with $\zeta_\ell=\sigma_{E'}
\smallcirc p_1\smallcirc t_\ell\smallcirc (\pi\times \id_T): E \times
T \to E'$ yields a graded morphism of degree $\ell$.  The $T$--point
$(\zeta_\ell,\ell) \in \Homgr(E,E')(T)$ is called \emph{the
  pseudo-zero of degree $\ell$}.

Notice that  $(\zeta_\ell,\ell)$ induces the zero morphism $0:E_T\to t_\ell^*E_T$.
\end{ej}

\begin{rem}\label{rem:yoneda}
Let  $ \mathcal
  V=\operatorname{Func}\bigl((\Schk)^{\op},\Sets\bigr)$; then we can endow
  $\mathcal V$ with the product  induced by the cartesian
  product in $\Sets$,  which  final object  is the constant functor
  equal to the final object in $\Sets$.

    Clearly, the Yoneda embedding $Y: \Schk \to \mathcal
  V=\operatorname{Func}\bigl((\Schk)^{\op},\Sets\bigr)$ preserves finite products.
\end{rem}

\begin{defn}\label{defn:vbgraded}
(1)     For $\mathcal V$ as above, we define the $\mathcal V$--category
     $\VBG(A)$ (i.e.~$\VBG(A)$ is enriched over $\mathcal V$) as
     follows:

\noindent The  \emph{objects} of $\VBG(A)$ are the same than $\VB_0(A)$.

Given  $(E,\pi)$, $(E',\pi')
     \in \VBG(A)$,  the \emph{hom-object} with domain $(E,\pi)$ and codomain
     $(E',\pi')$ is $\Homgr(E, E') \in \mathcal V$, i.e.~the
     \emph{functor of graded homomorphisms of vector bundles}, with
     compositions as defined before (Remark \ref{rem:compogr}).

\noindent (2) The
     $\mathcal V$--category $\HVBG(A)$ is the full $\mathcal V$--subcategory of
     $\VBG(A)$ with objects the homogeneous vector bundles 
     (compare with Definition \ref{defn:hmogvecbunbrion}).

Similarly than Definition \ref{defn:hmogvecbunbrion}, the category
$\HVBG(A)$ is the full subcategory of $\VBG(A)$ with objects the
homogeneous vector bundles.
\end{defn}

\begin{rem}
  \label{rem:clasicosygraded}
  Let $\pi:E\to A$ and $\pi':E'\to A$ be two vector bundles and
  $f:E\to E'\in \HomO(E,E')$ be a morphism of vector bundles. Then
  $f_T=f\times \id_T: E_T\to E'_T\in \Homgr(E, E')(T)$. Thus
  $\HomO(E,E')$ represents a subfunctor of $\Homgr(E,E')$, with
  $\HomO(E,E')(T)=\bigl\{ f\in \Homgr(E, E')(T)\mathrel{:}
  d(T)(f)=0\bigr\}$.  Thus, $\VB_0(A) \subseteq \VBG(A)$ is a
  \emph{wide} ($\mathcal V$--enriched) subcategory --- in the sense
  that has the same objects but less morphisms. Similarly for the
  homogeneous situation.
    \end{rem}

\begin{nota}
 If $E=E'$, then $\Endgr(E):=\Homgr(E,E)$ and
 $\EndO(E):=\HomO(E,E)$.
 \end{nota}

 \begin{rem}
   It is clear that $\Endgr(E)$ and $\EndO(E)$ are functors on
   monoids, and that the group $\Autgr(E)$ (resp.~$\AutO(E)$) 
is a subfunctor on monoids of $\Endgr(E)$ (resp.~$\End_0(E)$).
\end{rem}

The relationships between the (enriched) categories we just defined is
illustrated in the diagram below, where the vertical arrows are full
subcategories and the horizontal are wide subcategories. 
\[
  {\xymatrixrowsep{1pc}\xymatrixcolsep{1pc}      \xymatrix{\VB_0(A)
      \ar@{}[r]|-*[@]{\subseteq} & \VBG(A) 
        \\ \HVB_0(A)\ar@{}[u]|-*[@]{\subseteq}\ar@{}[r]|-*[@]{\subseteq}&
        \HVBG(A)\ar@{}[u]|-*[@]{\subseteq}}}
  \]

\begin{ej}\label{exam:casoaffin1}
It is clear that if $A=\Spec(\Bbbk)$, then all these categories
collapse into $\operatorname{Vect}_\Bbbk=
\VB_0\bigl(\Spec(\Bbbk)\bigr)=\HVB_0\bigl(\Spec(\Bbbk)\bigr)=\HVBG\bigl(\Spec(\Bbbk)\bigr)=\VBG\bigl(\Spec(\Bbbk)\bigr)$.

Indeed, if $T=\Spec(R)\in
\Schkaff$, and $V,W\in \operatorname{Vect}_{\Bbbk}$, then $V_T=
(V\times T\to T)$ and $\Homgr(V,W)(T)\cong 
\Hom_\Bbbk(V,W)\otimes_\Bbbk R$ --- in other words, the functor
$\Homgr(V,W)$ is represented by the vector space $\Hom_\Bbbk(V,W)$.
\end{ej}

  \begin{rem}
    \label{rem:pullback-algclo1}
    (1) In the category $\VB_0(A_{\overline{\Bbbk}})$, for $\ell \in A$ we
    denote as $T_\ell$ the ``pullback by the translation $t_\ell$'' functor
    (compare with Definition \ref{defn:autgr},(ii)). Thus,
    $T_\ell:\VB_0(A_{\overline{\Bbbk}}) \to \VB_0(A_{\overline{\Bbbk}})$ is
    given at the level of objects by: \[\xymatrix{
      T_\ell(E)\ar@{->}[r]^-{p_E}\ar@{->}_{\widehat{\pi_\ell}}[d]
      &E\ar@{->}^{\pi}[d]\\ A\ar@{->}[r]_{t_\ell }& A.}\]
   
It is clear that the vector bundle
      $\bigl(T_\ell(E),\widehat{\pi_\ell}\bigr) \cong (E,t_{-\ell}\smallcirc\pi)$; when there
    is no danger of confusion, the structure map $\widehat{\pi_\ell}$ is
    denoted simply as $\pi_\ell$.

If
      $(E,\pi)$ and $(E',\pi')$ are objects in $\VB_0(A_{\overline{\Bbbk}})$ and
      $f:(E,\pi)\to(E',\pi')$ is an arrow in $\VB_0(A_{\overline{\Bbbk}})$, then $T_\ell(f)=f:
      \bigl(T_\ell(E),\pi_\ell\bigr) \to \bigl(T_\ell(E'),\pi'_\ell\bigr)$ is an
      arrow in $\VB_0(A_{\overline{\Bbbk}})$ as shown in the diagram below.
      \[
   \xymatrix{E \ar[rr]^{T_\ell(f)=f}
    \ar[rd]^{\pi}\ar[rdd]_{\pi_ell}& &
    E'\ar[dl]_{\pi'}\ar[ddl]^{\pi'_\ell}\\    
    &A\ar[d]|<<<<{t_{-\ell}}&\\&A&}
\]

   The map $\ell \to T_\ell: A \to \operatorname{Fun}\bigl(\VB_0(A_{\overline{\Bbbk}})\bigr)$ is a morphism
   of the monoid $(A,+)$ to $\bigl(\operatorname{Fun}(\VB_0(A_{\overline{\Bbbk}})),\smallcirc\bigr)$
   ($\smallcirc$ denotes the composition of functors). In particular for each $\ell \in A$ the functor $T_\ell$ is invertible and its inverse is $T_{-\ell}$.

   \noindent (2) Let  $\ell \in A$, and $(E,\pi)$,
  $(E',\pi')$ be two objects in $\VB_0(A_{\overline{\Bbbk}})$ and $f:E \to E'$ a morphism of
  the underlying schemes. The diagram (whose rightmost triangle is
  commutative):
  \[
    \xymatrix{E\ar[d]_\pi \ar[rr]^{f}&&E'\ar[d]^{\pi'}\ar[dll]|{\pi'_{\ell}}\\A\ar[rr]_{t_\ell}&&A}\]
proves that
   $\HomO\bigl(E,T_\ell(E')\bigr) = \bigl\{f:E \to E':  \pi'
   f= t_\ell \pi, f|_{_{E_b}}:E_b\to E_{\ell+b} \text{ linear}\bigr\}$.
  
\end{rem}

 In view of the preceding remark, if $E,E'$ are homogeneous vector
 bundles over an algebraically closed field one can
 work with \emph{sets of graded morphisms}, i.e.~to consider the
 set of morphisms $f:E\to E'$ such that $\pi'\smallcirc f=t_\ell\smallcirc \pi$ for
 some $\ell\in A$, rather than with the functor $\Endgr(E,E')$. This is
 the approach taken by L.~Brambila-Paz and A.~Rittatore in
 \cite{kn:bprit}, for examination of the geometry and algebraic
 structure of $\Endgr(E)$ and $\Homgr(E, E')$.

 If $\Bbbk$ is an arbitrary field, and $E,E'\in \HVBG(A)$, we present
 below the proof that $\Homgr(E,E')$ is representable by the vector
 bundle $L_{\HomO(E,E')}=\Autgr(E')\times^{ \AutO(E')} \HomO(E,E')$
 (see Lemma \ref{lem:homfibbun}). It is similar to the proof in
 \cite{kn:bprit} (in the hypothesis that $\Bbbk=\overline{\Bbbk}$),
 with the necessary adaptations to the general situation.

\begin{lem}
\label{lem:homgrhomog}
Let $\pi: E\to A$, $\pi':E'\to A$ be two homogeneous vector bundles
over the abelian variety $A$. Then the homogeneous vector bundle
$L_{\HomO(E,E')}=\Autgr(E')\times^{ 
  \AutO(E')} \HomO(E,E')$ (see Lemma \ref{lem:homfibbun})  represents
  $\Homgr(E,E')$. 

Moreover, $\Homgr(E,E') \cong R_{\HomO(E,E')}=
\Autgr(E)\times^{\AutO(E)}\HomO(E,E')\in \HVB_0(A)$. 
  \end{lem}
    \pf
We adapt  the strategy used  in \cite{kn:bprit} for the algebraically
closed field case to this general case.

Let $\varphi :\Autgr(E')\times \HomO(E,E')\to \Homgr(E,E')$ the
morphism of fpqc sheaves given by composition. Then clearly $\varphi$
is $\AutO(E)$--invariant (see Remark \ref{rem:fpqcquot}), and
therefore induces a morphism of fpqc sheaves 
$\phi: L_{\HomO(E,E')}\to \Homgr(E,E')$.

We prove now that $\varphi$ is a monomorphism. Let $y_1: T\to L_{\HomO(E,E')}$, $y_2:T\to  L_{\HomO(E,E')}$ be two
points in $ L_{\HomO(E,E')}(T)$ such that
$\phi(T)(y_1)=\phi(T)(y_2)\in \Homgr(E,E')(T)$. Let $\sigma_i: T_i\to
T$, $i=1,2$, be fpqc morphisms and $x_1=(g_1,f_1),x_2=(g_2,f_2)\in \Autgr(E')\times
\HomO(E,E')(T_i)$ be such that $\pi(x_i)=y_i\smallcirc \sigma_i$. 
Then as points in $\Homgr(E,E')(T_1\times_T T_2)$, we have that 
\[
g_1\smallcirc f_1 =\phi(T_1\times_TT_2)(x_1)=\phi(y_1)=\phi(y_2)=
\phi(T_1\times_TT_2)(x_2)=g_2\smallcirc f_2.
\]

It follows that $f_2=g_2^{-1}\smallcirc g_1\smallcirc f_1\in
\HomO(E,E')(T_1\times_TT_2)$, with $g_2^{-1}\smallcirc g_1\in
\AutO(E')$. Thus, $y_1=y_2\in L_{\HomO(E,E')}(T_1\times_TT_2)$ and it
follows that $y_1=y_2\in L_{\HomO(E,E')}(T)$.

In order to prove that $\varphi(T)$ is surjective for all $T$, let $
(f,\ell)\in \Homgr(E,E')(T)$. Let $\sigma :T'\to T$ a fpqc morphism
and  $g \in \Autgr(E')( T')$ such that $q(T)(g)=\ell\smallcirc
\sigma\in A(T')$. Then  $h=g^{-1}\smallcirc
(f\smallcirc (\id_E\times \sigma) ,\ell\smallcirc
\sigma\bigr)\in \HomO(E,E')(T')$. It follows that 
$\varphi(T')\bigl(g, h)= = \bigl((f\smallcirc (\id_E\times \sigma), \ell\smallcirc \sigma)\in
\Homgr(E,E')(T')$.
From the commutative diagram:
\[
  \xymatrix{
  L_{\HomO(E,E')}(T)\ar[r]^{\phi(T)}\ar[d]_{\sigma^*}&\Homgr(E,E')(T)\ar[d]^{\sigma^*}\\
  L_{\HomO(E,E')}(T')\ar[r]^{\phi(T')}&\Homgr(E,E')(T')
}
\]
we deduce that there exists $y\in   L_{\HomO(E,E')}(T)$ such that
$\phi(y)=(f,\ell)$ by descent.

The  last assertion can be proved by a similar argument.\qed

\begin{rem}
  Let $E,E'\in \HVBG(A)$.  Since
  $\Homgr(E,E')=\Autgr(E')\times^{\AutO(E')}\HomO(E,E')$ is a
    homogeneous vector bundle, it follows that $\HVBG(A)$ is a closed
    category.
  \end{rem}

\begin{cor}
  \label{cor:endgrismon}
Let $\pi:E\to A$ be a homogeneous vector bundle. Then $\Endgr(E)$ is a
smooth monoid scheme of finite type, such that the following diagram is
commutative, where the vertical arrows are open immersions.
\[
  \xymatrix{
    1\ar[r]&  \EndO(E)\ar[r]& \Endgr(E)\ar[r]^-d&A\ar[r]& 0\\
    1\ar[r]& \AutO(E)\ar@{^(->}[u]\ar[r]& \Autgr(E)\ar@{^(->}[u]\ar[r]_-d& A\ar@{=}[u]\ar[r]&  0
    }
  \]
\end{cor}
\pf
Once we know that $\Endgr(E)$ is a smooth scheme of finite type, and
taking into account \cite[Theorem 1]{kn:brionoasam}, the
result follows easily.
\qed

\begin{rem}
  \label{rem:hvbnotmonoidal}
  We observed in Remark \ref{rem:hbvabelian} that $\HVB_0(A)$ is a
  abelian, monoidal, rigid category. Nevertheless, these
structures cannot be defined in the (wide) extension of the category
$\HVB_0(A)$ that we denoted as $\HVBG(A)$.

However, for homogeneous morphisms of the \emph{same} degree it is clear that the following holds.
\end{rem}

\begin{lem}
\label{lem:tensordualhvb}
Let $E,E', F, F'\in \HVBG(A)$ and
$(f, \ell)\in \Homgr(E,F)(T)$, $(f', \ell)\in
\Homgr(E',F')(T)$ be graded morphisms. Then the following maps are
graded morphisms in $\HVBG(A)$:

\noindent (i) $(f\oplus f', \ell)$, where $f \oplus f' :(E\oplus
E')_T\cong E_T\oplus E'_T\to (F
\oplus F')_T$ is given by $(f\oplus f')(e+e')= f(e)+f'(e')$;

\noindent (ii) $(f\otimes f',\ell)$; where $f\otimes f':(E\otimes
E')_T\cong E_T\otimes E'_T\to
(F\otimes F')_T$ is given by $(f\otimes
f')(e \otimes e')=f(e)\otimes f(e')$. 

\noindent (iii) $(f^\vee,-\ell)$, where $f^\vee: (E'_T)^\vee\cong
((E')^\vee)_T\to E_T$. \qed
\end{lem}

 \begin{rem}
 \label{rem:authbmiya1}
 Let $E\to A$ be a homogeneous vector bundle and assume that
 $\Authbext(E)$ admits a section $\sigma :A\to \Autgr(E)$, $d\smallcirc
 \sigma=\id_A$. Then $(\sigma, \id_{E_0}):A\times E_0\to
 \Autgr(E)\times E_0$ clearly induces a morphism of vector bundles $A\times
 E_0\to E=\Autgr(E)\times^{\AutO(E)} E_0$. Thus, we have  proved that a
 homogeneous vector bundle  is trivial if and only if $\Authbext(E)$
 admits a section. This is a well known result when $\Bbbk $ is
 algebraically closed field (see \cite{kn:miy} and \cite{kn:bprit}).
 \end{rem}

\begin{defn}
  \label{defn:maincat}
  Given an object $E$ in the category
  $\HVB_0(A)$, we call $\HVB_0(A)_E$ the full abelian monoidal rigid 
  category generated by $E$. We call $\HVBG(A)_E$
  the full subcategory of $\HVBG(A)$ that has the
  same objects that $\HVB_0(A)_E$.
\end{defn} 

\begin{rem}
    (1)  By definition the category $\HVB_0(A)_E$ is characterized by
    the  following universal property: for every abelian monoidal
    rigid  category $\mathcal C$ and any object $c \in \mathcal C$ there is one and only one additive monoidal functor $F_c: \HVB_0(A)_E \to \mathcal C$ such that $F_c(E)=c$.

    \noindent (2)  The relations between the above categories is depicted in the diagram below:
      \[
        {\xymatrixcolsep{1pc}\xymatrixrowsep{1pc}
          \xymatrix{
            \HVB_0(A) \ar@{}[r]|-*[@]{\subseteq} &
            \HVBG(A) \\
        \HVB_0(A)_E\ar@{}[u]|-*[@]{\subseteq}\ar@{}[r]|-*[@]{\subseteq}&
        \HVBG(A)_E\ar@{}[u]|-*[@]{\subseteq}
       } }
    \]
        where the horizontal maps are wide inclusions and the vertical
    ones are full.  
    \end{rem}

\begin{rem}\label{rem:homislinfib}
  Let $\pi_E:E\to A$, $\pi_{E'}: E'\to A\in \HVB_0(A)$. Then we have a morphism of schemes
  $a:\Homgr(E, E')\times E\to E'$ as follows:

  If $T$ is a $\Bbbk$--scheme, then $a(T): \Homgr(E, E')(T)\times
  E(T)\to E'(T)$ is given by $a(T)\bigl((f,\ell), e\bigr)=f\smallcirc (e\times
  \id_T)$. Notice that $\pi_{E'}\bigl(f(e\times
  \id_T)\bigr)=d(f)+\pi_{E}(e)$. 

  In other words, we have a commutative diagram
  \[
    \xymatrix{
      \Homgr(E, E')\times E\ar[r]^-a\ar[d]_{d\times \pi_E}&
      E'\ar[d]^{\pi_{E'}}\\
      A\times A\ar[r]^-s& A
      }
    \]

 If $E=E'$, then $a: \Endgr(E)\times E\to E$ is an action of the
 smooth monoid $\Endgr(E)$ (see Corollary \ref{cor:endgrismon}), we
 say that the action is \emph{linear on the fibers}.
  \end{rem}

\subsection{Representations of affine  extensions}\ %
\label{subsect:repaffext}

\begin{defn}
  \label{defn:repaffext}
Let $\mathcal S:$ $\ate$ be an affine  extension of the abelian
variety $A$. A \emph{representation of $\mathcal S$} or
\emph{$\mathcal S$--module}, is a  homogeneous vector
bundle $\pi_E:E\to A$ equipped with a morphism of affine 
extensions $\varrho: \mathcal S\to \Authbext(E)$
\[
\xymatrix{\mathcal S:\ar[d]_{\varrho}& 1\ar[r] & H\ar[r]\ar[d]
  &G\ar[r]^{q} \ar[d]_{\rho}&A\ar[r]\ar@{=}[d] &0\\
\Authbext(E) & 1\ar[r] & \AutO(E)\ar[r]
  &\Autgr(E)\ar[r]^-{d_E} &A\ar[r] &0}
\]
  \end{defn}

  \begin{rem}
    \label{rem:repsiact}
(1) To give a representation of $\mathcal S$ on an homogeneous vector
bundle $\pi_E:E\to A$ is equivalent to give an action of  
$a:G\times E\to E$, linear on the fibers (see Remark
\ref{rem:autgrisrep}),  such that the following diagram is commutative
\[
\xymatrix{
G\times E \ar[r]^(0.6){a}\ar[d]_{q\times \pi_E} & E\ar[d]^{\pi_E} \\
A\times A\ar[r]_(0.6){s} & A
}
\]

Therefore, when we talk
about a representation of $\mathcal S$ we mean either a morphism of
affine extensions
schemes $\varrho:\mathcal S \to \Authbext (E)$ or  a vector bundle $E$ together with
the action $a_\varrho$ of $G$ associated to $\varrho$.

In particular in the above perspective, if $g\in G(T)$ and
$\rho(T)(g)=\bigl(f_g ,\ell\bigr)$, then
$\ell=d\bigl(\rho(T)(g)\bigr)=q(g)\in A(T)$ and 
 $\bigl(a_\varrho(g,-),p_2\bigr)=f_g:E_T\to E_T$  is such that the
 induced morphism $E_T\to t^*_{q(g)}E_T$ is an  
 automorphism of vector bundles (i.e.~an isomorphism in the category $\HVB_0$).

\noindent (2) By construction, if $\rho(G)$ is the scheme theoretic
image of $\rho:G\to \Autgr(E)$, then $\varrho(\mathcal S)$:  
$ {\xymatrixcolsep{1.2pc}\xymatrix{ 1\ar[r] & \rho(H)\ar[r]
  &\rho(G)\ar[rr]^-{{d_E}|_{_{\rho(G)}}} &&A\ar[r] &0} }$ is a 
  closed sub-extension of $\Authbext(E)$.

\end{rem}

\begin{ej}
\label{eje:reptriv}
Let $\mathcal S$: $\ate$ be an affine extension and ${\mathbb I}:=
(p_2:\Bbbk\times A\to A)$ be the trivial bundle.  Then
$\Authbext({\mathbb I})$ is the extension $
  \xymatrix{ 0\ar[r]&  G_m\ar[rr]^-{(\id,0_A \smallcirc \st)}&& G_m\times
    A\ar[r]^-{p_2} & A\ar[r] & 0 }$, where $\st:G\to \Spec(\Bbbk)$ is the
  structure morphism of $G$ as a $\Bbbk$--scheme.

  It is clear that the representations $\varrho: \mathcal S\to
  \Authbext({\mathbb I})$ are in bijective correspondence with the
  \emph{characters} of $G$, i.e. with the group scheme homomorphisms
  $\chi: G \to G_m$, this identification is given by $\chi \mapsto
  (\chi,q): G \to G_m \times A$ for $\chi$ a character as above.

  The \emph{trivial character} i.e. the morphism $\chi_0 =
  1_{G_m}\smallcirc \st$ , induces the representation
  $\varrho_0=(\chi_0,q): \mathcal S\to \Authbext({\mathbb I})$, and
  the associated action ${a_\varrho}_0: G\times {\mathbb I}\to
  {\mathbb I}$ is called the \emph{trivial representation} or
  \emph{trivial $\mathcal S$--module}.
\end{ej}

  \begin{rem}
\label{rem:Gequivariant}
    Let  $(E,\varrho_E), (E',\varrho_{E'})$ be two $\mathcal
S$--modules. Then $G$ acts on the vector bundle
$\Homgr(E,E')$ as follows.

If $g: T \to G \in G(T)$ and $(f,\ell)\in \Homgr(E,E')(T)$, then
$a_\varrho\bigl( g,(f,\ell)\bigr)= \varrho_E'(g)\smallcirc (f,\ell)\smallcirc
\varrho_E(g^{-1})\in \Homgr(E,E')(T)$.

Notice  that $a_\varrho(g,-): \Homgr(E,E')(T)\to
\Homgr(E,E')(T)$ is a morphism of $A_T$--vector bundles and that  the diagram below is commutative
\[
  \xymatrix{
    G\times \Homgr(E,E') \ar[rr]^-{a_\varrho}\ar[dr]_{d\smallcirc p_2}& & 
    \Homgr(E,E')\ar[dl]^d\\
    &A &
    }
  \]

  In particular, $a_\varrho$ induces a
morphism of group schemes $G\to\AutO\bigl( \Homgr(E,E')\bigr)\subset \Autgr\bigl(
\Homgr(E,E')\bigr)$.

 \end{rem}

 \begin{lem}
   \label{lem:Smorfarebundle}
Let   $(E,\varrho_E), (E',\varrho_{E'})$ be two $\mathcal
S$--modules and consider the action $a_\varrho :G\times \Homgr(E,E')\to \Homgr(E,E') $ defined in Remark \ref{rem:Gequivariant}. Then  ${}^G\Homgr(E,E')\cong
G_{\ant}\times^{G_{\ant}\cap H}\, {}^G\HomO(E,E')$, where
${}^G\Homgr(E,E')$ denotes as usual the  fixed points subscheme. In particular,
${}^G\Homgr(E,E')$   is an  $\mathcal 
S_{\ant}$--module and hence it is a 
homogeneous vector
sub-bundle of $\Homgr(E,E')$. 
\end{lem}

   \pf Consider the action of $G$ given by post-composition by
   $\rho_{E'}(g)$, that is $a_{\rho_{E'}}= \rho_{E'}\smallcirc - :
   G\times \Homgr(E,E')\to \Homgr(E,E')$.
 
 Let $\mathcal S_{\ant}$ be the closed subextension associated to the
  Rosenlicht decomposition of $\mathcal S$ (see Theorem
 \ref{thm:rosenftgs2}), and notice that  ${}^G\Homgr(E,E')$ is stable
 by the action of $G_{\ant}$, since $G_{\ant}$ is central in $G$. In
 particular, ${}^G\HomO(E,E')$ is $G_{\ant}\cap H$--submodule (clearly
 ${}^G\HomO(E,E')$ is a vector space), and 
 $R_{\,{}^G\HomO(E,E')}= G_{\ant}\times^{G_{\ant}\cap H}  {}^G\HomO(E,E')\to A$
 is a homogeneous vector bundle (see Lemma \ref{lem:homfibbun}). The
 action $a_{\rho_{E'}}$ clearly induces an injective  morphism of vector
 bundles $\widetilde{a_{\rho_{E'}}}: R_{\,{}^G\HomO(E,E')}\to
 \Homgr(E,E')$ , with image contained in
 ${}^G\HomO(E,E')$. Since $\bigl(R_{\,{}^G\HomO(E,E')}\bigr)_0=
 {}^G\HomO(E,E')$, it follows that $R_{\,{}^G\HomO(E,E')}\cong
 {}^G\Homgr(E,E')$.
 \qed

 \begin{defn}
In the notations of Lemma \ref{lem:Smorfarebundle}, the sub-bundle
${}^G\Homgr(E,E')$ is called  the \emph{(homogeneous) vector bundle of
  $G$--equivariant morphisms}.
   \end{defn}

\begin{rem}\label{rem:equivmorasmono}
Let $(E,\varrho_E)$ and $(E,\varrho_{E'})$ be 
$\mathcal S$--modules and  denote by $a_E:G\times E\to E$  and
$a_{E'}:G\times E'\to E'$ the corresponding linear actions. Let us call
$a:=a_{{}^G\Homgr(E,E')}:{}^G\Homgr(E,E')\times E\to E'$ the morphism
associated to $ {}^G\Homgr(E,E')$ (see remarks \ref{rem:repsiact} and
\ref{rem:homislinfib}). Then we have the following 
commutative diagram:
\[
{\xymatrixrowsep{1.5em}  \xymatrix{
    {}^G\Homgr(E,E')\times G\times E
 \ar[d]_{\sigma_{12}}   \ar[rrr]^-{id_{ {}^G\Homgr(E,E')}\times a_E}& && {}^G\Homgr(E,E')\times
    E\ar[dd]^{a}\\
G\times {}^G\Homgr(E,E')\times
    E\ar[d]_{\id_G\times a} &&&\\
     G\times E'\ar[rrr]^-{a_{E'}} &&&E'
}}
    \]
where $\sigma_{12}: {}^G\Homgr(E,E')\times G\times E\to G\times {}^G\Homgr(E,E')\times
    E$ is the isomorphism given by the permutation of the first two
    coordinates. 
\end{rem}

\begin{cor}
  \label{cor:restriction}
Let $\mathcal S$: $\ate$  be an affine extension and  $E,E'\in \Rep(\mathcal
S)$. Then a morphism $f\in \Hom_{\Rep(\mathcal S)}(E,E')$ is
determined by its restriction to $E_0=\pi^{-1}(0)$, the fiber over $0\in A$.
  \end{cor}
\proof  Indeed,  $f\bigl([g,e]\bigr)=
f\bigl(g\cdot [1,e]\bigr)= g\cdot f\bigl([1,e]\bigr)$ for all
$(g,e)\in G\times E_0$.\qed

\begin{defn}
  \label{defn:catrepaffext}
Let $\mathcal S:$ $\ate$ be an affine  extension of the abelian
variety $A$. We define the category  (enriched over $\Sch|\Bbbk$) of
   \emph{representations of $\mathcal S$} or \emph{$\mathcal S$--modules}, 
denoted as $\Rep(\mathcal S)$, as  follows: 

The \emph{objects} are the
\emph{representations} of $\mathcal S$.

If $(E,\varrho_E), (E',\varrho_{E'})\in \Rep(\mathcal S)$, then
$\Hom_{\Rep(\mathcal S)}(E,E') := {}^G\Homgr(E,E')$, the (homogeneous)
vector bundle of $G$--equivariant morphisms.

 We define
  $\RepO(\mathcal S)$ as  the wide subcategory of $\Rep(\mathcal S)$
  that has as  morphisms $\Hom_{\RepO(\mathcal
    S)}(E,E') := \Hom_{\Rep(\mathcal S)}(E,E')_0$. Notice that 
  $\Hom_{\RepO(\mathcal S)}(E,E')$  can be identified with the vector
  space of morphisms of vector bundles $f:E\to E'$ that commute with
  the action (compare with remarks
   \ref{rem:clasicosygraded} and \ref{rem:equivmorasmono}). 
 \end{defn}
 
 \begin{rem}\label{rem:notationqS}
   Later --- for reasons of notational uniformity --- we represent an
   affine extension such as $\mathcal S:$ $\ate$ simply as $q:G \to A$
   and $\Rep(\mathcal S)$ simply as $\Rep(q)$ and the same for
   $\RepO(\mathcal S)=\RepO(q)$. See the comments at the
   introduction of Section \ref{sect:Hopf} and also Example
   \ref{ej:qS}.   
 \end{rem}
 
The following theorem exhibits the relationship between $\Rep(\mathcal
S)$ and $\Rep(H)$ --- it can be seen as a generalization of
\cite[Theorem 2.9]{kn:brionrepbund}.

\begin{thm}
\label{thm:GaffG}
Let $\mathcal S:$ $\ate$ be an affine  extension and $V\in
\Rep_{\operatorname{fin}}(H)$ a finite dimensional
(rational) $H$--module. Then  $\pi_V:E_V=G\times^{H}
V=(G\times V)/H \to  A$ is a representation of $\mathcal S$ --- recall from
Theorem \ref{thm:indesp1} that the 
quotient $\pi_{G\times V}: G\times V\to E_V$  exists and that $\pi_V:E_V\to A$
is a vector bundle of fiber
isomorphic to $V$.

Conversely, if the vector bundle  $\pi :E\to A$ is a representation of
$\mathcal S$, then $E$ and $G\times^{H} E_0$ are isomorphic in the
category $\HVB_0(A)$, where  the action $H\times E_0\to
E_0$ is given by restriction.

Moreover, the category $\RepO(\mathcal S)$ is equivalent to
$\Rep_{\operatorname{fin}}(H)$. In
particular, $\RepO(\mathcal S)$  is an abelian, mo\-noi\-dal, rigid, category.
\end{thm}

\pf
The first assertion is the content of Theorem \ref{thm:indesp1}. 

Conversely, if $E\to A$ is a  $\mathcal S$--module, then $E_0$ is an
$H$--module and therefore, again by Theorem \ref{thm:indesp1}, the
induced space $E_V=G\times^HE_0$ is a representation of $\mathcal
S$. Moreover, the morphism $f:G\times E_0\to E$, $(g,v)\mapsto g\cdot v$
is $H$--invariant and therefore induces a morphism $\widetilde{f}:E_V\to E$ (given
by $\widetilde{f}\bigl([g,v]\bigr)=g\cdot v$).  Since  $\widetilde{f}$ is clearly a
bijective morphism of vector bundles, it follows that  $\widetilde{f}$ is an
isomorphism.

It is now an easy exercise to verify that a morphism of $H$--modules
$f: V\to W$  induces the  morphism of $\mathcal S$--modules
$\widetilde{f}: G\times^HV\to G\times^HW$
$\widetilde{f}\bigl([g,v]\bigr)=\bigl[ g, f(v)\bigr]$. Therefore, we have just 
constructed a functor $\Rep_{\operatorname{fin}}(H)\to \RepO(\mathcal
S)$ such that $V\mapsto
G\times^HV$ and $\Hom_{\Rep_{\operatorname{fin}}(H)}(V,W)\ni f\mapsto
\widetilde{f}\in
\HomO(G\times^HV,G\times^HW)$. This functor is clearly
 the inverse functor of the ``restriction to the fiber'' functor $\RepO(\mathcal S)\to
\Rep_{\operatorname{fin}}(H)$.  \qed

\begin{cor}
  \label{cor:GaffGimpliesiso}
Let $\mathcal S$: $\ate$ be an affine extension and
$E,E'\in\Rep(\mathcal S)$. Then $E\cong E'$ as $\mathcal S$--modules
if and only if $E_0\cong E'_0$ as $H$--modules.
  \end{cor}

  \pf
Since the $H$--action of $E_0,E'_0$ is by restriction, it is clear
that if $E\cong E'$, then $E_0\cong E'_0$. Assume now that $f: E_0\to
E'_0$ is an isomorphism of  $H$--modules. Then the canonical morphism
$G\times E_0\to E'=G\times^HE'_0$ given by $(g,v)\mapsto
\bigl[g,f(v)\bigr]$ induces a morphism $f: E=G\times^HE_0\to
E'$, that is clearly a $G$--equivariant isomorphism of vector bundles.\qed

  \begin{cor}
\label{cor:repasindGant}
Let $\mathcal S$: $\ate$ be an affine extension and $E\in\Rep(\mathcal
S)$. Then $E\cong G_{\ant}\times^{G_{\ant}\cap H}E_0$.
\end{cor}
\proof
Let $(G_{\ant}\cap H)\times E_0\to E_0$  be the $(G_{\ant}\cap
H)$--module obtained by restriction of the $H$--action. Since
$\mathcal S$ has a  Rosenlicht decomposition,  $G/(G_{\ant}\cap
H)=A$ and $G_{\ant}\times^{G_{\ant}\cap H}E_0\to A$ is a vector
bundle. The canonical inclusion $G_{\ant}\times E_0\hookrightarrow
G\times E_0$ induces a morphism of vector bundles
$G_{\ant}\times^{G_{\ant}\cap H}E_0\to G\times^HE_0=E$ (of the same
dimension) that is clearly
an isomorphism. \qed

\begin{cor}
\label{con:repisind}
 Let $\mathcal S:$ $\ate$ be an affine extension, and
 $\bigl(\pi: E\to A,\varrho_E:G\to \Authbext(E)\bigr) $  an $\mathcal S$--module. Then
 $E\cong \rho(G)\times^{\rho(H)}E_0$, where, as usual,  $\rho:G\to \Autgr(E)$ is
 the mid morphism of $\varrho_E$ and $E_0=\pi^{-1}(0)$.
 \end{cor}
\pf
Immediate.\qed

Combining Theorem \ref{thm:GaffG} with Corollary \ref{con:repisind} we
obtain the following characterization of an homogeneous vector bundle.

\begin{prop}\label{prop:charachb} Let $\pi:E\to A$ be vector
bundle. Then $E$ is homogeneous if and only if there exists an affine
extension $\mathcal S$: $\ate$ and an action $a: G\times E\to E$, such
that the following diagram is commutative.
\[ \xymatrix{ G\times E\ar[r]^a\ar[d]_{q\times \pi} & E\ar[d]^{\pi}\\
A\times A\ar[r]^s& A }
  \] and the restriction $a|_{_{H\times E_0}}H\times E_0\to E_0$ is a
linear representation of $H$.
\end{prop} \pf Let $E$, $\mathcal S$ and $a$ be as in the
hypothesis. Then, as in the proof of Theorem \ref{thm:GaffG},
$a|_{_{G\times E_0}} :G\times E_0\to E$ induces an isomorphisms
$G\times^HE_0\cong E$.  Conversely, if If $\pi:E\to A$ is homogeneous,
then $E$ is a $\Authbext(E)$--module by Corollary
\ref{con:repisind}.\qed

\begin{ejs}
\label{exam:rep}
(1) Let $\mathcal G_{\aff}:$  ${\xymatrixcolsep{1.5pc} \xymatrix{ 1\ar[r] & G\ar[r]^(0.45){\id} &G\ar[r] &\Spec(\Bbbk)\ar[r]
    &0}}$ be an affine group scheme viewed as an affine extension. Then
$\Rep(\mathcal G_{\aff})=\Rep(G)$, the ``classical'' category of 
representations of an affine group scheme (see Example \ref{exam:casoaffin1}).

\noindent (2) Let $\mathcal A$ be the trivial extension
${\xymatrixcolsep{1.5pc} \xymatrix{ 0\ar[r] & 0\ar[r]& A\ar[r]^(.45){\id} &A\ar[r] 
    &0}}$. Since a homogeneous vector bundle $E$ is trivial if and
only if there exists a section $A\hookrightarrow \Autgr(E)$ (see
Remark \ref{rem:authbmiya1}), it
follows that $\Rep(
\mathcal A)$ has as objects the trivial bundles $A\times V$,
with action $a: A\times(A\times V)\to A\times V$, $b\cdot
(c,v)=(b+c,v)$. On the other hand $\Hom_{\Rep(\mathcal
  S)}(E,E')=\Homgr(E,E')= \Hom_{\Bbbk}(E_0,E'_0)\times A$.

\noindent (3)  Consider an isogeny
$g:A\to A$ and the corresponding affine 
 extension  $\mathcal S_N:$ ${\xymatrixcolsep{1.5pc}\xymatrix{
     1\ar[r]& N\ar[r]& A\ar[r]^-{g}&A\cong A/N\ar[r]& 0}}$, where $N$ is a normal finite
 subgroup scheme. If $E\in \Rep (\mathcal S_N)$, then $E=A\times^NV$, where $V\in
\Rep(N)$.

It follows that $\Rep(\mathcal S_N)$ can be obtained as follows. Let
$\mathcal N$ be the category of the
trivial homogeneous vector bundles built on $\Rep(N)$: (i) $E\in
\operatorname{Obj}(\mathcal N)$ if $E= A\times V$, with $V\in
\Rep(N)$; (ii) $(f,\ell) \in \Hom_{\mathcal N}(A\times V,A\times V')(T)$ if
and only if $f=(t_\ell\times h):A\times T\times V\to A\times T\times
V'$, with $h\in \Hom_{\Rep(N)}(V,V')$. Consider the functor $Q: \mathcal N\to
\HVBG(A)$ given by the  quotient by the diagonal action $n\cdot
(a,v)=(an^{-1},nv)$. Then the $\Rep(\mathcal S_N)$ is the image of
$\mathcal N$ by $Q$.

\noindent (4) Assume that $\Bbbk=\overline{\Bbbk}$ and let $L\in
\operatorname{Pic}(A)$ be an invertible  
homogeneous vector bundle. Then $L^\times=L\setminus \theta(L)$, where
$\theta:A\to L$ is the trivial section, is a smooth group scheme, with
Chevalley decomposition induced by the canonical projection
$\pi:L\to A$ (see \cite[Theorem 2]{kn:oam} and \cite[Corollary 6]{kn:bprit}):
$\mathcal L^\times:$
${\xymatrixcolsep{0.5pc} \xymatrix{
1\ar[rr]&& \Bbbk^*\ar[rr]&& L^\times\ar[rrr]^-{\pi|_{_{L^{^{\times}}}}}&&&A\ar[rr]&& 0} }$.

It follows from Theorem \ref{thm:GaffG} that $E$ is an
$\mathcal L^\times$--module if and only if $E\cong L^\times\times^{\Bbbk^*}V$, where $V$ is a
$\Bbbk^*$--module. On the other hand,  it is clear that $L^{\otimes n}$
is an $\mathcal L^\times$--module, with action $L^\times\times L^{\otimes n}\to
L^{\otimes n}$ given by $a\cdot( l_1\otimes\dots\otimes l_n)=  (a\cdot
l_1)\otimes\dots\otimes (a\cdot l_n)$.
It follows that if $V\cong \oplus V_i$, where $a\cdot v=a^iv$ for
$v\in V_i$, then $E\cong \bigoplus_i \bigoplus_{j=1}^{\dim V_i} L^{\otimes i}$.
\end{ejs}

For further use in the next section, we  introduce now the definition of
\emph{$T$--morphisms of homogeneous vector bundles} and some related
notions.

\begin{defn}\label{defn:C(T)}
Let $\mathcal C$ be a $\Schk$--category and $T\in \Schk$. The
\emph{category $\mathcal C(T)$} is defined as follows:

\noindent (1) its \emph{objects} are the objects of $\mathcal C$.

\noindent (2) given two objects $x,y\in\mathcal C$, $\Hom_{\mathcal
  C(T)}(x,y)=\Hom_{\mathcal C}(x,y)(T)$.
\end{defn}

\begin{rem}\label{rem:Tfunctors}
(1) Notice that if $\mathcal C$ is a $\Schk$--category,  $T\in \Schk$ and
$x\in \mathcal C $, then  the structure morphism $\st:T\to
\Spec(\Bbbk)$ induces an identity morphism in $\End_{\mathcal
  C(T)}(x)$, by post-composition with the identity morphism
$\Spec(\Bbbk)\to \End_{\mathcal C}(x)$. 

\noindent (2) Let $\mathcal C,\mathcal D$ be two  $\Schk$--categories, and
$F:\mathcal C\to  \mathcal D$ a functor.  It is clear that $F$ induces
a functor $F(T):\mathcal C(T)\to \mathcal D(T)$. 
  \end{rem}

\begin{defn}
  Let $\mathcal S$ be an affine extension and $T\in \Schk$. We define
  the \emph{category   of $\mathcal S$--modules with $T$--morphisms},
  as the category  $\Rep(\mathcal S)(T)$.
\end{defn}

\begin{rem}
Notice that the degree morphism $\Homgr(E, E')\to A$ induces a degree
morphism  $\Hom_{\Rep(\mathcal S)(T)}(E,E')= \Homgr(E, E')(T)\to A(T)$.
  \end{rem}
  
\begin{defn}
Let $\mathcal C,\mathcal D$ be two  $\Schk$--categories, and
$F,G:\mathcal C\to  \mathcal D$ two functors. If $T\in \Schk$, a
\emph{$T$--natural transformation} is a natural transformation
$\lambda: F(T)\Rightarrow G(T)$.

A \emph{functor on natural transformations $\lambda:F\Rightarrow G$ } is a
functor $\lambda: \Schk^{\op}\to  \Sets$, such that $\lambda
(T):F(T)\Rightarrow G(T)$ is a $T$--natural transformation.
  \end{defn}

\subsection{The category $\Rep(\mathcal S)$}\ %
\label{sct:repg}

In this paragraph we collect some basic properties of the categories
of representations $\Rep(\mathcal S)$ and $\RepO(\mathcal S)$ of an
affine extension $\mathcal S:$ $\ate$.

\begin{rem}
\label{rem:tensordual}
Let $\mathcal S$  be an affine
extension. Even though the category $\Rep(\mathcal S)$ is not monoidal, a
 situation similar to the one described in Lemma
 \ref{lem:tensordualhvb} holds, since $\RepO(\mathcal S)$ is an 
 abelian monoidal rigid  category.

 Indeed, once Theorem \ref{thm:GaffG} is established,  the assertion
 for $\RepO(\mathcal S)$ is proved  by
 transplanting the corresponding structure from
 $\Rep_{\operatorname{fin}}(H)$.  In that manner we obtain for $E,E'$
 objects in $\RepO(\mathcal S)$ 
 (that has the same objects than $\Rep(\mathcal S)$), other objects in
 the same categories called $E^\vee$, $E\oplus E'$ and $E\otimes E'$
 and for arrows in $\RepO(\mathcal S)$ we can define the arrows --- also
 in $\RepO(\mathcal S)$ ---: $f^\vee$, $f \otimes g$, $f + g$ as well
 as the functors $\Ker$ and $\operatorname{Coker}$ in $\RepO(\mathcal
 S)$. This construction at the level of arrows can be defined
  directly or by transporting them from $\Rep_{\operatorname{fin}}(H)$.

 On the other hand,  Lemma \ref{lem:tensordualhvb} implies the following
 weaker version of the universal properties.

\end{rem}

\begin{lem}
Let $\mathcal S$ be an affine  extension and $E,E', F, F'$ objects in
$\Rep(\mathcal S)$. Consider the graded morphisms
$(f,\ell)\in \Hom_{\Rep(\mathcal S)}(E,F)(T)$, $(f',\ell)\in  
\Hom_{\Rep(\mathcal S)}(E',F')(T)$, and $ (g, \ell) \in
\Hom_{\Rep(\mathcal S)}(E',F)(T)$. Then $(f\otimes f', \ell)\in
\Hom_{\Rep(\mathcal S)}(E\otimes E',F\otimes F')(T)$, $(f+ g, \ell)\in
\Hom_{\Rep(\mathcal S)}(E\oplus E',F)(T)$ and $(f^\vee,-\ell)\in \Hom_{\Rep(\mathcal S)}(F^\vee,E^\vee)(T)$.
\end{lem}

\pf
Immediate.
\qed

\defn{
\label{def:repS}
Let $\mathcal S$: $\ate$ be an affine  extension of the abelian
variety $A$. We  call $\omegagr: \Rep(\mathcal S)\to
\HVBG(A)$ the  forgetful functor  in the category of
homogeneous vector bundles over $A$;  and
  $\omegaO : \RepO(\mathcal S) \to \HVB_0(A)$ is the functor induced by
  restriction of $\omegagr$ --- notice that $\omegaO$ is a
        monoidal functor.
        \[{\xymatrixcolsep{1pc}\xymatrix{\RepO(\mathcal
            S)\ar[d]_{\omegaO}\ar@{^(->}[rr]&&\Rep{\mathcal
              S}\ar[d]^{\omegagr}\\ 
   \HVB_0(A)\ar@{^(->}[rr]&&\HVBG(A).}}
\] 
      }

  \begin{nota}
In the future and in order to simplify the notations, if $(E, \varrho_E)\in \Rep(\mathcal S)$ we
often omit the morphism $\varrho_E$ and write that \emph{$E$ is an
  $\mathcal S$--module}. The forgetful functor $\omegagr:\Rep(\mathcal
S)\to \HVBG(A)$ is given at the level of objects by
$(E,\varrho_E)\mapsto E$.  Occasionally and when it does not
produce confusions, the forgetful functor applied to
  objects might be omitted and we write $\omegagr(E):=E$, and
  similarly for the hom-objects.
    \end{nota}

  \begin{rem}\label{rem:interpretation}
     Consider the functor $\omegagr:\Rep(\mathcal S)\to \HVBG(A)$ and
    let $T\in \Schk$. Then a  $T$--natural transformation  
  $\lambda:\omegagr(T)\Rightarrow
  \omegagr(T)$ is  given by a family
  of graded morphisms $\lambda_E=(f_E, \ell_E)\in
  \Endgr\bigl(\omegagr(E)\bigr)(T)=\Endgr(E)(T)$, $E\in \Rep(\mathcal
  S)$, such that  the graded morphisms
  $\lambda_E$ satisfy the following compatibility condition:

   For all $E, F \in \Rep(\mathcal S)$ and $(g,a) \in
  \Hom_{\Rep(\mathcal S)}(E,F)(T)={}^G\Homgr(E,F)(T)$ the diagram below,
  that is a diagram in $\Sch|T$, commutes:
  \[
      \xymatrix{\omegagr(E)_T\ar[d]_{g}\ar[rr]^{f_E}&&
      \omegagr(E)_T\ar[d]^{g}\\ \omegagr(F)_T\ar[rr]_{f_F}&&
      \omegagr(F)_T}
\] 

Notice that in the diagram above we use the fact that
$\omegagr(T)(g,a)=(g,a)\in \Homgr(E,F)$.

 In particular, given $E_1, E_2 \in \Rep(\mathcal S)$, then the canonical
 inclusions $\operatorname{inc}_i: (E_i)_T\to (E_1\oplus E_2)_T\cong (E_1\oplus
 E_2)_T$ induce the following commutative diagram: 
  \[
      \xymatrix{\omegagr(E_i)_T\ar[d]_{\omegagr (\operatorname{inc}_i,0)}\ar[rr]^{\lambda_{E_i}}&&
      \omegagr(E_i)_T\ar[d]^{\omegagr(\operatorname{inc}_i,0)}\\ \omegagr(E_1\oplus
      E_2)_T\ar[rr]_{\lambda_{E_1\oplus E_2}}&&
      \omegagr(E_1\oplus E_2)_T}
\] 

It follows that $d(\lambda_{E_1})=d(\lambda_{E_1\oplus
  E_2})=d(\lambda_{E_2})$. In other words, the degree of the morphisms
$\lambda_E$ is constant.
\end{rem}

\begin{defn}
In  view of Remark \ref{rem:interpretation}, if  $\lambda
:\omegagr(T)\Rightarrow 
 \omegagr(T)$ is a $T$--natural transformation, the \emph{degree} of
 $\Lambda$  is defined as $d(\lambda_{\mathbb I})$, where $\mathbb
 I=\Bbbk\times A\to A$ is the trivial representation 
 (see Example \ref{eje:reptriv}).
  \end{defn}

\begin{defn}\label{defn:transfnatw}
   \noindent  Given the functor $\omegagr:\Rep(\mathcal S)\to \HVBG(A)$ we
    consider  the functor on natural transformations
    $\End(\omegagr): \Schk\to \operatorname{Mon} $, defined 
    as  
    \[
      \End(\omegagr)(T)= \bigl\{ \lambda: \omegagr(T) \Rightarrow
      \omegagr(T)\mathrel{:} \lambda \text { is $T$-natural transformation}\bigr\}.
    \]

If $g:T'\to T$, then
$\End(\omegagr)(f): \End(\omegagr)(T)\to \End(\omegagr)(T')$ is given
by $\End(\omegagr)(f)\bigl(\{\lambda_E\}\bigr)=\{ \lambda'_E\}$, where
$\lambda'_E=\Endgr\bigl(\omegagr(E)\bigr)(g)=\Endgr(E)(g)$ (see
Definition \ref{defn:gradedhom}).
  \end{defn}

  \begin{rem}
(1) By definition, an element of  $\End(\omegagr)(T) $ is a family $\bigl\{
\lambda_E=(f_E,\ell)\in \Endgr(E)(T)\mathrel{:} E\in \Rep(\mathcal
S)\bigr\}$ (see Remark \ref{rem:interpretation}).

\noindent (2) The monoid structure on $ \End(\omegagr)(T)$ is given by
vertical composition of the families: $\{\lambda_E\}\smallcirc
\{\mu_E\}=\{\lambda_E\smallcirc \mu_E\}$. The unit of the monoid is
the family $\bigl\{ (\id_{E\times T},0)\in \Endgr(E)(T)\bigr\}$. Notice
that $d(\lambda_E\smallcirc \mu_E)=d(\lambda_E)+d( \mu_E)\in A(T)$.
\end{rem}

\begin{defn}
  The \emph{degree map} $d_{\omegagr}: \End(\omegagr)\to A$ is given
  by $d_{\omegagr}(T)\bigl(\lambda \bigr)=d\bigl(\lambda_E\bigr)=
  d\bigl(\lambda_{\mathbb I}\bigr)$.

 We denote by $\End_0(\omegagr)\subset \End(\omegagr)$ the subfunctor
 of the families of degree $0$ $\End_0(\omegagr)(T)=\bigl\{\lambda
\in \End(\omegagr)(T) \mathrel{:} d_{\omegagr}(\lambda)=0\bigr\}$.
\end{defn}

  \begin{rem}
    It is clear the degree map $d$ is a morphism of functors on monoids,
    and that $\End_0(\omegagr)=\operatorname{Ker}(d)$.
  \end{rem}

  \begin{defn}
    Define the subfunctor on monoids
    $\Aut(\omegagr)\subset \End(\omegagr)$ by 
    \[
      \Aut(\omegagr)(T) =
\bigl\{\lambda=\{\lambda_E\} \mathrel{:} \lambda_E\in
\Autgr(E)(T)\bigr\} \subset \End(\omegagr)(T)
\]
and the
corresponding subfunctor $\Aut_0(\omegagr)\subset \End_0(\omegagr)$ by
$\Aut_0(\omegagr)=\operatorname{Ker}\bigl(d_{\omegagr}|_{_{\Aut(\omegagr)}}\bigr):\Aut(\omegagr)\to A$:
\[
\Aut_0(\omegagr)(T) =\{\lambda \in \Aut(\omegagr)(T):
d_{\omegagr}(\lambda)=0\}.
\]
\end{defn}

\begin{rem}\label{rem:interpretation2}
  In accordance with Corollary \ref{cor:endgrismon} and Lemma
  \ref{lem:Smorfarebundle},  if $E:=(\pi:E\to A)\in
  \Rep(\mathcal S)$, then $\Endgr\bigl(\omegagr(E)\bigr)=\Endgr(E)$
  and $\omegagr\bigl(\End_{\Rep(\mathcal S)}$ are a smooth
monoid scheme of finite type, and
$\Autgr\bigl(\omegagr(E)\bigr)\hookrightarrow
\Endgr\bigl(\omegagr(E)\bigr)$ and   $\omegagr\bigl(\Aut_{\Rep(\mathcal S)}(E)\bigr)\hookrightarrow
\omegagr\bigl(\End_{\Rep(\mathcal S)} (E)\bigr)$ are open
immersions.

On the other hand,  is not clear that the functor on monoids $\End(\omegagr)$ and
$\Aut(\omegagr)$  are representable, since the situation much more
complex as one must  take into account the
complete natural transformation --- i.e.~the family $\lambda= \{\lambda
_E \in \Endgr(E): E \in \Rep(\mathcal S)\}$ --- as well as all the
consistency conditions (see Remark \ref{rem:interpretation}).
\end{rem}

  Next we define a subfunctor on monoids of $\End(\omegagr)$ and the
  corresponding subfunctor on groups,   that will be crucial
  for the reconstruction process.

\begin{defn}\label{defn:endotimes}
 (1)  In the context above, we call $\Endwgr$ the subfunctor on monoids of $\End(\omegagr)$
  given  by the natural transformations  $\lambda
  \in \End(\omegagr)(T)$, $T\in \Schk$,  such that:

  \begin{itemize}
\item[(i)] $\lambda_{E_1\otimes E_2}=\lambda_{E_1}\otimes\lambda_{E_2}$
for all $E_1,E_2\in \Rep(\mathcal S)$ (see Lemma \ref{lem:tensordualhvb});  

\item[(ii)] If $\mathbb I=(p_2: \Bbbk \times A \to A)$ is the trivial
  representation (see Example \ref{eje:reptriv}), then
  $(\lambda_{{\mathbb I}},\ell)=(\id_{\Bbbk}\times t_\ell, \ell) \in
  \Endgr(\Bbbk\times A)(T)$.
\end{itemize}

  \noindent (2) We define the subfunctor on monoids  $\Autwgr(T) = \bigl\{\lambda_E:
 \lambda_E \text{ is an isomorphism}\bigr\}\subset \Endwgr(T)$, for
 $T\in \Schk$ --- notice that $\Autwgr$ is a functor on groups.
\end{defn}

\begin{rem}
It is clear that  $\Autwgr$ can also be seen as a subfunctor of
$\Aut(\omegagr)$. 
\end{rem}

\begin{ej}\label{exam:casoaffin3}
$G$ be an affine group and let $\mathcal S$ be the associated affine 
extension of $\Spec(\Bbbk)$, (see examples \ref{exam:rep} and
\ref{exam:casoaffin1}). In this case, $\omegagr: \Rep(\mathcal S)\to
\HVBG\bigl(\Spec(\Bbbk)\bigr))$ is the forgetful functor $\omega:\Rep(G)\to
\operatorname{Vect}_\Bbbk$. If $T=\Spec(R)\in \Schkaff$, since
$\Hom_{\Rep(S)}(V,W)(T)\cong \Hom_{\Rep(G)}(V,W)\otimes R$, we 
deduce  that a
$T$-natural transformation $\omegagr(T)\Rightarrow \omegagr(T)$ is a
family $\lambda_V:V\otimes R\to 
V\otimes R$ of $R$--linear morphisms, such that the following diagram
is commutative for all $f\in \Hom_{\Rep(G)}(V,W)$
\[
  \xymatrix{
    V\otimes R \ar[d]_{f\otimes \id_R}\ar[r]^{\lambda_V} & V\otimes R\ar[d]^{f\otimes
      \id_R}\\
    W\otimes R\ar[r]^{\lambda_W} &W\otimes R
    }
  \]

Moreover, in the notation of \cite[page 20]{kn:delmil}
$\{\lambda_R\}\in\Autwgr(T)$ if and and only if 
$\{\lambda_V\}\in \Aut^\otimes(\omega)$.
\end{ej}

\begin{defn}\label{defn:endotimes2} Adapting Definition
  \ref{defn:transfnatw} we can consider the forgetful functor
  $\omega_0:\Rep_0(\mathcal S) \to \HVB_0(A)$ and define    
  $\End(\omegaO)$
  as \[\End(\omegaO)=\bigl\{\zeta:\omegaO\Rightarrow \omegaO \mathrel{:} \zeta
  \textrm { is a natural transformation}\bigr\}.\]

  We can proceed similarly and define: $\Aut(\omegaO), \EndwO, \AutwO$.  
\end{defn}

\begin{rem}\label{rem:endotimes}
  It is easy to see that  the functor on  monoids $\End_0(\omegagr)$
  and $\End(\omegaO)$; 
  $\EndwgrO$ and $\EndwO$; $\Aut_0(\omegagr)$ and $\Aut(\omegaO)$ as
  well as $\AutwgrO$ and $\AutwO$ are isomorphic.  For example, it is
  clear that $d(\lambda_E, \ell)=0$ if and only if $\ell=0$ and
  $\lambda_E \in \End(\omegaO)$.
  \end{rem}

\begin{defn}
  \label{defn:monoidgen}
  \noindent (1) Given $E$ an object in $\RepO(\mathcal S)$ (or in $\Rep(\mathcal
  S)$) we call $\RepO(\mathcal S)_E$ the abelian (monoidal)
  subcategory of $\RepO(\mathcal S)$ generated by $E$ and
  $\Rep(\mathcal S)_E$ the wide extension (in $\Rep(\mathcal S)$)
  obtained by taking the   graded morphisms.

  \noindent (2) For $E\in \Rep(\mathcal S)$ we call $\omegagr|:
  \Rep(\mathcal S)_E \to \HVBG(A)$ the restriction of the forgetful
  functor $\omegagr:\Rep(\mathcal S)\to \HVBG(A)$ to the subcategory
  $\Rep(\mathcal S)_E$ and similarly for $\omegaO|: \Rep_0(\mathcal
  S)_E \to \HVBG(A)$ the restriction of 
  $\omegaO:\Rep_0(\mathcal S)\to \HVBG(A)$
\end{defn}

\begin{rem}
  \label{rems:endotimesE} \noindent (1) The two categories $\RepO(\mathcal S)_E$ and $\Rep(\mathcal
        S)_E$ are defined along the same lines than the constructions of
        Definition \ref{defn:maincat}.

  \noindent (2) The structures just defined are illustrated in the
  following commutative diagram:
\[
  \xymatrix@=1em{
            \RepO(\mathcal
            S)\ar[dd]_{\omegaO} \ar@{<-^)}[dr]\ar@{^(->}[rrr]&&&\Rep{\mathcal
              S}\ar[dd]^{\omegagr}\\ 
    &\RepO(\mathcal S)_E\ar[dl]_{\omegaO|}\ar@{}[r]|-*[@]{\subseteq}&\Rep(\mathcal S)_E\ar@{^(->}[ur]\ar[dr]^{\omegagr|}&\\
   \HVB_0(A)\ar@{^(->}[rrr]&&&\HVBG(A)}
 \]

  \noindent (3) If $\mu\in
\operatorname{End}^{\otimes}\bigl(\omegagr|_{_{\Rep(\mathcal
  S)_E}}\bigr)$,  from the conditions on the family $\mu$ it follows
that $\mu_E$ determines $\mu$. Moreover, the universal property of the category $\Rep(\mathcal S)_E$ guarantees that $\operatorname{End}^{\otimes}\bigl(\omegagr|_{_{\Rep(\mathcal
   S)_E}}\bigr)$ is isomorphic with a   closed  submonoid scheme  
of  $\Endgr\bigl(\omegagr(E)\bigr)$ and hence it  is a
 monoid scheme of finite type. 
\end{rem}

\begin{ej}[The universal extension of an abelian variety]
\label{ej:ueav}
In \cite{kn:brionfundgrp} and  \cite{kn:brionrepbund}, Brion
constructs  the projective
  cover of $A$ in the category of commutative pro-algebraic group
  schemes. This cover has associated an affine  extension $\mathcal
  G_A$ of anti-affine type,  called the \emph{universal extension of
    the abelian variety $A$}. We prove in this example   that
  $\Rep(\mathcal G_A)\cong \HBA$.

Given a homogeneous vector bundle $E\to A$, consider the smooth
  affine extension $\Authbext(E)$   (see Remark
\ref{rem:algclosforhomog}), and let
$\Authbext(E)_{\ant}$ be the
associated closed sub-extension of anti-affine type (see Theorem
\ref{thm:rosenftgs2}). Then, by  Corollary \ref{cor:repasindGant},  
 $E \cong \Autgr(E)_{\ant}\times^{\Autgr(E)_{\ant}\cap \AutO(E)}E_0$.

Consider an affine faithfully flat filtered system within the family of the
affine extensions $\Authbext(E)_{\ant}$, $E\in \HVBG(A)$ --- for example, such a family can by
constructed  using  the partial order $E\leq E'$
if $E\cong E'\oplus E''$ for some homogeneous vector bundle $E''$, see
the proof of Lemma \ref{lem:autw} ---. Then, taking limit on $E$ we  get a (commutative)
 affine extension $\mathcal G_A$ together with morphisms $ \varrho_E:
 \mathcal G_A\to \Authbext(E)_{\ant}$: 
\[
  \xymatrix{
   1\ar[r]&H_A\ar[d]^{\rho_E|_{_{H_A}}}\ar[r]&G_A\ar[d]^{\rho_E}\ar[r]^-{q}&
  \ar[r]\ar@{=}[d]
  A&0\\
1\ar[r]&\Autgr(E)_{\ant}\cap \AutO(E)\ar[r]&\Autgr(E)_{\ant}
 \ar[r]_-{q_E}&\ar[r]
    A&0
}
\]

The affine extension $\mathcal G_A$ is called the \emph{universal (anti-affine) extension
  of the abelian 
variety $A$}.

The equivalence of Brion's construction and the construction of
$\mathcal G_A$  as an  limit, is a direct consequence of the
Tannaka Duality Theorem \ref{thm:reconstruction}, see Example
\ref{ej:recforueav} below. 

Observe that the affine extension $\mathcal G_A$, being the  limit of 
extensions of anti-affine type, is also an extension of  anti-affine type,  by Theorem
\ref{thm:limantiaff}. 

Next, we prove that $\Rep(\mathcal G_A)\cong \HBA$.

If $E\to A$ is a homogeneous vector bundle, then the morphism
$\varrho_A:\mathcal G_A\to \Authbext(E)_{\ant}\subset \Authbext(E)$ is
a representation for $\mathcal G_A$. Consider the restricted action
$H_A\times E_0\to E_0$;  by Theorem \ref{thm:GaffG}
$G_A\times^{H_A}E_0\to A$ exists and is a $\mathcal
G_A$--module. Clearly, $E\cong G_A\times^{H_A}E_0$ in $\HVB_0(A)$, 
and therefore the vector bundles are isomorphic in $\HBA$.

Moreover, let $E,E'\in\HBA$ be two vector bundles and consider the
structures of $\mathcal G_A$--modules defined above. Then
${}^{G_A}\Homgr(E,E')=\Homgr(E,E')$. Indeed, the 
action  $G_A\times\Homgr(E,E')\to \Homgr(E,E')$ (given as in Remark
\ref{rem:Gequivariant}) is such that
$\Homgr(E,E')_a=\HomO(E,E')\otimes \Bbbk(a)$ is
$(G_A)_{\Bbbk(a)}$--stable for all $a\in A$. Thus the anti-affine
group $(G_A)_{\Bbbk(a)}$ acts trivially on $\Homgr(E,E')_a$, since
$\HomO(E,E')_a$ is an affine $\Bbbk(a)$--space. It follows that $G_A$
acts trivially on $\Homgr(E,E')$ and $\Hom_{\mathcal G_A}(E,E')=\Homgr(E,E')$.

   The remarks above clearly show that the category  $\Rep(\mathcal G_A)$ is
   equivalent to
   $\HBA$.
  \end{ej}

  \begin{ej}
   Recall that any affine group scheme $G$ can be interpreted as an
   affine  extension of the trivial abelian variety $A=\Spec
   (\Bbbk)$ (see Example \ref{exam:trivials}); in particular,
   the trivial group $\Spec(\Bbbk)$ corresponds to the sequence
   \[
     {\xymatrixcolsep{1.5pc} \xymatrix{\mathcal E:& 1\ar[r] &
     \Spec(\Bbbk)\ar[r] &\Spec(\Bbbk)\ar[r] &\Spec{(\Bbbk)}\ar[r]
     &0}}.
 \]
 Analogously, the category $\HVB_0\bigl(\Spec(\Bbbk)\bigr)$
   is equivalent to $\operatorname{Vect}_\Bbbk$.

   Moreover,  $\Rep(\mathcal
   E)=\HVB_0\bigl(\Spec(\Bbbk)\bigr)\cong
   \operatorname{Vect}_\Bbbk=\Rep\bigl(\Spec(\Bbbk)\bigr)$. 
 On the other hand, since 
$\Autgr(V)=\operatorname{GL}(V)$ and that
$\operatorname{GL}(V)_{\ant}=\Spec(\Bbbk)$, it follows that
$\mathcal G_{\Spec(\Bbbk)}$ is the  limit of the constant trivial
extension $\mathcal E$. Hence, $\mathcal
G_{\Spec(\Bbbk)}=\mathcal E$ and in particular
$G_{\Spec(\Bbbk)}=\Spec(\Bbbk)$ --- as expected 
from the Tannaka Duality Theorem for affine group schemes applied to
the category $\operatorname{Vect}_\Bbbk $ with the identity as
forgetful functor.
  \end{ej}

The
  definition that follows is the natural generalization of the  one
  referred to in the affine case.

\begin{defn}
An $\mathcal S$--module $E\in \Rep(\mathcal S)$ is \emph{faithful} is the
corresponding morphism $\mathcal S\to\Authbext (E)$ is a closed
immersion of affine extensions. 
\end{defn}

\begin{rem}
Let $\mathcal S:$ $\ate$ be an affine  extension
and $\varrho: \mathcal S\to \Authbext(E)$ be a  representation. Since
$H\hookrightarrow G$ is a  closed immersion, it follows that $\varrho$
is faithful if and only if $\rho:G\to\Autgr(E)$ is a closed immersion,
if and only if $\rho$ is an immersion (since $G$ is a quasi-compact
group scheme, see
Theorem \ref{thm:perrinigame}).
\end{rem}

\begin{thm}
\label{thm:ftypefaith}
Let $\mathcal S:$ $\ate$ be an affine  extension. Then $\mathcal S$ is of finite type if and only if there exists a
faithful $\mathcal S$--module $E\in \Rep(\mathcal S)$.
\end{thm}

\pf
Recall that $G$ is of finite type if and only if $H$ is so (see Remark
\ref{rem:torsoraffine1}).
If $H$ is of finite type, then there exists a faithful representation
$\rho_V: H \hookrightarrow \operatorname{GL}(V)$. Consider the induced
$\mathcal S$--module
$E_V=G\times^HV$ (see Theorem \ref{thm:GaffG}). Then
we have a morphism  of
affine  extensions
\[
\xymatrix{
\mathcal S:\ar[d]_{\varrho}& 1\ar[r]& H\ar[r]\ar[d]_{\rho |_{_H}} & G\ar[r]^q\ar[d]_\rho  & A\ar[r]\ar@{=}[d]& 0 \\
\Authbext (E_V): & 1\ar[r]& \AutO(E)\ar[r] & \Autgr(E)\ar[r]  & A\ar[r]& 0 
}
\]
where $\rho |_{_H}: H\to \AutO(E)$ is a closed
immersion. It follows that $\varrho$ is a closed immersion (since
$\operatorname{Ker}(\rho)\subset H$). 

On the other hand, if there exists a faithful representation $\varrho:
\mathcal S\to
\Autgr(E)$, then the restriction $\rho |_{_{H}}:H\to \AutO(E)$ is
a closed immersion. It follows that the restriction
$\overline{\varrho} |_{_{H\times
  E_0}}: H\times E_0\to E_0$ is  a faithful representation of
$H$. Therefore, $\mathcal S$ is of finite type.
\qed

\begin{lem}
\label{lem:stabilizer}
Let $\mathcal S:$ $\ate$ be an affine  extension, and
$\mathcal S':$ $ {\xymatrixcolsep{0.5pc}\xymatrix{1\ar[rr]&&
H' \ar[rr]&& G'\ar[rrr]^-{q|_{_{G'}}}&&& A\ar[rr]&& 0}}$  a closed
sub-extension  of $\mathcal S$. Then there exists a
homogeneous vector bundle $E\in \Rep(\mathcal S)$ and a homogeneous
line sub-bundle $L\subset E$, such that $G'$ is 
the stabilizer of $L$, that is  for all schemes $T$,
\[
G'(T)=\bigl\{
g\in G(T)\mathrel{:} g \text{ induces a $T$--automorphism in }L\times T\bigr\}
\]
(see for example \cite[\S\ 2.2]{kn:brionchev}).
\end{lem}

\pf It is well known that given the pair $H'\subset H$ as above, there
exists a finite dimensional $H$--module $V$ and a one dimensional
subspace $W\subset V$ such that $H'$ is the stabilizer of $W$,
i.e.~$H'$ is the largest closed group subscheme of $H$ such that
$H'\cdot W\subset W$ (see for example \cite[Chapter 8, Theorem
  2.3]{fer-ritt}).  Since $\mathcal S'$ is an affine extension, it
follows from Theorem \ref{thm:GaffG} that the quotients
$E_V=G\times^HV$ and $E_W=G'\times^{H'}W$ exist and are
representations of the extensions $\mathcal S$ and $\mathcal S'$
respectively.  We affirm that $\varphi: E_W\to E_V$, the morphism
induced by the canonical morphism $G' \times W\to E_V$, $(g,w)\mapsto
[g,w]$; is an immersion of vector bundles. Indeed, if
$\xi_i=[g_i,w_i]\in E_W= G' \times^{H'}W$, $i=1,2$, such that
$[g_1,w_1]=[g_2,w_2] \in E_V=G\times^{H}V$, then there exists $h\in H$
such that $g_2h=g_1$ and $w_1=h\cdot w_2$. It follows that $h\in G'$
and therefore $h\in H'$; hence, $\xi_1=\xi_2$.

 Let $L=\varphi(E_W)\subset E_V$ be the subvector
 bundle image of 
 $\varphi$; we prove that $L\subset E_V$  does the required job for $G$ and $G'$.
 Let $g\in G$ be such that $g\cdot L=L$; we want to prove
 that $g\in G'$.  Since $g$ stabilizes $L$, it follows that $g\cdot
 [g_1,w_1]=[gg_1,w_1]\in L$ for all $[g_1,w_1]\in L$; therefore there exist 
 $g_2\in G'$, $w_2\in W$ such
 that $[gg_1,w_1]= [g_2,w_2]$. 

 Assume that $g\in H$. If moreover $g_1=1$, then
$[g_2,w_2]=[g,w_1]=[1,gw_1]$, and  there exists $t\in H$ such that
$t=g_2$ and $t w_2= g w_1$. It follows that $t\in H\cap G'$,
and thus $gw_1\in
W$ for all $w_1\in W$. Therefore, $g\in H'$.

If $g\in G(T)$ is arbitrary,   let  $f:T'\to T$ a fpqc morphism and $c\in
 G'\bigl(T')$,  $q\smallcirc c = q\smallcirc g\smallcirc f \in A(T')$ (such a pair
 $(f,c)$ exists 
because $\mathcal S'$ is a short exact sequence). Then  $(g\smallcirc f) c^{-1}\in
H(T')$  
stabilizes $L\bigl(T'\bigr)$ and therefore $(g\smallcirc f)c^{-1}\in
H'(T')$. It follows that $g\smallcirc f \in
G'(T')$, and hence $g\in G'(T)$ (indeed $f$ is a faithfully flat
morphism and hence we can apply Lemma \ref{lem:ffandpoints}).
\qed

\section{Recovering an   affine  extension from its representations}
\label{sect:reconstruction}

In this section we fix an
   affine  extension $\mathcal S:$ $\ate$,  $\mathcal S=\lim
  \mathcal S_\alpha$, where $\bigl\{\mathcal S_\alpha:
  \!\!{\xymatrixcolsep{1.2pc}\xymatrix{1\ar[r]& H_\alpha \ar[r]&
      G_\alpha \ar[r]^{q_\alpha}& A \ar[r]& 0}};
  \phi_{\alpha,\beta}\}_{\alpha,\beta \in I}$ is an (affine)
  faithfully flat filtered system
  of  affine  extensions of finite type. Call
  $\phi_\alpha:\mathcal S \to \mathcal S_\alpha$ the canonical maps
  depicted in the diagram below:
\[
\xymatrix{\mathcal S:\ar[d]^{\phi_\alpha}& 1\ar[r]& H
  \ar[d]_{f_{\alpha}|_{_H}}\ar[r]&
  G\ar[d]^{f_{\alpha}}\ar[r]^{q} & A\ar@{=}[d]\ar[r] &0 \\ \mathcal
  S_\alpha:&1\ar[r]& H_\alpha \ar[r]& G_\alpha\ar[r]^{q_\alpha} &
  A\ar[r] &0.}
\]

 As in the classical case of Tannaka Duality
for affine group schemes, given now the more general situation of an
affine  extension $\mathcal S$ and the
category $\Rep(\mathcal S)$, we characterize
 $G$ as the group scheme consisting of all the (families of)
 automorphisms of the objects $E \in \Rep(\mathcal S)$ 
that commute with all the morphisms of the category  $f\in\Hom_{\Rep(\mathcal S)}(E,E')$ and
 that satisfy  additional compatibility conditions related to the abelian and
monoidal properties of $\Rep(\mathcal S)$. In order to formalize this
idea, we will make use of the forgetful functors
$\omegagr:\Rep(\mathcal S)\to \HVBG(A)$ and $\omegaO: \RepO(\mathcal
S)\to \HVB_0(A)$, as well as the associated functors on groups
$\Autwgr$ and $\AutwgrO\cong \AutwO$ (see definitions \ref{def:repS}
and \ref{defn:endotimes}).

Following the usual pattern and similarly to the classical case, we first treat the
problem in the ``finite type'' setting, and then take 
limits.

  \begin{rem}
\label{rems:autw}
(1) By definition, $(\lambda_E,\ell)_E\in \Autwgr(T)$ if

\noindent (i) The morphisms  $\lambda_{E}$ fit in
the commutative diagram of $T$--schemes
\[
{\xymatrixcolsep{2.8pc} 
\xymatrixrowsep{1.7pc} 
\xymatrix{ 
E_T=E\times T\ar[d]_{\pi_E\times \id_T}\ar[r]^-{\lambda_{E}}&E\times T\ar[d]^{\pi_E\times \id_T}\\
A_T=A\times T\ar[r]_-{t_\ell}& A\times T
}}
\]
for all $E\in\Rep(\mathcal S)$, and the induced morphisms
$\widehat{\lambda_E}: E_T\to
t_\ell^*(E_T)$ are isomorphisms of $A_T$--vector bundles (recall that
$\ell\in A(T)$);

\noindent (ii)   for all $E,E'\in
\Rep(\mathcal S)$ we have equalities of morphisms of $A_T$--vector bundles  $\widehat{\lambda_{E \otimes E'}} =
\widehat{ \lambda_{E}}\otimes\widehat{\lambda_{E'}} :(E_T\otimes E'_T)\to t_\ell^*(E_T\otimes
E'_T)$;

\noindent (iii)  
$\lambda_{{\mathbb I}}=\bigl(\id_\Bbbk\times t_\ell, \ell\bigr): (\Bbbk\times A) \times T\to (\Bbbk\times
A)\times T$, where $\mathbb I$ is the trivial representation, and

\noindent (iv) for every
$G$-equivariant morphism $(f, b) \in \Hom_{\Rep(\mathcal S)}(E,E')(T)$ the following diagram of morphisms
of $T$--schemes is commutative:
\[
{\xymatrixcolsep{3pc}
\xymatrix{ 
E\times T\ar[r]^{f}\ar[d]_{\lambda_{E}}& E'\times 
T\ar[d]^{\lambda_{E'}}\\
E\times T\ar[r]_{f}& E'\times 
T
 }}
\]

  \noindent (2) There exists a canonical morphism (natural
  transformation) from the group functor $G$ into
   $\Autwgr$, given as follows. If $T$ is a
  scheme, we consider the morphism of groups $(g:T\to G) \mapsto
  \overline{g}=\bigl(\rho_E(T)(g)\bigr)_E: G(T) \to \Autwgr(T)$,
  where $\varrho_E:\mathcal S\to \Authbext(E)$ is as usual the
  morphism of affine extensions associated to the representation $E$.

Observe that if $\rho_E(T)(g)=(\rho_g,b)$, then the morphisms of
$T$--schemes $\rho_g:E_T\to E_T$ satisfy the
following commutative diagram.
  \[
        {\xymatrixcolsep{2.8pc} \xymatrixrowsep{1.7pc} \xymatrix{ E
            \times T\ar[d]_{\pi_{E}\times
              \id_T}\ar[rr]^{\rho_g}&&E \times
            T\ar[d]^{\pi_{E}\times \id_T}\\ A \times
            T\ar[rr]_{t_b}&& A\times T,}}
      \]
        and induce   morphisms of $A_T$--vector bundles
        $\widehat{\rho_E(T)(g)}:E_T\to t_b^*E_T$. Moreover, by definition of
        $\Rep(\mathcal S)$,  the commutativity of
        the maps $\overline{g}$ with the maps that come from applying
        the forgetful functor (condition stated in Remark
        \ref{rems:autw}, (iv))  follows directly.  Regarding the other requirements in
        the remarks just mentioned we have that condition (i) was
        already checked, and conditions (ii) and (iii) are direct.  

\noindent (3) Let $E\in\Rep(\mathcal S)$ and consider the restriction
of the forgetful functor 
$\omegagr:\Rep(\mathcal S)\to \HVBG(A)$ to the subcategory
$\Rep(\mathcal S)_E$ (see Definition \ref{defn:monoidgen}
and Remark \ref{rems:endotimesE}) and construct
the corresponding group functor $\Autwe$. Then the map $\lambda\mapsto
\lambda_E$ identifies $\Autwe$ as a group subfunctor with its image in
$\Autgr(E)\subset \Aut(E)$. Moreover, it follows (in a similar manner
than in the mentioned remark)
  that $\Autwe$ can be identified with a closed subgroup scheme of
    the smooth group scheme of finite type  $\Autgr(E)$ and therefore $\Autwe$ is of
    finite type --- $\Autwe$ is the unit group of the algebraic monoid
  scheme $\operatorname{End}^{\otimes}\bigl(\omegagr|_{_{\Rep(\mathcal
    S)_E}}\bigr)$.
\end{rem}
  
\begin{rem}
  \label{rem:autw3}
Let $(E,\varrho_E)\in \Rep(\mathcal S)$. We denote the
scheme theoretic image  $\rho_E(G)$ by $G_E$. Since $\Autgr(E)$ of
finite type, it follows that $G_E$ is a group scheme  of finite type, 
and the morphism $\varrho_E$ factors through an affine 
subextension $\mathcal S_E$ as follows
\[
  {\xymatrixcolsep{1.5pc}
    \xymatrix{
\mathcal S:\ar[d]^{\varrho_E}& 1\ar[r]&  H\ar[r]\ar[d]&  G \ar[d]^{\rho_E}\ar[r]^q& A \ar@{=}[d]\ar[r]&  0\\
\mathcal S_E:\ar@{^(->}[d]& 1\ar[r]& (G_E)_0=\rho_E(H)\ar@{^(->}[d]\ar[r]& G_E\ar@{^(->}[d]\ar[r]^-{d_E|_{_{G_E}}}& A \ar@{=}[d]\ar[r]&  0\\
\Authbext(E): & 1\ar[r]& \AutO(E)\ar[r]& \Autgr(E)\ar[r]_-{d_E}& A \ar[r]&  0
} }
\]
\end{rem}

\begin{lem}
\label{lem:repGX}
Let $E\in \Rep (\mathcal S)$. Then $\Rep(\mathcal S)_E\cong
\Rep(\mathcal S_E)$. Moreover, the canonical inclusion
$G_E\hookrightarrow 
\Autwe$ is an isomorphism. In particular, the
corresponding 
affine  extensions are isomorphic.
\end{lem}

\pf
Recall that $G_E\subset \Autgr(E)$ is a closed  subgroup scheme, and
hence of finite type.
Since any representation of $G_E$ (resp.~$\Autgr(E)$) is a $G$-homogeneous vector bundle, and that
$E$ is a faithful representation of $G_E$  (resp.~$\Autgr(E)$), it
follows that any 
representation of $G_E$ (resp.~$\Autgr(E)$)  belongs to $\Rep(\mathcal
S)_E$. Indeed, 
it follows from Theorem \ref{thm:GaffG} that $E_0$ is a faithful
representation of $(G_E)_{0}$ and $\AutO(E)$; therefore,
any $(G_E)_{0}$--module (resp.~$\AutO(E)$--module) belongs
to $(\operatorname{Vect}_\Bbbk )_{E_0}$ (see for example 
\cite[\S~3.5]{kn:waterhouse}). Applying again Theorem \ref{thm:GaffG}
we deduce that 
$\operatorname{Obj}\bigl(\Rep(\mathcal S_E)\bigr)=\operatorname{Obj}\bigl(\RepO(\mathcal S_E)\bigr)$
and
$\operatorname{Obj}\bigl(\Rep\bigl(\Autgr(E)\bigr)\bigr)=\operatorname{Obj}\bigl(\RepO
\bigl(\Autgr(E)\bigr)\bigr)$
are contained in $\operatorname{Obj}\bigl(\Rep(\mathcal S)_E\bigr)$.

Let $F\in \Rep(\mathcal S)_E$ be a $\Autgr(E)$--homogeneous vector
bundle and $L\subset F$ a $G_E$--line sub-bundle. We
affirm  that $\Autwe$
stabilizes $L$. If this is the case, since $\Autwe\subset \Autgr(E)$
(by the assignment $(\lambda_{E'},\ell)\mapsto (\lambda_E, \ell)$) is an
closed subgroup scheme, it follows from Lemma
\ref{lem:stabilizer} applied to $G_E\subset \Autgr(E)$ that
$\Autwe=G_E$; in particular, notice  that $\AutO(E)=(G_E)_0$.

Let $L\subset F$ as before; then the morphism $\rho_E:G\to \Autgr(E)$
induces  $G$--linearizations on $L$ and $F$.
 Since the inclusion $\iota:L\hookrightarrow
F$ is $G_E$--equivariant, it is also $G$--equivariant, and it follows
that if $T$ is a $\Bbbk$--scheme, $g\in G(T)$,
$(\lambda_{E'},b)_{E'}\in \Autwe(T)$  and
$(\ell,t)\in L\times T$, then 
\[
\lambda_E(\ell,t)=\bigl(\lambda_E \smallcirc(\iota\times \id_T)\bigr)(\ell,t)=
\bigl((\iota\times \id_T)
\smallcirc \lambda_{L}\bigr)(\ell,t)\in L\times T .
\]
In other words, $(\lambda_E,b)$ stabilizes $L$, and therefore
$(\lambda_{E'},b)_{E'}\in \Autgr(E)$ stabilizes $L$.\qed

\begin{lem}
\label{lem:autw}
Let $\mathcal S$ be an  affine  extension. Let
$\AutwgrO\subset \Autwgr$ be the subgroup functor constructed in
Definition \ref{defn:endotimes}.
Then the sequence 
\[
{\xymatrixcolsep{0.7pc}\xymatrix{\operatorname {\mathit{Aut}}^\otimes(\omegagr) :&  1\ar[rr]&&\AutwgrO=\AutwO\ar[rr]&& \Autwgr \ar[rr]^(.60){q_\omegagr}&& A
  \ar[rr]&& 0}}
\]
is the  limit of  the affine faithfully flat filtered system of  affine extensions of finite type
$\operatorname{\mathit{Aut}}^\otimes\bigl(\omegagr |_{_{\Rep(\mathcal S_E)}}\bigr)$:
\[
{\xymatrixcolsep{1em}\xymatrix{  1\ar[r]& \AutwgrOe=\AutweO \ar[r]& \Autwe  \ar[rrr]^-{q_{\omegagr|_{_{\Rep(\mathcal S)_E}}}}&&& A\ar[r] & 0}}
\]
where the system is directed as follows: if $E,E' \in \Rep(\mathcal S)$, then
$E'\geq E$ if and only $E=E'\oplus F$ for some  $F\in \Rep(\mathcal
S)$, with transition morphisms given by restriction.

In particular, $\operatorname {\mathit{Aut}}^\otimes(\omegagr) $ is an affine  extension. 
\end{lem}

\pf
Ii is clear that if $E'\geq E$, then $ \Rep(\mathcal S)_{E'}\subset \Rep(\mathcal
S)_{E}$, and the system defined above is filtered, with transition 
morphisms given by restriction.
\[
{\xymatrixcolsep{1.3pc} 
\xymatrixrowsep{1.5pc} 
\xymatrix{ 
 1\ar[r]&  (G_E)_0\ar[r]\ar@{=}[d]& 
G_E\ar[rrr]\ar@{=}[d]&& &A \ar@{=}[d]  \ar[r]& 0\\
1\ar[r]&
  \AutwgrOe\ar[r]\ar[d]&
  \Autwe \ar[rrr]^-{q_{\omegagr|_{_{\Rep(\mathcal S)_E}}}}\ar[d]&&&
  A\ar[r]\ar@{=}[d] &  0\\
1\ar[r]& \Aut_0^{\otimes}\bigl(\omegagr|_{_{\Rep(\mathcal S)_{E'}}}\bigr)\ar[r]& \Aut^{\otimes}\bigl(\omegagr|_{_{\Rep(\mathcal S)_{E'}}}\bigr) \ar[rrr]_-{q_{\omegagr|_{_{\Rep(\mathcal S)_{E'}}}}}&&& A\ar[r] & 0}}
\]

Moreover, by the very definition of $\Autwgr$ and $\AutwgrO$ as
group functors,  it
follows that 
\[
\xymatrix{
\xymatrixrowsep{1.5pc} 
1\ar[r]&  \AutwgrO\ar[r]\ar@{=}[d]& 
\Autwgr\ar@{=}[d]\ar[r]& A \ar@{=}[d]\ar[r]&  0\\
1\ar[r]& \lim \AutwgrO\ar[r]& 
\lim \Autwe \ar[r]& A \ar[r]&  0
} 
\]
\qed

\begin{nota}
\label{nota:autw}
In what follows, $\mathcal K$  denotes the (affine, faithfully flat)
filtered system defined in Lemma \ref{lem:autw} above.  
\end{nota}

\begin{thm}[Reconstruction of  affine  extensions]
\label{thm:reconstruction}
Let $\mathcal S$ be an affine  extension. Then the natural map $\varphi: G\to \Autwgr$ is an isomorphism of functors
$G\cong \Autwgr:\operatorname{\mathit{Sch}}^{op}\to
\mathrm{Groups}$. Moreover, this isomorphism  induces
an 
isomorphism of affine  extensions
\[
{\xymatrixcolsep{1.5pc}\xymatrix{
\mathcal S:\ar[d]_(.4)\phi ^(.4){\cong}& 1\ar[r] & H\ar[r]\ar[d]_(.4){f|_{_H}}^(.4){\cong} & G\ar[r]^q\ar[d]_(.4)f^(.4){\cong} & A\ar[r]\ar@{=}[d] & 0\\
\operatorname{\mathit{Aut}}^\otimes(\omegagr)\big):&  1\ar[r] & \AutwgrO=\AutwO\ar[r] & \Autwgr\ar[r] & A\ar[r] & 0
}}
\]

In particular, two  affine  extensions $\mathcal S$
and $\mathcal 
S'$ of the abelian variety $A$ are isomorphic if and only if there
exists an equivalence of categories $F: \Rep(\mathcal S)\to
\Rep(\mathcal S')$ such that $F|_{_{\RepO(\mathcal S)}}: \RepO(\mathcal
S)\to \RepO(\mathcal S')$ is a monoidal functor and the following
diagram is commutative
\[
\xymatrix{
\Rep(\mathcal S)\ar[rd]_{\operatorname{\omega_{gr,\Rep(\mathcal
      S)}}}\ar[rr]^F && \Rep(\mathcal
S')\ar[ld]^{\operatorname{\omega_{gr,\Rep(\mathcal S')}}}\\ 
& \HBA &
 }
\]

\end{thm}

\pf
Let $E\in\Rep(\mathcal S)$ and $G_E\subset \Autgr(E)$ be the scheme theoretic
image of $\rho_E:G\to \Autgr(E)$. The group $G_E$ is by definition a
closed subgroup scheme of
$\Autgr(E)$, and fits into the    affine extension $\mathcal
S_E=\varrho_E(\mathcal S)$ (see Remark \ref{rem:autw3}). Moreover, $G_E=  \Aut(\omegagr|_{_{\Rep(\mathcal S)_E}}) \subset
\Autgr(E)$  by Lemma
\ref{lem:repGX}.

We direct the system of  affine  extensions $\{\mathcal S_E\}_{E\in
  \Rep(\mathcal S)}$
 by $E'\geq E$ if and only if the representation $E'$ factorizes
 trough $G_{E}$ --- i.e.~there exists a morphism of group schemes
 $\rho_{E,E'}:G_{E}\to G_{E'}$, with $\rho_{E'}=\rho_{E,E'}\smallcirc
 \rho_{E}$. In particular, if $E'\geq E$, then $E'\in \Rep(\mathcal S_{E})$; it follows that
 $\Rep(\mathcal S)_{E'}\subset \Rep(\mathcal S)_{E}$. Hence,
 we have the following commutative diagram of  group schemes (of
 finite type)
\[
\xymatrix{
 G_{E}\ar[r]^-{\cong}\ar[d]_{\rho_{E,E'}}&\Aut^\otimes\bigl(\omegagr|_{_{\Rep(\mathcal S)_{E}}}\bigr)\ar[d]^{f_{E,E'}}\\
G_{E'}\ar[r]_-{\cong}&\Aut^\otimes\bigl(\omegagr|_{_{\Rep(\mathcal S)_{E'}}}\bigr)
}
\]
that fits in a commutative diagram of affine  extensions. In
particular, one has that  $f_{E, E'}$ is an epimorphism if and 
only if $\rho_{E,E'}$ is so. It is clear that these morphisms induce
an affine faithfully flat filtered system indexed by $\Rep(\mathcal S)$, that we call $\mathcal J$. 

Since $\mathcal S$ is an affine  extension, we deduce from
Theorem \ref{thm:ftypefaith} that $\mathcal S$ is the 
 limit of a subsystem  of affine  extensions
$\{\mathcal S_E\}_{E\in I}$, $I\subset \mathcal J$, and therefore
$\lim_{\mathcal J} \mathcal S_E=\lim_{I} \mathcal
S_E=\mathcal S$.

On the other hand, it follows from Lemma \ref{lem:autw}  that the systems of  affine
 extensions  $\{\mathcal S_E\}_{\mathcal J}$ and $
\bigr\{\operatorname{\mathit{Aut}}^\otimes(\omegagr|_{_{\Rep(\mathcal
      S_E)}})\bigr\}_{\mathcal   K}$ (see Notation \ref{nota:autw}) have the same   limit
$\lim_{\mathcal J} \mathcal S_E=\lim_{\mathcal K}
\operatorname{\mathit{Aut}}^\otimes(\omegagr|_{_{\Rep(\mathcal
    S_E)}})=\operatorname{\mathit{Aut}}^\otimes(\omegagr)$.

The last assertion is clear.
\qed

\begin{defn}\label{defn:gene}
  Let  $\mathcal S:$ $\ate$ be an affine  extension and $E$ an
  object in $\Rep(\mathcal S)$. Call $\langle E \rangle$ the full
  subcategory of $\Rep(\mathcal S)$ generated by the objects of the
  form $E^n$ and its subquotients. 
\end{defn}

\begin{prop}
\label{prop:reconscons}
Let $\mathcal S:$ $\ate$ be an affine  extension. Then

\noindent (1) $H$ is a finite group if and only if there exists a
representation $E\in \Rep(\mathcal S)$ such that any object in
$\Rep(\mathcal S)$ is isomorphic to an object of $\langle E\rangle$.
In
particular, the extension $\mathcal S$ is of finite
type.

\noindent (2) $G$ is a group scheme  of finite type if and only if there
exist $E\in \Rep (\mathcal S)$ such
  that $\Rep (\mathcal S)=\Rep(\mathcal S)_E$ (see Definition \ref{defn:monoidgen}).

\noindent (3) $\mathcal S$: ${\xymatrixcolsep{1.5pc}\xymatrix{
    1\ar[r]& H\ar[r]&H\times A\ar[r]&A\ar[r]&0}}$ is a trivial extension of $A$,
  if and only if any representation of $\mathcal S $ is 
    constructed over a trivial bundle $\Bbbk^n\times
  A$ (compare with  Example \ref{exam:rep} (3)). 
\end{prop}

\pf (1) It is enough to
  prove the corresponding result for $\Rep_0(S)$ (see Theorem
  \ref{thm:GaffG}). For the proof in this situation of the classical
  representation theory of affine groups, see for example
  \cite[Prop. 2.20]{kn:delmil}.
  
  \noindent (2) Just combine Theorem \ref{thm:ftypefaith} and Lemma
\ref{lem:repGX}, together with the fact that if 
 $E\in \operatorname{Obj}\bigl(\Rep(\mathcal S)\bigr)$ is such that
$\Rep(\mathcal S)_E=\Rep(\mathcal S)$, then $G\cong G_E$ by the Reconstruction Theorem.

\noindent (3) If $G=H\times A$, and $E$ is a representation, then we
clearly 
have a section $A\to \Autgr(E)$ of the corresponding affine  extension. It follows that $E$ is a trivial
homogeneous vector bundle (see Remark \ref{rem:authbmiya1} above).

Assume now that any $\mathcal S$--representation is trivial. Since $\Autgr(\Bbbk^n\times A)=\operatorname{GL}_n(\Bbbk)\times A$, it
follows that $G_E=K_E\times A$ for some closed subgroup scheme
$K_E\subset \operatorname{GL}_n(\Bbbk)$. Therefore, $G\cong\lim
G_E=\lim K_E\times A=K\times A$, where $K$ is the affine group
scheme $K=\lim K_E$.
\qed

\section{The Recognition Theorem}
\label{sect:recognition}

Once that the Reconstruction Theorem
\ref{thm:reconstruction} has been proved, its combination with the structure Theorem
\ref{thm:GaffG} and with the Recognition Theorem for affine group schemes, yields the Recognition Theorem for affine  extensions.

\begin{thm}[Recognition Theorem]
\label{thm:recognition}
Let $(\mathcal C, \omegagr)$ be a 
  category $\mathcal C$, enriched over $\Sch|\Bbbk$, 
together with a fully faithful 
functor
$\omegagr :\mathcal C\to \HVBG(A)$, such that:

\noindent (1) $\Hom_{\mathcal C}(X,Y)$ is a homogeneous vector
bundle over $A$.

  \noindent (2) For any pair of objects
$X,Y\in\mathcal C$, 
\[
\omegagr\bigl(\Hom_{\mathcal C}(X,Y)\bigr)=\Hom_{\omegagr(\mathcal C)}
\bigl(\omegagr(X),\omegagr(Y)\bigr) \subset \Homgr\bigl(\omegagr(X),\omegagr(Y)\bigr)
\]
is a subvector bundle.

\noindent (3) The category $\mathcal C_0$ with objects
  $\operatorname{Obj}(\mathcal C_0)=\operatorname{Obj}(\mathcal C)$
  and morphisms 
\[
\Hom_{\mathcal
    C_0}(X,Y)=\Hom_{\mathcal
    C}(X,Y)_0=
  \omegagr^{-1}\bigl(\HomO\bigl(\omegagr(X),\omegagr(Y)\bigr)\bigr)
\]
  is abelian, monoidal,
  rigid.

  \noindent (4) $\End_{\mathcal C_0}(\mathbb I) \cong \Bbbk$.
  
  \noindent (5) The restriction of the forgetful functor
  $\omegaO=\omegagr|_{_{\mathcal C_0}}: \mathcal C_0\to \HVB_0(A)$ is a
  monoidal functor.

  \noindent (6) The functor $\omegaO$ remains fully faithful after taking
  restriction to the fiber over $0\in A$. In other words, the functor
  $\widetilde{\omega}: \mathcal C_0\to \operatorname{Vect}_\Bbbk$,
$\widetilde{\omega}(X)= \bigl(\omegaO(X)\bigr)_0$, $\widetilde{\omega}(f:X\to X')=
f|_{_{(\omegaO(X))_0}}:\bigl(\omegaO(X)\bigr)_0\to\bigl(\omegaO(X')\bigr)_0$ is a fully faithful abelian,  monoidal functor.

Then there exists an  affine  extension $\mathcal
S_{\mathcal C}$ and an equivalence of categories $F: \mathcal C\to
\Rep(\mathcal S_{\mathcal C})$ such that the following diagrams are
commutative
\[
  \xymatrix{
    \mathcal C\ar[rr]^F\ar[rd]_{\omegagr}& & \Rep(\mathcal S_{\mathcal
      C})\ar[dl]^{\omegagr} &  \mathcal C_0\ar[rr]^{F|_{_{\mathcal C_0}}}\ar[rd]_{\omegaO}&& \Rep_0(\mathcal S_{\mathcal
      C})\ar[dl]^{\omegaO} & \\
    & \HVBG(A) && & \HVB_0(A)    &
    }
  \]
  where the restriction $F|_{_{\mathcal C_0}}$ is a monoidal functor. 
\end{thm}

\pf
Since the pair  $\bigr(\mathcal C_0, \widetilde{\omega}:\mathcal C_0 \to
\operatorname{Vect}_\Bbbk\bigr)$ satisfies the hypothesis of the Recognition Theorem for
affine group schemes (see \cite[Proposition 2.8]{kn:delmil}), it
follows  that there exists an affine  group scheme $H$ such that
$\mathcal C_0\cong \Rep_{\operatorname{fin}}(H)$.

Let $\operatorname{\mathit{Aut}}^\otimes(\omegagr)$ be as presented in
Definition \ref{defn:endotimes} (for the category $\mathcal C$ instead of
$\Rep(\mathcal S)$), and for $X$ an object of $\mathcal C$ define
$\mathcal C_X\subset \mathcal C$ as in definitions
\ref{defn:monoidgen} and \ref{defn:maincat}.
Then, as in Remark
\ref{rem:autw3} and Lemma \ref{lem:autw}, it follows that we have a
 limit of  affine  extensions of finite type
\[
{\xymatrixrowsep{1.5pc} \xymatrixcolsep{0.7pc} 
\xymatrix
{
\operatorname{\mathit{Aut}}^\otimes(\omegagr):\ar[d]& 1\ar[rr] &&
\Aut^{\otimes}(\omegagr)_0=H \ar[rr]\ar[d]& &
\Aut^{\otimes}(\omegagr)\ar[rr]\ar[d]&& A\ar@{=}[d]\ar[rr]&& 0\\
\operatorname{\mathit{Aut}}^\otimes\bigl(\omegagr|_{_{\mathcal
    C_X}}\bigr) :& 1\ar[rr] &&
\Aut_0^\otimes\bigl(\omegagr|_{_{\mathcal C_X}}\bigr) \ar[rr]&& 
\Aut^{\otimes}\bigl(\omegagr|_{_{\mathcal C_X}}\bigr)\ar[rr]&& A\ar[rr]&& 0
}}
\]
Indeed,  since the functor
$\omegagr|_{_{\mathcal C_0}}$ is monoidal, 
the same calculations hold --- recall that  $  \Aut^\otimes(\omegagr)_0=H $ by the
Reconstruction Theorem for affine group schemes.

Next, we show that $\mathcal C$ (or equivalently $\omegagr(\mathcal C)$) is
equivalent to the representation theory of
$\operatorname{\mathit{Aut}}^\otimes(\omegagr)$. For this,
let $X\in \mathcal C$; then $\omegagr(X)$ is a
$\operatorname{\mathit{Aut}}^\otimes(\omegagr)$--module.

Conversely,
if $E$ is a 
$\operatorname{\mathit{Aut}}^\otimes(\omegagr)$--module, then $E_0$ is a
$H$--module, and $E\cong \Aut^{\otimes}(\omegagr) \times^HE_0$ by Theorem \ref{thm:GaffG}. Let 
$X\in\mathcal C$ be such that $\omegagr(X)_0\cong E_0$ as
$H$--modules --- recall that $\mathcal C_0$ is the representation
theory of $H$. Since $\omegagr(X)$ is an
$\operatorname{\mathit{Aut}}^\otimes(\omegagr)$--module, it follows
that 
$\omegagr(X)\cong E$,  by Corollary \ref{cor:GaffGimpliesiso}.

Let $X,Y\in\mathcal C$, be two objects. Since $\omegagr(\mathcal
C)_0=\omegagr(\mathcal C_0)\cong 
\Rep(H)=\Rep\bigl(\operatorname{\mathit{Aut}}^\otimes(\omegagr)_0\bigr)$,
it follows that  
\[
\begin{split}
\omegagr\bigl(\Hom_{\mathcal C_0}(X,Y)\bigr) &= \Hom_{\omegagr(\mathcal C)_0}\bigl(\omegagr(X),\omegagr(Y)\bigr)\cong\\
 \Hom_{\Rep(H)}(X_0,Y_0)&=\Hom_{\Rep\bigl(\operatorname{\mathit{Aut}}^\otimes(\omegagr)_0\bigr)}\bigl(\omegagr(X)_0,\omegagr(Y)_0\bigr)
\end{split}
\]

Recall that 
$
\Hom_{\omegagr(\mathcal
  C)}\bigl(\omegagr(X),\omegagr(Y)\bigr)=\omegagr\bigl(\Hom_{\mathcal
  C}(X,Y)\bigr)\in \HBA$ is a vector bundle, with fiber
$\omegagr\bigl(\Hom_{\mathcal
  C_0}(X,Y)\bigr) =\Hom_{\Rep(\operatorname{\mathit{Aut}}^\otimes(\omegagr)_0)}\bigl(\omegagr(X)_0,\omegagr(Y)_0\bigr)$.
On the other hand, by construction we have that 
\[
\Hom_{\omegagr(\mathcal
  C)}\bigl(\omegagr(X),\omegagr(Y)\bigr)\subset 
\Hom_{\Rep(\operatorname{\mathit{Aut}}^\otimes(\omegagr))}\bigl(\omegagr(X),\omegagr(Y)\bigr);
\]
the later being also a vector bundle of fiber
$\Hom_{\Rep(\operatorname{\mathit{Aut}}^\otimes(\omegagr)_0)}\bigl(\omegagr(X)_0,\omegagr(Y)_0\bigr)$
by definition. It follows that these vector bundles
coincide. In
other words, $\omegagr(\mathcal C)$ is the category of representations of
$\operatorname{\mathit{Aut}}^\otimes(\omegagr)$.
\qed

\begin{rem}
Condition (5) in Theorem \ref{thm:recognition} states that any
morphism $f\in \Hom_{\mathcal C_0}(X,Y)$ is determined (after taking
the forgetful functor) by its value in the fiber over $0\in A$.
\emph{A fortiori},  by Corollary \ref{cor:restriction}, this implies
that this condition holds for any morphism in $ \Hom_{\mathcal C}(X,Y)$. 
\end{rem}

We finish this section by  describing $\HVBG(A)$ as the category or
representations of $\mathcal G_A$, the universal extension of the abelian
variety $A$ (see Example \ref{ej:ueav}).

\begin{ej}
  \label{ej:recforueav}
  The identity functor $\operatorname{Id}: \HVBG(A)\to \HVBG(A)$ can
  be though as a forgetful functor. Therefore, $\operatorname{\it
    Aut}^{\otimes} (\operatorname{Id})$ is an   affine  extension,
  such that $\Rep\bigl(\operatorname{\it Aut}^{\otimes}
  (\operatorname{Id})\bigr) $ is equivalent as a category with the
  forgetful functor (in the sense of Theorem \ref{thm:reconstruction})
  with  $\HVBG(A)$ with the identity functor. 

Since $\Rep(\mathcal G_A)$,  the representation theory of
  the universal extension of $A$ (see Example \ref{ej:ueav}), is also equivalent to $\HVBG(A)$, it
follows by the Reconstruction Theorem \ref{thm:reconstruction} that
$\mathcal G_A\cong \operatorname{\it Aut}^{\otimes}
(\operatorname{Id})$.
\end{ej}

\section{Affine  extensions and Hopf sheaves}
\label{sect:Hopf}

The well known op--equivalence between the category of affine group
schemes over a field $\Bbbk$ and the category of Hopf algebras over
$\Bbbk$ has been generalized in \cite[Expos\'e I, Section
  4.2]{kn:SGA3-1} to the context of \emph{affine group schemes over a
  scheme $S$} --- that is, group objects in the category of affine
schemes over $S$ with respect to the monoidal structure given by the
fibered product over $S$. The algebraic counterpart of the group
object is in this case a sheaf of bialgebras in $\sSalg$.   
In this section we go one step further and establish an
op--equivalence  between the category $\GextaffA$ of
\emph{affine extensions of the abelian variety $A$} (see Definitions
\ref{defn:catextensions} and Notation \ref{nota:diffcateg} as well as
Definition \ref{defn:othermonoidal1}) and a category of $\mathcal
O_A$--algebras with additional structure that we call \emph{faithful
  (commutative) Hopf
  sheaves} --- named as $HQ\fsAalg$ in Definition \ref{defn:caths}.

In our situation $\GextaffA$ will appear as a subcategory of the category of bimonoids
(with an antipode) in a duoidal category based upon $\Schasqc$, the
category of separated, quasi-compact schemes over $A$  (see
Definition \ref{defi:duoidalcat}).  One of the two monoidal
structures is the fibered product over $A$ as 
in the classical case, but the other is defined taking into account
the additive structure of $A$: it is the composition of the product
over $\Bbbk$ with the base change by $s:A \times A \to A$, where $s$
is the sum in $A$.  This second structure is essential in order to capture in abstract terms the
fact that the base scheme has the additional structure of an
\emph{abelian variety} and that $q:G \to A$ is a group homomorphism
(see Definitions \ref{defn:othermonoidal1} and
\ref{defn:othermonoidal2}). See also remark \ref{rem:affbutqc}.

For the following undertakings, as was mentioned before, it is
better to view the affine extensions of an abelian variety $A$ as a
surjective (faithfully flat, separated) affine morphism of group 
schemes: $q:G\to A$ (see Remark \ref{rem:forequiv} and Section
\ref{subsect:affextassch}).

To make the exposition clearer, we deal first with the monoid
structure of $G$ --- and the bialgebra structure of the associated
sheaf ---, and after this is firmly secured, we present a formal treatment
of the inversion morphism of $G$ and the corresponding ``antipode'' in
the associated sheaf. As it happens frequently when dealing with
``generalized Hopf type objects'', it is harder to deal with the
antipode than with the bialgebra structure.

To be in safe ground from the viewpoint of the categorical
considerations, we will recall and use some definitions and concepts
pertaining to the theory of \emph{duoidal categories}, that are
categories with two monoidal structures, that are related by an \emph{interchange law}  (see
\cite{kn:aguiarspecies} 
and \cite{kn:garner}). The reader should be aware that  other names are used in the
literature for this concept, see Definition \ref{defi:duoidalcat} below.

\begin{nota}
In what follows we will deal with several monoidal structures, on
different categories. A monoidal structure in a category $\mathcal C$ will be
    denoted as a $3$-uple (\emph{e.g.}~$(\mathcal C, \otimes, \mathbb
    I)$) --- that is,   we omit the associativity and unit constraints in the
    formul\ae. We  write $\mathcal C^{\op}$ for the opposite category
    with the same monoidal structure. 
\end{nota}
  
  \subsection{Affine extensions as schemes over an abelian variety, revisited}\ %
\label{subsect:affextschoverA}

Even though our final goal is to work in the category $\Schaaff$ (of
affine schemes over $A$) we have to formulate our basic definitions in
larger categories such as $\Schaqc$ (of quasi-compact schemes
over $A$) and others. This is due to the fact that some of the basic
ingredients --- for example the construction of the new monoidal
structure --- do not live in the ``affine universe'' (see Remark \ref{rem:affbutqc}).

It is convenient to begin by setting the notation of the different
subcategories of $\Sch|S$ that we will use henceforth.

\begin{nota}
  \label{nota:diffcateg}
Let $S$ be a $\Bbbk$--scheme.

\noindent   (1) We denote  the 
    category of quasi-compact schemes over the $\Bbbk$--scheme $S$ 
    as $\Schqc$:
    its objects are the quasi-compact morphisms of $\Bbbk$-schemes
    $x:X\to S$ and its morphisms $f:(x:X \to S) \to (y:Y \to S)$ are
    morphisms of schemes $f:X \to Y$ such that $y\smallcirc f=x$.

We denote  an object $(x:X\to S)\in \Schqc$ as
   $x$, when no confusion arises. 
    
    \noindent (2) We denote as $\Schsqc$  the full
    subcategory of separated, quasi-compact schemes
    over $S$; the full subcategory of affine schemes over $S$ is
    denoted as $\Schaff$. Since  any affine morphism is separated
    and quasi-compact, we have that $\Schaff$ is fully embedded in
    $\Schsqc$.

    \noindent (3) We also consider the categories $\Schpqc$,
    $\Schfpqc$ defined by the conditions that the map $x:X \to S$ is
    flat (\emph {plate} in french) quasi-compact and faithfully flat
    (\emph{fid\`element plate}) quasi-compact respectively.

    \noindent (4) Also, we denote as $\Schpsqc$ and $\Schfpsqc$ the
    categories  defined by the conditions that the map $x: X \to S$ is
    flat separated and quasi-compact or  faithfully flat separated
    and quasi-compact, respectively.

      \noindent (5) Let $x:X \to S \in \Schsqc$. If there exists a
      closed point $s:\Spec (\Bbbk)\to S\in S(\Bbbk)$, such that $x $
      factors through $s$, way say that $x$ has constant structure
      morphism equal to $s$. See diagram below.
      \[\xymatrix@=7pt{&X\ar[ld]_(.35){\st}\ar[dd]^{x}\\
        \Speck\ar[rd]_{s}&\\ &S}\]

     \noindent (6) If $f: T \to S$ is a morphism of schemes, recall
     that the \emph{pull-back functor} $f^*:\Sch|S \to \Sch|T$ has the
     \emph{push forward functor} $f_*:\Sch|T \to \Sch|S$ as left
     adjoint --- the functor $f_*$ is defined as $f_*(x:X \to T)=f
     \smallcirc x$ for $x \in \Sch|T$ and $f^*(y:Y \to S)=(p_T: Y \times_S
     T \to T)$ is given by the pull-back diagram:
     \[\xymatrix@=15pt {Y \times_S T\ar[r]^-{p_T}
         \ar[d]_{p_Y}&T\ar[d]^(.35){f}\\ Y\ar[r]_{y}&S.}\]
     At the level
     of arrows the definitions are the standard ones.  Also recall
     that if $f$ is an isomorphism, then $f^*=(f^{-1})_*=(f_*)^{-1}$.
\end{nota}

In the case that $S=A$ an abelian variety, we have additional elements
to take into account.
 
\begin{defn}
    \label{defn:op-0-om}
  \noindent (1) Let $\op:A\to A$ be the inversion morphism,
  $\op(a)=-a$ and denote $\op_*(x)=-x=x^-$, where $\op_*: \Schaqc \to
  \Schaqc$ is $\op_*(x:X \to A)=\op\smallcirc{x}: X \to A$ and
  $\op_*(f)=f$.

  \noindent (2) Let $c_0=0\,\smallcirc \st:A\to A$, where
 $0=e_A:\Spec(\Bbbk) \to A$ and $\st:A\to \Spec(\Bbbk)$ is the
 structural morphism of the $\Bbbk$--scheme $A$ --- thus, $c_0$ is
 ``the constant morphism equal to $ 0\in A$''.
\end{defn}

\begin{defn}\label{defn:othermonoidal1}
(1) The  \emph{Cauchy monoidal structure  in $\Schaqc$} is defined as  follows:
 \[ 
   \wtimes:=s_*\times: \xymatrix{\Schaqc \times \Schaqc
     \ar[r]^(0.53){\times}&
     \operatorname{Sch}|_{_{\operatorname{qc}}}\,(A \times
     A)\ar[r]^(0.6){s_*}& \Schaqc},
 \]
 where $s$ denotes as usual the addition in $A$ and the functor
 $\times$ is the product in the category $\Schkqc$, i.e.~if $x:X \to
 A$, $y:Y \to A \in \Schaqc$, then \[(x: X \to A) \times (y:Y \to A)
 := \xymatrix{(X \times Y \ar[r]^{x \times y} & A \times A).}\]

The fact that the construction $\wtimes$ induces a monoidal
  structure on $\Schaqc$, with unit element $0: \Speck \to A$, is a
  straightforward calculation that we omit. We denote its unit element
  as $\uwtimes$.

\noindent (2) Similarly the fibered product $\times_A$, that we call
the \emph{Hadamard monoidal structure},  
induces a monoidal structure on $\Schaqc$, with unit element $\id_A:A \to A$ that we 
denote as $\uAtimes$.\footnote{The names of Hadamard and
    Cauchy monoidal structure, are used in 
  similar situations in other contexts, in particular in the theory of
  species (see \cite[Sect.~6.1,Ex.~6.22,Ex.~8.13.5]{kn:aguiarspecies}).} 
\end{defn}

\begin{rem}\label{rem:symmetric} Both monoidal structures presented in
  Definition \ref{defn:othermonoidal1} are symmetric braided. This
  fact is true in general for the fibered (Hadamard) product over any
  base, and in the case of the Cauchy product is due to the abelianity
  of the group structure in $A$.  
\end{rem}

Later, when working with the ``group type objects'' in the category
$\Schaqc$ we will concentrate our attention mainly to the case of
group extensions, i.e.~morphisms of group schemes $q:G\to A$ with
additional properties. As these morphisms are separated, it is natural
to consider the restriction of the Cauchy and Hadamard monoidal
structures to the category $\Schasqc$ of separated, quasi-compact
schemes over $A$. Similarly, one can consider the subcategory
$\Schapsqc$ (see Notation \ref{nota:diffcateg}). These 
restrictions will be necessary to deal with certain technical aspects
such as the ones considered in Section \ref{sect:quasicompandsheaves}.

\begin{lem}\label{lem:restrict}
The Cauchy and Hadamard monoidal structures restrict to 
$\Schapsqc$ and $\Schasqc$.
\end{lem}
\pf This is clear.\qed

The Cauchy and the Hadamard monoidal structures endow $\Schaqc$ with
a structure of \emph{duoidal category}: 

 \begin{defn}\label{defi:duoidalcat}
A \emph{duoidal category} --- also known as \emph{2--monoidal
  category} or \emph{two-fold monoidal category}, see
\cite[Sect.~4.9]{kn:garner} or \cite[Chapter 6]{kn:aguiarspecies} ---
is a quintuple $(\mathcal C,\diamond, \mathbb
I_{\diamond},\star,\mathbb I_{\star})$ with the following properties
--- the quintuple will be abbreviated as $\mathcal C$ when there is no
danger of confusion ---:

    \noindent (1) $(\mathcal C,\diamond,\mathbb I_{\diamond})$ and
      $(\mathcal C,\star,\mathbb I_{\star})$ are monoidal
      categories with their respective units.

    \noindent (2) There is a natural transformation, called the
      \emph{interchange law},
      \[\zeta_{a,b,c,d}: (a \star b) \diamond (c \star d) \Rightarrow
      (a \diamond c) \star (b \diamond d),\] defined for all $a,b,c,d
      \in \mathcal C$.

    \noindent (3) There are three morphisms: \[\Delta_{\diamond} : \mathbb
        I_{\diamond} \to \mathbb I_{\diamond} \star \mathbb
        I_{\diamond}\,,\, \mu_{\star} : \mathbb I_{\star}
        \diamond \mathbb I_{\star} \to \mathbb I_{\star}\,,\,
        u_{\mathbb I_{\star}}= \varepsilon_{\mathbb I_{\diamond}} :
        \mathbb I_{\diamond} \to \mathbb I_{\star}.\]

      \noindent (4)  All the data above satisfy additional conditions:

\noindent (i)  \emph{compatibility of units}, that amounts to the
following assertions:

$(\mathbb I_\star,\mu_\star,u_{\mathbb 
    I_{\star}})$ is a monoid in $(\mathcal C,\diamond,\mathbb
  I_{\diamond})$;

  $(\mathbb
  I_\diamond,\Delta_\diamond,\varepsilon_{\mathbb I_{\diamond}})$ is a
  comonoid in $(\mathcal C,\star,\mathbb I_{\star})$.

 \noindent (ii) \emph{associativity for $\zeta$}, that is expressed as
 the commutativity of the following diagrams for all objects
 $a,b,c,d,e,f\in \mathcal C$ --- for the sake of simplicity, all the
 diagrams are written omitting the 
  associativity constrains, i.e.~pretending that the monoidal
  structures are strict ---:
      \[
{\xymatrixcolsep{5em} \xymatrix{ (a\star b)\diamond (c\star d)\diamond
    (e\star f)\ar[r]^-{\id_{a\star b}\diamond
      \zeta_{c,d,e,f}}\ar[d]_{\zeta_{a,b,c,d}\diamond \id_{e\star f}}
    & (a\star b)\diamond\big((c\diamond e)\star (d\diamond
    f)\bigr)\ar[d]^{\zeta_{a,b,c\diamond e,d\diamond f }}\\ \bigl((a\diamond c)\star(b\diamond
    d)\bigr)\diamond (e\star f)\ar[r]_-{\zeta_{a\diamond c, b\diamond d,e,f}} & (a\diamond c\diamond
    e)\star (b\diamond d\diamond f) }}
      \]
 \[
{\xymatrixcolsep{5em} \xymatrix{ \left((a\star b)\star
    c\right)\diamond \left((d\star e) \star f\right)\ar[r]^-{
      \zeta_{a,b\star c,d,e\star f}}\ar[d]_{\zeta_{a\star b,c,d\star e,f}} & (a\diamond d)\star \bigl((b\star
    c)\diamond (e\star f)\bigr)\ar[d]^{\id_{a\diamond d}\star
      \zeta_{b,c,e,f}}\\ \bigl((a\star b)\diamond (d\star e)\bigr)\star
    (c\diamond f)\ar[r]_-{\zeta_{a,b,d,e}\star \id_{c\diamond f}} & (a\diamond
    d)\star (b\diamond e)\star (c \diamond f) }}
      \]

\noindent (iii) \emph{unitality/counitality for $\zeta$}, that is
expressed as the commutativity of the following diagrams for all
$a,b\in \mathcal C$ ---  again we omit the associativity constrains ---:
\[
  \xymatrixcolsep{5em}\xymatrix{
        (a \star b) = (a \star b)  \diamond \mathbb I_{\diamond} = 
        \mathbb {I_\diamond} \diamond (a\star b)
        \ar[r]^-{\Delta_{\diamond} \diamond \id_{a\star b}}
        \ar[d]_{\id_{a\star b} \diamond \Delta_{\diamond}} \ar@{<->}[rd]^{\id}
        & (\mathbb I_{\diamond} \star \mathbb I_{\diamond}) \diamond
        (a\star b) \ar[d]^{\zeta_{\mathbb I_{\diamond},\mathbb
            I_{\diamond},a,b}}\\
        (a\star b) \diamond (\mathbb I_{\diamond}
        \star \mathbb I_{\diamond}) \ar[r]_-{\zeta_{a,b,\mathbb
            I_{\diamond},\mathbb I_{\diamond}}} & a\star b }
      \]
      \[\xymatrixcolsep{3em}\xymatrix{
        (a \diamond b) = (a \diamond b) \star \mathbb I_{\star} =
        \mathbb {I_\star} \star (a\diamond b)
        \ar@{<-}[r]^-{\mu_{\star} \star \id_{a\diamond b}}
        \ar@{<-}[d]_{\id_{a\diamond b} \star \mu_{\star}}
        \ar@{<->}[rd]^{\id} & (\mathbb I_{\star} \diamond \mathbb
        I_{\star}) \star (a\diamond b) \ar@{<-}[d]^{\zeta_{\mathbb
            I_{\star},a,\mathbb I_{\star},b}}\\
        (a\diamond b)
        \star (\mathbb I_{\star} \diamond \mathbb I_{\star})
        \ar@{<-}[r]_-{\zeta_{a,\mathbb I_{\star},b,\mathbb I_{\star}}}
        & (a \star \mathbb I_\star)\diamond (b \star \mathbb I_\star)
        =  a\diamond b = (\mathbb I_\star \star a) \diamond (\mathbb
        I_\star \star b )} 
        \]

It is clear that if $(\mathcal C,\diamond, \mathbb
I_{\diamond},\star,\mathbb I_{\star})$ is a duoidal category, then
$(\mathcal C^{\op},\star,\mathbb I_{\star}, \diamond, \mathbb
I_{\diamond})$ is also a duoidal category that is written simply as ${\mathcal C}^{\op}$ and is called the opposite of the duoidal category $\mathcal C$ --- the interchange law of $\mathcal C$ and $\mathcal C^{\op}$ is the same morphism.  
        
We refer the reader to \cite[Paragraph 6.1.1]{kn:aguiarspecies} or
\cite[Section 4]{kn:garner} for more information on the properties of
duoidal categories.
\end{defn}

In the context of  duoidal categories, one can establish the notion of
bimonoid as follows.

 \begin{defn}
   \label{defi:bimonoid}
   Let $(\mathcal C,\diamond, \mathbb
   I_{\diamond},\star,\mathbb I_{\star})$ be a duoidal category. A
   quintuple $(b, \mu_b, u_b, \Delta_b, \varepsilon_b)$ consisting of an
   object $b \in \mathcal C$, and morphisms $\mu_b: b \diamond b \to b$,
   $u_b:\mathbb I_{\diamond} \to b$, $\Delta_b: b \to b \star b$ and
   $\varepsilon_b: b \to \mathbb I_{\star}$ is a \emph{bimonoid for the
     duoidal category} if:

  (1)  $(b, \mu_b, u_b)$ is a monoid in
     $(\mathcal C,\diamond,\mathbb I_{\diamond})$;
   
(2)  $(b, \Delta_b, \varepsilon_b)$ is a comonoid in
     $(\mathcal C,\star, \mathbb I_{\star})$;

     (3) The following conditions hold:

     \quad $\varepsilon_b: b \to \mathbb
     I_{\star}$ is a morphism of monoids in
     $(\mathcal C,\diamond,\mathbb I_{\diamond})$;

     \quad $u_b:\mathbb
     I_{\diamond} \to b$ is a morphism of comonoids in
     $(\mathcal C,\star,\mathbb I_{\star})$;

   (4)  The following diagram is commutative:
     \[\xymatrix{(b \star b)\diamond (b \star b)\ar[rr]^{\zeta_{b,b,b,b}}&&(b \diamond b)\star (b \diamond b)\ar[d]^{\mu_b \star \mu_b}\\
     b \diamond b\ar[u]^{\Delta_b \diamond
       \Delta_b}\ar[r]_{\mu_b}&b\ar[r]_{\Delta_b}&b \star b}\]

We call
  $\operatorname{Bimon}(\mathcal C)$ the category whose objects are
  the bimonoids in $\mathcal C$ and its arrows the morphisms of
  $\mathcal C$ that preserve the bimonoid structure.
\end{defn}

 \begin{rem} (1) In accordance with Definition \ref{defi:duoidalcat},4.(i) $(\mathbb I_\star,\mu_\star,u_{\mathbb 
    I_{\star}})$ is a monoid in $(\mathcal C,\diamond,\mathbb
   I_{\diamond})$ and in a trivial manner $\mathbb I_\star$ is also a comonoid in
   $(\mathcal C,\star,\mathbb I_{\star})$. It is clear that the
   compatibility conditions are satisfied and hence,
   $\mathbb I_\star$  is a bimonoid in the
   duoidal category $\mathcal C$.

   \noindent (2) Similarly $(\mathbb
   I_\diamond,\Delta_\diamond,\varepsilon_{\mathbb I_{\diamond}})$ is
   a comonoid in $(\mathcal C,\star,\mathbb I_{\star})$ and also has a
   natural structure of monoid in $(\mathcal C,\diamond,\mathbb
   I_{\diamond})$ and also of bimonoid in the duoidal category
   $\mathcal C$.

  \noindent (3) The map $\mathbb I_{\star} \to \mathbb I_{\diamond}$
      in Definition \ref{defi:duoidalcat}, is a morphism of bimonoids.
  \end{rem}

We proceed now to show that if $A$ is an abelian variety, then the
Hadamard and Cauchy monoidal structures on $\Schaqc$ combine into a
structure of duoidal category. In this duoidal category, bimonoids
correspond to morphisms of monoid schemes $M\to A$, where $M$ is a
quasi-compact monoid scheme.

 \begin{lem}\label{lem:duoschemes}
  The quintuple $(\Schaqc, \wtimes, \uwtimes, \times_A, \uAtimes)$
  consisting of the category of quasi-compact schemes over $A$ with
  the Cauchy and Hadamard monoidal structures, and their respective
  unit objects, together with the morphisms $\Delta_{\wtimes}: \mathbb
  I_{\wtimes} \to \mathbb I_{\wtimes}\times_A\mathbb I_{\wtimes}$ (the
  diagonal morphism); $\mu_{\times_A}:=s:\mathbb I_{\times_A}
  \wtimes\mathbb I_{\times_A} \to \mathbb I_{\times_A}$,
  $u_{\uAtimes}=\varepsilon_{\uwtimes}:=0 : \uwtimes \to \uAtimes$
  constitute a duoidal category.  Moreover, with the restricted
  structures (see Lemma \ref{lem:restrict}) the categories  $ (\Schasqc, \wtimes, \uwtimes,
\times_A, \uAtimes)$  and 
 $(\Schapsqc,
\wtimes, \uwtimes, \times_A, \uAtimes)$ are duoidal.
\end{lem}

\proof
   We only give an sketch of the proof.
 First of all, notice that the
morphisms $\Delta_{\wtimes}$, $\mu_{\times_A}$ and $u_{\uAtimes}=\varepsilon_{\uwtimes}$ are morphisms in $\Schaqc$, since the
following diagrams are commutative:
     \[\xymatrix{\Speck \ar[r]^-{\Delta_{\wtimes}}\ar[d]_{0}
       & \Speck \times_A \Speck \ar[d]^{0}\\A\ar[r]_{\id}&A}
        \xymatrix{A \times A\ar[r]^-{\mu_{\times_A}}\ar[d]_{\id \times
         \id}&A\ar[dd]^{\id}\\A\times
       A\ar[d]_s&\\A\ar[r]_{\id}&A} \xymatrix{\Speck
       \ar[rr]^-{u_{\uAtimes}=\varepsilon_{\uwtimes}}\ar[rd]_{0}&& A \ar[dl]^{\id}\\&A&}
   \]

   It is clear that $(\uwtimes, \Delta_{\wtimes},
 \varepsilon_{\wtimes}=0)$ is a comonoid for the Hadamard monoidal
 structure, and $(\uAtimes, \mu_{\times_A}, u_{\times_A}=0)$ is a
 monoid for the Cauchy monoidal structure.

     The interchange law is defined as follows: for $x:
   X\to A$, $y: Y\to A$, $z:Z\to A$ and $w:W\to A$, 
       $\zeta_{x,y,z,w}: (x \times_A y) \wtimes (z \times_A w) \to (x
       \wtimes z) \times_A (y \wtimes w)$ is the unique morphism given
       by the universal property of the fibered product:
       \[
\xymatrix{&&(X \times_A Y)\times (Z \times_A W)\ar[lld]_{p_X \times
    p_Z}\ar@{.>}[d]^{\zeta_{x,y,z,w}}\ar[rrd]^{p_Y \times p_W}&& \\
  X \times Z\ar[d]_{x \times z}&&
      (X \times Z)\times_A (Y \times
        W)\ar[ll]^(0.6){p_{X \times Z }}\ar[rr]_(0.6){p_{Y \times W}}
        && Y \times 
        W\ar[d]^{y \times w}\\
        A \times A\ar[rr]_s&&A&&A \times
        A\ar[ll]^s
      }\]

 Once the interchange law is established, the associativity and
 unitality of $\zeta$ follow easily.

 The fact that all the duoidal structure can be restricted to the
 subcategories is clear.  For example:
 \[
   \left(\Schaqc,
 \wtimes,\uwtimes:=\bigl(0: \Speck \to A\bigr), \times_A, \uAtimes:=(\id_A:A \to
 A)\right),\] restricts to a duoidal structure $(\Schasqc, \wtimes,
 \uwtimes, \times_A, \uAtimes)$ --- if $x:X\to A$, $y:Y\to A\in
 \Schaqc$, then the morphisms $s\smallcirc (x,y):X \times Y\to A$ and
 $x\smallcirc p_1=y\smallcirc p_2:X\times_AY\to A$
 are separated, as well as
 $\uwtimes=0:\Spec(\Bbbk)\to A $ and $\uAtimes=\id_A:A\to A$.
 \qed

 The following remark is of some relevance for future use.

 \begin{rem}\label{rem:allcomonoid}
   (1) All the objects in the monoidal category $(\Schaqc, \times_A,
   \id_A)$ can naturally be endowed with a unique comonoid structure,
   given by the morphisms depicted in the diagrams below, where
   $\delta_X:X \to X\times_A X$ is the canonical diagonal morphism and
   the counit is $\varepsilon_X:=q_X$.
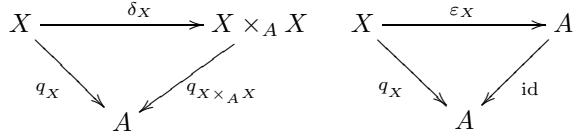
     \[\xymatrix{X\ar[rd]_{q_{_X}}\ar[rr]^{\delta_X}&&X \times_A X
       \ar[ld]^{q_{_{X \times_A X}}}\\&A&}\quad \xymatrix{X\ar[rd]_{q_{_X}}\ar[rr]^{\varepsilon_X}&& A
       \ar[ld]^{\id}\\&A&}
   \]

   Notice that $q_{X\times_A X }=q_X\smallcirc p_1=q_X\smallcirc
   p_2:X\times_AX\to A$.

 \noindent (2)  Similarly --- in a somewhat redundant way --- in the monoidal
        category of sheaves of $\mathcal O_A$--algebras $(\sAalg,
        \otimes_{\mathcal O_A}, \mathcal O_A)$ all objects
        are monoids in a unique way (see Definition  \ref{defn:othermonoidal2}
        below).

        \noindent(3) A peculiarity of this duoidal category is related
        to the element $\mathbb I_\diamond=\uwtimes$, which is the unit
        of the monoidal half $(\Schaqc, \wtimes, \uwtimes)$:
        $\uwtimes$  is
        idempotent with respect to the other half of the monoidal
        structure, i.e.~$\uwtimes \times_A \uwtimes \cong
        \uwtimes$.

        Also, $\uAtimes\!\!\!\wtimes\uAtimes = s:A \times A \to
        A$.
   \end{rem}

 Proposition  \ref{prop:catextcar} below encompasses the main
 properties of quasi-compact morphisms of monoid schemes $M\to A$ in a
 categorical framework,  and will 
 translate --- by the op-equivalence of categories mentioned before,
 once we take into consideration the inverse  when $M$ is a  group
 scheme
 ---  to the notion of Hopf sheaf.

 \begin{prop}\label{prop:catextcar}
   In the duoidal category $(\Schaqc, \wtimes,
  \uwtimes, \times_A, \uAtimes)$, an object $q_M:M \to A \in
  \Schaqc$ is a bimonoid if and only if $M$ is a monoid in $\Sch$ and
  the morphism $q_M:M \to A$ is a morphism of monoids; that is,
  $q_M$ is multiplicative and $q(1_M)=0_A$.

  Given two bimonoids $q_M:B\to A$ and $q_{M'}\to A$, a morphism
  $f:q_M\to q_{M'}$ is of bimonoids if and only if $f:M\to M'$ is a
  morphism of monoid schemes. 
  \end{prop}

  \begin{proof}
    Indeed,  a  structure of monoid in $q_M:M \to A$ is given by  two morphisms
    $\mu_M:M \times M \to M$ and $u_M: \Spec(\Bbbk) \to M$,
    such that $\mu_M$ and $u_M$
    satisfy the usual axioms of associativity and unitality, together with  
    the diagrams depicted below.
  
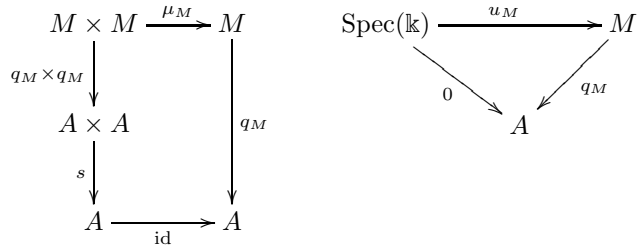
\[\xymatrix{M \times M \ar[d]_{q_M \times q_M}\ar[r]^-{\mu_M}&
      M\ar[dd]^{q_M}\\
      A \times A\ar[d]_s&&\\
      A\ar[r]_{\id}&A}
    \xymatrix{\Speck\ar[rr]^{u_M}\ar[dr]_-0&&M\ar[dl]^-{q_M}\\
      &A&}
  \]

  The last assertion follows easily.
\end{proof}

\begin{nota}
  \label{nota:MMor}
We denote $\MmoraffA\subset \operatorname{Bimon}(\Schaqc, \wtimes,
  \uwtimes, \times_A, \uAtimes)$ the full subcategory with objects the
  affine   morphisms of monoid schemes $q_M:M\to A$. 
  \end{nota}

The main properties of the functor $\op_*$ with respect to the duoidal
structure are expressed in the following easy proposition.

\begin{prop}\label{prop:opduo} Consider the duoidal category
  $(\Schaqc, \wtimes, \uwtimes, \times_A, \uAtimes)$.  The functor $\op_*:\Schaqc \to \Schaqc$ satisfies the following properties:

    \noindent (1) It is a strict monoidal involution with respect to
    $\wtimes$ and $\times_A$.

      \noindent (2)   $\op_*(\mathbb
      I_{\wtimes})=\mathbb I_{\wtimes}$ and  $\op_*(\mathbb
      I_{\times_A})\times_A\op_*(\mathbb I_{\times_A})\cong \op_*(\mathbb
      I_{\times_A})$.  \qed
    \end{prop}
    
Next, with the purpose of obtaining an adequate categorical
formulation for the inverse morphism for a group object in $\Schaqc$,
we need to establish some basic properties involving the functors
\[- \wtimes \uAtimes : \Schaqc \to \Schaqc\quad,
\quad - \times_A \uwtimes: \Schaqc \to \Schaqc, \] and others
associated to the duoidal structure.

\begin{lem}
\label{lema:smallproperties}
  Consider the duoidal category $(\Schaqc, \wtimes, \uwtimes,
  \times_A, \uAtimes)$. Then, with the above notations we have that:
  \begin{enumerate}
  \item    \begin{enumerate}
    \item $0_*0^* x = x \times_A \uwtimes$ for any $x:X\to A$;

    \item The map $\varepsilon_x=\id_x \times_A
      \,\varepsilon_{\uwtimes}: 0_*0^* x \to x$ is the counit of the adjunction 
      $0_* \dashv 0^*$;

    \item The unit $u_z: z \to 0^*0_*z$ of the adjunction  $0_* \dashv 0^*$ is
      an
      isomorphism.
    \end{enumerate}

  \item
    \begin{enumerate} \item
    $\st^*z=0_*z \wtimes \uAtimes$ for $z:Z\to \Spec(\Bbbk) \in \Schk$
      and $\st: A\to \Speck$;

    \item $0^*(c_0)_*=0^*0_*\st_*=\st_*$.
      \end{enumerate}

    \item $c_0^*x=(x \times_A \uwtimes)\wtimes
      \uAtimes$. Equivalently, if $x:X \to A$, $X_0=x^{-1}(0)$ and
      $x_0=x|_{_{X_0}}:X_0 \to A$, then $c_0^*x=p_A:X_0 \times A \to A$.
      
\item There is a natural transformation $\rho_x: (c_0)^*x \to x
  \wtimes \uAtimes$.
  \end{enumerate}
   \end{lem}
   \pf
    (1) In the commutative diagram
      \[\xymatrix{X \times_A \Speck \ar@{..>}[rd]\ar[r]^-{p_2}\ar[d]_{p_1}&
          \Speck \ar[d]^{0}\\ X\ar[r]_x&A,}\] the upper horizontal
      arrow is $0^*x$ and its composition with the vertical 0 arrow
      yields $0_*0^*x$ that is the arrow: $0p_2=xp_1: X \times_A
      \Speck \to A$. It is clear that $0_*0^*x=x \times_A
      \uwtimes$. The remaining parts of the proof are direct.

\noindent (2) The proof of the first part is direct and for the
second, the chain of equalities is guaranteed by 1(c).

\noindent (3) If $x:X\to A \in \Schaqc$, then
      $c_0^*x=\st^*0^*x=0_*0^* x \wtimes \uAtimes= (x \times_A
      \uwtimes)\wtimes \uAtimes$ the second equality follows from
      (2)(a) and the third from (1)(a). The proof of the second
      assertion is easy.

      \noindent (4) The natural transformation $\rho$ is obtained by
      considering the equality proved  in (3), and then applying  the unit morphism
      $u:\uwtimes \to \uAtimes$ to the second factor.\qed 
      
Next we define two natural transformations that are crucial for the
definition of the antipode in Theorem \ref{thm:antipoduoidal} below.
      
\begin{prop}\label{prop:smallproperties}
 Consider the duoidal category $(\Schaqc, \wtimes, \uwtimes,
 \times_A, \uAtimes)$ and let $x:X\to A$, $y:Y\to A \in
 \Schaqc$.

 \noindent (1) There is a natural transformation $\pi_{x,y}: (c_0)_*\bigl(x \times_A -y \bigr) \to x \wtimes y$.

 \noindent (2) There is a natural transformation $\widetilde{\gamma}_{x,y}: x \times_A -y \to (x \wtimes y) \wtimes \uAtimes$. 
 
\noindent (3) There is a natural transformation $\overline{\gamma}_{x,y}:(-x)\times_A y \to \uAtimes \wtimes\,\, (x \wtimes y).$
 \end{prop}

 \begin{proof} For the proof of (1) we consider the commutative
   diagram that follows --- in the next diagrams in order not to loose track of the structure
morphism $-y$ we are using the slightly unorthodox notation $\op_*(y:Y
\to A)=(-y: Y^-\to A$) even though as schemes $Y ^-=Y$.
  \[\xymatrix{
X\times_AY^-\ar[r]^-{p_2}\ar[d]_{p_1}& Y^-\ar[d]^{-y}\ar@/^1pc/[rdd] & \\
X\ar[r]_-{x}\ar@/_1pc/[drr]& A\ar[dr]^{\st} & \\
&& \Spec(\Bbbk)
}
\]
The canonical projections $p_i$ induce a morphism $\tau_{x,y} : X\times_AY^-\to
X\times Y$ such that $s\smallcirc (x \times y)\smallcirc \tau_{x,y}= 0\smallcirc
\st_{X\times_AY}$. In other words, the following diagram is
commutative:
\begin{equation}\label{eqn:important}
  \xymatrix{X \times_A Y^-\ar[d]_{\st x p_1 = \st
      (-y)p_2}\ar[rr]^{\tau_{x,y}}&&X\times Y\ar[d]^{x \times
      y}\\ \Spec(\Bbbk)\ar[rd]_{0}&&A\times A\ar[dl]^{s}\\ &A&}
\end{equation}

Diagram \eqref{eqn:important} means that the map $\tau_{x,y}$ is
a morphism $(c_0)_*\bigl(x \times_A -y\bigr)
\stackrel{\tau_{x,y}}{\longrightarrow} x \wtimes y$.

\noindent (2) It is clear that applying $(c_0)^*$ to the natural
transformation $\tau_{x,y}$  we obtain a natural transformation
$(c_0)^*(\tau_{x,y}): (c_0)^*(c_0)_*\bigl(x \times_A -y\bigr) \to
(c_0)^*(x \wtimes y)$. By composing with the unit of the adjunction
$(c_0)_* \dashv (c_0)^*$ we obtain a natural transformation: $x
\times_A -y \Rightarrow (c_0)^*(x \wtimes y)$. The conclusion of (2)
follows by post composition of this natural transformation with the natural transformation
$\rho_{x\wtimes y}$ defined in  Lemma
\ref{lema:smallproperties}. 

\noindent (3) To obtain $\overline{\gamma}$ one proceeds similarly.
 \end{proof}

\begin{rem}\label{rem:forantipduoidal} For a pair
  $x:X \to A, y:Y \to A$ the natural transformation
  $\widetilde{\gamma}_{x,y}$, is a morphism of schemes that has domain
  $x \times_A (-y):X \times_A Y^- \to A$ and codomain $x+y+\id_A:X
  \times Y \times A \to A$. Tracking down the above construction it is
  easy to see that $\widetilde{\gamma}_{x,y}=\bigl\langle \pi,x
  \times_A (-y)\bigr\rangle: X \times_A Y^- \to X \times Y \times
  A$. It is clear that the diagram below is commutative:
  \[\xymatrix{X \times_A Y^-\ar[rr]^{\langle \tau_{x,y},x \times_A (-y)\rangle}
    \ar[dr]_{x \times_A (-y)}&& X \times Y \times
    A\ar[dl]^{x+y+\id_A}\\&A&}\]

  Notice that if $(u,v)\in x\times_A (-y)$ and $a\in A$, then  
  $\bigl\langle \tau_{x,y}, x \times_A (-y) \bigr\rangle(u,v)=\bigl(u,v,x(u)\bigr)$, $\bigl(x \times_A
  (-y)\bigr)(u,v)=x(u)=-y(v)$ and  $(x+y+\id_A)(u,v,a)=x(u)+y(v)+a=a$.
\end{rem}

We are ready to wrap up the duoidal perspective for a group extension
by completing the result of Proposition \ref{prop:catextcar}.  

\begin{thm}
  \label{thm:antipoduoidal}
   In the duoidal category $(\Schaqc, \wtimes, \uwtimes, \times_A,
   \uAtimes)$ a bimonoid $(b: M \to
   A,\mu_b,u_b,\Delta_b,\varepsilon_b)$ is such that $M$ is a
   quasi-compact group scheme and $b$ a quasi-compact morphism of
   group
   schemes if and only if there is a morphism $\iota_b:b \to
   -b$ in $\Schaqc$, called an {\em antipode}, such that both diagrams below
   commute:
  \begin{equation}\label{eqn:firstantipode}
    \raisebox{10ex}{\xymatrixcolsep{1.5em}\xymatrix{& b\times_A
      b\ar[rr]^-{\id \times_A \iota_b}& &b \times_A (-b)
      \ar[rr]^-{\widetilde{\gamma}_{b,b}}&&(b \wtimes
        b)\wtimes \uAtimes\ar[dr]^{\mu_b\wtimes \id}&\\
        b\ar[dr]_{\varepsilon_b}\ar[ru]^-{\Delta_b}&&&&&& b\wtimes \uAtimes
        \\
        &\mathbb I_{\times_{\!A}}\ar[rrrr]_{\cong}&&&&
        \mathbb I_{\wtimes} \wtimes 
      \mathbb I_{\times_{\!A}}\ar[ur]_{(u_b
        \wtimes \id)}  &
      }}
    \end{equation}

    \begin{equation}\label{eqn:secondantipode}
      \raisebox{10ex}{\xymatrixcolsep{1.5em}\xymatrix{
          & b\times_A b\ar[rr]^-{\iota_b \times_A
          \id}&& (-b) \times_A b
        \ar[rr]^-{\overline{\gamma}_{b,b}}&&\uAtimes \wtimes(b
        \wtimes b) \ar[dr]^{\id \wtimes
          \mu_b}&\\
 b\ar[dr]_{\varepsilon_b}\ar[ru]^-{\Delta_b}&&&& &&
 \uAtimes \wtimes b \\
 &\mathbb       I_{\times_{\!A}}\ar[rrrr]_-{\cong}&&&& \mathbb I_{\times_{\!A}}
        \wtimes \mathbb I_{\wtimes} \ar[ur]_{(\id \wtimes u_b)} & }}
    \end{equation}
   where $\widetilde{\gamma}$ and $\overline{\gamma}$ are the natural
   transformations depicted in Proposition \ref{prop:smallproperties}
   (see also Remark \ref{rem:forantipduoidal}) and the bottom maps
   $\cong$ are the natural identifications associated to the unit of
   the $\wtimes$ monoidal structure.

\end{thm}
\begin{proof}
If $(M,\mu,u,\operatorname{inv})$ is a group scheme over $\Bbbk$ and
$b:M\to A \in \Schaqc$ is a quasi-compact group extension, in
accordance with Proposition \ref{prop:catextcar}, $b$ is a bimonoid in
the duoidal category $\Schaqc$. Additionally, if we define
$\iota_b:=\operatorname{inv}$, it is clear that $\iota_b$ is a
morphism in the category $\Schaqc$. Indeed, the diagram below is
commutative
\[
  \xymatrix{ M\ar[rr]^{\operatorname{inv}}\ar[rd]_b&&
    M\ar[ld]^{-b}\\ & A& }
  \]
A direct verification shows that diagrams \eqref{eqn:firstantipode}
and \eqref{eqn:secondantipode} are commutative provided that
$\mu_b=\mu, u_b=u, \iota_b=\operatorname{inv}$, $\Delta_b$ is the
diagonal morphism and $\varepsilon_b=b:M \to A$.
  
 Conversely, suppose we have a bimonoid $(b:M \to A,\mu_b,u_b,\Delta_b,\varepsilon_b)$ in the duoidal
 category, such that the bimonoid is equipped with a map $\iota_b:b\to -b$ satisfying
 diagrams such as \eqref{eqn:firstantipode}, \eqref{eqn:secondantipode}. A
 direct computation shows that the morphism associated to the upper
 path of the diagram \eqref{eqn:firstantipode} corresponds to the
 upper curved arrow of the diagram below, and similarly for the lower
 path and the lower curved arrow:
  \[\xymatrix{M\ar@/^.5pc/[rr]^{\langle d,b \rangle}\ar@/_.5pc/[rr]_{\langle c_1,b \rangle}\ar[dr]_{b}&&
           M \times A\ar[dl]^{p_A}\\&A&}\] with $d=\mu_b(\id \times
  \iota_b)$ and $c_1$ the $1$--morphism of $b$ (or in other words the
  constant morphism to the unit of $M$) --- here we use the structure
  and conclusions considered in Proposition \ref{prop:catextcar}. The
  commutativity of the diagram implies the equality of both curved
  arrows. Then, $\iota_b$ is a right inverse of the
  identity. Similarly, interpreting the second diagram
  \eqref{eqn:secondantipode}, we conclude the proof of the theorem.
  \end{proof}

\begin{defn}\label{defn:Gmor}
We denote $\GmorqcA\subset  \operatorname{Bimon}(\Schaqc, \wtimes,
  \uwtimes, \times_A, \uAtimes)$ the full subcategory with objects the
  bimonoids $b:M\to A$ such that $M$ is a (quasi-compact) group scheme
  and $b$ a morphisms of group schemes.

  We denote $\GmoraffA\subset \GmorqcA$ the full subcategory with
  objects the affine morphisms of group schemes --- notice that
  $\GmoraffA\subset \MmoraffA$ in a canonical way (see Notation
  \ref{nota:MMor}). 
  
Notice that  $\GextqcA$, the category of quasi-compact group
extensions of $A$, is a full
subcategory of $\GmorqcA$ and that the category of affine extensions
$\GextaffA$ is a full subcategory of $\GmoraffA$ (see Definition \ref{defn:catextensions}).
\end{defn}

\begin{rem}\label{rem:antipodedefi}
The procedure of Theorem \ref{thm:antipoduoidal} is related to the work of B\"ohm  and Lack
\cite{kn:bohmlack}, where the authors construct a
  certain duoidal category associated to a so-called Frobenius
  map--monoidale and introduce an antipode in that context.    

  Notice also that  our  construction of an antipode can be
  generalized as follows: let  $\mathcal C=(\mathcal C,\diamond, \mathbb
  I_{\diamond},\star,\mathbb I_{\star})$ be  a duoidal category and
   assume that it can be equipped with a functor $\op: \mathcal C \to
  \mathcal C$ and natural transformations similar to
  $\widetilde{\gamma}, \overline{\gamma}$ satisfying conditions such
  as the ones appearing in \ref{prop:opduo} and
  \ref{lema:smallproperties}.  For a bimonoid $b \in \mathcal C$ a
  morphism $\iota_b:b \to \op(b)$ that satisfies diagrams and
  properties such as the ones appearing in
  \eqref{eqn:firstantipode} and \eqref{eqn:secondantipode}, is called an
  \emph{antipode} in $b$.  We intend to explore this
  construction in further work.
\end{rem}

\subsection{Bimonoids in duoidal categories}\ %
\label{sect:modcomodduoidal}

We briefly recall the behavior of bilax functor between duoidal
categories, as well as  other monoidal structures related to
duoidality, such as the constructions of modules and
comodules over bialgebras in this context (see
\cite[Chap.~6, \S 6]{kn:aguiarspecies}).

\begin{defn}
(1) Given a functor $F:(\mathcal C,\otimes_{\mathcal
        C},\mathbb I_{\mathcal C}) \to (\mathcal D,\otimes_{\mathcal
        D},\mathbb I_{\mathcal D})$ between two monoidal categories, a
      \emph{lax monoidal} structure 
        consists of a map $\ell_0: \mathbb I_{\mathcal D}\to F(\mathbb
        I_{\mathcal C})$ and a natural transformation
        $\ell_{c,c'}:F(c) \otimes_{\mathcal D} F(c') \to F(c
        \otimes_{\mathcal C} c')$ subject to associative and unitality
        axioms. If such an structure exists, we say that $F$ is a
        \emph{lax functor}. A monoidal functor $F: \mathcal C \to \mathcal D$ is called 
        \emph{op-lax} (or \emph{colax}) if the induced functor
        $\mathcal C^{\op}\to \mathcal D^{\op}$ is lax (monoidal).

        If the maps
        $\ell_0,\ell_{c,c'}$ are isomorphisms we call $F$ a \emph{strong
        monoidal functor}.
        
      \noindent (2)  Given the above situation and two lax monoidal
      functors $(F,\ell_0,\ell), (G,\ell'_0,\ell')$ a \emph{monoidal natural
      transformation} is a natural transformation $\sigma:F \Rightarrow
      G: \mathcal C \to \mathcal D$, that satisfies the commutativity
      of the diagrams below:\
      
      \[\xymatrix{
          F(c)\otimes_\mathcal D F(c')\ar[r]^{\sigma_x
            \otimes_{\mathcal D}\sigma_y}\ar[d]_{\ell_{c,c'}} &
          G(c)\otimes_\mathcal D G(c')\ar[d]_{\ell'_{c,c'}}\\
          F(c \otimes_{\mathcal C} c') \ar[r]_{\sigma_{c
              \otimes_{\mathcal C}c'}} & G(c \otimes_{\mathcal C} c')}
        \quad
        \xymatrix{
          & \mathbb I_{\mathcal
            D}\ar[ld]^{\ell_0}\ar[rd]^{\ell'_0} & \\
          F(\mathbb
          I_{\mathcal C})\ar[rr]_{\sigma_{\mathbb I_{\mathcal
                C}}} & & G(\mathbb I_{\mathcal C})
        }
        \]

  \noindent (3) Suppose that $\mathcal C, \mathcal D$ are two monoidal
 categories and let $L$ and $R$ be lax monoidal functors. An
 adjunction   $\xymatrix{\mathcal C \ar@<5pt>[r]^L
   \ar@<-5pt>@{<-}[r]_R \ar@{}[r]|-\bot&\mathcal D}$ is
 \emph{monoidal} if the unit  $\eta$ and the counit $\varepsilon$  are
 monoidal natural transformations. In 
 that case we say that the \emph{abstract} adjunction
 ``lifts'' to a monoidal adjunction.

 Suppose now that $L: \mathcal C \to
 \mathcal D$ is colax with associated maps $c\ell_0:L(\mathbb
 I_{\mathcal C})\to \mathbb I_{\mathcal D}$ and
 $c\ell_{c,c'}:L (c \otimes c')\to L(c) \otimes L(c')$. Then,
 the maps defined by the diagrams below define a lax monoidal
 structure for $R$ and vice versa (see \cite[Prop (3.84)]{kn:aguiarspecies}
 for the proof).
 \[
   \xymatrix@C=.5pc{\mathbb I_{\mathcal C}
            \ar[rd]_{\eta_{\mathbb I_{\mathcal
                  C}}}\ar[rr]^{\ell'_0}&&R(\mathbb I_{\mathcal
              D})\\ &RL(\mathbb I_{\mathcal C})\ar[ru]_{R(c\ell_0)}&}
          \xymatrix@C=5pc{R(d)\otimes R(d')\ar[r]^{\ell'_{d,d'}}
          \ar[d]_{\eta_{R(d)\otimes R(d')}}&R(d \otimes
            d')\\ RL(R(d)\otimes
            R(d'))\ar[r]_{R(c\ell_{R(d),R(d')})}&R(LR(d) \otimes
            LR(d'))\ar[u]_{R(\varepsilon_d \otimes \varepsilon_d')}}
          \]
\end{defn}

\begin{defn}\label{defn:bilaxandtransf}
  Let $\mathcal C=(\mathcal C,\diamond, \mathbb
  I_{\diamond},\star,\mathbb I_{\star})$ and $\mathcal D=(\mathcal
  D,\diamond', \mathbb I_{\diamond'},\star',\mathbb I_{\star'})$ be
  duoidal categories.

 \noindent  (1) A functor $F:\mathcal C \to \mathcal D$ that is
  lax monoidal with respect to $(\diamond,\diamond')$ with structure
  $\tau_{c,d}:F(c) \diamond' F(d) \to F(c \diamond d)$, $\nu: \mathbb
  I_{\diamond'} \to F(\mathbb I_\diamond)$, and colax monoidal with
  respect to $(\star,\star')$ with structure $\rho_{c,d}:F(c \star d) \to F(c) \star' F(d)$,
  $\lambda: F(\mathbb I_\star) \to \mathbb I_{\star'}$, is called a {\em bilax functor} from $\mathcal C$ into $\mathcal D$, provided
  the following diagrams are commutative.
  \begin{enumerate}
    \item{\tt Interchange.}\[{\xymatrixcolsep{1em}\xymatrix{&&F(a\star b)\diamond'F(c\star d)\ar[lld]_{\tau_{a\star b,c \star d}}\ar[rrd]^{\rho_{a,b} \diamond' \rho_{c,d}}&&\\
   F((a\star b)\diamond (c\star d))\ar[d]_{F(\zeta_{a,b,c,d})}
   &&&&(F(a) \star' F(b))\diamond' (F(c) \star'
   F(d))\ar[d]^{\zeta'_{F(a),F(b),F(c),F(d)}}\\ F((a\diamond c)\star
   (b\diamond d))\ar[drr]_{\rho_{a\diamond c,b \diamond d}}
   &&&&(F(a)\diamond' F(c))\star'(F(b)\diamond'
   F(d))\ar[lld]^{\tau_{a,c} \diamond'\tau_{b,d}}\\ &&F(a\diamond
   c)\star'F(b\diamond d)&&}}\]
    \item{\tt Unitality.}
      \[\xymatrix{\mathbb I_{\diamond'}\ar[r]^{\nu}\ar[d]_{\Delta_{\diamond'}}&F(\mathbb I_{\diamond})\ar[r]^-{F(\Delta_{\diamond})}&F(\mathbb I_{\diamond}\star \mathbb I_{\diamond})\ar[d]^{\rho_{\mathbb I_{\diamond},\mathbb I_{\diamond}}}\\\mathbb I_{\diamond'}\star' \mathbb I_{\diamond'}\ar[rr]_{\nu\, \star' \nu}&&F(\mathbb I_{\diamond}) \star' F(\mathbb I_{\diamond})}\quad\xymatrix{\mathbb I_{\star'}\ar@{<-}[r]^{\lambda}\ar@{<-}[d]_{\mu_{\star'}}&F(\mathbb I_{\star})\ar@{<-}[r]^-{F(\mu_{\star})}&F(\mathbb I_{\star}\diamond \mathbb I_{\star})\ar@{<-}[d]^{\tau_{\mathbb I_{\star},\mathbb I_{\star}}}\\\mathbb I_{\star'}\diamond' \mathbb I_{\star'}
        \ar@{<-}[rr]_{\lambda\, \star' \lambda}&&F(\mathbb I_{\star})
        \diamond' F(\mathbb I_{\star})}\]
      \[\xymatrix{F(\mathbb I_{\star})\ar[rr]^{F(u_{\mathbb I_{\star}})=F(\varepsilon_{\mathbb I_{\star}})}&&F(\mathbb I_{\star})\ar[d]^{\lambda}\\\mathbb I_{\star'}\ar[rr]^{u_{\mathbb I_{\diamond'}}=\varepsilon_{\mathbb I_{\star'}}}\ar[u]^{\nu}&&\mathbb I_{\star'}}\]
    \end{enumerate}

    \noindent (2) 
  Let $F,G:(\mathcal C,\diamond,\star) \to (\mathcal
  D,\diamond',\star')$ be bilax functors. Call $(\tau,\nu) ,
  (\tau',\nu')$ the lax $(\diamond,\diamond')$--monoidal structures
  for $F$ and $G$ respectively and call
   $(\rho,\lambda) , (\rho',\lambda')$ the colax $(\star,\star')$--monoidal
  structures for $F$ and $G$ respectively. 

  A natural transformation $\sigma:F \Rightarrow G$ is a
  \emph{morphism of bilax functors} if it is at the same time a natural
  transformation of lax and colax functors. In other words if the following diagrams for the monoidal structures are commutative for all $c,c'\in C$.
  \[\xymatrix@=1.5pc{F(c)\diamond F(c')\ar[r]^-{\tau_{c,c'}}\ar[d]_{\sigma_c\diamond \sigma_{c'}}&F(c\diamond c')\ar[d]^{\sigma_{c\diamond c'}}\\G(c)\diamond G(c')\ar[r]_-{\tau'_{c,c'}}&G(c \diamond c')} \quad \xymatrix@=1.5pc{&\mathbb I_{\diamond'}\ar[ld]_{\nu}\ar[dr]^{\nu'}&\\F(\mathbb I_{\diamond})\ar[rr]_{\sigma_{\mathbb I_{\diamond}}}&&G(\mathbb I_{\diamond})}\]
  \[\xymatrix@=1.5pc{F(c)\star F(c')\ar@{<-}[r]^-{\rho_{c,c'}}\ar[d]_{\sigma_c\star \sigma_{c'}}&F(c\star c')\ar[d]^{\sigma_{c\star c'}}\\G(c)\star G(c')\ar@{<-}[r]_-{\rho'_{c,c'}}&G(c \star c')} \quad \xymatrix@=1.5pc{&\mathbb I_{\star'}\ar@{<-}[ld]_{\lambda}\ar@{<-}[dr]^{\lambda'}&\\F(\mathbb I_{\star})\ar[rr]_{\sigma_{\mathbb I_{\star}}}&&G(\mathbb I_{\star})}\]

\end{defn}
The details of the proof of the result that follows (and some
variations) can be found in \cite[Chap.~6, \S 8]{kn:aguiarspecies}.

\begin{lem}\label{lem:functbimonoid}
  Let $\mathcal C=(\mathcal C,\diamond, \mathbb
  I_{\diamond},\star,\mathbb I_{\star})$ and $\mathcal D=(\mathcal
  D,\diamond', \mathbb I_{\diamond'},\star',\mathbb I_{\star'})$ be
  duoidal categories and $b=(b, \mu_b, u_b, \Delta_b, \varepsilon_b)$
  a bimonoid in $\mathcal C$. Let $F:\mathcal C \to \mathcal D$ be a
  lax $(\diamond,\diamond')$--monoidal functor and colax
  $(\star,\star')$--monoidal functor and $\tau, \nu, \rho$ and
  $\lambda$ as in Definition \ref{defn:bilaxandtransf}. Then
  $F(b)=\bigl(F(b), F(\mu_b) \smallcirc \tau_{b,b}, F(u_b) \smallcirc \nu,
  \rho_{b,b} \smallcirc F(\Delta_b), \lambda \smallcirc F(\varepsilon_b)\bigr)$
  is at the same time a monoid and a comonoid with respect to
  $\diamond'$ and $\star'$ respectively. Moreover, if $F$ is bilax
  monoidal, then $F(b)$ is a {\em bimonoid} in $\mathcal D$.

  Moreover, if $f$ is a morphism of bimonoids so is $F(f)$, hence $F$
  restricts to a functor in the categories of bimonoids:
  $F:\operatorname{Bimon}(\mathcal C) \to
  \operatorname{Bimon}(\mathcal D)$.\qed
\end{lem}

\begin{defn}
  \label{defn:actionschqc}
(1) Assume that $(\mathcal C,{\scriptstyle \odot}\,, \mathbb
I_{\scriptscriptstyle \odot})$ is a monoidal category and let $m= (m,
\mu_m, u_m)$, be a monoid in $(\mathcal C,\, {\scriptstyle \odot}\,,
\mathbb I_{\scriptscriptstyle \odot})$.  A (left) {\em $m$--module
  structure} on $x \in \mathcal C$ or a (left) {\em action} of $m$ on
$x$ is a morphism $\alpha_x: m \, {\scriptstyle \odot}\, x \to x$ such
that the following diagrams are commutative:
  \[{\xymatrixcolsep{3.5em}\xymatrix{
      m \,{\scriptstyle \odot}\, m\,{\scriptstyle \odot}\, x
      \ar[r]^-{\id_m  {\scriptscriptstyle \odot}\,
        \alpha_x}\ar[d]_{\mu_m  \,{\scriptscriptstyle \odot}\, \id_x}& m \, {\scriptstyle \odot}\, 
      x\ar[d]^{\alpha_x} & \mathbb I_{\scriptscriptstyle \odot}  \,
      {\scriptstyle \odot}\, 
      x\ar[r]^{u_m \,{\scriptscriptstyle \odot}\, \id_x}\ar[rd]_{\cong}& m \, {\scriptstyle \odot}\,
      x\ar[d]^{\alpha_x}\\
      m \, {\scriptstyle \odot}\, x\ar[r]_{\alpha_x}& x & & x }}
  \]

The pair $(x,\alpha)$ is called an \emph{$m$--module} --- when there is no
ambiguity on the action, we will say that $x$ is an $m$--module. If
$(x,\alpha_x)$, $(y,\alpha_y)$ are left $m$--modules,  a
\emph{morphism of  left $m$--modules} $x\to y$ is a morphism $f\in 
\Hom(x,y)$ such that $\alpha_y\smallcirc(\id_m\,{\scriptstyle \odot}\, f)= f\smallcirc
\alpha_x:m\,{\scriptstyle \odot}\, x\to y$. 

We define  $\Lmodmon{m}$, the \emph{category of left $m$-modules} as
the category with objects the left $m$--modules and morphism the
morphisms the morphism of left $m$--modules.

In a similar way, we define $\Rmodmon{m}$, the \emph{category of right $m$-modules}.

\noindent (2)  If $(c,\Delta_c,\varepsilon_c)$ is a comonoid in $(\mathcal
 C, \, {\scriptstyle \odot}\,,\mathbb I_{\scriptscriptstyle \odot})$,
 a right {\em $c$--comodule structure} on 
 $y \in \mathcal C$ or a (right) {\em coaction} of $c$ on $y$ is a
 morphism $\chi_y: y \to y \,{\scriptstyle \odot}\, c$ such that the following diagrams
 are commutative:
  \[{\xymatrixcolsep{3.5em}\xymatrix{
     y \ar[r]^-{\chi}\ar[d]_{\chi}& y  \, {\scriptstyle \odot}\, c\ar[d]^{\chi  \, {\scriptscriptstyle \odot}\, \id_c}
     & y\ar[r]^-{\chi}\ar[rd]_-{\cong}& y \, {\scriptstyle \odot}\, c\ar[d]^{\id_y  {\scriptscriptstyle \odot}\,
       \varepsilon_c}\\
     y \, {\scriptstyle \odot}\, c\ar[r]_-{\id  {\scriptscriptstyle \odot}\, \Delta_c}& y  \, {\scriptstyle \odot}\,c
      \, {\scriptstyle \odot}\, c & &y  \, {\scriptstyle \odot}\, \mathbb I_{\scriptscriptstyle \odot}}}
\]

We say that the pair $(y,\chi)$ is a \emph{right $c$--comodule} --- we will
often omit the coaction and say that $y$ is a $c$--comodule.

If $(x, \chi_x)$ and $(y,\chi_y)$ are $c$--comodules, a \emph{morphism
  of $c$--comodules} between $x$ an $y$ is a morphism $f\in \Hom(x,y)$
such that $(f\,{\scriptstyle \odot}\, \id_c)\smallcirc \chi_x= \chi_y\smallcirc
f:x\to y\,{\scriptstyle \odot}\, c$.

We denote  $\Rcomod{c}$ the \emph{category of right $c$--comodules},
defined in the usual way. Analogously, we define its left counter
part $\Lcomod{c}$, the  \emph{category of left $c$--comodules},

\end{defn}

\begin{ej}\label{ej:qS}
  If $m:M\to A$ is a quasi-compact morphism of monoid (group) schemes,
  then to give an $m$--module in the monoidal category $(\Schaqc,
  \wtimes, \uwtimes)$ is equivalent to give a pair $(x:X\to
  A,a)$, where $a:M\times X\to X$ is an action --- in the
  usual sense --- such that
  the following diagram is commutative:
  \[
    \xymatrix{
      M\times X\ar[d]_{m \times x}\ar[r]^-a & X\ar[d]^x\\
      A\times A\ar[r]_-s& A
    }\quad\quad \xymatrix{\Speck \times X
    \ar[dr]_{\cong}\ar[r]^-{u_M \times \id} & M \times X
    \ar[d]^{a}\\&X}
  \]

  In particular, if $\mathcal S \in q:G \to A$ is an affine (or more
  generally quasi-compact) extension, then any representation $E\in
  \Rep(\mathcal S)$ --- in the nomenclature of Definition
  \ref{defn:catrepaffext} --- is a $q$--module in the monoidal
  category $\Schaqc$ and conversely.
\end{ej}

In the case that the monoidal category $(\mathcal C,\,{\scriptstyle \odot}\,,\mathbb
I_{\scriptscriptstyle \odot})$ is part of a duoidal category, more structure is
available as shown in the next proposition, that admits --- wherever it
is possible --- a left and a right version.

\begin{prop}\label{prop:fromaguiar} Let 
  $(\mathcal C,\diamond, \mathbb I_{\diamond},\star,\mathbb
  I_{\star})$ be a duoidal category and $m_1,m_2$ be monoids in
  $(\mathcal C,\diamond, \mathbb I_{\diamond})$, $c_1,c_2$ comonoids
  in $(\mathcal C,\star, \mathbb I_{\star})$ and $(b, \mu_b, u_b,
  \Delta_b, \varepsilon_b)$ a bimonoid in $(\mathcal C,\diamond, \mathbb
  I_{\diamond},\star,\mathbb I_{\star})$. Then:

  \noindent (1) $m_1\star
  m_2$ is a monoid in $(\mathcal C,\diamond, \mathbb I_{\diamond})$
  and $c_1 \diamond c_2$ is a comonoid in $(\mathcal C,\star, \mathbb
  I_{\star})$;

  \noindent (2) if $x_1,x_2$ are left
  modules for $m_1$ 
and $m_2$ with structures $\alpha_1, \alpha_2$ respectively, then $x_1 \star x_2$ is a left $m_1 \star m_2$ module with structure:
\[
  \xymatrix{\alpha_1\underline{\star}\alpha_2: (m_1 \star m_2)\diamond (x_1 \star x_2) \ar[rr]^-{\zeta_{m_1,m_2,x_1,x_2}}&&(m_1 \diamond
  x_1) \star (m_2 \diamond x_2) \ar[r]^-{\alpha_1 \star \alpha_2}&x_1
  \star x_2;}
  \]

  \noindent (3)
  if $y_1,y_2$ are right comodules for $c_1$ and $c_2$ with structures
  $\chi_1, \chi_2$ respectively, then $y_1 \diamond y_2$ is a right
  $c_1 \star c_2$ comodule with structure:
\[
  \xymatrix{\chi_1\underline{\star}\chi_2: y_1 \diamond y_2 \ar[r]^-{\chi_1 \diamond \chi_2} & (y_1 \star c_1)\diamond (y_2 \star c_2) \ar[rr]^-{\zeta_{y_1,c_1,y_2,c_2}}&&(y_1 \diamond
    y_2) \star (c_1 \diamond c_2);}\]

\noindent (4) if  $x_1,x_2$ are
left modules for $b$ with structures $\alpha_1,\alpha_2$ then, $x_1
\star x_2$ is also a left module for $b$ with structure:
\[\xymatrix{\alpha_{1,2}: b \diamond (x_1 \star x_2) \ar[rr]^-{\Delta_b \diamond\, \id} &&(b \star
  b) \diamond (x_1 \star x_2) \ar[r]^-{\alpha_1 \underline{\star}
    \alpha_2}&x_1 \star x_2.}\]

\noindent (5) if $y_1,y_2$ are
right comodules for $b$ with structures $\chi_1,\chi_2$ then, $y_1
\diamond y_2$ is also a right comodule for $b$ with structure:
\[\xymatrix{\chi_{1,2}: y_1 \diamond y_2 \ar[r]^-{\chi_1 \underline{\diamond} \chi_2}& (y_1 \diamond y_2) \star (b \diamond b) \ar[rr]^-{\id \star \mu_b} &&(y_1 \diamond y_2)\star b.}\]
\end{prop}
\proof The detailed proof of this result can be found in
\cite[Prop.6.25]{kn:aguiarspecies}. For example, the morphism given by
the composition of the interchange map and the $\star$ monoidal
product of the multiplications of $m_1$ and $m_2$, \[\xymatrix{(m_1
  \star m_2)\diamond (m_1 \star m_2)
  \ar[rr]^{\zeta_{m_1,m_2,m_1,m_2}}&&(m_1 \diamond m_1) \star (m_2
  \diamond m_2) \ar[r]^-{\mu_1 \star \mu_2}&m_1 \star m_2,}\] is the
multiplication morphism of $m_1 \star m_2$.  \qed

\begin{defn}\label{defn:bicomdalg}
   If $b$ is a bimonoid in $\mathcal C$, a {\em right $b$--comodule
     algebra} in the duoidal category $\mathcal C$ is a right
   $b$--comodule $(y,\chi)$ equipped also with a monoid structure
   $(y,\mu_y,u_y)$ in $(\mathcal C,\diamond,\mathbb I_\diamond)$ such
   that the diagrams below commute:
\[\xymatrix{y \diamond y \ar[r]^-{\chi_{12}}\ar[d]_{\mu_y}&(y \diamond y) \star b\ar[d]^{\mu_y\star \id}\\y\ar[r]^-{\chi} &y \star b}\quad\quad
\xymatrix{\mathbb I_\diamond
  \ar[r]^-{\Delta_{\diamond}}\ar[d]_{u_y}&\mathbb I_\diamond \star
  \mathbb
I_\diamond \ar[d]^{u_y\star u_b}\\y\ar[r]^-{\chi} &y \star b}\]

As usual, we have also the notion of \emph{left $b$--comodule algebra}.
\end{defn}

 The following Corollary follows easily from the proposition above and extends the results of Lemma \ref{lem:functbimonoid} to modules and comodules. 

 \begin{cor}\label{cor:complement}
   If $\mathcal C=(\mathcal C,\diamond, \mathbb
   I_{\diamond},\star,\mathbb I_{\star})$ be a duoidal category, then
   the subcategory $\operatorname{Mon}(\mathcal C, \diamond) \subset
   \mathcal C$ is monoidal with respect to the structure
   $(\star,\mathbb I_{\star})$ and the subcategory
   $\operatorname{Comon}(\mathcal C,\star) \subset \mathcal C$, is
   monoidal with respect to the structure $(\diamond,\mathbb
   I_{\diamond})$.

   Assume that $\mathcal D=(\mathcal D,\diamond',
   \mathbb I_{\diamond'},\star',\mathbb I_{\star'})$ is another duoidal
   category and that $F:\mathcal C \to \mathcal D$ is a functor that
   is lax $(\diamond,\diamond')$--monoidal and colax
   $(\star,\star')$--monoidal.

   Then $F$ restricts to a
   colax--$(\star,\star')$ monoidal functor
   $F:\bigl(\operatorname{Mon}(\mathcal C,\diamond, \mathbb
   I_{\diamond}),\star, \mathbb I_{\star}\bigr) \to
   \bigl(\operatorname{Mon}(\mathcal D,\diamond', \mathbb
   I_{\diamond'}),\star', \mathbb I_{\star'}\bigr)$.

   Dually, $F$ restricts to a lax--$(\diamond,\diamond')$ monoidal
   functor $F:\bigl(\operatorname{Comon}(\mathcal C,\star, \mathbb
   I_{\star}),\diamond, \mathbb I_{\diamond}\bigr) \to
   \bigl(\operatorname{Mon}(\mathcal D,\star', \mathbb
   I_{\star'}),\diamond', \mathbb I_{\diamond'}\bigr)$.

   Also, in the case that the functor $F$ is a bilax
   functor  and $b \in \operatorname{Bimod}(\mathcal C)$, we call
   $\operatorname{Mod(b)}$ and $\operatorname{Comod(b)}$ the
   categories of $b$--modules with respect to the $\diamond$--monoidal
   structure and of $b$--comodules with respect to the
   $\star$--monoidal structure. Then: 

   \noindent {\em (i)} $\bigl(\operatorname{Mod(b)},\star\bigr)$ and
   $\bigl(\operatorname{Comod(b)},\diamond\bigr)$ are monoidal categories. 

   \noindent {\em (ii)} If $F: \mathcal C \to \mathcal D$  is a bilax
   monoidal functor as above. Then $F$ induces a monoidal functor
   $F:\bigl(\operatorname{Mod}(b),\star \bigr) \to
   (\operatorname{Mod}(F(b)),\star')$ and
   $F:\bigl(\operatorname{Comod}(b),\diamond\bigr) \to
   \bigl(\operatorname{Comod}(F(b)),\diamond'\bigr)$.  
 \end{cor}
 \proof The proof is direct using the definitions and results of
 Lemma \ref{lem:functbimonoid} and Proposition \ref{prop:fromaguiar}.\qed

\subsection{Quasi-compact morphisms and their associated
   sheaves}\ %
\label{sect:quasicompandsheaves}

In this section we collect some results and definitions on
(separated) quasi-compact morphisms and their associated 
(quasi-coherent) sheaves,  that will be  used later. The basic
definitions can be found in  
\cite[Chap.~II]{kn:hartshorne} or \cite[Chap.~0, Chap.~1]{kn:EGAI}. The original reference for the adjunction results is \cite[\S 1.2,\S
  1.3]{kn:EGAII}, but they also appear as a series of exercises in
  \cite[Ex. II.5.17, II.5.18]{kn:hartshorne} as well as in many other references, e.g.~\cite{kn:stackproj}, \cite{kn:raising}. 
\begin{nota}
  If $S$ is a scheme, the category of $\mathcal O_S$--modules (algebras) will be
  denoted as $\sSmod$ ($\sSalg$) and the category of quasi-coherent $S$--modules (algebras)
  as $Q\sSmod$ ($Q\sSalg$).
\end{nota}

\begin{defn} \label{defn:defPP}
  We  denote $\PP: \Schqc\to \sSalg^{\op}$ the functor given as
  follows (see \cite{kn:EGAII}[Prop (1.3.1)] or \cite[Proposition
  II.5.8]{kn:hartshorne}):  
  If $x:X\to S \in \Schqc$, then $\PP(x):=x_*(\mathcal
  O_X)$;
  if  $(f,f^\#):(x:X \to S)\to (x':X' \to S)$ is a
  morphism in $\Schqc$ (recall that $ f^\#:\mathcal O_{X'} \to
  f_*\mathcal O_X$)  then 
 $\PP(f,f^\#):= x'_*(f^\#): \PP(x')= x'_*(\mathcal O_{X'}) \to
 \PP(x)=x_*(\mathcal O_{X})$.

It is well known that the restriction of  $\PP$ to $\Schsqc$ induces a
functor to $Q\sSalg^{\op}$, 
that we still call $\PP: \Schsqc\to Q\sSalg^{\op}$.
\end{defn}

  We recall now the  well known  construction of a right adjoint to
  $\PP: \Schsqc\to Q\sSalg^{\op}$.

  \begin{rem}(see
  \cite[Prop.~1.3.1 and Prop.~1.2.7]{kn:EGAII})
  \label{rem:affingen}
  The functor $\PP: \Schsqc\to Q\sSalg^{\op}$
  admits a right adjoint $\Spec: Q\sAalg^{\op}\to \Schsqc$. Given
  $\mathcal F\in Q\sSalg$, then $\Spec(\mathcal F)$ is the unique ---
  up to isomorphisms of $S$-schemes ---
  scheme over $S$ such that $\Spec(\mathcal F)\in \Schaff$  and
  $\PP\bigl(\Spec(\mathcal F)\bigr)=\mathcal F$. We denote $\Spec(\mathcal F)$ as 
   $\pi_{\mathcal F}:\Spec(\mathcal
   F)\to S$.

   Recall that the adjunction   ${\xymatrixcolsep{2.5em}\xymatrix{\Schsqc \ar@<5pt>[r]^{\PP}
     \ar@<-5pt>@{<-}[r]_{\Spec} \ar@{}[r]|-\bot&Q\sSalg^{\op}}}$ is
   given by a natural transformation   
   \[        \Hom_{Q\sSalg}\bigl(\mathcal F,\PP(x)\bigr)=
        \Hom_{(Q\sSalg)^{\operatorname{op}}}
        \bigl(\PP(x),\mathcal F\bigr)\cong \Hom_{\Schsqc}(x,\Spec
        \mathcal F),
      \]
      that has a  unit $\eta$  and counit $\varepsilon$
        the natural transformations:
      \[
        \begin{cases} \text{(i)}\,\, \eta_{x}: x \to
         \operatorname{Aff}_S(x):=\Spec\bigl(\PP(x)\bigr)\in \Schsqc
         ;\\ \text{(ii)} \,\, \varepsilon_{\mathcal F}:
         \PP\bigl(\Spec(\mathcal F)\bigr) \to \mathcal F \in
         (Q\sSalg)^{\operatorname{op}} \text{ or } \\ \,\,\,\,\quad
         \varepsilon_{\mathcal F}: \mathcal F \to
         \PP\bigl(\Spec(\mathcal F)\bigr)\in Q\sSalg,
      \end{cases}
    \]

    In this particular case, we have that for all $\mathcal F \in
      Q\sSalg\,,\,\varepsilon_{\mathcal F}$ is an isomorphism (see
      \cite[\S 1.2, \S 1.3, \S 1.4]{kn:EGAII}).

  \end{rem}

   \begin{defn}\label{defn:relativeaffi}
By construction,    for  $x \in
      \Schsqc$ the morphism $\eta_x:x \to \operatorname{Aff}_S(x) \in
      \Schsqc$ is affine; $\eta_x$ is called the \emph{(relative) affinization map} and its
      codomain is called the \emph{($S$--)affinization} of $x$ (frequently
      just called ``affinization'').
The functor $\operatorname{Aff}_S=\Spec\smallcirc \PP:\Schsqc \to \Schaff \subset \Schsqc$
is called the \emph{affinization over $S$} or the \emph{relative
  affinization (over $S$)}. Compare with Definition \ref{def:affinization}.
\end{defn}

 \begin{rem}      
The  affinization   satisfies the following
universal property (compare with Remark \ref{rem:gantqc}):

For any morphism $f:x \to y$ with $x \in
        \Schsqc, y \in \Schaff$ there is a unique morphism
        $\widehat{f}$ that makes commutative the diagram below:
         $\xymatrix@=15pt{x \ar[rr]^-{\eta_x}\ar[rd]_{f}&&
         \operatorname{Aff}_S(x)\ar@{.>}[ld]^{\widehat{f}}\\&y&}$
\end{rem}

     In view of Remark \ref{rem:affingen},  we have the following equivalence of categories: 
    
    \begin{prop}
      \label{prop:affschqcs}
      The
      adjunction $   {\xymatrixcolsep{3em}\xymatrix{ \Schsqc \ar@<5pt>[r]^-\PP
     \ar@<-5pt>@{<-}[r]_-\Spec \ar@{}[r]|-\bot&(Q\sSalg)^{\operatorname{op}}}}$
    restricts to functors
      $\PP|_{_{\Schaff}}: \Schaff\to Q\sSalg^{\op}$ and $\Spec:
      Q\sSalg^{\op}\to \Schaff$ that establish an adjoint
      op-equivalence between $\Schaff$ and $Q\sSalg$.

      In particular
      the counit $\varepsilon$ of the original adjunction is an isomorphism (see 
      Remark \ref{rem:affingen} and \cite[Prop.~1.3.1]{kn:EGAII} for a proof). \qed
      \end{prop}

We finish this section  presenting  (without proofs) some known results on flat (separated,
  quasi-compact) schemes over $S$ that will be needed; we follow \cite[\S 2.1 --
    2.3]{kn:EGAIV2}.

\begin{lem}
 Let  $x:X\to S\in \Schsqc$. Then $x$ is flat if and only if   the pull-back functor
  $x^*:\sSmod \to \sXmod$ is exact. \qed
 \end{lem}

 \begin{lem}  Let  $x:X\to S\in \Schsqc$ be a flat morphism. Then
$\PP(x)=x_*(\mathcal O_X)$ is a   quasi-coherent flat sheaf of algebras in $\sSmod$.

Conversely, if  $\mathcal   F$ is a quasi-coherent flat sheaf of algebras in $\sSmod$, then
          $\pi_{\mathcal F}: \Spec(\mathcal F)\to S$ is an
          affine (hence separated) flat morphism.\qed 
        \end{lem}

\begin{rem}
          \label{rem:monw}  
  Let $ X,Y\in \Schs$,  $\mathcal F\in \sXmod$, $\mathcal G\in \sYmod$, and let $h:X\to Y$
  $\ell:Y\to Y'$ be morphisms of 
  schemes over $S$.  Then the counits associated to the
  adjunction between $h^* \dashv h_*; \ell^* \dashv \ell_*$:
  $\varepsilon_{\mathcal F}:h^*h_* \mathcal F \to \mathcal F$ and
  $\varepsilon_{\mathcal G}:\ell^*\ell_* \mathcal G \to \mathcal G$, induce
  a  homomorphism of sheaves
  \[
    \varepsilon_{\mathcal F}\sboxtimes
  \varepsilon_{\mathcal G}:h^*h_* \mathcal F \sboxtimes \ell^*\ell_*
  \mathcal G = (h \times \ell)^*(h_* \mathcal F \sboxtimes \ell_* \mathcal
  G) \to \mathcal F \sboxtimes \mathcal G.
\]

Using the standard adjunction again we obtain the map:
\[
  \Gamma_{\mathcal F, \mathcal G}:=(h \times
  \ell)_*(\varepsilon_{\mathcal F}\sboxtimes \varepsilon_{\mathcal
    G})\nu_{_{h_* \mathcal F \sboxtimes \ell_* \mathcal G}}: h_* \mathcal
  F \sboxtimes \ell_* \mathcal G \to (h \times \ell)_*(\mathcal F
  \sboxtimes \mathcal G),\] where $\nu_{_{h_* \mathcal F \sboxtimes
    \ell_* \mathcal G}}$ is the unit of the adjunction.
\end{rem}

We are interested in conditions on $h,\ell$ and $\mathcal F,\mathcal G$ in order to guarantee that  $\Gamma_{\mathcal
  F,\mathcal G}$  is  an isomorphism. A set of conditions was
established by Brandenburg in response to a question in
``stackexchange/math'', in a more general setting. In our context,
Brandenburg answer can be stated as follows:  

\begin{lem}
  \label{lem:bran}
  Let $h,\ell$, $\mathcal F,\mathcal G$ as above and assume that either:
  
  \noindent (a) $h,\ell$ are quasi-compact and quasi-separated, or
  
  \noindent (b)  $h,\ell$ are affine morphisms.

Then  $\Gamma_{\mathcal
    F,\mathcal G}$  is  an isomorphism.
\end{lem}
\proof
See \cite{kn:bran}.\qed

\begin{rem}
 In the notations of Lemma \ref{lem:bran}, notice that since we are working with $\Bbbk$--schemes, it follows that 
$\mathcal F, \mathcal G$ are $S$--flat quasi-coherent sheaves --- that
is,
$h_*\mathcal F$ and $\ell_*\mathcal G$ are flat sheaves. 
  \end{rem}

\subsection{A duoidal structure for QA-mod}
\label{sect:hopfsheaves}\ %

In Section \ref{subsect:affextschoverA} we considered a duoidal
structure on $\Schaqc$ together with an additional functor
$\op_*:\Schaqc \to \Schaqc$ (denoted as $\op_*(x)=-x$) for which the
 subcategory $\operatorname{Bimon}(\Schaqc)
 \subseteq \Schaqc $ has as objects  the quasi-compact morphisms of
 monoids $q_M:M\to A$ (and as morphisms the morphisms over $A$, that
 are  of monoids).
Thus, the category $\GextqcA$ of quasi-compact group extensions of the abelian variety $A$
can be interpreted as the full subcategory of
$\operatorname{Bimon}(\Schaqc)$ that has objects the faithfully flat
morphisms  of group schemes. In terms of the category $\Schaqc$, a
group extension is a bimonoid $x$ in the category,   equipped
with an additional arrow $\iota_x:x \to (-x)$ satisfying the
supplementary conditions of Theorem \ref{thm:antipoduoidal}.  Since
the   category $\GextaffA$ of affine extensions of $A$ is a full
subcategory of $\GextqcA$, it can viewed also as a full subcategory of
 $\operatorname{Bimon}(\Schaqc)$.

In order to dualize the above situation to the category of sheaves, we
first restrict the above setting to the duoidal category based upon
$\Schasqc$ --- recall that since a since a group scheme
is separated $\GextqcA\subset \operatorname{Bimon}(\Schasqc)$. In this
context, the functors $\PP$ and $\Spec$ (see 
Remark \ref{rem:affingen}) establish an adjunction between
$\Schasqc$ and $Q\sAalg^{\op}$ --- our goal is  to describe the behavior of $\GextaffA$
under the mentioned adjunction. In order to describe the correlate of
$\operatorname{Bimon}(\Schasqc)$ in $Q\sAalg$, we introduce a duoidal
structure on $Q\sAmod$ such that the corresponding subcategory of
bimonoids $\operatorname{Bimon}(Q\sAmod)$ is the seeked correlate.
Finally, we introduce an additional structure that corresponds
under the adjunction to the antipode, and construct the category of
\emph{commutative Hopf sheaves}. In this  we end up with a monoidal
adjunction from the category $\GmorqcA$ (see Definition
\ref{defn:Gmor})  and the category of 
schemes \emph{commutative Hopf sheaves}. If moreover we restrict
ourselves the subcategory $\GextqcA$, we obtain an adjunction with the
subcategory of  \emph{flat commutative Hopf sheaves}, that restrict in
turn to an equivalence between the category of affine extensions of
$A$ and the category of flat commutative Hopf sheaves.

 Be begin by recalling the definition of the external tensor product
 of sheaves over a scheme $S$.

 \begin{defn}
  \label{defn:otherkind}
Let $S$ be a scheme and $X,Y\in \Schs$. If $\mathcal F\in Q\sXmod$,
and $\mathcal G\in Q\sYmod$, we define the sheaf $\mathcal F
\sboxtimes_{S} \mathcal G:=p_1^* \mathcal F \otimes_{\mathcal O_{X
    \times_S Y}}p_2^* \mathcal G \in Q\sXYmod$, where $p_1:X\times_S
Y\to X$ and $p_2: X\times _S Y\to Y$ are the canonical
projections. This correspondence can be extended to a functor
$\sboxtimes_{S}: Q\sXmod \times Q\sYmod \to Q\sXYmod$.  This functor
is called in \cite[Section 9]{kn:EGAI} the \emph{tensor product over
  $\mathcal O_S$} or the {\em tensor product over $S$}, but currently
it is called the \emph{external tensor product} (over $S$).  In the
particular case that $S=\Bbbk$, we usually write $\mathcal
F\sboxtimes_{\Spec(\Bbbk)}\mathcal G=\mathcal F \sboxtimes\, \mathcal
G$.
\end{defn}

 \begin{rem} In the situation that $X=Y=S$ and $\mathcal F,\mathcal G \in Q\sSmod$,\,$\mathcal F \sboxtimes_S \mathcal G=\mathcal F \otimes_{\mathcal O_S} \mathcal G$ the usual monoidal structure in the category of the sheaves of $\mathcal O_S$--mod.
   \end{rem}

  \begin{defn}\label{defn:othermonoidal2} Let $A$ be an abelian variety.
 \noindent(1) Along this section and to be consistent with Definition
 \ref{defn:othermonoidal1} we call the usual monoidal structure
 $\otimes_A:Q\sAmod\times Q\sAmod\to Q\sAmod$ the \emph{Hadamard monoidal structure} --- the unit
 $\uAotimes$ is $\mathcal O_A$.

    \noindent (2) We define the \emph{Cauchy monoidal structure in
      $Q\sAmod$} as follows:
 \[
   \wsboxtimes=s_* \smallcirc \sboxtimes_{\Speck}: \xymatrix{Q\sAmod \times Q\sAmod
     \ar[r]^(0.55){\sboxtimes_{\Speck}}&
     Q\sAAkmod \ar[r]^-{s_*}& Q\sAmod,}
 \]
 where  $s_*$ is the push-forward functor by
 the addition morphism $s:A\times A\to A$.

 It is easy to show that $\bigl(Q\sAmod,\wsboxtimes,
 \uwsboxtimes=\operatorname{skysc}_0(\Bbbk)\bigr)$, where 
 $\operatorname{skysc}_0(\Bbbk)\in Q\sAmod$ denotes the 
 skyscraper sheaf at $0 \in A$ with 
 stalk $\Bbbk$, is a monoidal structure.

 Indeed, if $\mathcal F,
\mathcal G$ are quasi-coherent sheaves on $A$, then 
$\mathcal F \sboxtimes_{\Speck} \mathcal G$ is also
quasi-coherent. Moreover,  since
$\mathcal F \wsboxtimes \mathcal G=s_*(\mathcal F
\sboxtimes_{\Speck}\,\mathcal G)$ is the push-forward of a
quasi-coherent sheaf by the
proper morphism $s:A\times A\to A$, the sheaf $\mathcal F \wsboxtimes \mathcal
G$ is also quasi-coherent (see for example \cite[Proposition
5.8]{kn:hartshorne}). 

On the other hand, if $\iota=(\id,0):A \to A \times_{\Bbbk}A$ is the
closed immersion $\iota(a)=(a,0)$ and $\mathcal F \in Q\sAmod$, then
$\iota_*(\mathcal F)\cong \mathcal F\sboxtimes_{\Speck}
\operatorname{skysc}_0(\Bbbk)$. Applying 
$s_*$ to the above isomorphism (and  using that $s\iota=\id_A$), we have
that $\mathcal F= s_*\bigl(\mathcal F\sboxtimes_{\Speck}
\operatorname{skysc}_0(\Bbbk)\bigr)=\mathcal F\wsboxtimes
\operatorname{skysc}_0(\Bbbk)$.  Therefore,
$\operatorname{skysc}_0(\Bbbk)$ is a right side unit, and similarly
one proves that $\operatorname{skysc}_0(\Bbbk)$ is also a  left side unit.
\end{defn}

The next elementary constructions and notations will be useful in what follows.

\begin{defn}    
 If $\op:A \to A$ is the
 inversion morphism of $A$ we consider the functor $\op_*: Q\sAmod \to Q\sAmod$ and define $-\mathcal F:=\op_*(\mathcal F)$.
\end{defn}

 \begin{rem}\label{rem:skyscisunit}
  Let $0:\Spec(\Bbbk)\to A$, $\st:A\to \Spec(\Bbbk)$ and $c_0=0\, \smallcirc \st:
  A \to A$ (see Definition \ref{defn:op-0-om}). Then:

\noindent (1) If $\mathcal F \in \sAmod$, then the sheaf
$\operatorname{st}_*(\mathcal F)$ has stalk $\mathcal F(A)$ at the
only point of $\Speck$. If $\mathcal V$ a sheaf on $\Speck$ of stalk
$V$, then $0_{*}(\mathcal V)=\operatorname{skysc}_0(V)$ and
$(c_0)_*(\mathcal F)=\operatorname{skysc}_0\bigl(\mathcal F(A)\bigr)$.

\noindent (2)  On the other hand, for $\mathcal V$ as above, 
$\operatorname{st}^{-1}(\mathcal V)(U)=V$ for all $U$ open in
$A$. Hence $\operatorname{st}^*(\mathcal V)=\mathcal O_A
\otimes_{\Bbbk} V$.
             
\noindent (3) If $\mathcal F \in \sAmod$, then $0^{-1}\mathcal F$
is the sheaf of $\mathcal O_{A,0}$--modules on $\Speck$ with stalk
$\mathcal F_0$. Hence, $0^*\mathcal F$ is the sheaf of $\mathcal
O_{\Speck}$--modules (i.e.~$\Bbbk$--spaces) with stalk $\Bbbk
\otimes_{\mathcal O_{A,0}} \mathcal F_0=\mathcal F_0/\mathcal
M_{A,0}\mathcal F_0$, where $\mathcal M_{A,0} \subseteq \mathcal
O_{A,0}$ is the maximal ideal of the local ring $\mathcal O_{A,0}$.

\noindent (4) Combining (2) and (3), we deduce that $c_0^*(\mathcal
F)=\mathcal O_A \otimes_{\Bbbk} \mathcal F_0/\mathcal M_{A,0}\mathcal
F_0$. In particular,
$c_0^*\bigl(\operatorname{skysc}_0(\Bbbk)\bigr)= \mathcal O_A$.
\end{rem}

\begin{prop}\label{prop:zetasheaf}
  Let $A$ be an abelian variety. Then
  $(Q\sAmod,\otimes_A,\uAotimes,\wsboxtimes, \uwsboxtimes)$ can be
  completed to a duoidal structure (see Definitions \ref{defi:duoidalcat},
  \ref{defn:otherkind}, \ref{defn:othermonoidal2}). Moreover, the
  subcategory $Q\sAalg$ inherits this duoidal category, that is\\
  $(Q\sAalg,\otimes_A|_{_{Q\sAalg\times
    Q\sAalg}},\uAotimes,\wsboxtimes|_{_{Q\sAalg\times Q\sAalg}},
  \uwsboxtimes)$  is a duoidal category.
\end{prop}

\begin{proof} We have already shown that
  $(Q\sAmod,\otimes_A,\uAotimes)$ and   $(Q\sAmod,\wsboxtimes,
  \uwsboxtimes)$ are monoidal structures. We concentrate now our
  attention in the description of the
  interchange law as presented in Definition
  \ref{defi:duoidalcat}: for sheaves $\mathcal A,\mathcal B, \mathcal
  C, \mathcal D \in Q\sAmod$ we need to define:
\[
\zeta_{\mathcal
    A,\mathcal B, \mathcal C, \mathcal D}: (\mathcal A \wsboxtimes
  \mathcal B)\otimes_A (\mathcal C \wsboxtimes \mathcal D) \to
  (\mathcal A \otimes_A \mathcal C)\wsboxtimes (\mathcal B \otimes_A
  \mathcal D).
\] 

Since $ p_1^*(\mathcal A \otimes_{\mathcal O_A} \mathcal C) = p_1^*
\mathcal A \otimes_{\mathcal O_{A \times A}}p_1^* \mathcal C$ and $
p_2^*(\mathcal B \otimes_{\mathcal O_A} \mathcal D) = p_2^* \mathcal B
\otimes_{\mathcal O_{A \times A}}p_2^* \mathcal D $, if we write
$p_1^* \mathcal A = \mathcal L$, $p_2^* \mathcal B = \mathcal R$,
$p_1^* \mathcal C = \mathcal M$, $p_2^* \mathcal D = \mathcal N$, we
need to check that there is a natural morphism
  \[
 {\xymatrixcolsep{1.2pc}  \xymatrix{ \zeta:s_* (\mathcal L
  \otimes_{\mathcal O_{A \times A}} \mathcal R) \otimes_{\mathcal O_A}
  s_*(\mathcal M \otimes_{\mathcal O_{A \times A}} \mathcal N)\ar[r]\ar[rd] &  s_* (\mathcal
  L \otimes_{\mathcal O_{A \times A}} \mathcal M \otimes_{\mathcal O_{A \times A}} \mathcal
  R \otimes_{\mathcal O_{A \times A}} \mathcal N)\ar[d]^{\cong}\\
  & s_* (\mathcal
  L \otimes_{\mathcal O_{A \times A}} \mathcal R \otimes_{\mathcal O_{A \times A}} \mathcal
  M \otimes_{\mathcal O_{A \times A}} \mathcal N)}}
\]

 The
  existence of this map follows from the general fact that in the
  context above if $\mathcal X,\mathcal Y$ are sheaves on ${A \times A}$ then,
  due to the existence of the morphism of algebras $s^\sharp: \mathcal
  O_A \to s_*\mathcal O_{A \times A}$, there is a natural map $s_*\mathcal X
  \otimes_{\mathcal O_A} s_*\mathcal Y \to s_*\mathcal X \otimes_{
    s_*\mathcal O_{A \times A}} s_* \mathcal Y=s_*(\mathcal X \otimes_{\mathcal
    O_{A \times A}} \mathcal Y)$.
  
  The
  $\wsboxtimes$--comonoidal structure $\Delta_{\otimes_A}:\uAotimes
  \to \uAotimes \wsboxtimes\uAotimes$ is
  $\Delta_{\otimes_A}:=s^\sharp:\mathcal O_A \to s_* \mathcal O_{A
    \times A}=\mathcal O_A \wsboxtimes \mathcal O_A$ and the
  $\otimes_A$--monoidal structure $\mu_{\wsboxtimes}:\uwsboxtimes
  \otimes_A \uwsboxtimes \to \uwsboxtimes$ is the map associated to
  the structure of $\mathcal O_A$--algebra in $\skyk$. Finally, the map
  $\varepsilon_{\uAtimes}=u_{\uwsboxtimes}:\mathcal O_A \to \skyk$ is
  defined by the multiplication of an element of $\mathcal O_A$ by the
  unit element of $\skyk$.

  The proofs of the  associativity, unitality  and counitality
  of $\zeta$ are easy exercises and therefore are
  omitted.

The fact that $(Q\sAalg,\otimes_A,\uAotimes,\wsboxtimes,
\uwsboxtimes)$ is also a duoidal category follows easily from 
Proposition \ref{prop:fromaguiar}.
  \end{proof}

The definitions of the duoidal structures on the categories 
$\Schasqc\subset \Schaqc$  and  $Q\sAalg\subset Q\sAmod$ are 
tailored to give the following result.

\begin{thm}\label{thm:mainspec}  Let $A$ be an abelian variety. Then:

  \noindent (1) The functor
  $
    \PP:(\Schasqc, \wtimes, \uwtimes, \times_A, \uAtimes) \to
    (Q\sAalg^{\op},\wsboxtimes, \uwsboxtimes,\otimes_A,\uAotimes)
  $
  is 
  bilax monoidal. More
  precisely, $\PP: (\Schasqc, \wtimes, \uwtimes) \to
    (Q\sAalg^{\op},\wsboxtimes, \uwsboxtimes)$  is strong monoidal, $\PP: (\Schasqc,  \times_A, \uAtimes) \to
    (Q\sAalg^{\op},\otimes_A,\uAotimes)$ is  colax monoidal, and the conditions of
  interchange and unitality in the definition
  \ref{defn:bilaxandtransf} are satisfied.

\noindent (2) The functor $\Spec: (Q\sAmod^{\op},\wsboxtimes,
\uwsboxtimes,\otimes_A,\uAotimes)\to  (\Schasqc, \wtimes, \uwtimes,
\times_A, \uAtimes)$  is bilax monoidal --- in this case, $\Spec: (Q\sAalg^{\op},\wsboxtimes,
\uwsboxtimes)\to  (\Schasqc, \wtimes, \uwtimes)$ is  lax monoidal and
$\Spec: (Q\sAalg^{\op},\otimes_A,\uAotimes)\to  (\Schasqc,
\times_A, \uAtimes)$ is  strong monoidal.

\noindent (3)  The adjunction $   {\xymatrixcolsep{3em}\xymatrix{ \Schsqc \ar@<5pt>[r]^-\PP
     \ar@<-5pt>@{<-}[r]_-\Spec
     \ar@{}[r]|-\bot&(Q\sSalg)^{\operatorname{op}}}}$
 $\PP \dashv \Spec$ is bimonoidal --- that is the unit and counit with respect to both structures are (lax) monoidal morphisms.
\end{thm}

\pf Let ${x:X \to A}$, ${y:Y \to A}\in \Schasqc$ and $x \times y:X
\times Y \to A\times A$. In this situation and in accordance with
Lemma \ref{lem:bran}, there is an
isomorphism $\Gamma_{\mathcal O_X,\mathcal O_Y}^{-1}: (x\times
y)_*(\mathcal O_{X \times Y})=(x\times y)_*(\mathcal O_X \sboxtimes
\mathcal O_Y) \to x_*(\mathcal O_X) \sboxtimes y_*(\mathcal O_Y)=\PP(x)\sboxtimes \PP(y)$ ---
the existence of such an isomorphism is due to the fact that $\mathcal
O_X, \mathcal O_Y$ are quasi-coherent, flat sheaves over $\Spec(\Bbbk)$  and
 $x,y$ are separated, quasi-compact morphisms (see Remark
\ref{rem:monw}).
 Then pushing forward the
above isomorphism by $s:A \times A \to A$ we have
$
   s_*(\Gamma_{\mathcal
  O_X,\mathcal O_Y}^{-1}):=s_*(x\times
y)_*(\mathcal O_{X \times Y})=(x\wtimes
y)_*(\mathcal O_{X \times Y}) \longrightarrow s_*\bigl(\PP(x) {\sboxtimes}
   \PP(y)\bigr)$.

 In other words we have a natural isomorphism:
 \[
   s_*(\Gamma_{\mathcal O_X,\mathcal O_Y}^{-1}):\PP(x \wtimes y)
   \longrightarrow \PP(x) \wsboxtimes \PP(y).
   \]
 
 This guarantees that $\PP$ is strong (and hence lax) monoidal in the
 category $Q\sAalg^{\op}$ with respect to $\wtimes$ and $\wsboxtimes$.
 For the second assertion take $p: X \times_A Y \to X$ that induces a
 morphism $p^\sharp:\mathcal O_X \to p_*(\mathcal O_{X \times_A
   Y})$. Then, $x_*(p^\sharp):x_*(\mathcal O_X)=\PP(x) \to
 x_*p_*(\mathcal O_{X \times_A Y})= (x \times_A y)_*(\mathcal O_{X
   \times_A Y})=\PP(x \times_A y)$. This map together with the map
 corresponding to $y$ yields a morphism
 \[\PP(x) \otimes_A \PP(y) \to \PP(x \times _A y) \in Q\sAalg,\]
 the required condition of colax in the category $Q\sAalg^{\op}$ with
 respect to $(\times_A, \otimes_A)$.  The proofs of the
 conditions in Definition \ref{defn:bilaxandtransf}[(1)(a),(b)] for $\PP$ is direct and left to the reader.

 It is well known that the functor $\Spec$ is strong monoidal with
 respect to $\otimes_A$ and $\times_A$ (see
 \cite{kn:EGAII}[Prop. 1.4.6]). The lax monoidality with respect to
 $\wsboxtimes$ and $\wtimes$ follows by ``doctrinal adjunction'' from
 the fact that it is the right adjoint of the colax monoidal functor
 $\PP$ (See Remark \ref{rem:next}).

 Assertion (3) is clearly a consequence of (1) and (2). \qed

 \begin{rem} \label{rem:next} 
  \noindent (1) The version of ``doctrinal adjunction'' that we are using
       can be stated as: let $L \dashv R: \mathcal C
       \stackrel{L}{\longrightarrow}\mathcal
       D\stackrel{R}{\longleftarrow}\mathcal C$ be an adjunction with
       $\mathcal C$ and $\mathcal D$ monoidal categories. Then $L:
       \mathcal C \to \mathcal D$ can be endowed with a structure of
       colax monoidal functor if and only if $R: \mathcal D \to
       \mathcal C$ can be endowed with a structure of lax monoidal
       functor that makes $(L,R)$ a monoidal adjunction. In any case
       each structure can be uniquely determined from the other. See
       \cite{kn:kelly} for the basic results on this subject and for
       example \cite{kn:aguiarspecies}[Prop. 3.84] for a direct proof.

       \noindent (2) 
           In explicit terms, if $\mathcal F,\mathcal F'\in Q\sAalg$, then
the lax monoidality of the functor $\Spec$ (see Theorem
\ref{thm:mainspec}) is implemented \emph{via}    the natural
transformation
\[
 \eta_{_{\Spec(\mathcal F)\wtimes
      \Spec(\mathcal F')}}: \Spec(\mathcal F)\wtimes \Spec(\mathcal F')
    \to \operatorname{Aff}_{_A} \bigl(\Spec(\mathcal F)\wtimes
    \Spec(\mathcal F')\bigr)=\Spec\bigl(\mathcal F\wsboxtimes \mathcal
    F').
  \]

 See also Proposition \ref{prop:catextcar}.
   \end{rem}

 \subsection{Bimonoid sheaves  and schemes of monoids over $A$}
\label{sect:hopfsheaves1}\ %

Once we have established in Theorem \ref{thm:mainspec} the adjunction
between $\Schasqc$ and $Q\sAmod^{\op}$, the notion of \emph{sheaf of
  bimonoids} as a bimonoid in the duoidal category of separable,
quasi-coherent $\sAalg$, will appear as the natural counterpart of the
notion of affine --- or quasi-compact separable --- bimonoid extension of
$A$. In this section we set the basic result for bimonoids, namely
that $\PP$ and $\Spec$ establish an adjunction between the categories
of quasi-compact separable bimonoids over $A$ (i.e.~bimonoids in  $\Schasqc$) and
bimonoids in $Q\sAalg$. The additional structure of the inversion
morphism for a ``group extension'', can also be added in a compatible
fashion in order to extend the result to group extensions.

We begin by displaying in explicit terms, the definition of bimonoid
in the duoidal category $Q\sAmod^{\op}$ (see Theorem
\ref{thm:mainspec}) --- we use the notations of Proposition
\ref{prop:zetasheaf}.

 \begin{defn}
\label{defn:bialgebrasheaf}
A \emph{sheaf of bimonoids} or a \emph{bimonoid sheaf} on $A$ is a
sheaf $\mathcal B \in Q\sAmod$ equipped with four sheaf morphisms
$\Delta_{\mathcal B}: \mathcal B \to \mathcal B \wsboxtimes \mathcal
B$, $\mu_{\mathcal B}: \mathcal B \otimes_A \mathcal B \to \mathcal
B$, $\varepsilon_{\mathcal B}: \mathcal B \to \uwsboxtimes$,
$u_{\mathcal B}:\uAotimes \to \mathcal B$ that make commutative the
diagrams below:
\[
  \xymatrix{\mathcal B\ar[rr]^{\Delta_{\mathcal
      B}}\ar[d]_{\Delta_{\mathcal B}}&& \mathcal B\wsboxtimes \mathcal
  B\ar[d]^{\Delta_{\mathcal B} \wsboxtimes \id}\\
  \mathcal
  B\wsboxtimes \mathcal B\ar[rr]^-{\id \wsboxtimes \Delta_{\mathcal B}}
  && \mathcal B\wsboxtimes \mathcal B \wsboxtimes \mathcal B
} \quad
\xymatrix{
  \mathcal B\ar@{<-}[rr]^{\mu_{\mathcal B}}\ar@{<-}[d]_{\mu_{\mathcal
      B}}&& \mathcal B\otimes_A
  \mathcal B\ar@{<-}[d]^{\mu_{\mathcal B}\otimes_A \id}\\
  \mathcal B\otimes_A \mathcal
  B\ar@{<-}[rr]^-{\id \otimes_A \mu_{\mathcal B}} && \mathcal B\otimes_A \mathcal B
  \otimes_A \mathcal B}
\] 
\[
  \xymatrix{
    &\mathcal B\ar[d]|-{\Delta_{\mathcal B}}\ar@{<->}[rd]^{\cong}\ar@{<->}[ld]_{\cong}&
    \\
    \mathcal B\wsboxtimes \uwsboxtimes&\mathcal B\wsboxtimes
    \mathcal B \ar[l]^{\id \otimes \varepsilon_{\mathcal
        B}}\ar[r]_-{\varepsilon_{\mathcal B} \otimes \id}&\uwsboxtimes
    \wsboxtimes \mathcal B
  } \quad \xymatrix{
    &\mathcal
  B\ar@{<-}[d]|-{\mu_{\mathcal
      B}}\ar@{<->}[rd]^{\cong}\ar@{<->}[ld]_{\cong}& \\
  \mathcal B \otimes_A \uAotimes&\mathcal B\otimes_A \mathcal B
  \ar@{<-}[l]^-{\id 
    \otimes u_{\mathcal B}}\ar@{<-}[r]_-{u_{\mathcal B} \otimes
    \id}&\uAotimes\otimes\mathcal B
}\]
\[
  \xymatrix{&\mathcal B\ar[dr]^{\Delta_{\mathcal B}}&\\
    \mathcal B\otimes_A \mathcal B\ar[ru]^{\mu_{\mathcal B}}\ar[d]_{\Delta_{\mathcal B} \otimes_A \Delta_{\mathcal B}}&&\mathcal B \wsboxtimes \mathcal B\\
  (\mathcal B\wsboxtimes \mathcal B)\otimes_A (\mathcal B\wsboxtimes
  \mathcal B)\ar[rr]^-{\zeta_{\mathcal B,\mathcal B,\mathcal B,
      \mathcal B}}&& (\mathcal B \otimes_A \mathcal B)\wsboxtimes
  (\mathcal B\otimes_A \mathcal B)\ar[u]_{\mu_{\mathcal B}\wsboxtimes
    \mu_{\mathcal B}}
}\]

\end{defn}

 Putting together the above results (see Proposition \ref{prop:catextcar}, Lemma
 \ref{lem:functbimonoid}, Remark
 \ref{rem:affingen} and  Theorem \ref{thm:mainspec}) we obtain the
 following consequence.

 \begin{prop}\label{prop:bimontobimon}
   The functor $\PP:\Schasqc \to Q\sAalg^{\op}$, takes bimonoids in
    $(\Schasqc, \wtimes, \uwtimes, 
  \times_A, \uAtimes)$ into bimonoids in $(Q\sAmod^{\op},\wsboxtimes,
  \uwsboxtimes,\otimes_A,\uAotimes)$.
  Similarly, the functor $\Spec$
  takes bimonoids in  $(Q\sAmod^{\op},\wsboxtimes,
  \uwsboxtimes,\otimes_A,\uAotimes)$ into affine bimonoids in  the
  category of $\Schasqc$.

Moreover, the adjunction  $   {\xymatrixcolsep{3em}\xymatrix{ \Schsqc \ar@<5pt>[r]^-\PP
     \ar@<-5pt>@{<-}[r]_-\Spec \ar@{}[r]|-\bot&(QS-\text{alg})^{\op}}}$
restricts to an adjunction as below:
\[
\xymatrix@=8pt{\operatorname{Bimon}(\Schasqc)\ar@/^1.5pc/[rr]^-{\PP}
  &\,\,\,
  \perp&\operatorname{Bimon}\bigl((Q\sAalg\bigr)^{\operatorname{op}})\ar@/^1.5pc/[ll]^{\Spec}}.
\]
    
In particular, the relative affinization over $A$ (see Definition
\ref{defn:relativeaffi}) takes bimonoids in $\Schasqc$ into bimonoids that are
affine schemes over
$A$.\qed
\end{prop}

 \begin{rem}
   \label{rem:productfor spec}
In view of Remark \ref{rem:next}, if $(\mathcal B, \Delta_{\mathcal
  B}, \mu_{\mathcal B}, \varepsilon_{\mathcal B}, u_{\mathcal B})$ is
a bimonoid in $Q\sAalg$, then the product  $m :  \Spec(\mathcal
B)\wtimes \Spec(\mathcal B)\to \Spec(\mathcal B)$ is
obtained as
\[
{\xymatrixcolsep{3em}  \xymatrix{
    \Spec(\mathcal B)\wtimes \Spec(\mathcal B)
    \ar[rr]^-{\eta_{\Spec(\mathcal B)\wtimes \Spec(\mathcal
        B)}}\ar[rrd]_m& &
    \operatorname{Aff}_A \bigl(\Spec(\mathcal B)\wtimes \Spec(\mathcal
    B)\bigr)= \Spec (\mathcal B\sboxtimes \mathcal B)\ar[d]^{\Spec(\Delta_B)}\\
 &&  \Spec(\mathcal B)
}}\]
\end{rem}
 
As an immediate consequence of Proposition \ref{prop:bimontobimon}, we
have the following.

 \begin{thm}\label{thm:grothequivalence} Let $A$ be an abelian
   variety.   Then the functors $\PP$ and $\Spec$ establish a contravariant
   isomorphism between $\MmoraffA$ and  the category of sheaves of
   bimonoids on $A$ (see Notation \ref{nota:MMor} and
   Definition \ref{defn:bialgebrasheaf}) with arrows the sheaf
   morphisms of bimonoids.
 \end{thm}

 \pf
 By Proposition \ref{prop:bimontobimon}, we have that if $q_M:M\to A$
 is a morphism of monoid schemes --- that is, a bimonoid in $\Schasqc$,
 see Proposition \ref{prop:catextcar} ---, then $\PP(q_M)$ is a sheaf of
 bimonoids, and if  $f:(q_M:M \to A) \to (q_N:N \to A)$ is a  morphism
 of bimonoids, then $\PP(f): \PP(q_M)\to \PP(q_N)$ is a morphism of
 sheaves of bimonoids. Conversely, if $\mathcal H$ is a sheaf of
 bimonoids then $\Spec{H}\to A$ is an affine morphism of monoid
 schemes, and if $f:\mathcal H\to \mathcal H'$ is a morphism of
 sheaves of bimonoids, then $\Spec(f):\Spec(\mathcal H)\to
 \Spec(\mathcal H')$ is a morphism of bimonoids in $\Schaqc$. Notice
 that $\Spec(f)$ is in particular a morphism of affine schemes over $A$.

 On the other hand, since the objects of $\MmoraffA$ are affine
 morphisms $q_M:M\to A$, clearly $\MmoraffA$ is isomorphic to a
 subcategory of $\Schaaff$, that we also denote $\MmoraffA$. Since
 $\PP$ and $\Spec$ induce a contravariant isomorphism between  
 $\Schaaff$ and $Q\sAalg$, in order to prove that
 $\PP|_{_{\MmoraffA}}:\MmoraffA\to
 \operatorname{Bimon}(Q\sAalg\bigr)^{\op}$ we can  use an elementary result on reflections of monoidal categories, that  we added below for lack of an adequate reference (see Remark \ref{rem:uniqueness}).
 \qed

 \begin{rem}\label{rem:affbutqc}
The reader should be aware that, as we pointed out in the beginning of
Section \ref{subsect:affextschoverA},  in order to define the
subcategory 
$\MmoraffA\subset \Schaaff$ (see Theorem \ref{thm:grothequivalence}),
we need to 
work in the category $\Schaqc$, since 
$s\smallcirc (q_M,q_M):M\times M\to A$ is \emph{not} an affine scheme over $A$.
\end{rem}

   \begin{rem}\label{rem:uniqueness}
     \noindent (1)  Assume that $\mathcal C_0 \subseteq \mathcal C$ is a pair of categories with the following additional conditions:

\noindent (a) $\mathcal C_0$ is a full and replete subcategory of $\mathcal C$;

     \noindent (b) $\mathcal C$ is monoidal with structure $(\times, \mathbb I)$ and $\mathcal C_0$ is monoidal with structure $(\times_0,\mathbb I_0)$;

     \noindent (c) The inclusion functor $\operatorname{inc}: \mathcal
     C_0 \to \mathcal C$ has a left adjoint $\mathcal A: \mathcal C
     \to \mathcal C_0$ with counit $\varepsilon: \mathcal A\smallcirc
     \operatorname{inc} \Rightarrow \id_{\mathcal C_0}$ and unit
     $\eta: \id_{\mathcal C}\Rightarrow  \operatorname{inc}\smallcirc \mathcal A$
     such that: $\varepsilon$ is an isomorphism and $\eta$ is strong
     monoidal, i.e. $\mathcal Ax \times_0 \mathcal Ay \cong \mathcal
     A(x \times y)$ for all $x,y \in \mathcal C$ and $\mathbb I_0
     \cong \mathcal A(\mathbb I)$.

     Then, in the above situation $\operatorname{Mon}(\mathcal C) \cap \mathcal C_0 \cong \operatorname{Mon}(\mathcal C_0)$.

     The proof of the assertion  above is easy: for $(x,m) \in \operatorname{Mon}(\mathcal C)
     \cap \mathcal C_0$, we consider the counit $\eta_{x \times x}: x
     \times x \to \mathcal A(x \times x)=\mathcal A(x) \times_0
     \mathcal A(x) = x \times_0 x$. Then, the structure morphism $m: x
     \times x \to x$ can be uniquely extended as in the diagram below:
     \[\xymatrix{x \times x \ar[rr]^{\eta_{x \times x}}\ar[rd]_m&&x \otimes_0 x\ar[dl]^{\widehat{m}}\\&x& }\]
     
     It is clear that given a monoid structure in $\mathcal C_0$ such as $\widehat{m}$ the monoid in $\mathcal C$ is obtained by composition with the unit.

    \noindent (2) Thus, we complete the proof of  Theorem
     \ref{thm:grothequivalence}  by considering in (1) above the
     categories  $\mathcal C= \Schasqc$ and $\mathcal
     C_0=\MmoraffA$ (or $\Schaff$) and $\mathcal A$ the affinization
     functor, we complete the.

     \noindent (3) In particular, if $\mathcal B$ is a
     sheaf of bimonoids in $A$ with coproduct $\Delta_{\mathcal B}$,
     then the product in $\Spec\bigl(\PP(m)\bigr)$ is obtained as 
     \[m= \operatorname{Aff}_A(m)=
        \Spec\bigl(\PP(m)\bigr)\smallcirc \eta_{M\wtimes M}= \Spec(\Delta_{\mathcal
          B_M})\smallcirc \eta_{\Spec(\mathcal B_M) \wtimes \Spec(\mathcal
          B_M)}.
      \]
     \end{rem}

     \subsection{Affine extensions of abelian varieties and Hopf
   sheaves}\ %
 \label{subsect:affexthopfsheaves}
 
 To finish our considerations on this topic, we define --- given an
 abelian variety $A$ --- the concept of Hopf sheaf on $A$ and show the
 category of commutative Hopf sheaves and its morphisms is op--equivalent with
 the category $\GmoraffA$ of affine morphisms of group schemes (see
 Definition \ref{defn:Gmor}).
 
Recall that if we call $\op:A \to A$ the map given by the inverse morphism in
 $A$, the antipode of $x:X \to A$ in the duoidal category $\Schasqc$
 is a morphism $\iota_x:x \to \op_*(x)$ that fits in the commutative
 diagrams \eqref{eqn:firstantipode}, \eqref{eqn:secondantipode} (see
 Theorem \ref{thm:antipoduoidal} and Remark \ref{rem:antipodedefi}). The
 situation is analogue in $Q\sAalg$.
 
 \begin{nota}
   Let $\op:A\to A$ be the morphism given by the inversion map in the
   abelian variety $A$ and consider  the push-forward functor $\op_*:Q\sAalg \to
   Q\sAalg$. We
   denote $\op_*(\mathcal F)=-\mathcal F$ and similarly for an arrow
   $F:\mathcal F \to \mathcal G$ we denote $\op_*(F:\mathcal F \to
   \mathcal G)=(-F: -\mathcal F \to -\mathcal G)$.

Notice that since the inversion map is an involution, then $-(-\mathcal F)=\mathcal F$.
 \end{nota}
 
 \begin{rem}\label{rem:adjst0} In order to fix notation, we recall the following easy
   properties of the functor $\op$:

   \noindent(1) $\op_*=\op^*:Q\sAalg \to Q\sAalg$;

   \noindent(2) The diagrams below are commutative:
   \[\xymatrix{\Schasqc
       \ar[r]^{\PP}\ar[d]_{\op_*}&Q\sAalg\ar[d]^{\op_*}\\
       \Schasqc \ar[r]_{\PP}&\sAalg}\quad\quad
     \xymatrix{\Schasqc
       \ar[d]_{\op_*}&Q\sAalg\ar[d]^{\op_*}\ar[l]_-{\Spec}\\
       \Schasqc &\ar[l]^-{\Spec}\sAalg.}\]
 
 \noindent(3) In the situation above we consider the  morphisms
 $A \stackrel{\st}{\longrightarrow} \Speck \stackrel{0}{\longrightarrow}A$ (see Definition \ref{defn:op-0-om}) and the associated
 adjunctions $   {\xymatrixcolsep{3em}\xymatrix{ \Bbbk-\operatorname{alg} \ar@<5pt>[r]^-{\st^*}
     \ar@<-5pt>@{<-}[r]_-{\st_*} \ar@{}[r]|-\bot&(Q\sAalg)^{\operatorname{op}}}}$
 and $   {\xymatrixcolsep{3em}\xymatrix{ (Q\sAalg)^{\operatorname{op}} \ar@<5pt>[r]^-{0^*}
     \ar@<-5pt>@{<-}[r]_-{0_*} \ar@{}[r]|-\bot& \Bbbk-\operatorname{alg}}}$.
 \end{rem}

 The adjunctions defined in Remark \ref{rem:adjst0} have the following
 properties,  analogous to the
 situation in Lemma \ref{lema:smallproperties} and Proposition
 \ref{prop:smallproperties}.

 \begin{rem} \label{rem:antipodesheaf}
   \noindent(1) Consider the following pull back diagram and the
   corresponding diagram of functors:
   \[\xymatrix{A \ar[rr]^{\delta}\ar[d]_{\st}&& A \times A
       \ar[d]^{s(\id \times \op)} \\
       \Speck\ar[rr]^(.5){0}&&
       A
     }\quad\xymatrix{
       \sAmod \ar[rr]^{0^*}\ar[d]_{(s(\id \times
       \op))^*}&& \Speck-{\operatorname{mod}}\ar[d]^{\st^*} \\
     A
     \times A-{\operatorname{mod}}\ar[rr]^(.5){\delta^*}&&
     \sAmod\,.}\]

 Evaluating at $\mathcal F \wsboxtimes -\mathcal
   G:=\bigl(s(\id \times \op)\bigr)_*(\mathcal F \sboxtimes \mathcal G)$ we
   obtain a natural transformation in $\mathcal F,\mathcal G$:
   $\st^*0^*(\mathcal F \wsboxtimes -\mathcal G)=\delta^*\bigl(s(\id \times
   \op)\bigr)^*\bigl(s(\id \times \op)\bigr)_*(\mathcal F \sboxtimes \mathcal G)\to
   \delta^*(\mathcal F \sboxtimes \mathcal G) \cong \mathcal F
   \otimes_A \mathcal G$, the penultimate arrow coming from the unit
   of the corresponding adjunction and the last equality follows from
   general properties of the external tensor product (see Definition
   \ref{defn:otherkind}). Indeed, it is well known that in the case of
   a morphism $f: X \to Y$ and a pair of sheaves $\mathcal F, \mathcal
   G \in QY\mathrm{-mod}$ we have: $f^*(\mathcal F \otimes_Y
   \mathcal G)\cong f^*(\mathcal F)\otimes_X f^*(\mathcal G)$ (see
   \cite[Theorem 16.3.7]{kn:raising}). In the case that we are dealing
   with the situation of $\delta:A \to A \otimes A$ and $\mathcal
   F,\mathcal G \in QA-{\operatorname{mod}}$, $\delta^*(\mathcal F
   \wsboxtimes\mathcal G)=\delta^*(p_1^*\mathcal F \otimes_{A \times
     A} p_2^*\mathcal G)=\delta^*p_1^*\mathcal F \otimes_{A}
   \delta^*p_2^*\mathcal G=\mathcal F \otimes_{A} \mathcal G$.

   \noindent(2) For $R \in \Bbbk\mathrm{-alg}$ we have that:
   $\st^* R = 0_* R \wsboxtimes\uAotimes$.   
 \end{rem}
 
 \begin{prop}\label{prop:sheafsmallproperties} Assume that
   $\mathcal F, \mathcal G$ are sheaves in $Q\sAalg$ and recall the
   notation $\op_*\mathcal G=-\mathcal G$. Then we can define two natural transformations as below:

   \noindent(1) $\widetilde{\gamma}_{\mathcal F,\mathcal G}: (\mathcal F
   \wsboxtimes\mathcal G) \wsboxtimes{\mathbb I_{\otimes_A}}\to
   \mathcal F \otimes_A -\mathcal G$;

   \noindent(2) $\overline{\gamma}_{\mathcal F, \mathcal G}: \uAotimes
   \wsboxtimes \mathcal F \wsboxtimes\mathcal G \to -\mathcal F
   \otimes_A\, \mathcal G$.

 \end{prop}
 \proof We sketch the proof of (1), the proof of (2) being similar.
 Using the first result of the Remark \ref{rem:antipodesheaf} we
 deduce the existence of a natural transformation $\st^*0^*(\mathcal F
 \wsboxtimes -\mathcal G)\to \mathcal F \otimes_A \mathcal G$ --- we
 are using that $-(-\mathcal G)=\mathcal G$. Using 
 now the second result of the mentioned remark we transform the above
 to: $0_*0^*(\mathcal F \wsboxtimes -\mathcal G)\wsboxtimes{\mathbb
   I}_{\otimes_A} \to \mathcal F \otimes_A \mathcal G$, and then using
 the adjunction $0^* \dashv 0_*$ we obtain a natural transformation
 $\gamma_{\mathcal F,\mathcal G}: (\mathcal F \wsboxtimes -\mathcal G)
 \wsboxtimes{\mathbb I_{\otimes_A}}\to \mathcal F \otimes_A
 \mathcal G$. \qed

 We are ready to define {\em Hopf sheaf on the abelian variety $A$}. We will use the nomenclature summarized in Definition \ref{defn:bialgebrasheaf}. 

 \begin{defn} \label{defn:hsh}
   Assume that $\mathcal H$ is a sheaf of bimonoids on $A$ (see Definition \ref{defn:bialgebrasheaf}). We say that $\mathcal H$ is a {\em Hopf sheaf} if there is a sheaf homomorphism $
   \sigma_\mathcal H: -\mathcal H \to \mathcal H$ --- called {\em
     the antipode} --- such the diagrams below are commutative.
   \begin{equation}
     \label{eqn:sfirstantipode}
    \raisebox{10ex}{\xymatrixcolsep{1.5em}\xymatrix{ & \mathcal
        H\otimes_A \mathcal H\ar@{<-}[rr]^-{\id \otimes_A \sigma_{\mathcal H}}& &\mathcal H\otimes_A -\mathcal H
        \ar@{<-}[rr]^-{\widetilde{\gamma}_{\mathcal H,\mathcal
            H}}&&(\mathcal H \wsboxtimes \mathcal H)\wsboxtimes\,
        \uAotimes\ar@{<-}[dr]^{\Delta_{\mathcal H}\wsboxtimes
          \id}&\\ \mathcal H\ar@{<-}[dr]_{u_\mathcal
          H}\ar@{<-}[ru]^-{\mu_\mathcal H}&&&&&& \mathcal
        H\wsboxtimes\uAotimes \\ &\mathbb
        I_{\otimes_A}\ar@{-}[rrrr]^-{\cong}&&&& \uwsboxtimes
        \wsboxtimes\uAotimes \ar@{<-}[ur]_{\varepsilon_{\mathcal H}
          \wsboxtimes \id} &}}
    \end{equation}                               

    \begin{equation}\label{eqn:ssecondantipode}
      \raisebox{10ex}{\xymatrixcolsep{1.5em}\xymatrix{ & \mathcal
          H\otimes_A \mathcal H\ar@{<-}[rr]^-{\sigma_\mathcal
            H\otimes_A \id}&& -\mathcal H\otimes_A \mathcal H
          \ar@{<-}[rr]^-{\overline{\gamma}_{\mathcal H,\mathcal
              H}}&&\uAotimes \wsboxtimes(\mathcal H \wsboxtimes\,
          \mathcal H) \ar@{<-}[dr]^{\id \wsboxtimes \Delta_\mathcal
            H}&\\ \mathcal H\ar@{<-}[dr]_{u_\mathcal
            H}\ar@{<-}[ru]^-{\mu_\mathcal H}&&&& && \uAotimes
          \wsboxtimes\,\, \mathcal H \\ &\mathbb
          I_{\otimes_A}\ar@{-}[rrrr]^-{\cong}&&&& \mathbb
          \uAotimes \wsboxtimes\,\, \uwsboxtimes \ar@{<-}[ur]_{\id
            \wsboxtimes\varepsilon_\mathcal H} &}}
    \end{equation}
   where $\widetilde{\gamma}$ and $\overline{\gamma}$ are the natural
   transformations defined in Proposition \ref{prop:sheafsmallproperties} and the
   bottom maps $\cong$ are the natural identifications associated to
   the unit of the $\wsboxtimes$ monoidal structure.

   A Hopf sheaf $\mathcal H$ is \emph{commutative} if $(\mathcal H,\mu_{\mathcal
      H},u_{\mathcal H})$ is a sheaf of 
    commutative $\mathcal O_A$--algebras, and a \emph{flat Hopf sheaf}
    is a Hopf sheaf that is  flat as sheaf of $\mathcal O_A$--modules
    --- that is, the stalks $\mathcal H_a$ are $\mathcal
    O_{a,A}$--flat modules for all $a\in A$. A Hopf sheaf $\mathcal H$
    is \emph{faithful} if the canonical morphism $\mathcal O_A\to 
    \mathcal H$ is injective --- in other words, $\mathcal H(U)$ is a
    faithful representation of $\mathcal
    O_A(U)$. 
   \end{defn}

As a summary we write down explicitly the conditions of a {\em Hopf sheaf} on an abelian variety $A$.

\begin{summary}
    \label{def:Hopfsheaf1}\label{sum:Hopfsheaf}
    
    Let $A$ be an abelian variety. A \emph{commutative Hopf sheaf} on $A$ is a
    sextuple $(\mathcal H,\Delta_{\mathcal   
      H},\varepsilon_{\mathcal H},\mu_{\mathcal H},u_{\mathcal
      H},\sigma_{\mathcal H})$, where $(\mathcal H,\mu_{\mathcal
      H},u_{\mathcal H})$ is a  sheaf of quasi-coherent commutative
    $\mathcal O_A$--algebras (i.e.~$\mathcal H \in
    Q\sAalg$) with multiplication $\mu_{\mathcal H}$ unit
    $u_{\mathcal H}$, and $\Delta_{\mathcal H } : \mathcal H \to
    \mathcal H \wsboxtimes \mathcal H$, $\varepsilon_{\mathcal H }:
    \mathcal H \to  \operatorname{skysc}_0(\Bbbk)$, $\sigma_{\mathcal H}: -\mathcal H \to \mathcal H$ are morphisms of
    sheaves satisfying the following additional conditions:

\noindent (1) The triple $(\mathcal H,\Delta_\mathcal H,\varepsilon_{\mathcal
  H})$ is a comonoid in
  $\bigl(Q\sAalg,\wsboxtimes,\uwsboxtimes=\operatorname{skysc}_0(\Bbbk)\bigr)$;

\noindent (2)  $\Delta_{\mathcal H}:\mathcal H \to \mathcal H \wsboxtimes
  \mathcal H$ and $\varepsilon_{\mathcal H }: \mathcal H \to
  \operatorname{skysc}_0(\Bbbk)$ are morphisms of $Q\sAalg$, that is:

  \begin{enumerate}
  \item[(a)] The morphism $\Delta_{\mathcal
    H}$ is such that the following diagrams are commutative:
  \[
{\xymatrixcolsep{1.2pc}    \xymatrix{&\mathcal H\ar[dr]^{\Delta_{\mathcal H}}&\\
    \mathcal H\otimes_A \mathcal H\ar[ru]^{\mu_{\mathcal
        H}}\ar[d]_{\Delta_{\mathcal H} \otimes_A \Delta_{\mathcal
        H}}&&\mathcal H \wsboxtimes \mathcal H\\ 
    (\mathcal H\wsboxtimes \mathcal H)\otimes_A (\mathcal H\wsboxtimes
    \mathcal H)\ar[rr]^-{\zeta_{\mathcal H,\mathcal H,\mathcal H,
        \mathcal H}}&&(\mathcal H \otimes_A \mathcal H)\wsboxtimes
    (\mathcal H\otimes_A \mathcal H)\ar[u]_{\mu_{\mathcal
        H}\wsboxtimes \mu_{\mathcal H}}} 
    \xymatrix@R=2.3cm@C=1.5cm{\uAotimes\ar[r]^-{\Delta_{\otimes_A}}\ar[d]_{u_{\mathcal
          H}}&\uAotimes\!\wsboxtimes\uAotimes\ar[d]^{u_{\mathcal
          H}\wsboxtimes u_{\mathcal H}}\\
      \mathcal H
      \ar[r]^-{\Delta_{\mathcal H}}&\mathcal H \wsboxtimes\mathcal
      H}}
  \] 

\item[(b)]  The morphism  $\varepsilon_{\mathcal H}$ is such that the
  following diagrams are commutative:
  \[
    \xymatrix{\mathcal H \otimes_A \mathcal
      H \ar[d]_{\varepsilon_{\mathcal H}}
      \ar[r]^-{\mu_{\mathcal H}}&\mathcal H
      \ar[d]^{\varepsilon_{\mathcal H}}\\
      \uwsboxtimes \otimes_A \uwsboxtimes \ar[r]_-{\mu_{\mathbb
          I_{\wsboxtimes}}}&\uwsboxtimes} 
    \xymatrix{\uAotimes\ar[r]^{\varepsilon_{\uAotimes}}
      \ar[d]_{u_{\mathcal H}}&\uwsboxtimes\ar[d]^{\id}\\
      \mathcal
      H\ar[r]_{\varepsilon_{\mathcal H}}&\uwsboxtimes.}
  \]
    \end{enumerate}
    
    \noindent (3) The antipode $\sigma_{\mathcal H}: -\mathcal H \to
      \mathcal H$ is a morphism in $Q\sAmod$ --- recall that
      $-\mathcal H=\op_*(\mathcal H)$ where $\op_*$ is the functor in
      $\sAmod$ given by push-forward (or pull-back) by $a \mapsto -a:A
      \stackrel{\op}{\rightarrow} A$. Moreover, the antipode map, fits
      in the commutative diagrams \eqref{eqn:sfirstantipode},
      \eqref{eqn:ssecondantipode}.

If moreover $\mathcal H$ is a flat $\mathcal O_A$--module, then we say
that the sextuple is a \emph{flat commutative Hopf sheaf}; if
$\mathcal O_A\to \mathcal H$ is an injective morphism, then the
sextuple is a \emph{faithful commutative Hopf sheaf}.

  \end{summary}

Given the abelian variety $A$ we define the category of Hopf sheaves
in the natural manner.

\begin{defn}\label{defn:caths}
If $A$ is a given abelian variety and $\mathcal H,\,\mathcal K$ are
flat commutative Hopf sheaves, a \emph{morphism from $\mathcal H$ into
  $\mathcal K$} is simply a morphism of bimonoids in the duoidal
category $(Q\sAmod,\otimes_A,\uAotimes,\wsboxtimes,
\uwsboxtimes)$. Explicitly it is a morphism of sheaves $F: \mathcal
H\to \mathcal K$ of $\mathcal O_A$-algebras, with the additional
property that the diagrams below commute:
\[
{\xymatrixcolsep{3pc}
\xymatrix{
\mathcal H\ar[r]^{F}\ar[d]_{\Delta_{\mathcal H}}& \mathcal
K \ar[d]^{\Delta_{\mathcal K}}  
\\
\mathcal  H\wsboxtimes \mathcal  H\ar[r]^-{F\wsboxtimes
  F}&\mathcal  K\wsboxtimes \mathcal
K
}}\quad\xymatrix{\mathcal H \ar[rr]^F\ar[dr]_{\varepsilon_{\mathcal H}}&& \mathcal K\ar[ld]^{\varepsilon_{\mathcal K}}\\&\uwsboxtimes &}
\]
We call $HQ\sAalg$ (resp.~$HQ\fsAalg$) the category whose objects are the
commutative Hopf sheaves (resp.~faithful commutative Hopf sheaves) on
$A$ and whose arrows are the morphisms of Hopf sheaves.
 \end{defn}

\begin{rem}\label{rem:goodantipode}
  In the context considered above, the following two assertions can be
  proved.
  
  \noindent (1) In the case that the antipode $\sigma_{\mathcal
    H}$ exists for the bimonoid $\mathcal H$, then it is unique ---
  for example, this can be proved using the equivalence given by
  Theorem \ref{thm:hopssheaf=affext} below  and the fact that the inverse
  morphism of a group scheme is unique. 

\noindent (2) If $F:\mathcal H \to \mathcal K$ is a morphism of Hopf
sheaves, then $\sigma_{\mathcal K}\smallcirc (-F)=F \smallcirc \sigma_{\mathcal H}$. In
other words, a morphism of sheaves that are Hopf sheaves and that
preserve the bimonoid structure, automatically preserves the
antipode.  The proof of this assertion is a consequence of (1).
\end{rem}

The close relationship between the affine extensions of an abelian
variety $A$ and the commutative Hopf sheaves on $A$ is expressed in the
theorem that follows.

\begin{thm}\label{thm:hopssheaf=affext}
  Let $A$ be an abelian variety, and  $\GextaffA$ and  $HQ\fsAalg$ the
  categories of  affine extensions of $A$ and  faithful commutative Hopf sheaves
   of $A$ respectively. 
Then, $\PP:\GextaffA \to (HQ\fsAalg)^{\op}$ and
$\Spec:(HQ\fsAalg)^{\op} \to \GextaffA$ constitute an adjoint
equivalence between $\GextaffA$ and $HQ\fsAalg$. 
\end{thm}
\proof If $q:G\to A$ is an affine extension, then $q$ is a surjective
morphism and therefore the sheaf $\PP(q)$  (see
Definition \ref{defn:defPP}) is a faithful sheaf of commutative
$\mathcal O_A$--algebras. On the other hand, by Theorem
\ref{thm:antipoduoidal} the inverse morphism $\iota_G:G\to G$ verifies
the commutative diagrams \eqref{eqn:firstantipode} and
\eqref{eqn:secondantipode}.  It follows by construction that
$\sigma_{\mathcal H}=\PP(\iota_G)$ satisfies commutative diagrams
\eqref{eqn:sfirstantipode} and \eqref{eqn:ssecondantipode} for
$\mathcal H=\PP(q)$. Indeed, it is easy to check that
$\PP(\widetilde{\gamma}_{q,q})=\widetilde{\gamma}_{\mathcal H,\mathcal
  H}$ and $\PP(\overline{\gamma}_{q,q})=\overline{\gamma}_{\mathcal H,\mathcal
  H}$ (see Remark \ref{rem:forantipduoidal} and Proposition
\ref{prop:sheafsmallproperties}), thus applying the functor $\PP$ to
the diagrams \eqref{eqn:firstantipode} and \eqref{eqn:secondantipode}
we obtain the diagrams \eqref{eqn:sfirstantipode} and
\eqref{eqn:ssecondantipode}.  Since $\PP$ takes affine morphisms of
monoids to sheaves of bimonoids, it follows that $\PP(q)$ is a
faithful commutative Hopf sheaf.

Conversely, if $\mathcal H\in HQ\fsAalg$, with antipode $\sigma_{\mathcal H}$,  then $\Spec \mathcal H$: $q:M\to A$  is a bimonoid in  
$\Schasqc$, with $q$ a faithful affine morphism (of monoid schemes), by   
Theorem \ref{thm:grothequivalence}. Moreover, applying $\Spec$ to the commutative diagrams  \eqref{eqn:sfirstantipode} and
\eqref{eqn:ssecondantipode}, we deduce that $\iota_q=\Spec (\sigma_{\mathcal H}): q\to -q$ satisfies the commutative diagrams
\eqref{eqn:firstantipode} and \eqref{eqn:secondantipode}. In other
words, $M$ is a group scheme and $q$ an affine extension of $A$ (see
Remark \ref{rem:forequiv}).\qed 

\begin{nota}\label{nota:defHq} The following notation will be used
  in the future. Assume that $q:G \to A$ is an affine extension, then
  $\mathcal P(q)$ the associated Hopf sheaf of $\sAalg$ will be
  denoted as $\mathcal H_q:=\mathcal P(q)$.
  \end{nota}

  Notice that Theorem \ref{thm:hopssheaf=affext}
implies in particular the following result.

\begin{cor}
Let $\mathcal H$ be a  commutative Hopf sheaf on the abelian
variety $A$. Then $\mathcal H$ is a flat sheaf if and only if $\mathcal H$ is faithful,  if and only if the unit
morphism $u_{\mathcal H}$ is monic.
\end{cor}
\proof
Indeed, since  a flat morphism of schemes $f:X\to Y$, with $Y$ N\oe therian is
dominant, it follows
from Proposition \ref{prop:affschqcs}  that a
flat commutative Hopf sheaf $\mathcal H$ is faithful. Conversely, if
$\mathcal H$ is faithful, Theorem \ref{thm:hopssheaf=affext} implies
that $\Spec(\mathcal H) $: $q:G\to A$ is an affine extension, and
therefore a flat morphism by   Theorem \ref{thm:perrinigame}.

Finally, notice that  the unit
morphism $u_{\mathcal H}$  is monic if and only if
$\mathcal O_A(U)\to \mathcal H(U)$ is an inclusion for any (affine)
open subset $U\subset A$.\qed

\begin{ejs}
   \label{ej:hopfgroups}
(1) Let $H$ be an affine group scheme, and consider the corresponding
affine  extension ${\xymatrixcolsep{1.5pc}
  \xymatrix{1\ar[r]&H\ar[r]^(0.45){\id}&H\ar[r]&\ar[r] 0&0}}$.
Then $\mathcal H$ is the Hopf algebra $\Bbbk[H]$ seen as a sheaf on
$\{*\}=\Spec(\Bbbk)$.

Conversely, given a Hopf algebra $R$, then $R$ can be seen as a Hopf
sheaf on $\{*\}=\Spec(\Bbbk)$, and  the affine group scheme
$\Spec(R)$ induces the affine  extension 
${\xymatrixcolsep{1.5pc}
  \xymatrix{1\ar[r]&\Spec(R)\ar[r]^(0.45){\id}&\Spec(R)\ar[r]&\ar[r] 0&0}}$.

\noindent (2) If $A$ is an abelian variety, then the structure sheaf $\mathcal O_A$ is a faithful
commutative Hopf sheaf on $A$; it corresponds to the trivial 
extension
$
  {\xymatrixcolsep{1.5pc}
  \xymatrix{0\ar[r]&0\ar[r]&A\ar[r]^(0.45){\id}&\ar[r] A&0}}$.

\noindent (3) More generally, if $R$ is a Hopf algebra then $\mathcal
R=R\otimes_\Bbbk\mathcal O_A$ is a flat Hopf sheaf; $\mathcal R$
corresponds to the direct product: \[\Spec(R) \times A:
{\xymatrixcolsep{1.5pc}
  \xymatrix{1\ar[r]&\Spec(R)\ar[r]&\Spec(R)\times
    A\ar[r]^-(0.45){p_2}& A\ar[r] &0}}.\]
\end{ejs}

\begin{rem}
(1) Since an  affine  extension $\mathcal S:$ $\ate$ is of finite type
if and only  if $H$ is of finite type (as
follows from descent theory, see \cite[Proposition
2.6.5]{kn:brionchev} and \cite[Prop. 2.7.1]{kn:EGAIV2}), it follows that $G$
if of finite type if and only if $\mathcal H(U)$ is a finitely
generated $\mathcal O_A(U)$--algebra for any affine open subset
$U\subset A$. 

\noindent (2)  If $\mathcal H$ is a faithful commutative Hopf
sheaf on $A$, then $G =\Spec(\mathcal H)$ is an anti-affine group scheme 
if and only if $\mathcal H(A)=\Bbbk$.
Indeed, if $q:G\to A$ is the
associated morphism of quasi-compact group schemes, then  $\mathcal
H(A)=q_*(\mathcal O_G)(A)=\mathcal O_G(G)$.
\end{rem}

\subsection{Hopf ideals and affine subextensions}\ %
\label{sect:hopfideals}

In this section we present the expected generalizations on the
 relationship between ideals of a Hopf algebra $H$ and closed
 subgroups of the affine algebraic group
 $\Spec(H)$,  to the context of Hopf sheaves and affine extensions.

We begin by recalling some  definitions and known results concerning
the correspondence 
between closed subschemes of $X$ and quasi-coherent  sheaves of ideals
on
$\mathcal O_X$ (see for example \cite[Proposition II.5.9]{kn:hartshorne}
  or \cite[\S 4]{kn:EGAI}).  

  \begin{rem} \label{rem:idealsubset}
Let $X$ be  a $\Bbbk$--scheme and 
$(i,i^\#): Y \subseteq X$, a closed subscheme. Then we have a  short
exact sequence of quasi-coherent sheaves of $\Bbbk$--algebras
\begin{equation}\label{eqn:idealofY}
    \xymatrix{0\ar[r]&\mathcal I_{X|Y}\ar[r]&\mathcal O_X\ar[r]^-{i^\#}&
      i_*(\mathcal O_Y) \ar[r]&0,}
    \end{equation}
were  $\mathcal I_{X|Y}$ is a 
sheaf of ideals in $\mathcal O_X$. In this manner we obtain a
bijective correspondence between quasi-coherent sheaves of ideals of
$\mathcal O_X$ and closed subschemes of $X$. The inverse map --- that
we call $\mathfrak V$ ---, sends an
ideal $\mathcal I \subseteq \mathcal O_X$ into the closed subscheme  of
$X$ given by $\operatorname{Supp}(\mathcal O_X/\mathcal I)$.
\end{rem}

In the case of schemes over a $\Bbbk$-scheme $S$, we can push forward
the short exact sequence \eqref{eqn:idealofY}, provided that we
impose additional conditions on $Y$ and $X$.

\begin{defn}
       Let  $(x: X \to S)\in \Schpsqc$. We define  $\mathcal C(x)$ as the
       poset of closed subschemes of $X$ in $\Schpsqc$ --- that is, we
       consider $y:Y\to S\in\Schpsqc$ with $(i,i^\#): y\to x$ a closed
       subscheme.

       If $\mathcal F\in Q\sSalg$, we define  $\mathcal{II}
       (\mathcal F)$ as the poset
       of quasi-coherent sheaves of ideals of $\mathcal F$. 
\end{defn}

     \begin{lem}\label{lem:idealsclosedsub}
       Let  $(x: X \to S)\in \Schpsqc$ and $\mathcal F \in Q\sSalg$.  Then:

  \noindent(1) If $ (y: Y \to S) \in \mathcal C(x)$, then the
  sequence in the category $Q\sSmod$: 
      \[\xymatrix{0\ar[r]&x_*\bigl(\mathcal I_{X|Y}\bigr)
          \ar[r]&\PP(x) \ar[r]^(.47){\PP(i)}& \PP(y) \ar[r]&0,}\]
      is exact.
      
      \noindent(2) The map $\mathfrak I: \mathcal{C}(x) \to
      \mathcal{II}\bigl(\PP(x)\bigr)$ given by $\mathfrak
      I(y)=x_*(\mathcal I_{X|Y})$ is a contravariant functor between
      the domain and codomain posets. 

  \noindent (3) The map  $\mathfrak V:
  \mathcal{II}(\mathcal F) \to \mathcal C\bigl(\Spec(\mathcal F)\bigr)$ given
  by $\mathfrak{V}(\mathcal I)= \Spec(\mathcal F/\mathcal I) \subset
  \Spec(\mathcal 
  F)$ is a contravariant functor between the domain and codomain
  posets --- recall that in this context $\Spec(\mathcal F)$ is a $\Bbbk$-scheme
  that is affine over $S$, with $\mathcal P\bigl({\Spec(\mathcal
    F)}\bigr)\cong 
  \mathcal F$.

  \noindent (4) If $ (y: Y \to S) \in \mathcal C(x)$, then  $y \cong
  \mathfrak V\mathfrak I (y)= \Spec\bigl(\PP(x)/x_*(\mathcal
  I_{X|Y})\bigr)$.
  If $\mathcal I\in \mathcal {II}(\mathcal F)$, then  
  $\PP(\mathfrak V(\mathcal I)) \cong \mathcal F/\mathcal I$.
\end{lem}
\proof This is an easy exercise and its proof is therefore omitted.\qed

\begin{defn}\label{defn:hopfideal}
Let $\mathcal H$ be a commutative flat Hopf sheaf on $A$. A  subsheaf
$\mathcal I\subset \mathcal H$ is a \emph{sheaf of Hopf ideals} if 
there exists a pair $(\mathcal K, F)$ where $\mathcal K$ is a Hopf
sheaf and $F: \mathcal H \to \mathcal K$ is a surjective morphism of
Hopf sheaves with $\Ker(F)=\mathcal I$ --- recall that in this case
$\mathcal K\cong \mathcal H/\mathcal I$.

We say that a sheaf of Hopf ideals $\mathcal I\subset \mathcal H$  is
\emph{faithful} if   $\mathcal K=\mathcal H/\mathcal I$ is a faithful Hopf sheaf.
\end{defn}

\begin{rem}  
 By definition, a sheaf of Hopf ideals is faithful if and
only if $\mathcal I(U)\cap \mathcal O(U)=\{0\}$ for all
open subset $U\subset A$.
\end{rem}

\begin{prop}\label{prop:equivhopfsheaf}
  If  $\mathcal H$ is a commutative flat Hopf sheaf of $A$, then a  subsheaf $\mathcal
  I\subset \mathcal H$ in $\sAmod$ is a sheaf of Hopf ideals if
  and only if the following conditions hold:

  \begin{itemize}
    \item[(i)] The subsheaf $\mathcal I \subset \mathcal H$ is a
      quasi-coherent sheaf
      of ideals;

\item[(ii)]  Let $\mathrm{inc}:\mathcal I \to \mathcal H$ be the
inclusion morphism and consider $\mathrm{inc} \wsboxtimes \id + \id
\wsboxtimes \mathrm{inc}: \mathcal I \wsboxtimes\mathcal H +
\mathcal H \wsboxtimes\mathcal I \to \mathcal H \wsboxtimes\mathcal
H$. Then the morphism $\Delta \smallcirc \mathrm{inc}: \mathcal I \to \mathcal H \wsboxtimes \mathcal H$ factors as in the diagram below:
\[
\xymatrix{\mathcal H \ar[r]^-{\Delta} & \mathcal H \wsboxtimes\,
  \mathcal H \\ \mathcal I
  \ar[u]^{\mathrm{inc}}\ar[r]^-{\Delta|_{_{\mathcal I}}}&\mathcal I
  \wsboxtimes\mathcal H + \mathcal H \wsboxtimes\mathcal
  I,\ar[u]_{\mathrm{inc} \wsboxtimes \id + \id \wsboxtimes\,
    \mathrm{inc}}}
\]

\item[(iii)] $\mathcal I\subset
\operatorname{Ker}(\varepsilon_{\mathcal H})$.

\end{itemize}
\end{prop}

\proof
Assume that $\mathcal I$ is a sheaf of Hopf ideals, i.e.~$\mathcal I=\Ker(F)$
  for some morphism of Hopf sheaves  $F:\mathcal H\to \mathcal K$. From the flatness
  hypothesis it follows that  $\mathcal I$ is an
  ideal of the sheaf of algebras  $\mathcal H$. Also from the
  flatness it follows 
  that in the monoidal abelian category $(Q\sAalg,\wsboxtimes)$,
  \begin{equation}
    \label{eq:kerFotimesF}
    \Ker(F\wsboxtimes F)= \mathcal I \wsboxtimes \mathcal H+ \mathcal
    H\wsboxtimes \mathcal I,
  \end{equation}
  and therefore (ii) is verified (since $F$
  is a morphism of Hopf sheaves). The proof of assertion (iii) is trivial.

  For the converse, assume that $\mathcal I$ satisfies conditions
  (i)--(iii) and call $F:\mathcal H \to \mathcal H/\mathcal I=\mathcal
  K$. Then $\mathcal K$ is a sheaf of algebras, with product
  $\mu_{\mathcal K}$ and unit  $u_{\mathcal
    K}$ induced by $\mu_{\mathcal H}$ and $u_{\mathcal H}$.

  It follows from the
  equality 
  \eqref{eq:kerFotimesF} that the map $(F\wsboxtimes F)\smallcirc \Delta$ factors through
  $\mathcal K$ and induces 
  a morphism of sheaves $\Delta_{\mathcal K}:\mathcal
  K\to (\mathcal H\wsboxtimes\mathcal H)/(\mathcal I
  \wsboxtimes \mathcal H + \mathcal H \wsboxtimes \mathcal I )\cong
  \mathcal K\wsboxtimes \mathcal K$.
  
From condition (iii) we deduce that $\varepsilon_{\mathcal H}: \mathcal H \to
  \operatorname{skysc}_0(\Bbbk)$ induces
  a morphism $\varepsilon_{\mathcal K}: \mathcal K \to
  \operatorname{skysc}_0(\Bbbk)$.
 On the other hand, since $-\mathcal K=(-\mathcal H)/(-\mathcal I)$,
 it is clear that $\sigma_{\mathcal H}:-\mathcal H\to \mathcal H$ induces a morphism $\sigma_{\mathcal K} :-\mathcal K\to \mathcal K$. 

Finally, it is clear that, by construction, the morphisms
$\Delta_{\mathcal K}$, $\varepsilon_{\mathcal K}$ and
$\sigma_{\mathcal K}$ satisfy the required commutative diagrams for $(\mathcal K,\Delta_{\mathcal   
      K},\varepsilon_{\mathcal K},\mu_{\mathcal K},u_{\mathcal
      K},\sigma_{\mathcal K})$ to be a commutative Hopf sheaf.
  \qed

\begin{prop}
\label{prop:hopfideals}
Let $\mathcal S:$ $\ate$ be an  affine
extension  and let  $\mathcal H=q_*\bigl(\mathcal O_G\bigr)$ be the
(faithful, commutative) Hopf sheaf  associated  to $\mathcal S$. Then
the poset of  closed
sub-extensions 
of $\mathcal S$ is op-equivalent to the poset of faithful Hopf ideals of
  $\mathcal H$.
\end{prop}
\pf
Let $\mathcal T:$ ${\xymatrixcolsep{1.5pc}
  \xymatrix{1\ar[r]&H'\ar[r]&G'\ar[r]^(.45)q&\ar[r] A&0}}$ be a closed
sub-extension of $\mathcal S$ and consider $\mathcal I_{G'}\subset
\mathcal O_G$, the subsheaf of ideals associated to $G'$. Clearly
$q_*(\mathcal I_{G'})\subset \mathcal H$ is a subsheaf of ideals.  On
the other hand, if we denote by $\operatorname{inc}: G'\to G$ the canonical inclusion, then we have a commutative diagram of sheaves of $\mathcal
O_G$--modules
\[
{\xymatrixcolsep{3.5pc}\xymatrix{
\mathcal O_G\ar[r]^-{m^\#}\ar[d]& m_*\bigl( \mathcal O_{G\times
  G}\bigr)=m_*\bigl(\mathcal O_G\sboxtimes \mathcal O_G\bigr)\ar[d]\\
\operatorname{inc}_*\bigl(\mathcal
O_{G'}\bigr)\ar[r]^-{\operatorname{inc}_*m_{G'}^\#}& \operatorname{inc}_*\bigl( (m_{G'})_*\bigl( \mathcal O_{G'\times
  G'}\bigr)\bigr)=\operatorname{inc}_*\bigl( (m_{G'})_*\bigl(\mathcal
O_{G'}\sboxtimes \mathcal O_{G'}\bigr)\bigr)
}}
\]
where the vertical arrows are the canonical projections induced by the
inclusions $G'\hookrightarrow G$ and $G'\times G'\hookrightarrow G\times
G$. Since $\mathcal I_{G'\times 
  G'}= \mathcal I_{G'} \sboxtimes \mathcal O_G+ \mathcal O_G\sboxtimes
\mathcal I_{G'}$, it follows that $\mathcal O_{G'}\sboxtimes \mathcal 
O_{G'} = (\mathcal O_{G}\sboxtimes \mathcal O_{G})/\bigl(\mathcal
I_{G'}\sboxtimes \mathcal O_{G} +\mathcal O_{G}\sboxtimes \mathcal
I_{G'}\bigr) $. Hence, $m_*(\mathcal I_{G'})\subset \mathcal
I_{G'}\sboxtimes \mathcal O_{G} +\mathcal O_{G}\sboxtimes \mathcal
I_{G'}$.

Also, it is easy to show that $\mathcal I_{G'}\subset
\operatorname{Ker}(\varepsilon_{G})$.  From the
functorial
properties of $q_*$, it follows that $q_*(\mathcal I_{G'})$ is a sheaf
of Hopf ideals and, $\mathcal T$ being a sub-extension, $q_*(\mathcal
I_{G'})$ is faithful.

Conversely, given a faithful sheaf of Hopf ideals  $\mathcal I\subset \mathcal
H$, let  
\[
\xymatrix{
\mathcal T:& 1\ar[r]& \Spec \bigl(\mathcal H/\mathcal I\bigr)_0\ar[r]&
\Spec \bigl(\mathcal H/\mathcal I\bigr)\ar[r]& A\ar[r]& 0
}
\]
be the affine extension associated to the Hopf sheaf $\mathcal H/\mathcal I$. 
If $U\subset A$ is an affine open subset, then  the canonical
projection $\mathcal H(U)\to (\mathcal H/\mathcal I)(U)$ induces a
closed immersion $\Spec \bigl((\mathcal H/\mathcal I)(U)\bigr)\to
\Spec \bigl(\mathcal H(U)\bigr)$. Therefore, $\mathcal T$ is a
closed sub-extension of $\mathcal S$.
\qed

\section{The category $\Rep(\mathcal S)$ as a category of sheaves}  
\label{sect:repsheaves}

If $G$ is an affine group scheme, it is well known that the category
of its (left) rational representations and the category of (right)
$\Bbbk[G]$--comodules are equivalent  --- in the usual notations:
$\Rep(G)\cong \Comod{\Bbbk[G]}$. One can also consider the
anti-equivalence between the category of vector spaces and the
category of symmetric algebras (given by $V\mapsto \mathcal
O_V(V)=S(V^\vee)$, the symmetric algebra generated by the dual
$V^\vee$) in order to  produce an anti-equivalence between   
$\Rep(G)$ and the category of $\Bbbk{G}$--comodule symmetric
algebras. 
On  the other hand, it is also well known that the
category of vector bundles over a scheme $T$ 
is equivalent to the category of locally free, coherent, sheaves of
$\mathcal O_T$--modules, see for example \cite[Chapter 13]{kn:raising}
(see also Proposition \ref{prop:equicomodrep0}). 

In view of the previous remarks, the objective of this section  
is two-fold:

In the light of Theorem
\ref{thm:hopssheaf=affext} (and in the nomenclature of Definition
\ref{defn:catrepaffext} and Remark \ref{rem:notationqS}), given an
affine extension $q:G\to A$ we want to establish an    equivalence
$\RepO(q) \cong \bigl( \Comod{\mathcal
  H_q}\bigr)_{\operatorname{fin}}$, the category of locally free,
  coherent sheaves that are $\mathcal H_q$--comodules, were $\mathcal
  H_q=q_*\bigl(\mathcal O_G\bigr)$   is the Hopf sheaf associated to
  $q$. Moreover, we want to extend this equivalence 
  to the graded setting --- which
involves the enlargement of the category $\Comod{\mathcal H_q}$ to an (enriched)
category with graded morphisms.

On the other hand, we also want to generalize
Mumford's equivalence between $G$--linearized line bundles and
$G$--linearized invertible sheaves to our context. Whereas the notion
of $G$--linearized sheaf is well established (see \cite[Tag
  03LE]{kn:stackproj} and \cite[page 30]{kn:GIT}), we need (again) to
develop the notions of graded morphisms between $G$--linearized
sheaves and of homogeneous sheaves, in order to construct a replacement
for $\HVBG(A)$ (see Definition \ref{defn:vbgraded}).

\subsection{The category of comodules of a Hopf  sheaf}\ %
\label{subsection:Hcomod}

In this section we consider the
usual morphisms of sheaves in $\sAmod$ that correspond with
$\RepO(q)$.  In the next Section \ref{sect:homogsheves}, we extend the
equivalence given below in Proposition \ref{prop:equicomodrep0}, to
categories with graded morphisms (see Lemma \ref{lem:vbspecgr}) in
order to obtain $\Rep(q)$.

We begin by writing  down the
definition of \emph{comodule algebra} for a sheaf of bimonoids in the duoidal category
$(Q\sAmod,\otimes_A,\uAotimes,\wsboxtimes, \uwsboxtimes)$ --- or more
particularly for a Hopf sheaf --- as considered in Section
\ref{sect:modcomodduoidal}, in particular definitions
\ref{defn:actionschqc} and \ref{defn:bicomdalg}.

\begin{rem}
\label{rem:comodhts1}

Consider a bimonoid $\mathcal B=(\mathcal B,\mu_{\mathcal B},
u_{\mathcal B}, \Delta_{\mathcal B},\varepsilon_{\mathcal B} )$ in the
duoidal category $(Q\sAmod,\otimes_A,\uAotimes,\wsboxtimes,
\uwsboxtimes)$.

\noindent  (1) A
\emph{left $\mathcal B$--comodule} is a pair $(\mathcal F,\chi)$, with
$\mathcal F\in
Q\sAmod$ and $\chi:\mathcal F \to \mathcal B\wsboxtimes \mathcal F$ a  morphism of 
sheaves, that   satisfies the corresponding commutative
diagrams  (as in Definition \ref{defn:actionschqc}).

\noindent (2) A \emph{morphism of left $\mathcal B$--comodules} is a morphism $\psi:
\mathcal M\to \mathcal M' \in Q\sAmod$ such that the diagram
\[
 \xymatrix{ \mathcal M\ar[rr]^{\psi}\ar[d]_{\chi}&
    &\mathcal M'\ar[d]^{\chi'}\\ \mathcal
    B\wsboxtimes \mathcal M\ar[rr]^-{\operatorname{id}_{\mathcal
        B}\wsboxtimes \psi }& & \mathcal B\wsboxtimes \mathcal M' }
\]
is commutative.

\noindent (3) The category \emph{$\Lcomod{\mathcal B}$ of left  $\mathcal
  B$--comodules} has as objects the $\mathcal B$--comodules and as
arrows $\Hom_{\Lcomod{\mathcal B}}(\mathcal M,\mathcal M')$ the
morphisms of $\mathcal B$--comodules. 

\noindent (4) If we take two $\mathcal B$--comodules (with respect to
  the $\wsboxtimes$ monoidal structure) $\mathcal M, \mathcal M'$
  their product $\mathcal M \otimes_A \mathcal M'$ is also a $\mathcal
  B$--comodule. In other words $\Acomodsf{\mathcal B}$ is a
  $\otimes_A$--monoidal category with unit $\mathcal O_A$, that is
  viewed as an object of $\Lcomod{\mathcal B}$ via the structure
  $\xymatrix{\mathcal O_A \ar[r]^-{s^\#} & \mathcal O_A \wsboxtimes
    \mathcal O_A \ar[r]^{u_{\mathcal B} \wsboxtimes \id}& \mathcal B
    \wsboxtimes \mathcal O_A}$ (see Proposition \ref{prop:fromaguiar}).

\noindent (5) A \emph{right $\mathcal B$--comodule algebra} is a
right  $\mathcal B$--comodule $(\mathcal F,\chi_{\mathcal F})$ such that
$(\mathcal F, \mu_{\mathcal F}, u_{\mathcal F}) \in Q\sAalg$ and
$\chi_{\mathcal F}\in \Hom_{\sAalg}(\mathcal F,\mathcal F \wsboxtimes \mathcal B)$
with the adequate algebra structure in $\mathcal F \wsboxtimes
\mathcal B$ (see  Proposition \ref{prop:fromaguiar} and
Definition \ref{defn:bicomdalg}).

In explicit terms, we ask the diagrams below to commute --- we are using
the notations of definitions \ref{defn:bicomdalg} and
 \ref{defn:othermonoidal2}:
\[
\xymatrix{\mathcal O_A \ar[r]^-{s^\#}\ar[d]_{u_{\mathcal F}} &\mathcal
  O_A \wsboxtimes \mathcal O_A \ar[d]^{u_{\mathcal F}\wsboxtimes
    u_{\mathcal B}}\\\mathcal F\ar[r]^-{\chi_{_\mathcal F}} &\mathcal F \wsboxtimes
  \mathcal B.} 
\]

\[
  \xymatrix{&(\mathcal F \wsboxtimes \mathcal
    B)\otimes_A(\mathcal F \wsboxtimes \mathcal
    B)\ar[r]^{\zeta_{\mathcal F,\mathcal B,\mathcal F,\mathcal
        B}}&(\mathcal F \otimes_A \mathcal F)\wsboxtimes\,(\mathcal B
    \otimes_A \mathcal B)\ar[rd]^{\id \wsboxtimes \mu_{\mathcal B}}&\\  
    \mathcal F {\otimes_A} \mathcal F \ar[ru]^{\chi_{_{\mathcal F}}
      {\otimes_A} \chi_{_{\mathcal F}}}\ar[d]_{\mu_\mathcal
      F}\ar[rrr]^{\chi_{_{\mathcal F \otimes_A \mathcal
          F}}}&&&(\mathcal F \otimes_A \mathcal F) \wsboxtimes
    \mathcal B\ar[d]^{\mu_\mathcal F\wsboxtimes \id}\\
    \mathcal
    F\ar[rrr]^-{\chi} &&&\mathcal F \wsboxtimes \mathcal B} 
\]

\noindent (6)  A \emph{morphism of left $\mathcal B$--comodule
  algebras} from $\mathcal F$ to $\mathcal F'$ is a morphism $f\in
 \Hom_{\Acomodsf{\mathcal
    B}}(\mathcal F, \mathcal F')$ that is also a morphism in $\Hom_{\sAalg}(\mathcal F,\mathcal F')$.

\noindent (7) We denote the category of left (resp.~right) $\mathcal B$--comodule algebras as
$\AQcoalgsf{\mathcal B}$ (resp.~$\RQAcoalgsf{\mathcal B}$).
\end{rem}

As expected, the adjunction between $\PP$ and $\Spec$  gives  a
correspondence between actions of bimonoids 
 $b:M\to A$ and structures of $\mathcal \PP(b)$-comodule algebras (see
 Theorem \ref{thm:mainspec} and Proposition \ref{prop:bimontobimon}). 

\begin{prop}
\label{prop:ratalg}
Let $b:M\to A\in \Schaqc$ be a bimonoid, $x:X\to A\in \Schaqc$ and
$a_X$ an action of $b$ on $x$ (see Definition
\ref{defn:actionschqc} and Example \ref{ej:qS}).  Then $\PP(a_X)$
endows $\PP(x)$ with a structure of $\PP(b)$--comodule algebra.

Conversely, let $\mathcal B\in Q\sAalg$ be a bimonoid,   $\mathcal
F\in Q\sAalg$ and $\chi:\mathcal F\to \mathcal B\wsboxtimes \mathcal F$  a
$\mathcal B$--comodule algebra. Then  $\Spec
(\chi)\smallcirc\eta_{\Spec(\mathcal B)\wtimes \Spec(\mathcal
  F)}:\Spec(\mathcal B)\wtimes \Spec(\mathcal F)\to \Spec(\mathcal F)$
is 
an $\Spec(\mathcal B)$--action --- recall that
$\operatorname{Aff}\bigl(\Spec(\mathcal B)\wtimes \Spec(\mathcal
F)\bigr)=\Spec(\mathcal B\wsboxtimes \mathcal F)$.

In particular, $\PP$ induces an (op)-equivalence between the following two
categories:

\noindent (i) $b-\Schaaff$, with objects the pairs $(x,a_X)$ where
$x:X\to A\in \Schaaff$ and $a_X$ is a $b$--action on $x$, and with
arrows the $b$--equivariant morphisms;

\noindent (ii) $\AQcoalgsf{\PP(b)}$, the category of quasi-coherent
$\PP(b)$--comodule algebras.

Under this equivalence, flat $\PP(b)$--comodules correspond to flat $b$--objects.
\end{prop}
\proof The proof is straightforward and therefore, it is omitted.
\qed

Our  objective is to combine Proposition \ref{prop:ratalg} with the
well known (monoidal) equivalences  
between the category of vector bundles over $A$, the category  of
 locally free sheaves  of $\mathcal O_A$--mod, of finite rank,  and  the
category of the symmetric algebras generated by these sheaves, in
order to describe the category $\RepO(q)$, where $q:G\to A$ is an
affine extension, as a category of sheaves with additional structure.

\begin{rem}
  \label{rem:vvbandsalg}
(1) Recall that if $\mathcal F\in Q\sTmod$, then one can consider the Hadamard monoidal
structure $\otimes_T$ and construct $ S(\mathcal F)$, the \emph{symmetric algebra
  generated by $\mathcal F$}.   

\noindent (2) It is well known that if  $\mathcal F\in Q\sSmod$ is
moreover a   
locally free sheaf of finite rank $\operatorname{rk}\mathcal F=n$, then
$\mathcal F$ is a coherent sheaf  and $\Spec\bigl(S(\mathcal
F)\bigr) $ is a vector bundle over $T$ of rank $n$. Notice also that
$S(\mathcal F)$ is a flat sheaf.

\noindent (3) Conversely, if $\pi: E\to T$ is a vector bundle of rank $n$,  then
$\PP(E)$ is a 
symmetric algebra generated by a  locally free sheaf of
$\mathcal O_T$--modules, of rank $n$. Namely,
$\PP(\pi)=S\bigl(\Gamma_E^\vee\bigr)$, where $\Gamma_E$ is the
\emph{sheaf of sections} of the vector bundle $E$,
$\Gamma_E(U)=\Gamma(E,U)=\{s:U\to E\mathrel{:} \pi\smallcirc u=\id_U\}$.

In this way, the restriction of  the functors $\Spec$ and $\PP$ gives
a monoidal equivalence between $\VB_0(T)$ and $\operatorname{ST-alg}$, the category of
sheaves of $\mathcal O_T$--module symmetric algebras generated by locally free
sheaves of finite rank.

\noindent (4) Taking into account the previous remark, we can produce
another useful equivalence of categories:  if we denote by
$C_{\operatorname{lf}}\sTmod$ the category of (necessarily coherent)
  locally free sheaves of 
  $\mathcal O_T$--modules, of finite rank, then  the functor $\VBSpec:
C_{\operatorname{lf}}\sTmod\to \VB_0(T)$, given by  $\VBSpec(\mathcal F)=\Spec
\bigl(\mathcal S(\mathcal 
F^\vee)\bigr)$  and $\VBSpec(f: \mathcal F\to \mathcal F')=\Spec
\bigl(\mathcal S(f^\vee)\bigr): \mathcal S(\mathcal F'^\vee)\to \mathcal
S(\mathcal F^\vee)$ is an equivalence of categories.
\end{rem}

\begin{defn}
Let $A$ be an abelian variety.  If $\mathcal B$ is a bimonoid in $Q\sAmod$,  we denote
  $\Comodfin{\mathcal B}$ the category of left $\mathcal
  B$--comodules with support on locally free sheaves of $\mathcal
  O_A$--modules, of finite rank.

  We denote $\Comod{\mathcal B}_{\SimA}\subset \AQcoalg{\mathcal B}$  the full subcategory of
symmetric algebras generated by the locally free
 $\mathcal B$--comodules of  finite rank --- notice that if $\mathcal
 F\in \Comodfin{\mathcal B}$, then the $\mathcal
 B$--comodule structure on $\mathcal F$ induces a $\mathcal
 B$--comodule structure on $S(\mathcal F)$ in a natural way. 
\end{defn}

We are now in condition to present our first result concerning the
equivalence of the category of representations of an affine extension
$q:G\to A$
with some categories of $\mathcal H_q$--comodules.

\begin{prop}
\label{prop:equicomodrep0}
Let  $\mathcal S$: $q:G\to A$ be an affine extension. Then:

\noindent (1) the equivalence of
Proposition \ref{prop:ratalg} restricts to a monoidal (op)-equivalence between
$\RepO(\mathcal S)$ and  $\Comod{\mathcal H_q}_{\SimA}$, where $\mathcal
H_q$ denotes as usual  the sheaf of Hopf algebras associated to $q$;

\noindent (2) the equivalence $\VBSpec:
C_{\operatorname{lf}}\sAmod\to \VB_0(A)$ induces an equivalence of
  categories, that we also denote $\VBSpec: 
  \Comodfin{\mathcal H_q}\to \RepO(\mathcal S)$.
\end{prop}
\pf
The proof of (1) is a straightforward consequence of propositions
\ref{prop:charachb}  and \ref{prop:ratalg},  and hence it is  omitted.

In order to prove (2), we first observe that if $(\mathcal
F,\chi_{\mathcal F})\in \Comodfin{\mathcal H_q}$, 
then $\Spec\bigl(\mathcal S(\mathcal F)\bigr)=p_E:E\to A$ supports a
$\mathcal S$--module structure 
$a: q\wtimes p_E\to p_E$ --- notice that we \emph{are not}
applying the functor $\VBSpec$. It follows that  $p_E^\vee: E^\vee\to
A$ is a $\mathcal S$--module, and hence $\PP(E^\vee)$ is a $\mathcal
H_q$--comodule algebra. But by construction, $\PP(E^\vee)=\mathcal
S(\mathcal F^\vee)$ and the corresponding $\mathcal H_q$--coaction
being linear, it restricts to a
coaction $\chi_{\mathcal H_q}:\mathcal F^\vee\to \mathcal
H_q\wsboxtimes \mathcal F^\vee$.

It is easy to see that the
functor ``take dual'' induces an equivalence  $\cdot ^\vee:
\Comodfin{\mathcal B}\to \Comodfin{\mathcal B}$. The proof of  (2) is
now straightforward, since $\VBSpec(\mathcal F)= \Spec\bigl(S(\mathcal
F^\vee)\bigr)$. 
\qed

If $\mathcal S$: $q:G\to A$ is an affine extension, in view of
Proposition \ref{prop:equicomodrep0}, it makes sense to 
 define a  category of  ``infinite 
dimensional'' $\mathcal 
S$--modules either as a certain full subcategory of  $\AQcoalg{\mathcal
  H_q}$ containing  $\Comod{\mathcal H_q}_{\SimA}$ or as a full
subcategory of $\AQcomod{\mathcal H_q}$ containing $\Comodfin{\mathcal
  H_q}$. 

As proposed by V.~Drinfeld in \cite{kn:drinfeldfvid} (in the context
of the definition of an  infinite dimensional vector bundle''), we take the
second approach and consider 
the full subcategory of $\AQcomod{\mathcal H_q}$ with objects the  quasi-coherent,
flat sheaves  of $\mathcal H_q$--comodules, that we denote
$\AQpcomod{\mathcal H_q}$, as a
replacement for 
$\VB_0(A)$ -- recall that if $\mathcal F$ is a coherent flat sheaf of
$\mathcal O_A$--modules, then $\mathcal F$ is locally free (see
\cite[Proposition 2]{kn:serrefvid}).

\subsection{Homogeneous sheaves on an abelian variety}\ %
\label{sect:homogsheves}

In this section we define the category of homogeneous sheaves on an
abelian variety $A$.
We fix the following notation: if $T$
    is a $\Bbbk$--scheme and $\mathcal F \in \sAmod$, we define  $\mathcal F_T=\mathcal F\sboxtimes
    \mathcal O_T\in \sATmod$ --- recall that $A_T=A\times T$. Then
    $\mathcal F_T=p_1^*\mathcal
    F\otimes_{\mathcal O_{A\times T}} p_2^*(\mathcal
    O_T)=p_1^*\mathcal F\otimes_{\mathcal O_{A\times T}} \mathcal O_{A
      \times T}=p_1^*\mathcal F$,   where
    $p_1:A_T\to A$ and $p_2:A_T\to T$ are the canonical
    projections.

  \begin{defn}
    \label{defn:hsgraded1} \noindent (1) If  $A$ is an abelian
    variety and $\mathcal F,\mathcal G\in \sAmod$, we define the
    \emph{functor of graded morphisms} $\Homgr(\mathcal F,\mathcal
    G):\Sch^{\op}\to \Sets$ as follows: for $T \in \Sch^{\op}$ an element of the set
    $\Homgr(\mathcal F,\mathcal G)(T)$ is a pair $(f, \ell)$, with
    $\ell\in A(T)$ and $f\in \Hom_{\sATmod}( t_\ell^*\mathcal F_T ,
    \mathcal G_T)$ --- recall that the translation $t_\ell:A_T\to A_T$
    is a morphism of $T$--schemes and that $f$ is a morphism   
    of sheaves of $A_T$--modules.

For an arrow  $j:T'\to T$ the functor is defined as follows:
\[
 \Homgr(\mathcal F,\mathcal
G)(j): \Homgr(\mathcal F,\mathcal G)(T)\to \Homgr(\mathcal
F,\mathcal G) (T')
\]
 is given by $\Homgr(\mathcal F,\mathcal
G)(j)(f,\ell)= \bigl((\id_A\times j)^*f,\ell\smallcirc j\bigr)$ --- notice that
$\ell \smallcirc j\in A(T')$ and that since  $p_1\smallcirc t_\ell \smallcirc (\id_A \times
j)=p'_1\smallcirc t_{\ell \smallcirc j}$, then   $(\id_A\times j)^*f\in \Hom_{sATmod}(t_{\ell \smallcirc
  j}^*\mathcal F_{T'}, \mathcal G_{T'})$.

\noindent (2) If $(f,\ell)\in \Homgr(\mathcal F,\mathcal G)(T)$, then
the element $\ell \in A(T)$ is called the \emph{degree of}
$(f,\ell)$.

The \emph{degree maps} $d(T): \Homgr(\mathcal F,\mathcal G)(T)\to
A(T)$, $d(f,\ell)=\ell$, conform --- by definition --- a natural
transformation, that is  also called the \emph{degree map}.

\noindent (3) We call $\HomO(\mathcal F,\mathcal G)$ the subfunctor  of $\Homgr(\mathcal F,\mathcal G)$ given by the elements of
degree zero: i.e.~the set $\HomO(\mathcal F,\mathcal G)(T)=\{(f,0): (f,0) \in \Homgr(\mathcal F,\mathcal G)(T)\}$ with $f \in \Hom_{\sATmod}\bigl(\mathcal F_T , \mathcal
G_T\bigr)$.
\end{defn}

\begin{defn}
  \label{defn:variashaces}
  Let $A$ be an abelian variety and $\mathcal F,\mathcal G , \mathcal
  E\in \sAmod$. We define a natural transformation $\smallcirc:
  \Homgr(\mathcal G,\mathcal E)\times \Homgr(\mathcal F,\mathcal G)
  \Rightarrow \Homgr(\mathcal F,\mathcal E)$ as follows: given a
  $\Bbbk$--scheme $T$ and $(f,\ell)\in \Homgr(\mathcal F,\mathcal
  G)(T)$, $(g,b)\in \Homgr(\mathcal G,\mathcal E) $, then
  $f:t_\ell^*\mathcal F_T \to \mathcal G_T$ and $t_b^*f:
  t_{\ell+b}^*\mathcal F_T \to t_b^*\mathcal G_T$. Hence, we can
  define $(g, b) \smallcirc (f,\ell): = ( g\smallcirc t_b^* f, \ell+b)\in
  \Homgr(\mathcal F,\mathcal E)$.
\end{defn}
\begin{rem} It follows by definition that the degree map satisfies the
  following compatibility condition with the composition defined
  above. 
  \[
    \xymatrix{\Homgr(\mathcal G,\mathcal E)\times
    \Homgr(\mathcal F,\mathcal G) \ar[rr]^(.6)\smallcirc\ar[d]_{d \times d}&&\Homgr(\mathcal
    F,\mathcal E)\ar[d]^d\\A \times A\ar[rr]^s&&A}
\]
\end{rem}

  \begin{nota}\label{notahom} 
    (1) We denote  $\Endgr(\mathcal F):=\Homgr(\mathcal F,\mathcal
    F)$; notice that $\Endgr(\mathcal F)$ with the composition of
    graded morphisms is  a
monoid functor,  with unit $e_{\Endgr(\mathcal F)}(T)=(\id_{\mathcal
  F_T}, 0)$.

\noindent (2) The group functor $\Autgr(\mathcal F)$ is the unit subfunctor of
$\Endgr(\mathcal F)$;  that is, the subgroup functor of
$\Endgr(\mathcal F)$ given (for each scheme $T$) 
        by the pairs $(f,\ell)$ such that $f: t_\ell^*\mathcal F_T\to
        \mathcal F_T $ is an isomorphism.
        
\noindent (3) We also consider the functor on monoids $\EndO(\mathcal F)=\HomO(\mathcal
  F,\mathcal F)$, and the functor on groups $\AutO(\mathcal F)$ given
  by the pull back of the inclusions $\Autgr(\mathcal F)
  \subset \Endgr(\mathcal F)$ and $\EndO(\mathcal F)
  \subset \Endgr(\mathcal F)$ --- therefore, $\AutO(\mathcal F)$ is the
  subfunctor of $\Autgr(\mathcal F)$ given by the elements of degree zero.
        
    \end{nota}

  As in the case of vector bundles, it is easy to see that the
  morphism of degree zero between two sheaves  $\mathcal F,\mathcal G$
  are  in bijection
  with the morphisms   of sheaves of $\mathcal O_A$--modules between
  $\mathcal F$ and $\mathcal G$.

    \begin{rem}\label{rem:hom0OK}
Let $\mathcal F,\mathcal G\in \sAmod$. Then the functor
$\HomO(\mathcal F,\mathcal G)$ is represented by the $\Bbbk$-vector
space 
$\Hom_{\sAmod}(\mathcal F,\mathcal G)$.
In particular, $\EndO(\mathcal F)$ is a smooth scheme on monoids.
Indeed,    the $\Bbbk$--vector space
      $\Hom_{\sAmod}(\mathcal F,\mathcal G)$ 
      represents the functor  $T\mapsto
      \Hom_{\sATmod}(\mathcal F_T,\mathcal G_T)$.
\end{rem}

\begin{defn}
As in Definition \ref{defn:vbgraded}, we consider the monoidal
category $\mathcal
V=\operatorname{Func}(\Sch^{\op},\Sets)$ and we define the $\mathcal
V$--category $\sAmodgr$, that is called the
\emph{category of sheaves on $A$ with graded 
  morphisms},  with  objects the sheaves of $\mathcal O_A$--modules
and with hom-object  
$\Homgr(\mathcal F,\mathcal G) \in \mathcal V$, with  composition
defined as above. In other words,  for $\mathcal F$ and $\mathcal G$
sheaves on $A$, $\Hom_{\sAmodgr}(\mathcal F,\mathcal G)=
\Homgr(\mathcal F,\mathcal G)$.

    We define the full   subcategory  $Q\sAmodgr\subset \sAmodgr$ of \emph{quasi-coherent
  sheaves of $\mathcal O_A$--modules with graded morphisms} with 
objects  the quasi-coherent sheaves of $\mathcal O_A$--modules.

Similarly, we define the subcategory $Q\sAalggr$ of
\emph{quasi-coherent sheaves of $\mathcal O_A$--algebras with graded
  morphisms} by taking as objects the quasi--coherent sheaves of
$\mathcal O_A$--algebras, and if $\mathcal F,\mathcal G$ are two
objects, then $\Hom_{\sAalggr}(\mathcal F,\mathcal G)$ is the
subfunctor of $\Homgr(\mathcal F, \mathcal G)$ given by $(f,\ell)\in
\Hom_{\sAalggr}(\mathcal F,\mathcal G)$ if $f:
t_\ell^*\mathcal F_T\to \mathcal G_T$ is a morphism of
$\mathcal O_{A_T}$--algebras (over the $T$--scheme $A_T$).
 \end{defn}
   
 \begin{rem}
   Call $\sAmodz \subseteq \sAmodgr$ the wide subcategory of
   $\sAmodgr$ with morphisms between $\mathcal F, \mathcal G$ the
   functor $\HomO(\mathcal F,\mathcal G)$. It is clear that the usual
   category $\sAmod=\mathcal O_A-\mathrm{mod}$ is equivalent to
   $\sAmodz$ (see Remark \ref{rem:hom0OK}). 
   Similarly for the analogous situation but in $Q\sAmod$ and $Q\sAalg$ with respect to
   $Q\sAmodz$ and $Q\sAalgz$.

Clearly,  $\EndO(\mathcal F) 
$   is  the kernel of the morphism of
 functors on monoids $d:\Endgr(\mathcal F)\to A$, and $\AutO(\mathcal F)$
the kernel of  $d:\Autgr(\mathcal F)\to A$.
\end{rem}

 \begin{rem}
    \label{rem:hgomgrfpqc}
As in the case of  vector bundles, if $\mathcal F, \mathcal
F'\in \sAmodgr$, then $\Homgr(\mathcal F, \mathcal F')$ is a fpqc
sheaf. 
\end{rem}

In  particular,  when we deal with locally free sheaves of finite
rank we have the following equivalence result:

\begin{lem}
  \label{lem:vbspecgr0}
In the above situation, the equivalence  $\VBSpec:
C_{\operatorname{lf}}\sAmod\to \VB_0(A)$  extends to an
equivalence $\VBSpecgr: C_{\operatorname{lf}}\sAmodgr\to \VBG(A)$. 
\end{lem}
\pf
Fist, we observe that if $\mathcal F\in C_{\operatorname{lf}}\sAmod$ and $T$ is a
$\Bbbk$--scheme,  then $\VBSpec(\mathcal F_T)\cong \VBSpec(\mathcal
F)_T$. Indeed, 
\[
\Spec\bigr(\mathcal 
S\bigl((\mathcal F_T)^\vee\bigr)\bigr)=\Spec\bigl(\mathcal S
\bigl((p_1^*\mathcal F)^\vee\bigr)\bigr) = \Spec\bigl(\mathcal S
\bigl(p_1^*(\mathcal
F^\vee)\bigr)\bigr)=p_1^*\bigl(\Spec\bigl(\mathcal S (\mathcal
F^\vee)\bigr)\bigr),\]
 where the second of
  the above chain of equalities is due to the commutation of pullback
  with duals and the third by the commutation of base change with
  $\Spec$ (see \cite{kn:commute} and \cite{kn:raising}[Thm. 17.1.3,
    Ex. 17.1.F)] respectively). In other words,  
 \[
\VBSpec(\mathcal
  F_T) \cong\VBSpec(\mathcal F)_T=\bigl(\pi_{\mathcal F}\times \id:
  \VBSpec(\mathcal F) \times T \to A \times T\bigr).
\]

  Next we show that for all $\mathcal F \in C_{\operatorname{lf}}\sAmod$  and  $b \in
  A(T)$ we have the commutation relation
  $\VBSpec\bigl(t_b^*(\mathcal F_T)\bigr)\cong t_b^*\VBSpec(\mathcal 
  F)_T$:
  \[
  \begin{split}
    \VBSpec (t_b^*\mathcal F_T) = &\ 
\Spec\bigr(t_b^*\mathcal S\bigl(\mathcal F_T^\vee\bigr)\bigr)= t_b^*\bigl(\Spec\bigl(\mathcal S \bigl((p_1^*\mathcal
  F)^\vee\bigr)\bigr)\bigr)=\\
  & \ t_b^*\VBSpec(\mathcal F_T) = \bigl(t_b^*\VBSpec(\mathcal F)\bigr)_T, 
\end{split}
  \]
where we  used again the commutation of duals and pullback for locally
free shaves.

Once we have the identifications above, it is clear that  the
equivalence  $\VBSpec:
C_{\operatorname{lf}}\sAmod\to \VB_0(A)$  induces  a natural
isomorphism between the functors $\Homgr(\mathcal F, \mathcal G)$
and $\Homgr\bigl(\VBSpec(\mathcal F),\VBSpec(\mathcal G)\bigr)$.

Indeed, if  $T$ is a $\Bbbk$--scheme, then the family of bijections
(indexed on $\ell\in A(T)$) given by 
\[
  \Hom_{\sATmod}\bigl(t_\ell^*(\mathcal F_T),\mathcal
G_T\bigr)\to \Hom_0\bigl(\VBSpec\bigl(t_\ell^*\mathcal F_T\bigl),\VBSpec(\mathcal
G_T)\bigr)= \Hom_0\bigl(t_\ell^*\bigl(\VBSpec(\mathcal F)_T\bigl),\VBSpec(\mathcal
G)_T\bigr),
\]
combine into a bijection
\[
 \bigl\{ (f,\ell)\mathrel{:}
 f\in\Hom_{_{\sATmod}}\bigl( t_\ell^* (\mathcal F_T) , \mathcal
 G_T\bigr)\,,\ \ell\in A(T)\bigr\} \to \Homgr\bigl(\VBSpec(\mathcal
 F),\VBSpec(\mathcal G)\bigr).
\]

The family  of bijections above (indexed on $T$) conform the  seeked natural
 transformations.\qed
 
 \begin{rem}
   \label{rem:casolf}
   As an immediate consequence of Lemma \ref{lem:vbspecgr0}, we have
   that if $\mathcal F\in C_{\operatorname{lf}}\sAmod$, then
 $\Autgr(\mathcal F)$ is a smooth group scheme of finite type and the
 degree map $d:\Autgr(\mathcal F)\to A$ is an affine morphism of group
 schemes, with kernel $\AutO(\mathcal F)$, since $\Autgr(\mathcal
 F)\cong \Autgr\bigl(\VBSpec(\mathcal F)\bigr)$.

 It is an open question whether  $\Autgr(\mathcal F)$ is a  group
 scheme or not if $\mathcal F\in Q\sAmod$. However, we have the
 following description.
 \end{rem}

\begin{lem}\label{lem:autgrfpqcseq}
Let $A$ be an abelian variety and $\mathcal F\in \sAmod$. Then
$\AutO(\mathcal F)$ is a smooth affine group scheme,
and the 
degree map 
$d: \Autgr(\mathcal F)\to A$ is a morphism of fpqc sheaves, with kernel $
\AutO(\mathcal F)$.
  \end{lem}
  \pf
  Recall that   $\EndO(\mathcal F)$ is a smooth monoid (see Remark
  \ref{rem:hom0OK}),  and that 
  $\AutO(\mathcal F)$ is its unit group. It follows from
  \cite[II.3.7]{kn:demgab} that $\EndO(\mathcal F)$ is the 
  limit in the category of monoid schemes of a family $M_i$ of finite
  type; therefore, $\AutO(\mathcal F)$ is the  limit of the
  unit groups $G(M_i)$. Since $G(M_i) \subset M_i$ is open by
  \cite[II.3.6]{kn:demgab}, it follows that $\AutO(\mathcal F)$ is an
  open subscheme of $\EndO(\mathcal F)$.

It is clear that $\Autgr(\mathcal F)$ is a sheaf for the fpqc
topology (e.g.~by descent for morphisms). Moreover, $\AutO(\mathcal
F)$  is a subsheaf and the quotient morphism $\pi:\Autgr(\mathcal F)\to
\Autgr(\mathcal F)/\AutO(\mathcal F)$ is  a torsor under $\AutO(\mathcal F)$, in view of 
\cite[III.4.1.8]{kn:demgab}. Now is clear that $d$ factorizes through $\pi$.\qed

\begin{defn}
  \label{defn:homogsheaf}
  Let $A$ be an abelian variety. A sheaf $\mathcal F\in \sAmodgr$ is
  \emph{homogeneous} if  the degree map $d:\Autgr(\mathcal F)\to A$
  surjective in the fpqc  topology. 

In view of  Lemma \ref{lem:autgrfpqcseq}, it follows that the 
  $\Autgr(\mathcal F)$ is a quasi-compact group scheme, the   sequence 
  \[
    \xymatrix{
    \Authbext(\mathcal F):&      1\ar[r]&  \AutO(\mathcal F)\ar[r]& \Autgr(\mathcal F)\ar[r]^-d&
      A\ar[r]& 0
    }
\]
 is an affine extension. In particular $d$ is
  a quasi-compact, faithfully flat morphism of group schemes.
\end{defn}

\begin{ej} 
(1) Clearly, the structure sheaf $\mathcal O_A$ is homogeneous.

\noindent (2) It is clear that under the equivalence $\VBSpecgr:
C_{\operatorname{lf}}\sAmodgr\to \VBG(A)$, homogeneous sheaves
correspond to homogeneous vector bundles (see Remark
\ref{rem:casolf}). 
\end{ej}

\begin{defn}
\label{def:cathomogsheaves}
We define $\hsAmod$, \emph{the category of homogeneous sheaves on
  $A$}, as the full subcategory of $\sAmod$ with objects the
homogeneous sheaves on $A$.  Similarly, we define the category
$Q\hsAmod$ of \emph{homogeneous quasi-coherent sheaves of $\mathcal
  O_A$--modules} as the full subcategory of $Q\sAmod$ with objects the
homogeneous, quasi-coherent sheaves on $A$, and $C_{lf}\hsAmod\subset
Q\hsAmod$ the 
full subcategory of locally free, homogeneous, sheaves of finite rank. We denote
$Q\hsAalg$ the category of \emph{homogeneous quasi-coherent sheaves of
  $\mathcal O_A$--algebras}.

As in Definition \ref{defn:vbgraded} we define the
$\mathcal V$--categories $\hsAmodgr$, $Q\hsAmodgr$,
$C_{\operatorname{lf}}\hsAmodgr$, $\hsAalggr$ and 
$Q\hsAalggr$ as 
the full subcategories of $\sAmodgr$, $Q\sAmodgr$,
$C_{\operatorname{lf}}\sAmodgr$,  $\sAalggr$ and
$Q\sAalggr$ where on each situation the objects are limited to the
homogeneous sheaves.
\end{defn}

\begin{lem}
\label{lem:homgrhomogsheaves}
Let $\mathcal F,\mathcal F'\in \hsAmodgr$ be two homogeneous sheaves. Then the homogeneous vector bundle
$R_{\HomO(\mathcal F,\mathcal F')}=\Autgr(\mathcal F')\times^{ 
  \AutO(\mathcal F')} \HomO(\mathcal F,\mathcal F')$   represents
  $\Homgr(\mathcal F,\mathcal F')$. 

Moreover, $  R_{\HomO(\mathcal F,\mathcal F')}\cong L_{\HomO(\mathcal F,\mathcal F')}=
\Autgr(\mathcal F)\times^{\AutO(\mathcal F)}\HomO(
\mathcal F,\mathcal F')\in \HVB_0(A)$. 
  \end{lem}
    \pf
    We replicate the proof of Lemma \ref{lem:homgrhomog}.
    
Let $\varphi :\Autgr(\mathcal F')\times \HomO(\mathcal F,\mathcal
F')\to \Homgr(\mathcal F, \mathcal F')$ the
morphism of fpqc sheaves given by composition. Then clearly $\varphi$
is $\AutO(\mathcal F)$--invariant, and therefore induces a morphism of fpqc sheaves
$\phi: R_{\HomO(\mathcal F,\mathcal F')}\to \Homgr(\mathcal
F,\mathcal F')$.

Let  $y_1:T\to \Autgr(\mathcal F')\times \HomO(\mathcal F,\mathcal
F')$ and  $y_2:T\to \Autgr(\mathcal F')\times \HomO(\mathcal F,\mathcal
F')$ be   such that $\varphi(T)(y_1)=\varphi(T)(y_2)\in \Homgr(\mathcal F,\mathcal F')(T)$. 
Define 
$f_j:=p_1(T)(y_j) \in \Autgr(\mathcal F')(T)$,  $\ell_j:=d\smallcirc p_1(T) (y_j) \in A(T)$,
$ g_j:=p_2(T)(y_j)\in \HomO(\mathcal F,\mathcal F')(T) $ and $
\varphi(T)(y_j)= (f_j\smallcirc g_j, \ell_j)\in \Homgr(\mathcal F,\mathcal
F')(T)$, for $j=1,2$.

By hypothesis we have $\ell_1= \ell_2$ and  $ f_1\smallcirc g_1 = f_2\smallcirc g_2$, but since the $f_j$'s are invertible, we get
$g_2=f_2^{-1}\smallcirc f_1\smallcirc g_1$, and obviously $f_2= f_1\smallcirc (f_2^{-1}\smallcirc f_1)^{-1}$ with $f_2^{-1}\smallcirc f_1\in \AutO(\mathcal F')(T)$.

On the other hand, given $(f,\ell)\in \Homgr(\mathcal F, \mathcal F')(T)$, since the sequence
\[
1\to\AutO(\mathcal F')\to \Autgr(F')\to A \to 0
\]
is exact  there exist $h:T'\to T$ fpqc and
$g:T'\to \Autgr(\mathcal F')$ such that $d\smallcirc g=\ell\smallcirc h = \ell|_{_{T'}}$. In particular,
$g:t_{\ell|_{_{T'}}}^*\mathcal F_{T'}\to \mathcal F_{T'}$
is invertible, therefore $f\smallcirc g^{-1}\in\HomO(\mathcal F,\mathcal
F')(T')$, $g\in \Autgr(\mathcal F')(T')$ and $\varphi(g,f\smallcirc g^{-1})
= (f,\ell)|_{_{T'}}$, showing that $\varphi$ is locally surjective. Since
both sides are sheaves on the fpqc topology, they turn out to be
isomorphic and 
$R_{\HomO(\mathcal F,\mathcal F')}$ represents the functor
$\Homgr(\mathcal F,\mathcal F')$ (see Lemma \ref{lem:homgrhomog}). The claim for $L_{\HomO(\mathcal
  F,\mathcal F')}$ follows in a similar way. The isomorphism of
representing schemes is now obvious. 
\qed

\begin{rem}
Let $\mathcal F\hsAmod$ be a  homogeneous sheaf. Then $\Autgr(\mathcal
F)$ is a quasi compact group scheme and form the proof of Lemma
\ref{lem:homgrhomogsheaves} above we deduce that    $L_{\HomO(\mathcal
  F,\mathcal F')}$ represents the functor $\Homgr(\mathcal F,\mathcal
F')$, regardless on whether $\mathcal F'$ is homogeneous or
not. Similarly if $\mathcal F'$ is homogeneous, then
$R_{\HomO(\mathcal F,\mathcal F')}$ represents the functor
$\Homgr(\mathcal F,\mathcal F')$, regardless on whether $\mathcal F$
is homogeneous or not. The existence of a geometric structure for $\Homgr(\mathcal
F,\mathcal F')$ when neither of the sheaves is homogeneous (i.e.~the
representability of this functor) remains open.
\end{rem}

We finish this paragraph studying the relationship between homogeneous
vector bundles and homogeneous sheaves --- in particular, we present
the correspondence between 
homogeneous vector bundles and homogeneous locally free sheaves of
finite rank as an equivalence of categories.

\begin{lem}
  \label{lem:vbspecgr}
In the above situation, the equivalence  $\VBSpec:
C_{\operatorname{lf}}\sAmod\to \VB_0(A)$ restricts to an equivalence $\VBSpec|_{_{C_{\operatorname{lf}}}\hsAmod} :
C_{\operatorname{lf}}\hsAmod\to \HVB_0(A)$. Moreover, $\VBSpec|_{_{C_{\operatorname{lf}}\hsAmod}}$ extends (as an
equivalence) to $\VBSpecgr: C_{\operatorname{lf}}\hsAmodgr\to \HVBG(A)$. 
\end{lem}
\pf
The result follows immediately from Lemma \ref{lem:autgrfpqcseq}  and Remark
\ref{rem:casolf}.\qed

\begin{cor}
\label{coro:homogeneity}
A vector bundle $E\to A$ is homogeneous if and only if $\PP(E)\in
C\SimA$ is  
homogeneous, and this happens 
if and only if its corresponding coherent, locally free sheaf $\mathcal
F_E$ is homogeneous.

In particular, if $\mathcal H$ is a flat Hopf sheaf, then any coherent $\mathcal
H$--comodule is homogeneous.
\end{cor}
\pf
This a direct consequence of Lemma
\ref{lem:vbspecgr} and    Proposition
\ref{prop:equicomodrep0}.
\qed

\subsection{Linearization of sheaves}\ %
\label{subsect:linerarishaeves}

We begin by recalling the definition of the category of
$G$--linearized sheaves (see \cite[page 30]{kn:GIT} and \cite[Tag
  03LE]{kn:stackproj}). We work over schemes defined over $\Bbbk$ and
the sheaves will be in general quasi--coherent $\mathcal O_X$--modules
but the definitions can be performed for general $\mathcal O_X$--modules.

 Let $G$ be a group scheme with $m,e_G$ the multiplication and the unit
 of the group, assume $X$ a $G$--scheme with an  action $a:G \times X \to
 X$.

 In order to simplify notations, we  display the following (rather
 obvious) commutation
 relations:

\begin{enumerate}

\item $ a\smallcirc (\id_G\times a)= a\smallcirc (m\times \id_X):G\times
  G\times X\to X$ and $a\smallcirc( e_G\times \id_X)=
  p_2:\Spec(\Bbbk)\times X\to X$ (that is, $a$ is an action).

\item $a\smallcirc p_{23}= p_2\smallcirc (\id_G\times a): G\times G\times X\to
  X$, where $p_{ij}:X_1\times X_2\times X_3\to X_i\times X_j$ denotes
  the canonical projection. 

  \item  $ p_2\smallcirc (\id_G\times p_2)= p_2\smallcirc (m\times \id_X)=p_3:G\times
  G\times X\to X$ and $p_2\smallcirc( e_G\times \id_X)=
  p_2:\Spec(\Bbbk)\times X\to X$ (that is, the canonical projection $p_2$ is trivial action).

\end{enumerate}

\begin{defn}
  \label{def:Gbundle}
Let $G$ be a group scheme,  $X$  a $G$--scheme and
  $\mathcal F$ a  sheaf of  $\mathcal O_X$-modules; denote the
  $G$--action of $G$ on $X$ by $a:G\times X\to X$.

  A \emph{$G$--linearization of the sheaf}  $\mathcal F$ (or a
  linearization {\em compatible with the action $a: G \times X \to X$}) is  an  isomorphism of
  sheaves of $\mathcal O_{G\times X}$--modules  $\Phi:
  a^*(\mathcal F)\to p_2^* (\mathcal F)$ such that:

\noindent (1) The diagram below is commutative:
  \[\resizebox{\displaywidth}{!}{%
  \xymatrix@C=-1.75pc{(\id_G \times a)^*a^*(\mathcal F) = (m \times
    \id_X)^*a^*(\mathcal F)\ar[rdd]_{\hspace{-1.5pc}(\id_G \times
      a)^*(\Phi)}\ar[rr]^(.52){(m \times \id_X)^*(\Phi)}&&(m\times
    \id_X)^*p_{2}^*(\mathcal F)=p_3^*(\mathcal F)= p_{23}^*p_2^*(\mathcal F)\\&&\\
    &(\id_G \times a)^*p_2^*(\mathcal F)=p_{23}^*a^*(\mathcal F)
    \ar[ruu]_{\hspace{1pc}p_{23}^*(\Phi)}&}}\]

  \noindent (2) The pullback by the morphism of schemes $e_G\times
  \id_X: \Speck \times X \to G \times X$ of the diagram $\Phi:
  a^*(\mathcal F)\to p_2^* (\mathcal F)$ yields $(e_G\times
  \id_X)^*(\Phi)=\id : \mathcal F \to \mathcal F$.
\end{defn}

\begin{rem}
  Conditions (1) and (2) of Definition
\ref{def:Gbundle} are the same than the ``cocycle condition'' stated originally in GIT
\cite[\S 1.3]{kn:GIT}.
\end{rem}

 Next we define the category of $G$--linearized sheaves for a
 $G$--scheme $X$ with a fixed action $a: G \times X \to X$. 

 \begin{defn} If $a:G\times X\to X$ is an action of the group scheme
   $G$ on the scheme $X$,  define $\operatorname{G-Q}\sXmod$, the \emph{category of $G$--linearized
   sheaves for the $G$--scheme $X$},  as having: 
   
 \noindent (1) {\tt as objects:} the  pairs $(\mathcal F, \Phi)$,
 where 
 $\mathcal F$ is a sheaf of quasi-coherent $\mathcal O_X$--modules and $\Phi$ a
 linearization of $\mathcal F$ as defined above;

\noindent (2) {\tt as morphisms:} the morphisms from 
 $(\mathcal F, \Phi)$ to $(\mathcal F',\Phi')$, 
 are the morphisms $f\in \Hom_{Q\sXmod}(\mathcal F,\mathcal F')$
 such that the following diagram is
commutative
\[
\xymatrix{ a^*(\mathcal F)\ar[r]^-{\Phi}\ar[d]_(.45){a^*(f)} &
  p_2^* (\mathcal F)\ar[d]^(.45){p_2^*(f)}\\ a^*(\mathcal
  F')\ar[r]^-{\Phi'} & p_2^* (\mathcal F'). }
\]
\end{defn}

\begin{defn}
  \label{defi:linsheaf}
  Let $\mathcal S$: $q:G\to A$ be an affine extension. Then $G$ acts on $A$
by $a_q=s\smallcirc (q\times id_A):G\times A\to A$. An \emph{$\mathcal
  S$--linearized sheaf} or \emph{$q$--linearized sheaf} is a pair $(\mathcal F, \Phi)$, where $\mathcal F$  a
quasi-coherent sheaf of $\mathcal O_A$--modules and $\Phi: a_q^*(\mathcal F)\to
p_2^* (\mathcal  F)$ is a 
 $G$--linearization compatible with $a_q$.

Given two $q$--linearized sheaves    $\mathcal F, \mathcal
F'$, then a \emph{morphism of $q$--linearized sheaves} is
a morphism $f:\mathcal F\to \mathcal F'$ of $G$--linearized sheaves
with respect to $a_q$.

The \emph{category $\operatorname{q-}Q\sAmod$
   of quasi-coherent $q$--linearized
  sheaves} has as
objects the quasi-coherent $q$--linearized sheaves,
as morphisms $\Hom_{\operatorname{q-}Q\sAmod}(\mathcal F,\mathcal F')$ the
morphisms of $q$--linearized sheaves. We denote
$\operatorname{q-}C_{\operatorname{lf}}\sAmod\subset
  \operatorname{q-}Q\sAmod$ the  full subcategory with objects the
  locally free sheaves of finite rank.
\end{defn}

\begin{rem}\label{rem:linearizations}
  Let $\mathcal S$: $q:G\to A$ be an affine extension.
  
   \noindent (1)  It is easy to see that  kernels, images, cokernels
   of homomorphisms of $q$-linearized sheaves as well as tensor
   products and 
symmetric powers (for the Hadamard monoidal structure $\otimes_A$) 
of $q$-linearized sheaves inherit $q$-linearizations in a natural way.

\noindent (2) Clearly, if $\mathcal F\in
C_{\operatorname{lf}\sAmod}$ is $q$--linearized, $\mathcal F^\vee$
inherits a $q$-linearization in a natural way.

\noindent (3) If $X$ is a $G$--scheme of finite type and $f:X\to A$ is
a $G$--equivariant morphism, then  a $q$--linearization on $\mathcal F\in
Q\sAmod $ induces a 
$G$--linearization on $f^*\mathcal F\in Q\sXmod$ (see
\cite[p.~94]{kn:huybrechtslehn}). 

\noindent (4) A direct generalization of the considerations in
\cite[\S 1.3]{kn:GIT} shows that a $q$--linearization in  $\mathcal
  F\in Q\SimA$ induces a $G$--action  $a:G\times \Spec(\mathcal F)\to \Spec(\mathcal F)$, such
  that the following diagram is commutative
  \[
    \xymatrix{
      G\times \Spec(\mathcal
      F)\ar[r]^-a\ar[d]_{q\times\pi}&\Spec(\mathcal F)\ar[d]^\pi\\
      A\times A \ar[r]^-s& A
      }
    \]
     In other words, a
    $q$--linearization  of $\mathcal F$ induces on $\Spec(\mathcal F)$
    a structure of $q$--module in the duoidal category    $(\Schaqc,
    \wtimes, \uwtimes, \times_A, \uAtimes)$ --- that is,  an action $a: q\wtimes
    \Spec(\mathcal F)\to \Spec(\mathcal F)$. 

    Conversely, if $x:X\to A$ is a $q$--module, then the action
    $a:q\wtimes x\to x$ induces a $q$--linearization on $\PP(x)$.

   In  \cite[p.~94]{kn:huybrechtslehn} the reader will find  a proof
   (where
   $G$ is assumed of finite 
    type) that  is valid in our context.

    \noindent (5) In particular, we deduce from (2) and (4) that if a sheaf
    $\mathcal F\in C_{\operatorname{lf}}\sAmod$  admits a
    $q$--linearization, then it induces a structure of $q$-module on
    $\VBSpec(\mathcal F)$. Conversely, if $(\pi:E\to
    A)\in \RepO(\mathcal S)$, then the action $a:q\wtimes \pi\to \pi$
    induces a $q$--linearization on $\PP(\pi)$. It is easy to show that
   if $f\in \Hom_{\operatorname{q-}C_{\operatorname{lf}}\sAmod}(\mathcal F,\mathcal G)$,
   then $\VBSpec(f)\in \HomO\bigl(\VBSpec(\mathcal F),\VBSpec(\mathcal
   G)\bigr)$, and that  we have
   an equivalence between $\operatorname{q-}C_{\operatorname{lf}}\sAmod$ and
   $\RepO(\mathcal S)$.
\end{rem}

 We proceed now to establish the notion
of graded morphisms of $\mathcal 
S$--linearized sheaves. Before doing so,  it is convenient to set some
equalities 
between different pull-backs involved it the mentioned definition. 

\begin{rem}
\label{rem:indlin}
  \noindent (1)  Let $\mathcal S$: $q:G\to A$ be an affine extension
 and $\mathcal F\in Q\sAmod$  a quasi-coherent sheaf of $\mathcal O_A$--modules. Let
 $T$ be a $\Bbbk$--scheme and 
 $\ell:T\to A$  a $T$--point. From  the
  equalities of morphisms $G\times A\times T\to A$:
  \begin{align*}
    p_1\smallcirc (a_q\times\id_T)  &  =a_q\smallcirc p_{12}\\
p_1\smallcirc p_{23} & = p_2\smallcirc p_{12}\\
   p_2\smallcirc p_{12}(\id_G \times t_{\ell}) &=p_1\smallcirc p_{23}\smallcirc (\id_G \times t_{\ell})= p_1\smallcirc t_{\ell}\smallcirc p_{23}
  \end{align*}
  and the equality
  \[
    t_{\ell}\smallcirc (a_q \times \id_T)  =(a_q \times \id_T) \smallcirc (\id_G \times t_\ell): G\times A\times T\to A_T=A\times T
  \]
       we deduce --- in the same order --- the following equalities of sheaves of $\mathcal
  O_{G\times A\times T}$--modules (recall that $\mathcal F_T=p_1^*\mathcal F$):
  \begin{equation}
    \label{eq:glinear}
    \begin{aligned} (a_q \times\id_T)^*(\mathcal F_T) & =
        p_{12}^* (a_q^* \mathcal F)= a_q^*(\mathcal
        F)_T\\
        p_{23}^*(\mathcal F_T) & = p_{12}^*\bigl(p_2^*(\mathcal
        F)\bigr)\\
        p_{23}^*(t_{\ell})^*(\mathcal F_T) &=(\id_G \times t_{\ell})^*p_{23}^*(\mathcal F_T)= (\id_G \times t_{\ell})^*p_{12}^*\bigl(p_2^*(\mathcal F)\bigr)\\
              (a_q\times\id_T)^*t_{\ell}^*(\mathcal F_T) & =
                 (\id_G\times t_{\ell})^* \bigl((a_q^*\mathcal
                 F)_T\bigr)
       \end{aligned}
  \end{equation}

  \noindent (2) In the situation above, suppose now that $\mathcal F$
  is  linearized with respect to $a_q$, and consider the  action
  $a_q\times \id_T: G\times A\times T\to A\times T$. Then the
  linearization map $\Phi: a_q^*\mathcal F \to 
  p_2^*\mathcal F \in \mathcal O_{G \times A}$ induces (by applying
  $p_{12}^*$) a   linearization map on $\mathcal F_T\in QA_T\mathrm{-mod}$ given as: $p_{12}^*\Phi: (a_q
  \times \id_T)^* \mathcal F_T \to p_{23}^* \mathcal F_T$. In what follows we will sometimes write $\Phi_{\mathcal F}$ instead of
  just $\Phi$, specially if there is more than one sheaf under consideration. 
  
  If we take a $T$--point $\ell: T \to A$, and apply to the above
  linearization for $\mathcal F_T$, the functor $(\id_G \times
  t_\ell)^*$ we obtain a linearization for $t_{\ell}^*\mathcal F_T\in Q\sATmod$ as:
  \[
    (\id_G \times t_\ell)^*p_{12}^*(\Phi_{\mathcal
    F}):(a_q\times\id_T)^*(t_{\ell})^*(\mathcal F_T) \to p_{23}^*
  (t_{\ell})^*(\mathcal F_T).
  \]

This is an easy consequence of the formul\ae\ established just above. 

\end{rem}

\begin{defn}
Let $\mathcal S$: $q:G\to A$ be an affine extension and      $\mathcal
F, \mathcal 
F'\in Q\sAmodgr$ two $q$--linearized sheaves. The \emph{functor of
  graded morphisms of $q$--linearized sheaves}   is the 
subfunctor $\Hom_{\operatorname{q-}Q\sAmodgr}(\mathcal  F,\mathcal F')\subset \Homgr(\mathcal  F,\mathcal F')$ given as follows:
$(f,\ell)\in  \Homgr (\mathcal  F,\mathcal F')(T)$ belongs to
$\Hom_{\operatorname{q-}Q\sAmodgr}(\mathcal  F,\mathcal F')(T)$ if the
following  diagram of sheaves on the $T$--scheme $G\times A\times T$ is
  commutative
  \[
 \xymatrix{
   (a_q\times\id_T)^*t_\ell^*(\mathcal F_T)
   \ar[rrr]^-{(\id_G\times t_\ell)^*p_{12}^*(\Phi_{\mathcal F})}
   \ar[d]_{(a_q\times
   \id_T)^*(f)}&&& p_{23}^* t_\ell^*(\mathcal F_T)\ar[d]^{p_{23}^*(f)}\\ 
   (a_q\times \id_T)^*(\mathcal
   F'_T)\ar[rrr]^-{p_{12}^*\Phi_{\mathcal F'}}&& &p_{23}^* (\mathcal F'_T)
 }
 \]
where we used the equalities \eqref{eq:glinear}.

\end{defn}

\begin{defn}
Let $\mathcal S$: $q:G\to A$ be an affine extension. The
\emph{(enriched) category $\operatorname{q-}Q\sAmodgr$ of $q$--linearized
(or \emph{$\mathcal S$--linearized})  sheaves with graded morphisms} 
has as objects the $\mathcal 
S$--linearized sheaves and as morphisms the functor of graded
morphisms of $\mathcal S$--linearized sheaves.
\end{defn}

We finish this section by showing the relationship between the
concepts of homogeneous  and $\mathcal S$--linearized shaves.

  \begin{lem}
    \label{lem:glinhomog}
  Let $\mathcal S$:  $q:G\to A$ be an affine  extension and $\mathcal
  F$  a $\mathcal S$--linearized sheaf. Then $\mathcal F$ is homogeneous.   

Conversely, let $\mathcal F\in \hsAmod$ be a homogeneous sheaf. Then
$\mathcal F$ admits an $\Authbext (\mathcal F)$--linearization.
\end{lem}

\proof
If  $\ell\in A(T)$ then there exists a  fpqc morphism  $f:T'\to T$
and $g\in G(T')$ such that $q\smallcirc g= \ell\smallcirc f$.
Let  $a_{T}:G\times A\times T\to A\times T$ and $a_{T'}:G\times
A\times T'\to A\times T'$ be the actions induced by  
$a_q$. Then the linearization
$\Phi: a_q^*(\mathcal F)\cong p_2^*(\mathcal F) \in \mathcal O_{G
  \times A}$ induces a $G$--linearization
\[
\Psi: a_{T'}^*(\mathcal  F_{T'})\cong p_{23}^*(\mathcal F_{T'}) \in \mathcal O_{G
  \times A\times T'}
  \]
\emph{via} the $G$--equivariant morphism $p_2\smallcirc (\id_A\times
f): A\times
T'\to A$ --- here we use  remarks \ref{rem:indlin} (2) and
\ref{rem:linearizations}(3)).

Now consider the following commutative diagram of $T'$--schemes:
\begin{equation}
  \label{eqn:forlin}
\raisebox{8.5ex}{  {\xymatrixrowsep{1.5pc}\xymatrixcolsep{3pc}
  \xymatrix{
& & & A\times T'\ar[dd]^{t_{\ell\circ f}}\\
    A\times
    T'\ar[rr]^(0.52){(g \circ p_2, \id_{A\times T'})}
    \ar@/^0.75pc/[rrru]^{\id_{A\times T'}}\ar@/_0.75pc/[rrrd]_{t_{\ell\circ
       f}} &  & G\times
    A\times T'\ar[ru]^{p_{23}}\ar[rd]_{a_{T'}} & \\
   & &&A\times T'
}}}
  \end{equation}

  It follows from diagram \eqref{eqn:forlin} that
  $\Psi:a_{T'}^*\mathcal F_{T'}\to p_{23}^*\mathcal F_{T'}$ induces an
  isomorphism
  \[
      \psi: (g \smallcirc p_2, \id_{A\times T'})^*a_{T'}^*\mathcal F_{T'} \cong
  t_{a\circ\sigma}^* \mathcal F_{T'}\to (g \smallcirc p_2,
  \id_{A\times T'})^*p_{23}^*\mathcal F_{T'} \cong \mathcal F_{T'}.
\]

 It follows that 
  $(\psi, \ell \smallcirc f)\in \Autgr(\mathcal F)(T')$ and is such that 
  $d(\psi, \ell\smallcirc f)=\ell\smallcirc f$. We deduce that $d$ is
  surjective in the fpqc topology, and therefore $\mathcal F$ is
  homogeneous.

In order to prove the converse,   assume  that $\mathcal F\in Q\sAmod$
is homogeneous. Then $\Authbext (\mathcal F)$ is an affine extension
of $A$. In particular, $d:\Autgr(\mathcal F)\to A$ is a fpqc
morphism, and it follows from the homogeneity applied to the $\Autgr(\mathcal
  F)$--point $d$ that there exists an isomorphism $\varphi :
t_d^*\mathcal F_{\Autgr(\mathcal F)}\cong \mathcal F_{\Autgr(\mathcal
  F)}$.

Let $\sigma: \Autgr(\mathcal F)\times A  \to A\times \Autgr(\mathcal
F)$ be the switch morphism. Then
$a_d=p_1\smallcirc t_d\smallcirc \sigma:  \Autgr(\mathcal F)\times
A\to A$. Since  $\mathcal F_{\Autgr(\mathcal
  F)}=p_1^*\mathcal F\in QA_{\Autgr(\mathcal F)}\operatorname{-mod}$,
it follows that $\Psi= \sigma^*\varphi: 
a_d^*\mathcal F\to p_2^*\mathcal F\in Q\bigl(\Autgr(\mathcal F)\times
A)\operatorname{-mod}$ is the seeked $d$--linearization.\qed

\subsection{The category of sheaf representations of an affine  extension}\ %
\label{sect:repassheaves}

\begin{defn}
  \label{defi:Slinsheaf}
Let $\mathcal S$: $q:G\to A$ be an affine  extension of the abelian variety
$A$.

\noindent (1) The \emph{category $\ShRepgr{q}$ (or $\ShRepgr{\mathcal
    S}$)  of sheaf 
  representations of $\mathcal S$} is the full subcategory of
$Q\hsAmodgr(\mathcal S)$  with objects the  quasi-coherent, flat $\mathcal S$--linearized 
sheaves --- recall that,  by Lemma \ref{lem:glinhomog}), the $\mathcal
S$--linearized sheaves are homogeneous.

\noindent (2) We denote as $\ShRepgrfin{q} \subset \ShRepgr{\mathcal
  S}$ the full subcategory of coherent, flat (necessarily locally
  free)  sheaves --- in other words, the objects of
  $\ShRepgrfin{q}$ are the sheaf
representations that are locally free, of finite rank. We
will denote also $\ShRepgrfin{\mathcal S}=\ShRepgrfin{q}$.

  \noindent (3) The subcategory $\ShRepO{q}\subset \ShRepgr{\mathcal
    S}$ is defined by taking as  morphisms the pairs
$(f, 0)\in
 \Hom_{\ShRepgr{\mathcal S}}(\mathcal S)(\mathcal F,\mathcal G)$
  belonging to $\HomO(\mathcal
  F,\mathcal G)$. We will also denote $\ShRepO{\mathcal S}=\ShRepO{q}$.

  \noindent (4) The subcategory $
\ShRepOfin{q}=\ShRepOfin{\mathcal S}\subset \ShRepO{q}$ is defined as the
full subcategory of coherent sheaf representations (i.e.~locally free
of finite rank) of $\mathcal S$
with morphisms the pairs $(f, 0)$.   Notice that $
\ShRepOfin{q}$ can be seen also as a subcategory of
$\ShRepgrfin{q} $. 
\end{defn}

\begin{defn}
  \label {defi:rationalsheafrep}
  A $\mathcal S$--sheaf representation $\mathcal F$ is \emph{rational}
if there exists a filtered system of coherent subrepresentations
$\mathcal F_\alpha\subset \mathcal F_\beta\subset \mathcal F$, of
finite rank $n_\alpha$, such that $\mathcal F \cong
\operatorname{colim}_\alpha \mathcal F_\alpha$.
  In relation with the
categories considered in Definition \ref{defi:Slinsheaf} when we
restrict the objects to the rational sheaves, we add the prefix $r$ to the notations
e.g. we write $\ShRepgr{rq}$, etc.
\end{defn}

\begin{thm}
\label{thm:ratalg1}
Let $\mathcal S$:  $q:G\to A$ be an affine extension. 
Then the
equivalence of categories $\VBSpecgr : C_{\operatorname{lf}}\sAmodgr\to \HVBG(A)$ (see
Lemma \ref{lem:vbspecgr})
induces  an equivalence 
$\ShRepgrfin{\mathcal S} \cong \Rep(\mathcal S)$.
\end{thm}
\pf
By Remark \ref{rem:linearizations}, we have that $\VBSpec
C_{\operatorname{lf}}\sAmod\to \HVB_0(A)$ induces an equivalence
$\ShRepOfin{\mathcal S} \cong \RepO(\mathcal S)$. Hence, we only need
to show that if $\mathcal F,\mathcal F'\in \ShRepOfin{\mathcal S}$ are
$q$--linearized sheaves, then $\VBSpecgr :\Homgr(\mathcal F,\mathcal
F')\to \Homgr\bigl(\VBSpec(\mathcal F), \VBSpec(\mathcal F')\bigr)$
restricts to a bijection $\Hom_{\ShRepgrfin{\mathcal S}}(\mathcal F,\mathcal
F') \to \Hom_{\Rep(\mathcal S)} \bigl(\VBSpec(\mathcal F),
\VBSpec(\mathcal F')\bigr)$.

Denote $(\pi: F\to A):= \VBSpec(\mathcal F)$ and $(\pi': F'\to A):=
\VBSpec(\mathcal F')$. Let   $T$ be a $\Bbbk$--scheme and  $(f,\ell)\in 
\Homgr( F, F')(T)$. Let $a_F: G\wtimes F\to F$ and $a_{F'}:G\wtimes
F'\to F'$ be the $q$-actions
on $\VBSpec(\mathcal F)$ and $\VBSpec(\mathcal F')$ induced by
$\Phi_{\mathcal F}$ and $\Phi_{\mathcal F'}$, the
$q$--linearizations on $\mathcal F$ and $\mathcal F'$
respectively. Consider the  morphisms $a_{q,T}=a_q\times
\id_T:G\times A_T=G\times A\times T\to A_T=A\times T$, $a_{F,T}=a_1\times
\id_T:G\times F_T=G\times F\times T\to A_T$ and
$\pi_T=\pi\times\id_T: F_T\to A_T$.
Then, from the  cartesian diagram of $T$--schemes
\[
\xymatrix{
G\times F_T \ar[d]_{\id_G\times \pi_T}
\ar[rr]^-{a_{F,T}}& & F_T
\ar[d]^{\pi_T}
\\
G\times A_T \ar[rr]^-{a_{q,T}} & &A_T
}
\]
where $\pi_T=\pi\times \id_T$, we deduce that   $(f,\ell)$ is $G$--equivariant if
the following  diagram of $T$-schemes is commutative:
\begin{equation}
  \label{eqn:paralin}
\raisebox{5ex}{\xymatrixcolsep{3pc} 
\xymatrix{ 
G\times F_T
\ar[d]_-{a_{F,T}}
\ar[r]^-{\id_G\times f}&G\times F'_T
\ar[d]^{a_{F',T}}  
\\
F_T\ar[r]^-{f}\ar[d]_{\pi_T} &  F'_T\ar[d]^{\pi_T}\\
A_T\ar[r]^{t_\ell}&A_T}}
\end{equation}
where we consider the $T$-structure given by projection on the last coordinate.

Taking into account that $ t_\ell\smallcirc (\pi\times
\id_T)=(\pi'\times \id_T)\smallcirc f$,
 we deduce
that  the commutativity of Diagram \eqref{eqn:paralin} is equivalent
to the commutativity of the  diagram:
\[
  \xymatrix{
   (a_q\times\id_T)^*t_\ell^*(\mathcal F_T)
   \ar[rrr]^-{(\id_G\times t_\ell)^*p_{12}^*(\Phi_{\mathcal F}^\vee)}
   \ar[d]_{(a_q \times
   \id_T)^*(\widetilde{f})}&&& p_{23}^* t_\ell^*(\mathcal F_T)\ar[d]^{p_{23}^*(\widetilde{f})}\\ 
   (a_q\times \id_T)^*(\mathcal
   F'_T)\ar[rrr]^-{p_{12}^*(\Phi_{\mathcal F'}^\vee)}&& &p_{23}^* (\mathcal F'_T)
 }
\]
where $\VBSpec(\widetilde{f},\ell)=(f,\ell)$. In other words,
$(\widetilde{f},\ell)\in \Hom_{\ShRepgrfin{\mathcal S}}(\mathcal
F,\mathcal F')$.
\qed

Combining Proposition \ref{prop:equicomodrep0} and Theorem
\ref{thm:ratalg1}, we get the following:

\begin{cor}
\label{cor:differenteqiv}
  Let $\mathcal S$: $q:G\to A$ be an affine extension and $\mathcal H_q$
its associated Hopf sheaf. Then  the equivalences $\VBSpec: \Comodfin{\mathcal
  H_q}\to \RepO(\mathcal S)$ and $\VBSpecgr :
C_{\operatorname{lf}}\hsAmodgr\to \HVBG(A)$ induce equivalences
\[
  \Comodfin{\mathcal
  H_q}\cong \Rep(\mathcal S) \cong \ShRepOfin{\mathcal S}. \hfill  \qed
\]
\end{cor}

 We would like to extend now the equivalence  $\Comodfin{\mathcal
  H_q}\cong \ShRepOfin{q} $  to the category of rational
sheaf representations. First, we need to establish the 
notion of rational comodule.

\begin{defn}
\label {defi:rationalcomod} 
Let$\mathcal H$ be a Hopf sheaf on the abelian variety $A$. A
$\mathcal H$--comodule $\mathcal F$ is
\emph{rational}  if there exists a filtered system of coherent flat
sub-comodules  $\mathcal F_\alpha\subset \mathcal F_\beta\subset
\mathcal F$, of
finite rank $n_\alpha$, such that $\mathcal F$ is the filtered union
of the subsheaves $\mathcal F_\alpha$; that is, $\mathcal F=\operatorname{colim}_\alpha \mathcal
F_\alpha$.

Notice that since the limit of flat modules is flat, a rational
$\mathcal H$--comodule is necessarily flat.

We denote by $\RatAcomodgr{\mathcal H}\subset \Acomodgr{\mathcal H}$ the
full category of rational $\mathcal H$--comodules, and  by $\RatAcomodO{\mathcal H}\subset
\RatAcomodgr{\mathcal H}$ the wide subcategory with the same objects and
morphisms $\operatorname{Ker}d$, where $d: \Acomodgr{\mathcal H}\to A$ is the
degree map. In other words, 
\[
  \Hom_{\RatAcomodO{\mathcal
    H}}(\mathcal F,\mathcal G)(T)=\bigl\{(f,0)\mathrel{:} (f,0)\in \Hom_{\Acomodgr{\mathcal
    H}}(\mathcal F,\mathcal G)(T)\bigr\}.
\]
\end{defn}

\begin{prop}
\label{prop:ratshafpobre}
Let $\mathcal S$: $q:G\to A$ be an affine  extension and $\mathcal
H_q$ the associated Hopf  sheaf. If $(\mathcal F,\Phi_{\mathcal F})\in \ShRepO{q}$,
then $\mathcal F$ admits a structure of (rational) $\mathcal H_q$-comodule.

Conversely, if $(\mathcal F,\chi_{\mathcal F})\in \RatAcomodO{\mathcal
  H_q}$ then  $\mathcal F$ admits a structure of (rational) $\mathcal
S$-- sheaf representation.
\end{prop}

\pf
Let $(\mathcal F,\Phi_{\mathcal F})\in \ShRepO{q}$, and
$(\mathcal F_\alpha,\phi_\alpha=\phi|_{_{a_q^*{\mathcal F_\alpha}}})$
be a filtered system of 
coherent $\mathcal S$--sheaf subrepresentations with $\operatorname{colim}
\mathcal F_\alpha=\cup \mathcal
F_\alpha=\mathcal F$ and $(a_q)_\alpha: G\times \Spec\bigl(\mathcal
S(\mathcal F_\alpha^\vee)\bigr)\to \Spec\bigl(\mathcal
S(\mathcal F_\alpha^\vee)\bigr) $ the associated $\mathcal S$--action
(see Theorem \ref{thm:ratalg1}). By Proposition
\ref{prop:equicomodrep0}, $(a_q)_\alpha$ induces a structure of  $\mathcal H_q$--comodule   $\chi_{\mathcal
  F_\alpha} : 
F_\alpha\to \mathcal H_q\wsboxtimes \mathcal F_\alpha\subset \mathcal
H_q\wsboxtimes \mathcal F$. Since this 
association is of functorial nature, it follows that is $\mathcal 
F_\alpha\subset \mathcal F_\beta$, then $\chi_{\mathcal
  F_\alpha}= \chi_{\mathcal  F_\beta}|_{_{\mathcal F_\alpha}}$. It
follows that the family $\chi_{\mathcal F_\alpha}$ induces a structure
of $\mathcal H_q$--comodule  $\chi_{\mathcal F}:\mathcal F=\operatorname{colim}
\mathcal F_\alpha\to \mathcal H_q\wsboxtimes \mathcal F$, such that
$\chi_{\mathcal F}|_{_{\mathcal F_\alpha}}= \chi_{\mathcal F_\alpha}$.

The proof of the converse is similar and therefore is omitted.
\qed

\end{document}